\documentclass{amsart}[12pt]

\usepackage{amsmath}
\usepackage{amssymb}
\usepackage{amsthm}
\usepackage{amsfonts}
\usepackage{calc}
\usepackage{cobordism}
\usetikzlibrary{fit}

\newtheorem{theorem}{Theorem}
\newtheorem{corollary}[theorem]{Corollary}
\newtheorem{lemma}[theorem]{Lemma}

\theoremstyle{definition}
\newtheorem{definition}[theorem]{Definition}
\newtheorem{remark}[theorem]{Remark}
\newtheorem{example}[theorem]{Example}

\DeclareMathOperator{\id}{id}
\DeclareMathOperator{\Hom}{Hom}

\usepackage[font=scriptsize,labelfont=bf]{subcaption}
\usepackage[font=scriptsize,labelfont=bf,margin=3cm]{caption}
\usepackage{pdflscape}

\usepackage[colorlinks, linkcolor=red, citecolor=blue, backref]{hyperref}

\usepackage{tikz}
\usetikzlibrary{arrows.meta}

\usepackage{graphicx}
\graphicspath{ {images/} }

\newcommand{\lab}{\left\langle}
\newcommand{\rab}{\right\rangle}

\def\inv{{\hspace{0.5pt}\text{-} \hspace{-0.5pt}1}} 

\newcommand{\Rarrow}[1]{
    \begin{aligned}
	    \begin{tikzpicture}
		    \node [scale=0.9] at (0.5,0.75) {\ensuremath{#1}};
		    \draw [double equal sign distance, -Implies] (0,0.5) -- (1,0.5);
	    \end{tikzpicture}
	\end{aligned}
		                }

\newcommand{\RLarrow}[2]{
    \begin{aligned}
	    \begin{tikzpicture}
		    \node [scale=0.9] at (0.5,0.75) {\ensuremath{#1}};
		    \draw [double equal sign distance, -Implies] (0,0.5) -- (1,0.5);
		    \draw [double equal sign distance, -Implies] (1,0) -- (0,0);
		    \node [scale=0.9] at (0.5,-0.35) {\ensuremath{#2}};
	    \end{tikzpicture}
	\end{aligned}
		                }

\def\quotient#1#2{
    \raise1ex\hbox{$#1$}\Big/\lower1ex\hbox{$#2$}   
}                                                   

\usepackage{mathtools}                                                  
\makeatletter                                                           
\def\slashedarrowfill@#1#2#3#4#5{
  $\m@th\thickmuskip0mu\medmuskip\thickmuskip\thinmuskip\thickmuskip    
   \relax#5#1\mkern-7mu
   \cleaders\hbox{$#5\mkern-2mu#2\mkern-2mu$}\hfill                     
   \mathclap{#3}\mathclap{#2}
   \cleaders\hbox{$#5\mkern-2mu#2\mkern-2mu$}\hfill                     
   \mkern-7mu#4$
}                                                                       
\def\rightslashedarrowfill@{
  \slashedarrowfill@\relbar\relbar\mapstochar\rightarrow}               
\newcommand\xslashedrightarrow[2][]{
  \ext@arrow 0055{\rightslashedarrowfill@}{#1}{#2}}                     
\newcommand{\profarrow}{\xslashedrightarrow{}}     						
\makeatother                                                            

\tikzset{tiny label/.style={
        shape=circle,
        text=black,
        draw=black,
        line width=1,
        fill=white,
        inner sep=0.05cm,
        minimum width=0.25cm,
        font=\small}}

\tikzset{tiny red label/.style={
        shape=circle,
        text=red,
        draw=red,
        line width=0.5,
        fill=white,
        inner sep=0.05cm,
        minimum width=0.25cm,
        font=\small}}

\newenvironment{tz}{\begin{aligned} \begin{tikzpicture}}{\end{tikzpicture} \end{aligned}}

\newcommand{\CC}{\ensuremath{\mathbb{C}}}

\newcommand{\cC}{\mathcal{C}}
\newcommand{\cE}{\mathcal{E}}
\newcommand{\cD}{\mathcal{D}}

\newcommand{\interchangor}{\varphi}
\newcommand{\phileft}{\phi_l}
\newcommand{\phiright}{\phi_r}

\tikzset{double arrow scope/.style={every path/.style={double equal sign distance, -Implies}}}

\DeclareMathOperator{\Aut}{Aut}
\DeclareMathOperator{\II}{II}

\DeclareMathOperator{\Curves}{Curves}
\DeclareMathOperator{\Dim}{Dim}

\newcommand{\BV}{\ensuremath{\text{BV}(S^1)}}

\colorlet{LightGray}{black!15}
\colorlet{LightRed}{red!15}
\colorlet{LightBlue}{blue!30}
\colorlet{LightOrange}{orange!15}
\colorlet{DarkGreen}{green!50!black!50!}
\colorlet{LightGreen}{green!45}
\colorlet{VLG}{black!5}
\colorlet{VLB}{black!10}

\newcommand{\BlackDot}{%
\begin{tikzpicture}[scale=0.75]
\draw [fill] (0,0) circle (1.5pt);
\end{tikzpicture}%
}

\newcommand{\Bord}{\ensuremath{\mathbf{Bord}^{or}_{123}}}
\newcommand{\Bordn}{\ensuremath{\text{Bord}^{or}_{n}}}
\newcommand{\Prof}{\ensuremath{\mathbf{Prof}}}
\newcommand{\Vect}{\ensuremath{\text{Vect}}}
\newcommand{\TVect}{\ensuremath{\mathbf{2Vect}}}
\newcommand{\KV}{\ensuremath{\mathbf{KV}}}


\newcommand{\sgn}{\ensuremath{\Sigma_{g,n}}}
\newcommand{\scup}{\ensuremath{\Sigma_{0,1}}}
\newcommand{\spants}{\ensuremath{\Sigma_{0,3}}}
\newcommand{\shandle}{\ensuremath{\Sigma_{1,2}}}

\tikzset{bot=true}

\newcommand{\cobscale}{1}
\newcommand{\gscalecobordisms}[1]{\renewcommand{\cobscale}{#1}}

\tikzset{
    partial ellipse/.style args={#1:#2:#3}{
        insert path={+ (#1:#3) arc (#1:#2:#3)}
    }
}

\def\gpdashlength{0.4*\pgflinewidth}                          
\tikzset{                                                     
         gp path/.style={dash pattern=on 7.5*\gpdashlength    
                                     off 7.5*\gpdashlength}}  

\newcommand{\tinypants}{\hspace{0.2mm}\raisebox{-0.75mm}{\ensuremath{  
\begin{tikzpicture}[scale=0.1]                                         
		\draw (2.5,3) ellipse (1cm and 0.25cm);                        
		\draw [gp path] (1,0) [partial ellipse=0:180:1cm and 0.25cm];  
		\draw (1,0) [partial ellipse=180:360:1cm and 0.25cm];          
		\draw [gp path] (4,0) [partial ellipse=0:180:1cm and 0.25cm];  
		\draw (4,0) [partial ellipse=180:360:1cm and 0.25cm];          
		\draw (0,0) to [out=90, in=270, looseness=1] (1.5,3);          
		\draw (5,0) to [out=90, in=270, looseness=1] (3.5,3);          
		\draw (2,0) to [out=90, in=90, looseness=2.5] (3,0);           
\end{tikzpicture}                                                      
}                                                                      
}\hspace{-1mm}                                                         
}                                                                      

\newcommand{\tinycopants}{\hspace{0.2mm}\raisebox{-0.75mm}{\ensuremath{
\begin{tikzpicture}[scale=0.1]                                         
		\draw (1,3) ellipse (1cm and 0.25cm);                          
		\draw (4,3) ellipse (1cm and 0.25cm);                          
		\draw [gp path] (2.5,0) [partial ellipse=0:180:1cm and 0.25cm];
		\draw (2.5,0) [partial ellipse=180:360:1cm and 0.25cm];        
		\draw (1.5,0) to [out=90, in=270, looseness=1] (0,3);          
		\draw (3.5,0) to [out=90, in=270, looseness=1] (5,3);          
		\draw (2,3) to [out=270, in=270, looseness=2.5] (3,3);         
\end{tikzpicture}                                                      
}                                                                      
}\hspace{-1mm}                                                         
}                                                                      

\newcommand{\tinycap}{\hspace{0.2mm}\raisebox{-0.75mm}{\ensuremath{    
\begin{tikzpicture}[scale=0.18]                                        
		\draw (0,0) to [out=90, in=90, looseness=3] (2,0);             
		\draw [gp path] (1,0) [partial ellipse=0:180:1cm and 0.25cm];  
		\draw (1,0) [partial ellipse=180:360:1cm and 0.25cm];          
\end{tikzpicture}                                                      
}                                                                      
}\hspace{-1mm}                                                         
}                                                                      

\newcommand{\tinycup}{\hspace{0.2mm}\raisebox{-1.9mm}{\ensuremath{     
\begin{tikzpicture}[scale=0.18]                                        
		\draw (0,3) to [out=270, in=270, looseness=3] (2,3);           
		\draw (1,3) ellipse (1cm and 0.25cm);                          
\end{tikzpicture}                                                      
}                                                                      
}\hspace{-1mm}                                                         
}                                                                      

\newcommand{\tinycyl}{\hspace{0.2mm}\raisebox{-0.9mm}{\ensuremath{     
\begin{tikzpicture}[scale=0.1]                                         
		\draw (1,3) ellipse (1cm and 0.25cm);                          
		\draw (0,3) -- (0,0);                                          
		\draw (2,3) -- (2,0);                                          
		\draw [gp path] (1,0) [partial ellipse=0:180:1cm and 0.25cm];  
		\draw (1,0) [partial ellipse=180:360:1cm and 0.25cm];          
\end{tikzpicture}                                                      
}                                                                      
}\hspace{-1mm}                                                         
}                                                                      
\newcommand{\cylwithtopcup}{
\begin{tikzpicture}[scale=\cobscale]                                                
		\draw (1,3) ellipse (1cm and 0.25cm);                                       
		\draw (0,3) -- (0,0);                                                       
		\draw (2,3) -- (2,0);                                                       
		\draw [gp path] (1,0) [partial ellipse=0:180:1cm and 0.25cm];               
		\draw (1,0) [partial ellipse=180:360:1cm and 0.25cm];                       
		\draw [blue, thick] (0.65,2.75) to [out=270,in=270,looseness=2] (1.35,2.75);       
\end{tikzpicture}                                                                   
}                                                                                   

\begin{document}

\pagestyle{plain}


\title{Extended TQFTs via Generators and Relations I: The Extended Toric Code.}

\author{Bruce Bartlett}
\address{Mathematics Division, Stellenbosch University, South Africa}
\email{bbartlett@sun.ac.za}

\author{Gerrit Goosen}
\address{Department of Mathematics and Applied Mathematics, School of Mathematical and Statistical Sciences,
North-West University, Potchefstroom, South Africa}
\email{Gerrit.Goosen@nwu.ac.za}

\thanks{This work was partially supported by NRF South Africa.}
\subjclass[2010]{Primary 57R56, Secondary 57M27}
\keywords{Topological Quantum Field Theory, Quantum Topology, String-net}

\begin{abstract}
In his PhD thesis \cite{goosen2018oriented}, Goosen combined the string-net and the generators-and-relations formalisms for arbitrary once-extended 3-dimensional TQFTs. In this paper we work this out in detail for the simplest non-trivial example, where the underlying spherical fusion category is the category of $\mathbb{Z}/2\mathbb{Z}$-graded vector spaces. This allows us to give an elementary string-net description of the linear maps associated to 3-dimensional bordisms. The string-net formalism also simplifies the description of the mapping class group action in the resulting TQFT. We conclude the paper by performing some example calculations from this viewpoint.

\end{abstract}

\maketitle

\section{Introduction}


The purpose of this paper is to explain the string-net description of the extended Turaev-Viro topological quantum field theory from \cite{goosen2018oriented} in the simplest non-trivial case, namely the {\em extended toric code}. 

\subsection*{Topological quantum field theories}

In mathematics, an $n$-dimensional {\em topological quantum field theory} (TQFT) is defined (following Atiyah and Segal) as a representation of the $n$-dimensional bordism category $\Bordn$. That is, an $n$-dimensional topological quantum field theory is a symmetric monoidal functor
\begin{equation} \label{TQFT_functor}
 Z : \Bordn \rightarrow \Vect
\end{equation}
where $\Bordn$ is the symmetric monoidal category whose objects are closed oriented $(n-1)$-manifolds and whose morphisms are equivalence classes of $n$-dimensional cobordisms between them, and $\Vect$ is the symmetric monoidal category  of finite-dimensional vector spaces and linear maps. 

Two-dimensional TQFTs are classified by commutative Frobenius algebras, while the study of three-dimensional TQFTs is particularly rich, bringing together various topics in low-dimensional topology and representation theory such as the Jones polynomial, Chern-Simons gauge theory, loop group representations, quantum groups, and affine Lie algebras. 

\subsection*{Extended topological quantum field theories}

A $n$-dimensional TQFT defined as in \eqref{TQFT_functor} is only required to respect gluing along $(n-1)$-dimensional boundaries. A TQFT which additionally respects gluing along $(n-2)$-dimensional boundaries (and possibly also higher codimensional boundaries) is called an {\em extended TQFT}. Different authors have defined extended TQFTs in slightly different ways. One common approach is to express the extra gluing laws in the language of higher categories, which is the approach we adopt in this paper. Specifically, by an {\em extended 3-dimensional TQFT} (also known as a 1-2-3 TQFT) we will mean a symmetric monoidal bifunctor
\begin{equation} \label{extendedTQFTdefn}
 Z : \Bord \rightarrow \Prof.
\end{equation}
Here $\Bord$ is the symmetric monoidal bicategory whose objects are closed oriented $1$-dimensional manifolds, whose $1$-morphisms are $2$-dimensional compact oriented cobordisms, and whose 2-morphisms are diffeomorphism classes of 3-dimensio\-nal compact oriented cobordisms (i.e. $3$-manifolds with corners, see \cite{CSPthesis}), while $\Prof$ is the symmetric monoidal bicategory whose objects are linear categories, morphisms are profunctors  and 2-morphisms are natural transformations (see Section \ref{modular_presentation_section}).  Note that $\Prof$ is equivalent to the bicategory $\TVect$ of 2-vector spaces (see Lemma \ref{Profequals2Vect}), which is more commonly used as the target bicategory for extended TQFTs. Using $\Prof$ as target is more convenient in our setting as it allows the graphical calculus to employ only string-nets labelled by the original spherical fusion category $A$, thereby avoiding Cauchy completion and string-nets labelled by the Drinfeld double $Z(A)$ as in \cite{kirillov2011string} (see Remark \ref{prof_remark}).

\subsection*{The Turaev-Viro model}

For every {\em spherical fusion category} $C$ (a finitely semimsimple tensor category with coherent duals, see \cite{ENO}), there is an associated 3-dimensional TQFT $Z^{TV}_C$ called the {\em Turaev-Viro} model. (In fact, Turaev and Viro only constructed it for the special case where $C = \text{Rep}\, U_q sl_2(\mathbb{C})$ at a root of unity $q$, while Barrett and Westbury generalized the construction to an arbitrary spherical fusion category $C$). The vector space $Z^{TV}_C(\Sigma)$ associated to closed surface $\Sigma$ in the  model is defined by first choosing a triangulation $\Delta$ for $\Sigma$, which gives rise to a vector space $W_{\Delta}$ and a projector $P$ acting on $W_{\Delta}$. The vector space $Z^{TV}_C(\Sigma)$ is the image of this projector, and is independent of the triangulation $\Delta$. In particular, when $G$ is a finite group, then we can take $C = \text{Vect}[G]$, the category of $G$-graded vector spaces. The resulting 3-dimensional TQFT $Z^{TV}_{\text{Vect}[G]}$ is known to be isomorphic to the Dijkgraaf-Witten model (which builds a TQFT by counting principal $G$-bundles). 

\subsection*{The extended Turaev-Viro model}

In a series of papers \cite{balsamphd, BalKirI, BalII, BalIII}, as an important part of their proof that the Turaev-Viro model associated to a spherical fusion category $C$ is equivalent to the Reshetikhin-Turaev model associated to the Drinfeld centre $Z(C)$ (a different proof is given in \cite{turaev2010two}), Balsam and Kirillov gave a definition of the Turaev-Viro model as an {\em extended} TQFT. Note that their definition of an extended TQFT is not expressed in the language of higher categories and cannot be directly compared with Definition \ref{extendedTQFTdefn}, but it is nevertheless very similar.

\subsection*{The toric code}
The {\em toric code} is a lattice model for fault-tolerant topological quantum computation introduced by Kitaev \cite{Kit03}. Given a closed surface $\Sigma$ with a triangulation $\Delta$, one can define a Hamiltonian (a sum of mutually commuting projectors) acting on a large vector space $\mathcal{H}(\Sigma, \Delta)$ whose ground state space $V(\Sigma)$ (lowest eigenvalue eigenspace) is the `topologically protected subspace', which is resistant to errors and does not depend on the triangulation $\Delta$. Levin and Wen showed \cite{LW05} that this ground state space $V(\Sigma)$ of the toric code can be interpreted as the TQFT vector space $Z(\Sigma)$ where $Z$ is the 3-dimensional Turaev-Viro TQFT coming from the finite group $\mathbb{Z}/2\mathbb{Z}$. Moreover, they showed that $V(\Sigma)$ can be defined elegantly without the construction of a lattice, in the language of {\em string-nets}.

\subsection*{Kirillov's string-net model}

This relationship between the Kitaev lattice model, the Levin-Wen string-net model, and the Turaev-Viro TQFT associated to $\mathbb{Z}/2\mathbb{Z}$, was explored further by Kirillov in \cite{Kir11} (see also \cite{balsam2012kitaev}). In this paper a careful mathematical treatment of the Levin-Wen string-net formulation of the model was given, for an arbitrary spherical fusion category $C$, and its relationship to the Turaev-Viro TQFT vector spaces explored. In particular, Kirillov showed how to formulate the string-net model on a {\em surface with boundary} by imposing {\em boundary conditions} on the string-nets.  Moreover, he showed that the resulting vector space of string-nets on a surface with boundary is isomorphic to the vector spaces in the extended Turaev-Viro model, as defined in \cite{BalIII}. 

\subsection*{The action of cobordisms on string-nets} 

Although \cite{Kir11} described the vector spaces of string-nets associated to closed surfaces as well as surfaces with boundary, it did not describe how the linear maps associated to 3-dimensional cobordisms act on these vector spaces.   This is the problem that was investigated by Goosen \cite{goosen2018oriented} in his PhD thesis. By using the generators-and-relations description of a 1-2-3 TQFT from \cite{BDSPV2014, BDSPV2015}, Goosen was able to give explicit graphical rules for how the linear maps $Z(M)$ associated to 3-dimensional cobordisms $M$ act on string-nets. This gave, for the first time, a graphical description of the Turaev-Viro TQFT {\em completely in the language of string-nets}. In addition, this is a description of Turaev-Viro theory as an {\em extended} TQFT in the sense of \eqref{extendedTQFTdefn}.

\subsection*{Extended toric code via string-nets}

The purpose of this paper is to spell out the string-net description of the extended Turaev-Viro TQFT from \cite{goosen2018oriented} in the simplest non-trivial case. Namely, the finite group model associated to $G=\mathbb{Z}/2\mathbb{Z}$; that is, the {\em extended toric code}. It makes sense to work out the theory explicitly for the toric code as this is the example which motivated the entire formalism of string-nets, as explained above. In this example, the spherical fusion category is $\text{Vect}[\mathbb{Z}/2\mathbb{Z}]$, which has a very simple graphical calculus, making the linear maps associated to the generating 3-dimensional cobordisms particularly easy to understand (see Remark \ref{rem:pushforward}). 

\subsection*{Outline of paper}

In Section \ref{modular_presentation_section} we recall the generators-and-relations presentation of the oriented 1-2-3 bordism bicategory $\Bord$ from \cite{BDSPV2014}, and we describe explicitly what it means to formulate an extended TQFT in this language. In Section \ref{string_net_section} we specialize the general definition of the string-net space of a surface with boundary from \cite{Kir11} to the case of the toric code, resulting in an elementary graphical definition (see Definition \ref{defn_string_net_space}). In Section \ref{action_of_generators_section} we spell out explicitly, at the level of string-nets, what the generating objects, generating 1-morphisms and generating 2-morphisms of $\Bord$ are sent to in $\Prof$, and show that they satisfy the relations from \cite{BDSPV2014}. This proves that these assignments define an extended TQFT. Finally, in Section \ref{computations_section} we perform some example calculations to illustrate how this formalism works in practice. Namely, we compute the dimension of the string-net vector spaces, the mapping class group action on the torus and the quantum invariants of lens spaces, using the string-net formalism that we have defined in the preceding  sections.

\section{The Modular Presentation}

\label{modular_presentation_section}

The generators-and-relations description of 1-2-3 TQFTs, discovered by Bartlett, Douglas, Schommer-Pries, and Vicary \cite{BDSPV, BDSPV2, BDSPV2014, BDSPV2015} is one of our main ingredients, and is used throughout the remainder of the paper. In order to make the paper as self-contained as possible (as well as to establish notational conventions), we review their construction briefly here.\\

We recall some definitions. In what follows, we write $\Vect$ for the category of finite-dimensional complex vector spaces. 

\begin{definition} The bicategory $\Prof$ is defined as follows:
\begin{itemize}
	\item An object is a $\Vect$-enriched category $C$. 
	\item A morphism $F : C \profarrow D$ is a $\Vect$-enriched profunctor\footnote{We shall simply call them profunctors from now on, with values in $\Vect$ being understood.}, i.e. a $\Vect$-enriched functor $F : \cD^{op} \boxtimes \cC \to \Vect$, where $\boxtimes$ is the enriched tensor product of $\Vect$-enriched categories\footnote{The objects of $\cC \boxtimes \cD$ consist of pairs of objects $(c,d) \in \cC \times \cD$, and the morphism vector spaces are given by $\text{Hom}_{\cC \boxtimes \cD}((c,d),(c',d')) = \text{Hom}_{\cC}(c,c') \otimes_k \text{Hom}_{\cD}(d,d')$.}. For two profunctors $F : \cC \profarrow \cD$ and $G : \cD \profarrow \cE$, the composite
	
\begin{equation}
G \circ F : \cE^{op} \boxtimes \cC \to \Vect
\end{equation}

\noindent is defined by

\begin{equation} \label{eqn:prof_comp}
(G \circ F)(e,c) := \bigoplus\limits_{d \in \cD} \quotient{(G(e,d) \otimes F(d,c))}{\sim}
\end{equation}

\noindent where $\sim$ is the equivalence relation generated by the relation $g \cdot x \otimes f \sim g \otimes x \cdot f$ for all $g \in G(e,d)$, $f \in F(d',c)$, and $x \in \text{Hom}_{\cD}(d,d')$.
  \item The 2-morphisms in $\Prof$ are natural transformations between the associated $\Vect$-valued ordinary functors.
  \item The monoidal structure is given by the enriched tensor product and the symmetric monoidal structure is given by the canonical `swap' profunctors.
\end{itemize}
\end{definition}

\begin{definition} The bicategory $\Bord$ is defined\footnote{We only provide an intuitive definition of $\Bord$. For all the technical details, see \cite{CSPthesis}.} as follows:
\begin{itemize}
\item An object is a closed oriented $1$-manifold, i.e. a disjoint unions of a finite (possibly zero) number of circles.
\begin{equation}
\scalecobordisms{1}
\begin{tz}
\node[Cyl, top, height scale=0]  at (0,0) {};
\node[Cyl, top, height scale=0]  at (1.5,0) {};
\node[Cyl, top, height scale=0]  at (3,0) {};
\node at (4.5,0) {$\cdots$};
\node[Cyl, top, height scale=0]  at (6,0) {};
\end{tz}
\end{equation}
\item A $1$-morphism is a compact oriented $2$-dimensional cobordisms between the objects. For instance, the picture below is a $1$-morphism from one copy of $S^1$ to two copies of $S^1$.
\begin{equation}
\begin{tz} 
\node[Copants, top, anchor=belt] at (0,0) {};
\end{tz}
\end{equation}
\noindent Composition works by gluing the manifolds together in the obvious way.
\item A $2$-morphism is a $3$-manifold-with-corners which is a cobordism between the $1$-morphisms. For example, the $3$-manifold realizing the $2$-morphism
\begin{equation}
\begin{tz} 
\node[Cyl, tall, top, bot] (A) at (0,0) {};
\node[Cyl, tall, top, bot] (B) at (2*\cobwidth, 0) {};
\end{tz}
\quad
\Longrightarrow
\quad
\begin{tz} 
\node[Pants, bot] (A) at (0,0) {};
\node[Copants, top, bot, anchor=belt] at (A.belt) {};
\end{tz}
\end{equation}
\noindent may be visualized as the ``fusing together'' of the two cylinders over time.
\item The monoidal structure is given by disjoint union and the symmetric monoidal structure is given by the canonical `swap' cobordisms.
\end{itemize}
\end{definition}

\noindent In this paper, we adopt the following `profunctor' definition of an oriented 1-2-3 TQFT (see Remark \ref{prof_remark}). 

\begin{definition}
An oriented 1-2-3 TQFT is a symmetric monoidal pseudofunctor%
\begin{equation}%
Z : \Bord \to \Prof.
\end{equation}
\end{definition}
Note that an alternative approach would be to define an oriented 1-2-3 TQFT as a symmetric monoidal pseudofunctor
\[
  Z : \Bord \to \TVect
\]
where $\TVect$ is the symmetric monoidal bicategory which has Cauchy-complete $\Vect$-enriched categories as objects, $\Vect$-functors as morphisms, and natural transformations as 2-morphisms\footnote{See for instance \cite{tillmann1998sscr, BDSPV2015}. Note that the collection $\KV$ of Kapranov-Voevodsky 2-vector spaces \cite{kapranov1991braided} (i.e. categories equivalent to $\Vect^n$ for some $n$) forms a full monoidal sub-bicategory of $\TVect$. The properties of $\Bord$ imply that $Z(S^1)$ must be finitely semisimple \cite{tillmann1998sscr, BDSPV2015}, so it makes no difference if one uses $\KV$ or the bigger bicategory $\TVect$ as the target.}.

However, the two definitions are equivalent by the following lemma. Given a Vect-enriched linear category $C$, we write $\hat{C}$ for its Cauchy completion (see \cite{borceux1994handbook}).
\begin{lemma} \label{Profequals2Vect} There is a canonical equivalence of symmetric monoidal bicategories $\TVect \simeq \Prof$. 
\end{lemma}
\begin{proof} The functor
\begin{align*}
    \Prof & \rightarrow \TVect     \\
    C &\mapsto \hat{C}
\end{align*}
yields the desired equivalence. This follows from
\begin{align*}  \Hom_{\Prof} (C, D) &\simeq \Hom_{\Prof}(\hat{C}, \hat{D}) && \text{(because $C \simeq \hat{C}$ and $D \simeq \hat{D}$ in $\Prof$)} \\
   & \simeq \Hom_{\TVect} (\hat{C}, \hat{D}) && \text{(because $\hat{D}$ is Cauchy complete)} 
\end{align*}
where the last equivalence uses the fact that enriching over {\em finite}-dimensional vector spaces means every profunctor in $\Prof$ has a right adjoint (by taking linear duals), and hence is representable as a functor \cite[Theorem 7.9.3]{borceux1994handbook}. 
\end{proof}
 
\begin{remark} \label{prof_remark}
 The reason we adopt $\Prof$ as target is a geometric one. In this paper, the category that the TQFT $Z$ will assign to the circle $S^1$ will be the `category of boundary conditions' $\BV$ (see \cite{kirillov2011string}), which can be phrased purely in terms of string nets (see Definition \ref{category_of_boundary_values}). However, $\BV$ is not a Cauchy complete category. In \cite{kirillov2011string}, Kirillov passes to the Cauchy completion $\widehat{\BV}$, and shows that in general, $\widehat{\BV} \simeq Z(A)$ where $A$ is the spherical fusion category used as input and $Z(A)$ is its Drinfeld centre\footnote{Warning: in this paper $\widehat{\BV}$ is the Cauchy completion of $\BV$, whereas in \cite{kirillov2011string} the notation works the other way.} (in our case, $A = \Vect[\mathbb{Z}/2\mathbb{Z}]$). This makes contact with the standard approach to extended 3-dimensional TQFT, where the category assigned to $S^1$ is $Z(A)$. 

However, passing to the Cauchy completion is less geometric as it requires us to complicate the string diagram graphical calculus by incorporating projectors on the boundary circles (this was the approach of \cite{goosen2018oriented}). By using profunctors, we can remain entirely in the world of string nets and avoid the Drinfeld double entirely. 
\end{remark}

\noindent Since $\Bord$ is very complicated, finding symmetric monoidal pseudofunctors out of it is a difficult problem. In this paper, our approach to simplify the problem shall be to make use of a presentation for $G$; a finite collection of data, generators and relations, from which $\Bord$ may be reconstructed.

\scalecobordisms{0.4}

\begin{theorem} \cite{BDSPV2014} The 3-dimensional oriented bordism bicategory $\Bord$ admits the following presentation.
\begin{itemize}
\item Generating object:

\[
\begin{tz}
        \node[Cyl, top, height scale=0]  at (0,0) {};
\end{tz}
\]
 
 \item Generating 1\-morphisms:
    \[
    \begin{tz}
        \node[Pants, top, bot] (A) at (0,0) {};
        \node[Copants, top, bot] (B) at (2,0) {};
        \node[Cup, top] (C) at (4,0.1) {};
        \node[Cap, bot] (D) at (6,-0.1) {};
    \end{tz}
    \]
 \item Invertible generating 2-morphisms:
    \begin{equation}
    \begin{tz}
        \node[Pants, top, bot, wide] (A) at (0,0) {};
        \node[Pants,  bot, anchor=belt] (B) at (A.leftleg) {};    
        \node[Cyl, bot, anchor=top] at (A.rightleg) {}; 
    \end{tz}
    \quad
    \RLarrow{\alpha}{\alpha^\inv}
    \quad
    \begin{tz}
        \node[Pants, top, bot, wide] (A) at (0,0) {};
        \node[Pants,  bot, anchor=belt] (B) at (A.rightleg) {};    
        \node[Cyl, bot, anchor=top] at (A.leftleg) {}; 
    \end{tz} 
    \qquad
    \qquad
    \begin{tz}
        \node[Pants, top, bot] (A) at (0,0) {};
        \node[Cyl, bot, anchor=top] at (A.leftleg) {};
        \node[Cup] at (A.rightleg) {};  
    \end{tz}
    \quad
    \RLarrow{\rho}{\rho^\inv}
    \quad
    \begin{tz}
        \node[Cyl, bot, top, tall] at (0,0) {};
    \end{tz}
    \quad
    \RLarrow{\lambda^\inv}{\lambda}
    \quad
    \begin{tz}
        \node[Pants, top, bot] (A) at (0,0) {};
        \node[Cyl, bot, anchor=top] at (A.rightleg) {};
        \node[Cup] at (A.leftleg) {};   
    \end{tz} \label{MC}
    \end{equation}

    \begin{equation}
    \begin{tz}
        \node[Pants, top, bot] (A) at (0,0) {};
    \end{tz}
    \quad
    \RLarrow{\beta}{\beta^\inv}
    \quad
    \begin{tz}
        \node[Pants, top, bot] (A) at (0,0) {};
        \node[BraidB, anchor=topleft, bot] at (A.leftleg) {};
    \end{tz}
    \qquad
    \qquad
    \begin{tz}
        \node[Cyl, top, bot, tall] (A) at (0,0) {};
    \end{tz}
    \quad
    \RLarrow{\theta}{\theta^\inv}
    \quad
    \begin{tz}
        \node[Cyl, top, bot, tall] (A) at (0,0) {};
    \end{tz} \label{RC}
    \end{equation}

 \item Non-invertible generating 2-morphisms:
    \begin{equation} 
    \begin{tz} 
        \node[Cyl, tall, top, bot] (A) at (0,0) {};
        \node[Cyl, tall, top, bot] (B) at (2*\cobwidth, 0) {};
    \end{tz}
    \quad
    \RLarrow{\eta}{\eta^\dagger}
    \quad
    \begin{tz} 
        \node[Pants, bot] (A) at (0,0) {};
        \node[Copants, top, bot, anchor=belt] at (A.belt) {};
    \end{tz}
    \qquad
    \qquad
    \begin{tz} 
        \node[Pants, top, bot] (A) at (0,0) {};
        \node[Copants, bot, anchor=leftleg] at (A.leftleg) {};
    \end{tz}
    \quad
    \RLarrow{\epsilon}{\epsilon^\dagger}
    \quad
    \begin{tz} 
        \node[Cyl, top, bot, tall] (A) at (0,0) {};
    \end{tz} \label{A1}
    \end{equation}
    
    \begin{equation} 
    \begin{tz}
        \draw[green] (0,0) rectangle (0.6, 0.6);  
    \end{tz}
    \quad
    \RLarrow{\nu}{\nu^\dagger}
    \quad
    \begin{tz}
        \node[Cap, bot] (A) at (0,0) {};
        \node[Cup] at (0,0) {};
    \end{tz}
    \qquad
    \qquad
    \begin{tz}
        \node[Cup, top] (A) at (0,0) {};
        \node[Cap, bot] (B) at (0,-2*\cobheight) {};
    \end{tz}
    \quad
    \RLarrow{\mu}{\mu^\dagger}
    \quad
    \begin{tz}
        \node[Cyl, top, bot, tall] (A) at (0,0) {};
    \end{tz}
    \label{A2}
    \end{equation}
 \end{itemize}
The relations are as follows:
\begin{itemize}
\item (Inverses) Each of the invertible generating 2-morphisms $\omega$ satisfies \mbox{$\omega \circ \omega^\inv = \id$} and $\omega^\inv \circ \omega = \id$. 
\item (Monoidal) The generating 2-morphisms in \eqref{MC} obey the pentagon and unit equations:
\begin{equation} \label{eq:pentagon}
\begin{tz}[xscale=2.1, yscale=1.2]
\node (1) at (0,0)
{
$\begin{tikzpicture}
        \node [Pants, wider, top] (A) at (0,0) {};
        \node [Pants, wide, anchor=belt] (B) at (A.leftleg) {};
        \node [Cyl, tall, anchor=top] (C) at (A.rightleg) {};
        \node[Pants, anchor=belt] (D) at (B.leftleg) {};
        \node[Cyl, anchor=top] (E) at (B.rightleg) {};
        \selectpart[green] {(A-belt) (B-leftleg) (A-rightleg)};
        \selectpart[red] {(A-leftleg) (D-leftleg) (B-rightleg)};
\end{tikzpicture}$
};
\node (2) at (1,1)
{
$\begin{tikzpicture}
        \node [Pants, verywide, top] (A) at (0,0) {};
        \node [Pants, anchor=belt] (B) at (A.rightleg) {};
        \node [Cyl, anchor=top] (C) at (A.leftleg) {};
        \node[Pants, anchor=belt] (D) at (C.bottom) {};
        \node[Cyl, anchor=top] (E) at (B.leftleg) {};
        \node[Cyl, anchor=top] (F) at (B.rightleg) {};
        \selectpart[green] {(A-leftleg) (A-rightleg) (D-leftleg) (F-bottom)};
\end{tikzpicture}$
};
\node (3) at (2,1)
{
$\begin{tikzpicture}
        \node [Pants, verywide, top] (A) at (0,0) {};
        \node [Pants, anchor=belt] (B) at (A.leftleg) {};
        \node[Cyl, anchor=top] (C) at (A.rightleg) {};
        \node[Cyl, anchor=top] (D) at (B.leftleg) {};
        \node[Cyl, anchor=top] (E) at (B.rightleg) {};
        \node[Pants, anchor=belt] (F) at (C.bottom) {};
        \selectpart[green] {(A-belt) (B-leftleg) (A-rightleg)};
\end{tikzpicture}$
};
\node (4) at (3,0)
{
$\begin{tikzpicture}
        \node [Pants, wider, top] (A) at (0,0) {};
        \node [Cyl, tall, anchor=top] (B) at (A.leftleg) {};
        \node[Pants, anchor=belt, wide] (C) at (A.rightleg) {};
        \node[Cyl, anchor=top] (D) at (C.leftleg) {};
        \node[Pants, anchor=belt] (E) at (C.rightleg) {};
\end{tikzpicture}$
};
\node (5) at (1,-1)
{
$\begin{tikzpicture}
        \node [Pants, veryverywide, top] (A) at (0,0) {};
        \node [Pants, wide, anchor=belt] (B) at (A.leftleg) {};
        \node [Cyl, tall, anchor=top] (C) at (A.rightleg) {};
        \node[Pants, anchor=belt] (D) at (B.rightleg) {};
        \node[Cyl, anchor=top] (E) at (B.leftleg) {};
                \selectpart[green] {(A-belt) (B-leftleg) (A-rightleg)};
\end{tikzpicture}$
};
\node (6) at (2,-1)
{
$\begin{tikzpicture}
        \node [Pants, veryverywide, top] (A) at (0,0) {};
        \node [Pants, wide, anchor=belt] (B) at (A.rightleg) {};
        \node [Cyl, tall, anchor=top] (C) at (A.leftleg) {};
        \node[Pants, anchor=belt] (D) at (B.leftleg) {};
        \node[Cyl, anchor=top] (E) at (B.rightleg) {};
                \selectpart[green] {(A-rightleg) (D-leftleg) (B-rightleg)};
\end{tikzpicture}$
};
\begin{scope}[double arrow scope]
    \draw (1) --  node[above left, green]{$\alpha$} (2);
    \draw (2) --  node[above]{$\interchangor$} (3);
    \draw (3) --  node[above]{$\alpha$} (4);
    \draw (1) --  node[below left, red]{$\alpha$} (5);
    \draw (5) --  node[below]{$\alpha$} (6);
    \draw (6) --  node[below right]{$\alpha$} (4);
\end{scope}
\end{tz}
\end{equation}
\begin{equation}
\label{eq:triangle}
\begin{tz}[xscale=1, yscale=2, every to/.style={out=down,in=up}]
\node (1) at (-1,1)
{
$\begin{tikzpicture}
    \node (A) [Pants] at (0,0) {};
    \node (B) [Pants, wide, top, anchor=leftleg] at (A.belt) {};
    \node (C) [Cup] at (A.rightleg) {};
    \node [Cyl, anchor=top] (D) at (A.leftleg) {};
    \node [Cyl, tall, anchor=top] (E) at (B.rightleg) {};
        \selectpart[green] {(B-belt) (A-leftleg) (B-rightleg)};
    \selectpart[red] {(B-leftleg) (D-bottom) (A-rightleg)};
\end{tikzpicture}$
};
\node (2) at (1,1)
{
$\begin{tikzpicture}
    \node (A) [Pants] at (0,0) {};
    \node (B) [Pants, wide, top, anchor=rightleg] at (A.belt) {};
    \node (C) [Cup] at (A.leftleg) {};
    \node (D) [Cyl, anchor=top] at (A.rightleg) {};
    \node [Cyl, tall, anchor=top] at (B.leftleg) {};
    \selectpart[green] {(B-rightleg) (A-leftleg) (D-bottom)};
\end{tikzpicture}$
};
\node (3) at (0,0)
{
$\begin{tikzpicture}
    \node (A) [Pants, top] at (0,0) {};
\end{tikzpicture}$
};
\begin{scope}[double arrow scope]
    \draw (1) --  node[above, green]{$\alpha$} (2);
    \draw (2) --  node[below right]{$\lambda$} (3);
    \draw (1) --  node[below left, red]{$\rho$} (3);
\end{scope}
\end{tz}
\end{equation}
 \item (Balanced) The data \eqref{MC} and \eqref{RC} forms a braided monoidal object equipped with a compatible twist:
\begin{equation}
\label{eq:hexagon}
\begin{tz}[xscale=2.2, yscale=1.2]
\node (1) at (0,0)
{
$\begin{tikzpicture}
    \node [Pants, wide, top] (A) at (0,0) {};
    \node [Pants, anchor=belt] (B) at (A.leftleg) {};
    \node [Cyl, anchor=top] (C) at (A.rightleg) {};
        \selectpart[red]{(A-leftleg) (B-leftleg) (B-rightleg)};
\end{tikzpicture}$
};

\node (2) at (0.75,1)
{
$\begin{tikzpicture}
    \node [Pants, wide, top] (A) at (0,0) {};
    \node [Pants, anchor=belt] (B) at (A.rightleg) {};
    \node [Cyl, anchor=top] (C) at (A.leftleg) {};
        \selectpart[green]{(A-belt) (A-leftleg) (A-rightleg)};
\end{tikzpicture}$
};

\node (3) at (1.5,1)
{
$\begin{tikzpicture}
    \node [Pants, wide, top] (A) at (0,0) {};
    \node[BraidB, wide, anchor=topleft] (B) at (A.leftleg) {};
    \node [Pants, anchor=belt] (C) at (B.bottomright) {};
    \node [Cyl, anchor=top] (D) at (B.bottomleft) {};
\end{tikzpicture}$
};

\node (4) at (2.25,1)
{
$\begin{tikzpicture}
    \node [Pants, wide, top] (A) at (0,0) {};
    \node [Pants, anchor=belt] (C) at (A.leftleg) {};
    \node [Cyl, anchor=top] (D) at (A.rightleg) {};
    \node [BraidB, anchor=topleft] (E) at (C.rightleg) {};
    \node[Cyl, anchor=top] (F) at (C.leftleg) {};
    \node[BraidB, anchor=topleft] (G) at (F.bottom) {};
    \node[Cyl, anchor=top] (H) at (G.bottomright) {};
    \node[Cyl, anchor=top] (I) at (G.bottomleft) {};
    \node[Cyl, anchor=top, tall] (J) at (E.bottomright) {};
    \selectpart[green]{(A-belt) (C-leftleg) (D-bottom)};
\end{tikzpicture}$
};

\node (5) at (3,0)
{
$\begin{tikzpicture}
    \node [Pants, wide, top] (A) at (0,0) {};
    \node [Pants, anchor=belt] (C) at (A.rightleg) {};
    \node [Cyl, anchor=top] (D) at (A.leftleg) {};
    \node [BraidB, anchor=topleft] (E) at (C.leftleg) {};
    \node[Cyl, tall, anchor=top] (F) at (A.leftleg) {};
    \node[BraidB, anchor=topleft] (G) at (F.bottom) {};
    \node[Cyl, anchor=top] (H) at (E.bottomright) {};
\end{tikzpicture}$
};

\node (6) at (1,-1)
{
$\begin{tikzpicture}
    \node [Pants, wide, top] (A) at (0,0) {};
    \node [Pants, anchor=belt] (B) at (A.leftleg) {};
    \node [Cyl, anchor=top, tall] (C) at (A.rightleg) {};
    \node [BraidB, anchor=topleft] (D) at (B.leftleg) {};
        \selectpart[green]{(A-belt) (B-leftleg) (A-rightleg)};
\end{tikzpicture}$
};

\node (7) at (2.0,-1)
{
$\begin{tikzpicture}
    \node [Pants, wide, top] (A) at (0,0) {};
    \node [Pants, anchor=belt] (B) at (A.rightleg) {};
    \node [Cyl, anchor=top] (C) at (A.leftleg) {};
    \node[BraidB, anchor=topleft] (D) at (C.bottom) {};
    \node[Cyl, anchor=top] (E) at (B.rightleg) {};
    \selectpart[green]{(A-rightleg) (B-leftleg) (B-rightleg)};
\end{tikzpicture}$
};

\begin{scope}[double arrow scope]
    \draw (1) --  node[above left]{$\alpha$} (2);
    \draw (2) --  node[above]{$\beta$} (3);
    \draw[-=] (3) --  (4);
    \draw (4) --  node[above right]{$\alpha$} (5);
    \draw (1) --  node[below left, red]{$\beta$} (6);
    \draw (6) --  node[below]{$\alpha$} (7);
    \draw (7) --  node[below right]{$\beta$} (5);
\end{scope}
\end{tz}
\end{equation}
\begin{equation} \label{balanced1}
\begin{tz}[xscale=1.4, yscale=1.5]

\node (1) at (0,0)
{
$\begin{tikzpicture}
        \node[Pants, top] (A) at (0,0) {};
        \selectpart[green, inner sep=1pt]{(A-belt)};
        \selectpart[red] {(A-leftleg) (A-rightleg) (A-belt)};
\end{tikzpicture}$
};
\node (2) at (1,0)
{
$\begin{tikzpicture}
        \node[Pants, top] (A) at (0,0) {};
\end{tikzpicture}$
};

\node (3) at (0,-1)
{
$\begin{tikzpicture}
        \node[Pants, top] (A) at (0,0) {};
        \selectpart[green, inner sep=1pt]{(A-leftleg)};
\end{tikzpicture}$
};

\node (4) at (1,-1)
{
$\begin{tikzpicture}
        \node[Pants, top] (A) at (0,0) {};
        \selectpart[green, inner sep=1pt]{(A-rightleg)};
\end{tikzpicture}$
};

\begin{scope}[double arrow scope]
    \draw (1) -- node[above, green] {$\theta$} (2);
    \draw (1) -- node[left, red] {$\beta^2$} (3);
    \draw (3) -- node[below] {$\theta$} (4);
    \draw (4) -- node[right] {$\theta$} (2);
\end{scope}
\end{tz}
\end{equation}
\begin{equation}
\label{balanced2}
\begin{tz}[xscale=1.4, yscale=2]
\node (1) at (0,0)
{
$\begin{tikzpicture}
    \node[Cup, top] (C) at (0,0) {};
        \selectpart[green, inner sep=1pt]{(C-center)};
\end{tikzpicture}$
};
\node (2) at (1,0)
{
$\begin{tikzpicture}
    \node[Cup, top] (C) at (0,0) {};
    \selectpart[green, inner sep=1pt]{(C-center)};

  \end{tikzpicture}$
};
\begin{scope}[double arrow scope]
    \draw (1) --  node[above]{$\theta$} (2);
\end{scope}d
\end{tz}
\quad = \quad
\begin{tz}[xscale=1.4, yscale=2]
\node (1) at (0,0)
{
$\begin{tikzpicture}
    \node[Cup, top] (C) at (0,0) {};
\end{tikzpicture}$
};
\node (2) at (1,0)
{
$\begin{tikzpicture}
    \node[Cup, top] (C) at (0,0) {};
  \end{tikzpicture}$
};
\begin{scope}[double arrow scope]
    \draw (1) --  node[above]{$\id$} (2);
\end{scope}
\end{tz}
\end{equation}
Note that the second hexagon axiom is redundant in the presence of a twist~~\cite{joyal1993braided}.
 \item (Rigidity) Write $\phileft$ for the following composite (`left Frobeniusator'):
 \begin{align} \label{defn_of_phileft}
 \phileft \quad&:=\quad
\begin{tz}
 \node[Pants, top, bot] (A) at (0,0) {};
 \node[Cyl, bot, anchor=top] (B) at (A.leftleg) {};
 \node[Copants, bot, anchor=leftleg] (C) at (A.rightleg) {};
 \node[Cyl, top, bot, anchor=bottom] (D) at (C.rightleg) {}; 
 \selectpart[green, inner sep=1pt] {(A-belt) (D-top)};
\end{tz}
\,\, \Rarrow{\eta} \,\,
\begin{tz}
 \node[Copants, top, wide, bot] (F) at (0,0) {};
 \node[Pants, bot, wide, anchor=belt] (G) at (F.belt) {};
 \node[Pants, bot, anchor=belt] (A) at (G.leftleg) {};
 \node[Cyl, bot, anchor=top] (B) at (A.leftleg) {};
 \node[Copants, bot, anchor=leftleg] (C) at (A.rightleg) {};
 \node[Cyl, bot, anchor=bottom] (X) at (C.rightleg) {}; 
 \selectpart[green] {(F-belt) (A-leftleg) (X-bottom)};
\end{tz}
\,\, \Rarrow{\alpha} \,\,
\begin{tz}
 \node[Copants, top, wide, bot] (F) at (0,0) {};
 \node[Pants, bot, wide, anchor=belt] (G) at (F.belt) {};
 \node[Pants, bot, anchor=belt] (A) at (G.rightleg) {};
 \node[Cyl, tall, bot, anchor=top] (B) at (G.leftleg) {};
 \node[Copants, bot, anchor=leftleg] (C) at (A.leftleg) {};
 \selectpart[green] {(G-rightleg) (A-leftleg) (A-rightleg) (C-belt)};
\end{tz}
\,\, \Rarrow{\epsilon} \,\,
\begin{tz}
        \node[Copants, top, bot] (A) at (0,0) {};
        \node[Pants, bot, anchor=belt] (B) at (A.belt) {};
\end{tz}
\intertext{The left rigidity relation says that $\phileft$ is invertible, with the following explicit inverse:}
\label{explicit_phileft_inverse}
\phileft^\inv \quad&=\quad  \begin{tz}
                \node[Copants, top] (A) at (0,0) {};
                \node[Pants, anchor=belt] (B) at (A.belt) {};
                \selectpart[green, inner sep=1pt] {(B-rightleg)};
\end{tz} 
\,\, \Rarrow{\epsilon^\dagger} \,\,
\begin{tz}
 \node[Copants, top, wide, bot] (F) at (0,0) {};
 \node[Pants, bot, wide, anchor=belt] (G) at (F.belt) {};
 \node[Pants, bot, anchor=belt] (A) at (G.rightleg) {};
 \node[Cyl, tall, bot, anchor=top] (B) at (G.leftleg) {};
 \node[Copants, bot, anchor=leftleg] (C) at (A.leftleg) {};
 \selectpart[green]{(F-belt) (G-leftleg) (A-rightleg)};
\end{tz}
\,\, \Rarrow{\alpha^\inv} \,\,
\begin{tz}
 \node[Copants, top, wide, bot] (F) at (0,0) {};
 \node[Pants, bot, wide, anchor=belt] (G) at (F.belt) {};
 \node[Pants, bot, anchor=belt] (A) at (G.leftleg) {};
 \node[Cyl, bot, anchor=top] (B) at (A.leftleg) {};
 \node[Copants, bot, anchor=leftleg] (C) at (A.rightleg) {};
 \node[Cyl, bot, anchor=bottom] (X) at (C.rightleg) {}; 
 \selectpart[green]{(F-leftleg) (F-rightleg) (G-leftleg) (G-rightleg)};
\end{tz}
\,\, \Rarrow{\eta^\dagger} \,\,
\begin{tz}
 \node[Pants, top, bot] (A) at (0,0) {};
 \node[Cyl, bot, anchor=top] (B) at (A.leftleg) {};
 \node[Copants, bot, anchor=leftleg] (C) at (A.rightleg) {};
 \node[Cyl, top, bot, anchor=bottom] (D) at (C.rightleg) {}; 
\end{tz}
\intertext{Similarly, write $\phiright$ for $\phileft$ rotated about the $z$-axis (`right Frobeniusator'):}
 \label{defn_of_phiright}
 \phiright \quad&:=\quad
\begin{tz}
 \node[Pants, top, bot] (A) at (0,0) {};
 \node[Cyl, bot, anchor=top] (B) at (A.rightleg) {};
 \node[Copants, bot, anchor=rightleg] (C) at (A.leftleg) {};
 \node[Cyl, top, bot, anchor=bottom] (D) at (C.leftleg) {}; 
 \selectpart[inner sep=1pt, green] {(D-top) (A-belt)};
 \end{tz}
\,\, \Rarrow{\eta} \,\,
\begin{tz}
 \node[Copants, top, wide, bot] (F) at (0,0) {};
 \node[Pants, bot, wide, anchor=belt] (G) at (F.belt) {};
 \node[Pants, bot, anchor=belt] (A) at (G.rightleg) {};
 \node[Cyl, bot, anchor=top] (B) at (A.rightleg) {};
 \node[Copants, bot, anchor=rightleg] (C) at (A.leftleg) {};
 \node[Cyl, bot, anchor=bottom] (X) at (C.leftleg) {}; 
 \selectpart[green] {(F-belt) (A-rightleg) (X-bottom)};
\end{tz}
\,\, \Rarrow{\alpha^\inv} \,\,
\begin{tz}
 \node[Copants, top, wide, bot] (F) at (0,0) {};
 \node[Pants, bot, wide, anchor=belt] (G) at (F.belt) {};
 \node[Pants, bot, anchor=belt] (A) at (G.leftleg) {};
 \node[Cyl, tall, bot, anchor=top] (B) at (G.rightleg) {};
 \node[Copants, bot, anchor=rightleg] (C) at (A.rightleg) {};
 \selectpart[green] {(G-leftleg) (A-leftleg) (A-rightleg) (C-belt)};
\end{tz}
\,\, \Rarrow{\epsilon} \,\,
\begin{tz}
        \node[Copants, top, bot] (A) at (0,0) {};
        \node[Pants, bot, anchor=belt] (B) at (A.belt) {};
\end{tz}
\intertext{The right rigidity relation says that $\phiright$ is invertible, with the following explicit inverse :}
\label{explicit_phiright_inverse}
\phiright^\inv \quad&=\quad
\begin{tz}
                \node[Copants, top] (A) at (0,0) {};
                \node[Pants, anchor=belt] (B) at (A.belt) {};
                \selectpart[green, inner sep=1pt] {(B-leftleg)};
\end{tz} 
\,\, \Rarrow{\epsilon^\dagger} \,\,
\begin{tz}
 \node[Copants, top, wide, bot] (F) at (0,0) {};
 \node[Pants, bot, wide, anchor=belt] (G) at (F.belt) {};
 \node[Pants, bot, anchor=belt] (A) at (G.leftleg) {};
 \node[Cyl, tall, bot, anchor=top] (B) at (G.rightleg) {};
 \node[Copants, bot, anchor=rightleg] (C) at (A.rightleg) {};
 \selectpart[green]{(F-belt) (A-leftleg) (G-rightleg)};
\end{tz}
\,\, \Rarrow{\alpha} \,\,
\begin{tz}
 \node[Copants, top, wide, bot] (F) at (0,0) {};
 \node[Pants, bot, wide, anchor=belt] (G) at (F.belt) {};
 \node[Pants, bot, anchor=belt] (A) at (G.rightleg) {};
 \node[Cyl, bot, anchor=top] (B) at (A.rightleg) {};
 \node[Copants, bot, anchor=rightleg] (C) at (A.leftleg) {};
 \node[Cyl, top, bot, anchor=bottom] (X) at (C.leftleg) {}; 
 \selectpart[green]{(F-leftleg) (F-rightleg) (G-leftleg) (G-rightleg)};
\end{tz}
\,\, \Rarrow{\eta^\dagger} \,\,
\begin{tz}
 \node[Pants, top, bot] (A) at (0,0) {};
 \node[Cyl, bot, anchor=top] (B) at (A.rightleg) {};
 \node[Copants, bot, anchor=rightleg] (C) at (A.leftleg) {};
 \node[Cyl, top, bot, anchor=bottom] (D) at (C.leftleg) {}; 
\end{tz}
\end{align}

\item (Ribbon) The twist satisfies the following equation:
\begin{equation}
\label{tortile}
\begin{tz}[xscale=1.4, yscale=2]
\node (1) at (0,0)
{
$\begin{tikzpicture}
        \node[Pants] (A) at (0,0) {};
        \node[Cap] at (A.belt) {};
        \selectpart[green, inner sep=1pt]{(A-leftleg)};
\end{tikzpicture}$
};
\node (2) at (1,0)
{
$\begin{tikzpicture}
        \node[Pants] (A) at (0,0) {};
        \node[Cap] at (A.belt) {};
\end{tikzpicture}$
};
\begin{scope}[double arrow scope]
    \draw (1) -- node[above] {$\theta$} (2);
\end{scope}
\end{tz}
\quad = \quad
\begin{tz}[xscale=1.4, yscale=2]
\node (1) at (0,0)
{
$\begin{tikzpicture}
        \node[Pants] (A) at (0,0) {};
        \node[Cap] at (A.belt) {};
        \selectpart[green, inner sep=1pt]{(A-rightleg)};
\end{tikzpicture}$
};
\node (2) at (1,0)
{
$\begin{tikzpicture}
        \node[Pants] (A) at (0,0) {};
        \node[Cap] at (A.belt) {};
\end{tikzpicture}$
};
\begin{scope}[double arrow scope]
    \draw (1) -- node[above] {$\theta$} (2);
\end{scope}
\end{tz} 
\end{equation}

\item (Biadjoint) The data \eqref{A1} expresses $\tikztinypants$ as the biadjoint of $\tikztinycopants$, while \eqref{A2} expresses $\tikztinycup$ as the biadjoint of $\tikztinycap$. That is, the following equations hold, along with daggers and rotations about the $x$-axis: 
\begin{align}
\label{adj_nu_mu}
\begin{aligned}
\begin{tikzpicture}[xscale=1.6, yscale=4, every to/.style={out=down,in=up}]
\node (A) at (0,0)
{
$\begin{tikzpicture}
    \node (1) [Cup, top] at (0,0) {};
    \node (2) [Cup, invisible] at (0,-1.5\cobheight) {};
    \node (3) [Cap, invisible] at (0,-1.5\cobheight) {};
    \fixboundingbox
    \selectpart[green]{(3) (2)}
\end{tikzpicture}$
};
\node (2) at (1,0)
{
$\begin{tikzpicture}
    \node (1) [Cup, top] at (0,0) {};
    \node (2) [Cap] at (0,-1.5\cobheight) {};
    \node (3) [Cup] at (0,-1.5\cobheight) {};
    \fixboundingbox
    \selectpart[green]{(1-center) (2-center)}
\end{tikzpicture}$
};
\node (3) at (2,0)
{
$\begin{tikzpicture}
        \node[Cyl, top] (X) at (0,0) {};
    \node[Cup] at (X.bottom) {};
\end{tikzpicture}$
};
\begin{scope}[double arrow scope]
    \draw (A) --  node[above]{$\nu$} (2);
    \draw (2) --  node[above]{$\mu$} (3);
\end{scope}
\end{tikzpicture}
\end{aligned}
\quad &= \quad
\begin{aligned}
\begin{tikzpicture}[xscale=1.6, yscale=4, every to/.style={out=down,in=up}]
\node (A) at (0,0)
{
$\begin{tikzpicture}
    \node (1) [Cup, top] at (0,0) {};
\end{tikzpicture}$
};
\node (2) at (1,0)
{
$\begin{tikzpicture}
    \node[Cup, top] at (0,0) {};
\end{tikzpicture}$
};
\begin{scope}[double arrow scope]
    \draw (A) --  node[above]{$\id$} (2);
\end{scope}
\end{tikzpicture}
\end{aligned}
\\
\label{adj_eta_epsilon}
\begin{aligned}
\begin{tikzpicture}[xscale=1.6, yscale=4]
\node (1) at (0,0)
{
$\begin{tikzpicture}
    \node [Pants, top] (A) at (0,0) {};
    \node[Cyl, anchor=top] (B) at (A.leftleg) {};
    \node[Cyl, anchor=top] (C) at (A.rightleg) {};
    \fixboundingbox
    \selectpart[green]{(A-leftleg) (A-rightleg) (B-bottom) (C-bottom)}
\end{tikzpicture}$
};
\node (2) at (1,0)
{
$\begin{tikzpicture}
    \node [Pants, bot, top] (A) at (0,0) {};
    \node [Copants, anchor=leftleg] (B) at (A.leftleg) {};
    \node [Pants, anchor=belt, bot] at (B.belt) {};
    \selectpart [green] {(A-leftleg) (A-belt) (A-rightleg) (B-belt)}; 
\end{tikzpicture}$
};
\node (3) at (2,0)
{
$\begin{tikzpicture}
    \node [Pants] (A) at (0,0) {};
    \node [Cyl, tall, bot, anchor=bot, top] at (A.belt) {};
\end{tikzpicture}$
};
\begin{scope}[double arrow scope]
    \draw (1) --  node[above]{$\eta$} (2);
    \draw (2) --  node[above]{$\epsilon$} (3);
\end{scope}
\end{tikzpicture}
\end{aligned}
\quad &= \quad
\begin{aligned}
\begin{tikzpicture}[xscale=1.6, yscale=4]
\node (1) at (0,0)
{
$\begin{tikzpicture}
    \node [Pants, top] (A) at (0,0) {};
\end{tikzpicture}$
};
\node (2) at (1,0)
{
$\begin{tikzpicture}
    \node [Pants, bot, top] (A) at (0,0) {};
\end{tikzpicture}$
};
\begin{scope}[double arrow scope]
    \draw (1) --  node[above]{$\id$} (2);
\end{scope}
\end{tikzpicture}
\end{aligned}
\end{align}
These are 8 equations in total.

 \item (Pivotality) The following equation holds, together with its rotation about the $z$-axis $\eqref{piv_on_sphere} {}^z$:  \begin{equation} \label{piv_on_sphere}
\begin{tz}[xscale=1.4, yscale=2]
\node (1) at (0,0)
{
$\begin{tikzpicture}
 \node[Cap] (A) at (0,0) {};
 \node[Cup, bot=false] (B) at (0,0) {};
 \node[Cobordism Bottom End 3D] (AA) at (0,0) {};              
 \selectpart[green, inner sep=1pt] {(AA)};
\end{tikzpicture}$
};

\node (2) at (1,0)
{
$\begin{tikzpicture}
    \node [Pants] (A) at (0,0) {};
        \node[Cap] at (A.belt) {};
    \node[Copants, anchor=leftleg] (B) at (A.leftleg) {};
    \node[Cup] at (B.belt) {};
    \selectpart[green, inner sep=1pt] {(A-rightleg)};
\end{tikzpicture}$
};

\node (3) at (2,0)
{
$\begin{tikzpicture}
    \node [Pants] (A) at (0,0) {};
        \node[Cap] at (A.belt) {};
    \node[Cyl, tall, anchor=top] (B) at (A.leftleg) {};
    \node[Cup] (C) at (A.rightleg) {};
    \node[Copants, anchor=leftleg] (D) at (B.bottom) {};
    \node[Cap, bot=false] (E) at (D.rightleg) {};
    \node [Cobordism Bottom End 3D] (B) at (D.rightleg) {};
    \node[Cup] at (D.belt) {};
    \selectpart[green] {(A-rightleg) (B)};
\end{tikzpicture}$
};

\node (4) at (3,0)
{
$\begin{tikzpicture}
    \node [Pants] (A) at (0,0) {};
        \node[Cap] (X) at (A.belt) {};
    \node[Copants, anchor=leftleg] (B) at (A.leftleg) {};
    \node[Cup] at (B.belt) {};
    \selectpart[green] {(X-center) (A-leftleg) (A-rightleg) (B-belt)};
\end{tikzpicture}$
};

\node (5) at (4,0)
{
$\begin{tikzpicture}
    \node [Cap] at (0,0) {};
    \node[Cup] at (0,0) {};
    \end{tikzpicture}$
};

\begin{scope}[double arrow scope]
    \draw (1) -- node[above] {$\epsilon^\dagger$} (2);
    \draw (2) -- node[above] {$\mu^\dagger$} (3);
    \draw (3) -- node[above] {$\mu$} (4);
    \draw (4) -- node[above] {$\epsilon$} (5);
\end{scope}
\end{tz}
\quad = \quad 
\begin{tz}[xscale=1.4, yscale=2]
\node (1) at (0,0)
{
$\begin{tikzpicture}
 \node[Cap] (A) at (0,0) {};
 \node[Cup] (B) at (0,0) {};
\end{tikzpicture}$
};

\node (2) at (1,0)
{
$\begin{tikzpicture}
 \node[Cap] (A) at (0,0) {};
 \node[Cup] (B) at (0,0) {};
\end{tikzpicture}$
};

\begin{scope}[double arrow scope]
    \draw (1) -- node[above] {$\id$} (2);
\end{scope}
\end{tz}
\end{equation}

\item (Modularity) The following equation holds, together with its rotation about the $z$-axis $\eqref{MOD}^z$:\begin{equation}
\label{MOD}
\begin{tz}[xscale=2, yscale=2]
\node (1) at (0,-0.5)
{$\begin{tikzpicture}
        \node[Cyl, top, tall] (A) at (0,0) {};
\end{tikzpicture}$};

\node (2) at (1,0)
{$\begin{tikzpicture}
        \node[Pants, top] (A) at (0,0) {};
        \node[Copants, anchor=leftleg] (B) at (A.leftleg) {};
        \selectpart[green, inner sep=1pt]{(A-leftleg)};
        \selectpart[red, inner sep=1pt] {(A-rightleg)};
\end{tikzpicture}$};

\node (3) at (2,0)
{$\begin{tikzpicture}
        \node[Pants, top] (A) at (0,0) {};
        \node[Copants, anchor=leftleg] (B) at (A.leftleg) {};
\end{tikzpicture}$};

\node (4) at (3,-0.5)
{$\begin{tikzpicture}
        \node[Cyl, top, tall] (A) at (0,0) {};
\end{tikzpicture}$};

\node(5) at (1.5, -1)
{$\begin{tikzpicture}
        \node[Cup, top] (A) at (0,0) {};
        \node[Cap] (B) at (0, -1.5*\cobheight) {};
\end{tikzpicture}$};
    
\begin{scope}[double arrow scope]
    \draw (1) -- node [above] {$\epsilon^\dagger$} (2);
    \draw[] ([xshift=0pt] 2.0) to node [above, inner sep=1pt] {${\color{green}\theta}, \color{red}{\theta^\inv}$} (3);
    \draw (3) -- node [above] {$\epsilon$} (4);
    \draw (1) -- node [below left] {$\mu^\dagger$} (5);
    \draw (5) -- node [below right] {$\mu$} (4);
\end{scope}
\end{tz}
\end{equation}

\item (Anomaly-freeness) The following equation holds:
\begin{equation} \label{AF}
\begin{tz}
        \node[Cap] (A) at (0,0) {};
        \node[Cup] (B) at (0,0) {};
        \node [Cobordism Bottom End 3D] (C) at (0,0) {};
        \selectpart[green, inner sep=1pt] {(C)};
\end{tz}
\,\,
\Rarrow{\epsilon^\dagger}
\,\,
\begin{tz}
        \node[Cap] (A) at (0,0) {};
        \node[Pants, anchor=belt] (B) at (A.center) {};
        \node[Copants, anchor=leftleg] (C) at (B.leftleg) {};
        \node[Cup] (D) at (C.belt) {};
        \selectpart[green, inner sep=1pt] {(B-leftleg)};
\end{tz}
\,\, \Rarrow{\theta} \,\,
\begin{tz}
        \node[Cap] (A) at (0,0) {};
        \node[Pants, anchor=belt] (B) at (A.center) {};
        \node[Copants, anchor=leftleg] (C) at (B.leftleg) {};
        \node[Cup] (B) at (C.belt) {};
        \selectpart[green] {(A-center) (B-leftleg) (B-rightleg) (C-belt)};
\end{tz}
\,\, \Rarrow{\epsilon} \,\,
\begin{tz}
        \node[Cap] (A) at (0,0) {};
        \node[Cup] (B) at (0,0) {};
\end{tz}
\quad = \quad
 \begin{tz}
        \node[Cap] (A) at (0,0) {};
        \node[Cup] (B) at (0,0) {};
\end{tz}
\,\, \Rarrow{\id} \,\,
 \begin{tz}
        \node[Cap] (A) at (0,0) {};
        \node[Cup] (B) at (0,0) {};
\end{tz}
\end{equation}

\end{itemize}
\end{theorem}

The generators have the following geometric interpretation:
\tikzset{surfacecurve/.style={line width=0.7pt}}
\tikzset{surfacedotted/.style={surfacecurve, densely dotted}}
\begin{itemize}
\item the generating object represents a circle;
\item the generating 1-morphisms represent the 2-dimensional bordisms naively suggested by the pictured surfaces, namely the pants, copants, cup, and cap;
\item the generating 2-morphisms $\alpha$, $\rho$, $\lambda$, and their inverses represent invertible 3-dimensional bordisms induced by the (boundary relative) ambient isotopies suggested by the pictured surfaces; 
\item the generating 2-morphisms $\eta ^\dagger$, $\epsilon$, and $\mu ^\dagger$ represent the bordism implementing the addition of a 2-handle about the curves
\begin{equation} \label{fig:surgery_curves}
\begin{aligned}
\includegraphics{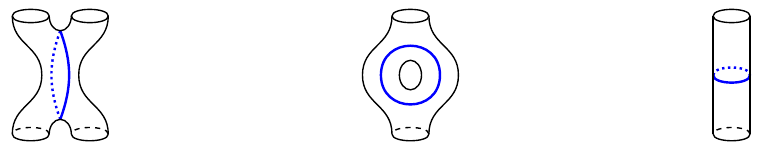}
\end{aligned}
\end{equation}
respectively; the 2-morphism $\nu ^\dagger$ represents the bordism implementing the addition of a 3-handle;
\item the generating 2-morphisms $\beta$ and $\theta$ represent the invertible 3-dimensional bordisms arising as mapping cylinders of the following diffeomorphisms:
\[
\begin{aligned}
\includegraphics[width=0.3\textwidth]{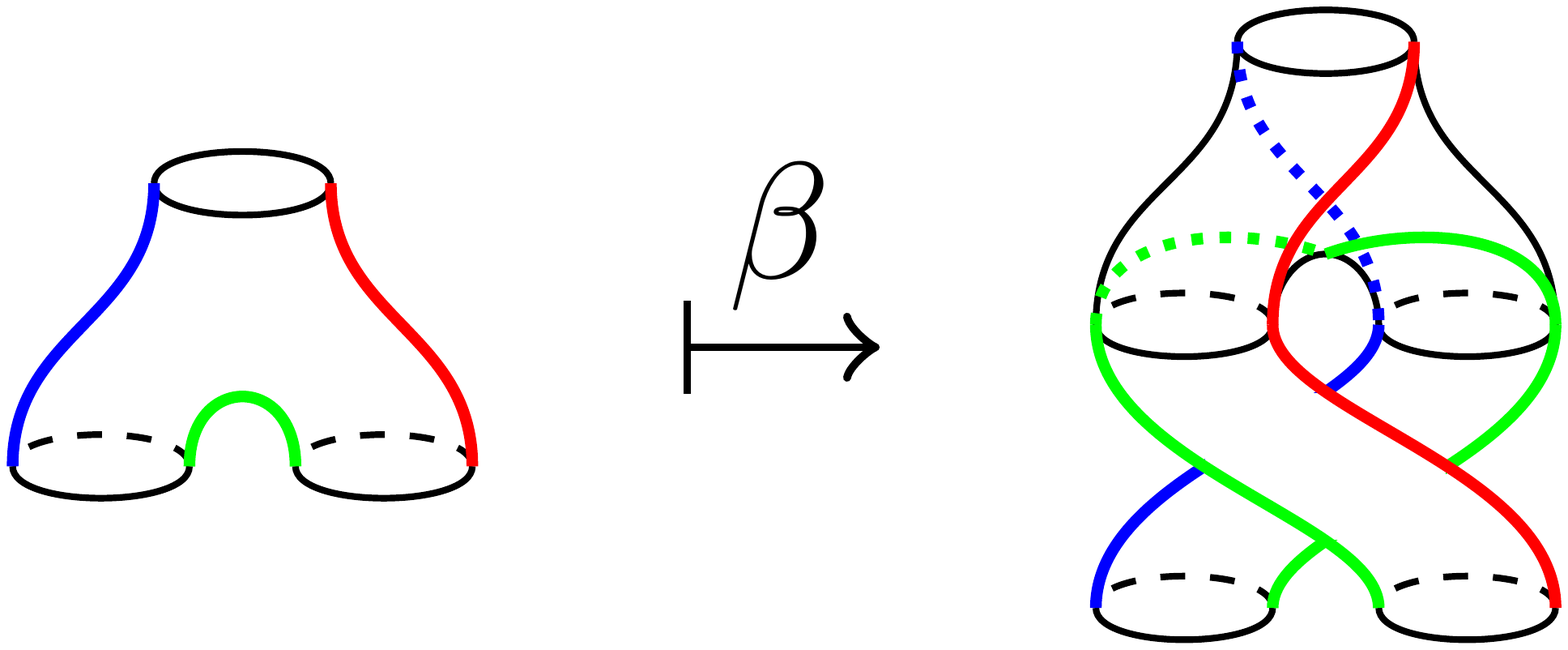}
\end{aligned} 
\qquad
\begin{aligned}
\includegraphics[width=0.2\textwidth]{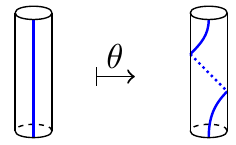}
\end{aligned}
\]
\item for a 2-morphism $\sigma$ representing a particular bordism, the 2-morphism $\sigma ^\dag$ represents the time-reversed bordism.
\end{itemize}

\begin{corollary} \label{co:123TQFTdata}
To give an oriented 1-2-3 TQFT $Z$, it suffices to give the following data:

\begin{itemize}
	\item (objects) A \Vect-enriched category $Z(S^1)$;
	\item (1-morphisms) for any obejects $A, B, C \in Z(S^1)$, vector spaces\footnote{These spaces are required to be functorial in the boundary labels.}
	    \begin{equation}
	        Z(\tinypants)^{A}_{B \boxtimes C}
	        \qquad
	        Z(\tinycopants)^{A \boxtimes B}_{C}
	        \qquad
	        Z(\tinycap)_{A}
	        \qquad
	        Z(\tinycup)^{A}
	    \end{equation}
	\item (2-morphisms) and for each generating 2-morphism $\Sigma_s \xRightarrow{~\kappa~} \Sigma_t$, a linear map
	    \begin{equation}
	        Z(\Sigma_s) \xrightarrow{~Z(\kappa)~} Z(\Sigma_t),
	    \end{equation}
	such that all the relations are satisfied.
\end{itemize}

\end{corollary}

\section{The String-net Space}


\label{string_net_section}

In this section we take the general definition of the string-net space for a surface with boundary (following \cite{Kir11}) and specialize it to the case of the toric code, where various simplifications can be made.

Throughout this section, $\Sigma$ denotes an oriented surface with boundary.

\begin{definition} \label{def:boundvalue}
A \emph{boundary value} for $\Sigma$ is a finite subset $B \subset \partial \Sigma$.
\end{definition}
 
 
\begin{definition} \label{def:graph}
Let $B$ be a choice of boundary value for $\Sigma$. A \emph{multicurve on $\Sigma$ subject to $B$} is a finite collection of disjoint smoothly embedded arcs and loops in $\Sigma$, such that 
\begin{itemize}
	\item[(1)] for any arc, each of its endpoints is an element of $B$;
	\item[(2)] each point in $B$ has an incident arc; and
	\item[(3)] arcs meet $\partial \Sigma$ transversely.
\end{itemize}
We write $\Curves(\Sigma; B)$ for the collection of all multicurves on $\Sigma$ which are subject to the boundary condition $B$.
\end{definition}

See Figure \ref{fig:graph_example} for a simple example.

\begin{figure}
    \includegraphics[scale=0.6]{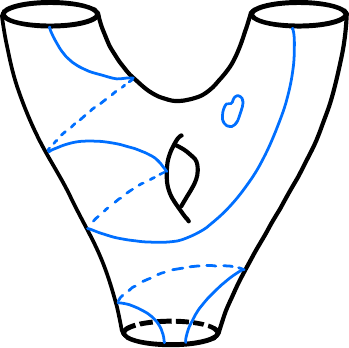}
		\caption{An example of a multicurve on $\Sigma$ subject to the boundary value $B$ consisting of a single point on each of the upper boundary circles and no point on the lower boundary circle.}
		\label{fig:graph_example}
\end{figure}

\begin{remark}
It is easy to see that there exist no multicurves on $\Sigma$ subject to $B$ when $|B|$ is odd.
\end{remark}

\begin{definition} \label{defn_string_net_space}
Let $\Sigma$ be an oriented surface with boundary and $B$ a boundary value for $\Sigma$. Two multicurves $\Gamma, \Gamma' \in \Curves(\Sigma; B)$ are {\em equivalent} if they differ by a finite sequence of the following moves: 
\begin{itemize}
    \item (isotopy invariance)
        \begin{equation} \label{isotopy_move}
        \begin{aligned}
        \begin{tikzpicture}[scale=0.33]
                \draw [thick, dashed] (0,0) circle (2cm);
                \draw [thick, blue] (0,-1.9) to [out=90,in=270,looseness=1] (-1.5,0);
                \draw [thick, blue] (-1.5,0) to [out=90,in=270,looseness=1] (0,1.9);
        \end{tikzpicture}
        \end{aligned}
        \quad
        \sim
        \quad
        \begin{aligned}
        \begin{tikzpicture}[scale=0.33]
                \draw [thick, dashed] (0,0) circle (2cm);
                \draw [thick, blue] (0,-1.9) to [out=90,in=270,looseness=1] (1.5,0);
                \draw [thick, blue] (1.5,0) to [out=90,in=270,looseness=1] (0,1.9);
        \end{tikzpicture}
        \end{aligned}
        \end{equation}
        
        
    \item (F-move\footnote{The name ``F-move'' is traditionally used in physics for the associator data in a fusion category. See \cite{LW05}.})
        \begin{equation} \label{F_move}
        \begin{aligned}
        \begin{tikzpicture}[scale=0.33]
                \draw [thick, dashed] (0,0) circle (2cm);
                \draw [thick, blue] (-1.35,1.35) to [out=300,in=240,looseness=1.5] (1.35,1.35);
                \draw [thick, blue] (-1.35,-1.35) to [out=60,in=120,looseness=1.5] (1.35,-1.35);
        \end{tikzpicture}
        \end{aligned}
        \quad
        \sim
        \quad
        \begin{aligned}
        \begin{tikzpicture}[scale=0.33, rotate=90]
                \draw [thick, dashed] (0,0) circle (2cm);
                \draw [thick, blue] (-1.35,1.35) to [out=300,in=240,looseness=1.5] (1.35,1.35);
                \draw [thick, blue] (-1.35,-1.35) to [out=60,in=120,looseness=1.5] (1.35,-1.35);
        \end{tikzpicture}
        \end{aligned}
        \end{equation}
    \item (loop contraction)
        \begin{equation} \label{loop_move}
        \begin{aligned}
        \begin{tikzpicture}[scale=0.33]
                \draw [thick, dashed] (0,0) circle (2cm);
                \draw [thick, blue] (0,0) circle (1.25cm);
        \end{tikzpicture}
        \end{aligned}
        \quad
        \sim
        \quad
        \begin{aligned}
        \begin{tikzpicture}[scale=0.33]
                \draw [thick, dashed] (0,0) circle (2cm);
        \end{tikzpicture}
        \end{aligned}
        \end{equation}
\end{itemize}
We define the \emph{string-net space of $\Sigma$ subject to $B$} as
\begin{equation} \label{string-net-space-defn}
 H(\Sigma; B) := \mathbb{C}[\Curves(M; B)/\sim] \, .
\end{equation}
Elements of $H^{string}(\Sigma;B)$ are called \emph{string-nets}. 
\end{definition}

\begin{remark}
For brevity, we shall usually omit saying ``subject to $B$'' whenever possible when referring to the string-net space. 
\end{remark}

\begin{remark} \label{homology_remark} Let $\Sigma$ be an oriented surface with boundary, and let $B_0$ be the empty boundary condition. Then a multicurve on $\Sigma$ subject to $B_0$ is the same thing as a 1-cycle in homology with $\mathbb{Z}/2\mathbb{Z}$ coefficients:
\[
 \Curves(\Sigma; B_0) = Z_1(\Sigma, \mathbb{Z}/2\mathbb{Z}) 
\]
Moreover, two multicurves $\Gamma, \Gamma'$ are equivalent according to the moves \eqref{isotopy_move}-\eqref{loop_move} precisely when they are homologous when considered as 1-cycles (see \cite{johnson1980spin}). Therefore,  
\[
\Curves(\Sigma; B_0) / \sim \,\, = \,\, H_1(\Sigma, \mathbb{Z}/2\mathbb{Z}).
\]
so that we can interpret the string-net space in terms of homology as
\[
 H(\Sigma; B_0) = \mathbb{C}[H_1(\Sigma, \mathbb{Z}/2\mathbb{Z})].
\]
For other boundary conditions, it is less clear how to interpret the string-net space in terms of homology (but see Lemma \ref{lem:zerotoone}). 
\end{remark}

\begin{remark} \label{rem:pushforward} The string-net space is functorial with respect to diffeomorphisms of surfaces. Let $\Sigma$ and $\Sigma'$ be oriented surfaces with boundary, with boundary conditions $B$ and $B'$ respectively, and let $\gamma: \Sigma \rightarrow \Sigma'$ be a diffeomorphism which restricts to a bijection $B \rightarrow B'$.  If $\Gamma$ is a multicurve on $\Sigma$ subject to $B$, then $\gamma(\Gamma)$ is a multicurve on $\Sigma'$ subject to $B'$. Thus $\gamma$ gives rise to an invertible push-forward linear map 
\[
 \gamma_* : H(\Sigma; B) \rightarrow H(\Sigma'; B').
\]
See Corollary \ref{push-forward-map-lem} for an application.
\end{remark}

\begin{remark} Given any spherical fusion category $A$, one can construct the string-net space associated to an oriented surface $\Sigma$ as the quotient of the vector space of formal combinations of $A$-labelled graphs on $\Sigma$ by the subspace spanned by null graphs supported in a disk (see \cite{Kir11}):
\begin{equation} \label{general-string-net}
H_A(\Sigma; B) = \mathbb{C}[\mathrm{Graphs}_A(\Sigma; B)] / \mathrm{NullGraphs}_A (\Sigma)
\end{equation}
The toric code corresponds to the case $A = \mathrm{Vect}[\mathbb{Z}/2\mathbb{Z}]$. In this case, the general definition \eqref{general-string-net} simplifies to \eqref{string-net-space-defn}, namely the free vector space on a set of equivalence classes of multicurves, as opposed to a quotient of vector spaces of graphs. 
\end{remark}

\begin{example}
If $\Sigma = S^2$ then any multicurve on $\Sigma$ is a collection of loops, all of whom contract to $1$ since $S^2$ is simply-connected. It follows that $H(S^2) \cong \CC$.
\end{example}

\begin{example}
The string-net space of the torus is
\begin{equation*}
H(T) = \text{span}\left\{ \raisebox{-2.5mm}{\includegraphics[scale=0.1]{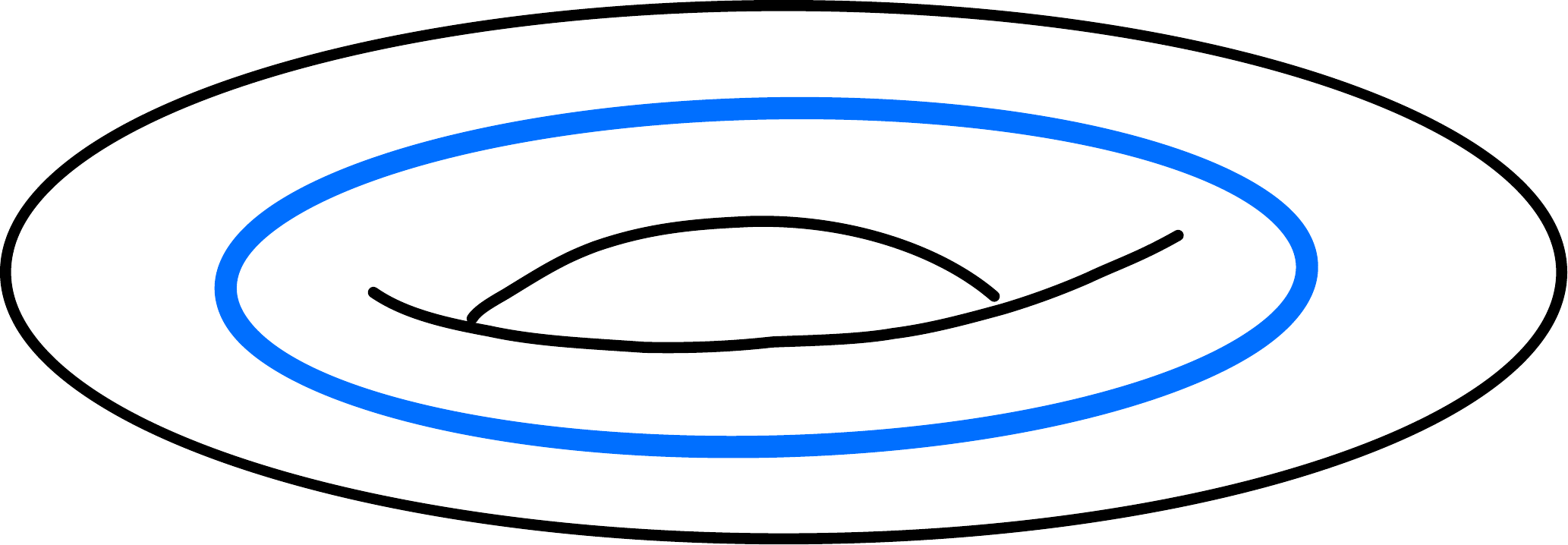}},\raisebox{-2.5mm}{\includegraphics[scale=0.1]{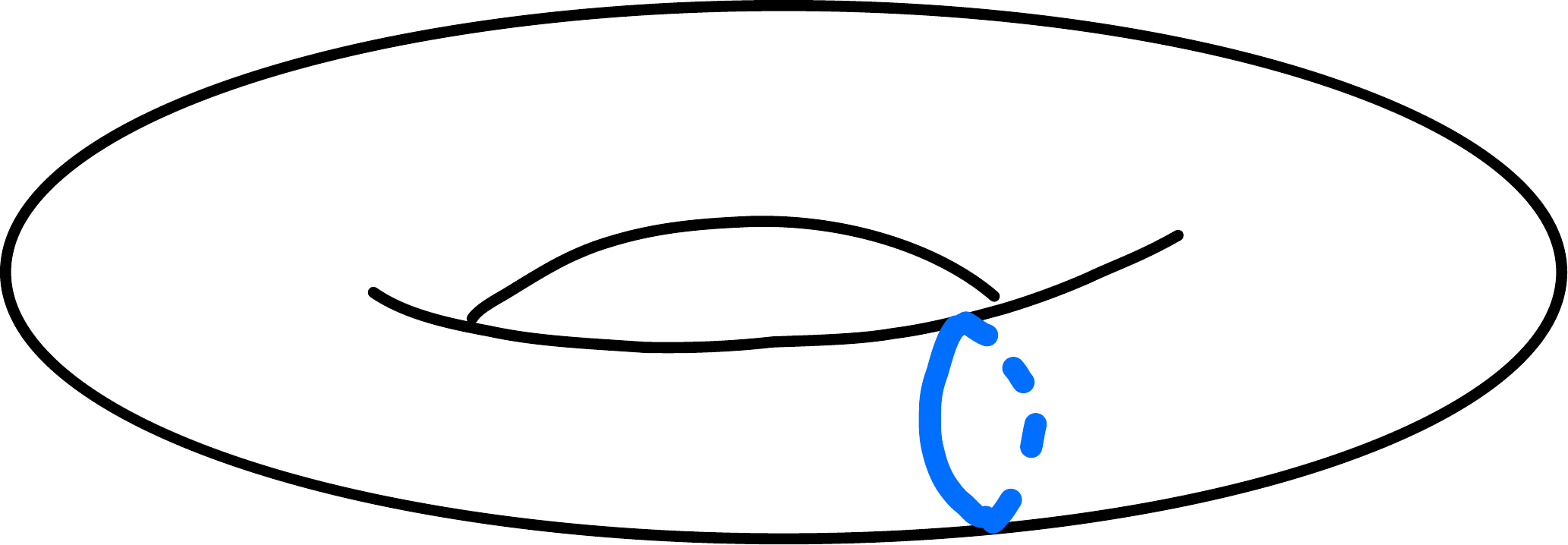}},\raisebox{-2.5mm}{\includegraphics[scale=0.1]{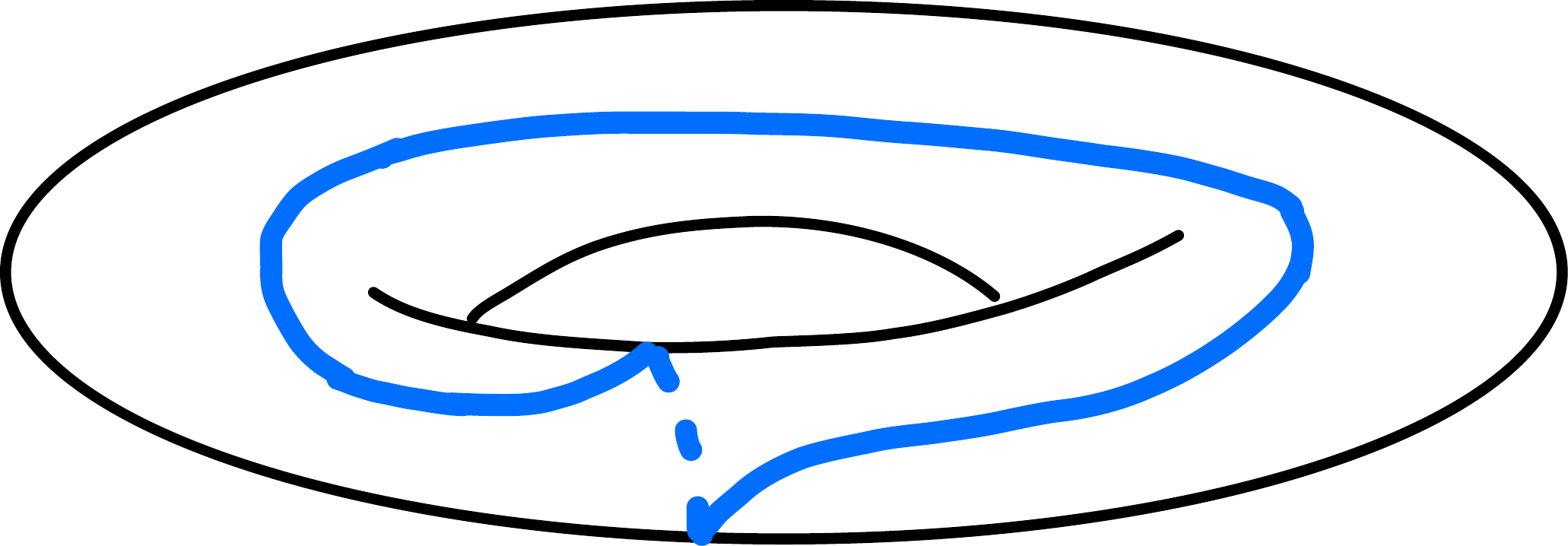}},\raisebox{-2.5mm}{\includegraphics[scale=0.1]{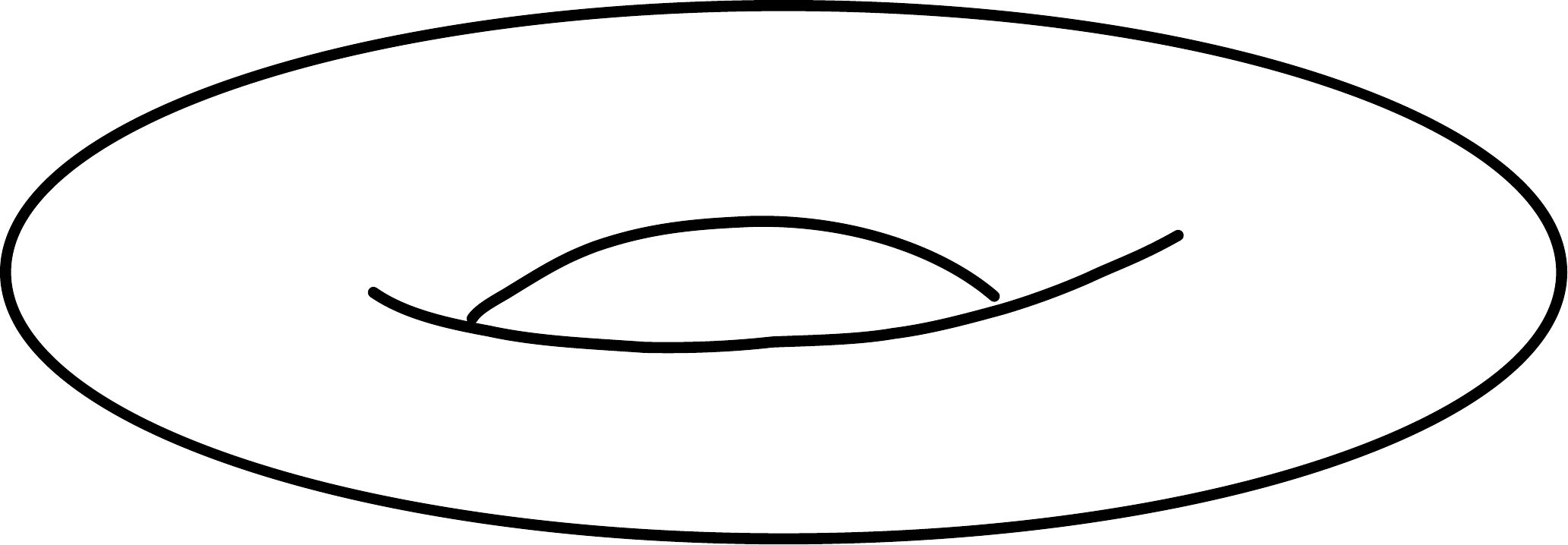}} \right\}.
\end{equation*}
\end{example}

\begin{lemma}[Cloaking] \label{lem:cloaking}
Let the orange string be defined by

\begin{equation}
        \begin{aligned}
        \begin{tikzpicture}[scale=0.33]
                \draw [thick, dashed] (0,0) circle (2cm);
                \draw [thick, orange] (0,-1.9) -- (0,1.9);
        \end{tikzpicture}
        \end{aligned}
        \quad
        =
        \quad
        \begin{aligned}
        \begin{tikzpicture}[scale=0.33]
                \draw [thick, dashed] (0,0) circle (2cm);
								\draw [thick, blue] (0,-1.9) -- (0,1.9);
        \end{tikzpicture}
        \end{aligned}
				\quad
				+
				\quad
				\begin{aligned}
        \begin{tikzpicture}[scale=0.33]
                \draw [thick, dashed] (0,0) circle (2cm);								
        \end{tikzpicture}
        \end{aligned}
        \end{equation}
Then the relation
\begin{equation} \label{eqn:cloaking}
\begin{tz}
\filldraw [LightGray] (1,1.5) circle (0.25cm);
\draw [thick, orange] (1,1.5) circle (0.4cm);
\draw [thick, blue] (1,2.5) to [out=270,in=90,looseness=1] (0.25,1.5);
\draw [thick, blue] (0.25,1.5) to [out=270,in=90,looseness=1] (1,0.5);
\end{tz}
\quad
=
\quad
\begin{tz}
\filldraw [LightGray] (1,1.5) circle (0.25cm);
\draw [thick, orange] (1,1.5) circle (0.4cm);
\draw [thick, blue] (1,2.5) to [out=270,in=90,looseness=1] (1.75,1.5);
\draw [thick, blue] (1.75,1.5) to [out=270,in=90,looseness=1] (1,0.5);
\end{tz}
\end{equation}
holds in any string-net space, irrespective of the contents of the shaded region.\footnote{It may contain any subgraph and any number of holes or punctures.}
\end{lemma}
\begin{proof}

\begin{align}
\begin{tz}
\filldraw [LightGray] (1,1.5) circle (0.25cm);
\draw [thick, orange] (1,1.5) circle (0.4cm);
\draw [thick, blue] (1,2.5) to [out=270,in=90,looseness=1] (0.25,1.5);
\draw [thick, blue] (0.25,1.5) to [out=270,in=90,looseness=1] (1,0.5);
\end{tz}
\quad &=
\quad
\begin{tz}
\filldraw [LightGray] (1,1.5) circle (0.25cm);
\draw [thick, blue] (1,1.5) circle (0.4cm);
\draw [thick, blue] (1,2.5) to [out=270,in=90,looseness=1] (0.25,1.5);
\draw [thick, blue] (0.25,1.5) to [out=270,in=90,looseness=1] (1,0.5);
\end{tz}
\quad
+
\quad
\begin{tz}
\filldraw [LightGray] (1,1.5) circle (0.25cm);
\draw [thick, blue] (1,2.5) to [out=270,in=90,looseness=1] (0.25,1.5);
\draw [thick, blue] (0.25,1.5) to [out=270,in=90,looseness=1] (1,0.5);
\end{tz} \\
&=
\quad
\begin{tz}
\filldraw [LightGray] (1,1.5) circle (0.25cm);
\draw [thick, blue] (1,2.5) to [out=270,in=90,looseness=1] (0.25,1.75) to [out=270,in=180,looseness=1] (1,1.9);
\draw [thick, blue] (1,1.5) [partial ellipse=0:90:0.4cm and 0.4cm];
\draw [thick, blue] (1,1.5) [partial ellipse=270:360:0.4cm and 0.4cm];
\draw [thick, blue] (1,0.5) to [out=90,in=270,looseness=1] (0.25,1.25) to [out=90,in=180,looseness=1] (1,1.1);
\end{tz}
\quad
+
\quad
\begin{tz}[rotate=180]
\filldraw [LightGray] (1,1.5) circle (0.25cm);
\draw [thick, blue] (1,2.5) to [out=270,in=90,looseness=1] (0.25,1.75) to [out=270,in=180,looseness=1] (1,1.9);
\draw [thick, blue] (1,1.5) [partial ellipse=0:90:0.4cm and 0.4cm];
\draw [thick, blue] (1,1.5) [partial ellipse=270:360:0.4cm and 0.4cm];
\draw [thick, blue] (1,0.5) to [out=90,in=270,looseness=1] (0.25,1.25) to [out=90,in=180,looseness=1] (1,1.1);
\end{tz} \\
&=
\quad
\begin{tz}
\filldraw [LightGray] (1,1.5) circle (0.25cm);
\draw [thick, blue] (1,2.5) to [out=270,in=90,looseness=1] (1.75,1.5);
\draw [thick, blue] (1.75,1.5) to [out=270,in=90,looseness=1] (1,0.5);
\end{tz}
\quad
+
\quad
\begin{tz}
\filldraw [LightGray] (1,1.5) circle (0.25cm);
\draw [thick, blue] (1,1.5) circle (0.4cm);
\draw [thick, blue] (1,2.5) to [out=270,in=90,looseness=1] (1.75,1.5);
\draw [thick, blue] (1.75,1.5) to [out=270,in=90,looseness=1] (1,0.5);
\end{tz}
\quad
=
\quad
\begin{tz}
\filldraw [LightGray] (1,1.5) circle (0.25cm);
\draw [thick, orange] (1,1.5) circle (0.4cm);
\draw [thick, blue] (1,2.5) to [out=270,in=90,looseness=1] (1.75,1.5);
\draw [thick, blue] (1.75,1.5) to [out=270,in=90,looseness=1] (1,0.5);
\end{tz}
\end{align}
\end{proof}

We call Relation \ref{eqn:cloaking} above ``cloaking'', since the presence of the orange loop effectively cloaks the contents of the shaded region, allowing strings to pass over it unperturbed. Cloaking will play an important role in the following section.

\section{The Extended Toric Code}


\label{action_of_generators_section}

In this section we make use of Corollary \ref{co:123TQFTdata} in order to define our desired extended 3-dimensional TQFT, which we denote by $Z_{SN}$, based on string-nets.

\subsection{The Objects}

We require a \Vect-enriched category to assign to the generating object (i.e. the circle). 

\begin{definition} \label{category_of_boundary_values}
The \emph{category $\BV$ of boundary values}\footnote{This is a specialization of Definition 6.1 in \cite{Kir11}} is defined as follows:

\gscalecobordisms{0.25}

\begin{itemize}
    \item The objects of $\BV$ are finite subsets of $S^1$.
    \item Given objects $B$ and $B'$ of $\BV$, the space $\text{hom}(B,B')$ is defined to be the string-net space $H(S^1 \times [0,1]; B, B')$ with the boundary value $B$ at $S^1 \times \{1\}$ and $B'$ at $S^1 \times \{0\}$.
\end{itemize}

Composition in $\BV$ works by gluing the cylinders vertically and then rescaling. The identity morphisms are string-nets where each point of $B$ at the top of the cylinder is connected by a vertical arc to the corresponding point of $B$ at the bottom. 

\end{definition}

Note that the category $\BV$ is not Cauchy complete. (In fact, its Cauchy completion is the usual Drinfeld center of the category $\Vect[\mathbb{Z}/2\mathbb{Z}]$ of $\mathbb{Z}/2\mathbb{Z}$ graded vector spaces, which has 4 simple objects, but we do not adopt that viewpoint here. See Remark \ref{prof_remark}.)

\gscalecobordisms{0.25}
\begin{lemma} \label{lem:sec4_parityiso}
Let $B \in \BV$ where $|B| = k$. Let $B_0$ and $B_1$ denote the objects

\begin{equation}
B_0
\quad
=
\quad
\begin{tikzpicture}[scale=\cobscale]
		\draw (1,0) ellipse (1cm and 0.25cm);
\end{tikzpicture}
\qquad
\text{and}
\qquad
B_1
\quad
=
\quad
\raisebox{-1.35mm}{%
\begin{tikzpicture}[scale=\cobscale]
		\draw (1,0) ellipse (1cm and 0.25cm);
		\node at (1,-0.25) {\BlackDot};
\end{tikzpicture}
}
\end{equation}
Then in the category $\BV$, $B$ is canonically isomorphic to $B_0$ if $k$ is even, and $B$ is canonically isomorphic to  $B_1$ if $k$ is odd.
\end{lemma}
\begin{proof}
If $k = 2n$, then the string-net consisting of $n$ cups at the top of the cylinder is an isomorphim from $B$ to $B_0$.\footnote{See figure \ref{fig:two_to_zero} for an example when $k = 2$.} This isomorphism is canonical since any two ways of cupping off the $2n$ boundary points are equivalent under repeated application of the $F$-move. If $k = 2n + 1$, then the string-net consisting of $n$ consecutive cups at the top of the cylinder, followed by connecting the final point of $B$ with the lone point on $B_1$, is an isomorphism from $B$ to $B_1$. This isomorphism is also canonical, again due to the $F$-move.
\end{proof}

\begin{remark}
A useful consequence of Lemma \ref{lem:sec4_parityiso} is that it suffices to define the actions of the generating 2-morphisms on string-net spaces subject to boundary conditions $B_0$ or $B_1$ on its boundary circles.
\end{remark}

\begin{figure}
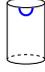

    \centering
    \cylwithtopcup
    \caption{A morphism in $\BV$ from the object with two elements to the object with no elements.}
    \label{fig:two_to_zero}
\end{figure}
\gscalecobordisms{0.25}

On objects we set

\begin{equation}
\begin{tikzpicture}[scale=\cobscale]
		\draw (1,0) ellipse (1cm and 0.25cm);
\end{tikzpicture}
\qquad
\xmapsto{~Z_{SN}~}
\qquad
\BV.
\end{equation}

\subsection{The Generating 1-Morphisms}

Below we specify the vector spaces assigned to the generating 1-morphisms. We define

\begin{align}
    Z_{SN}\left(\tinycup\right)^A &= H\left(\tinycup;A\right) \\
    Z_{SN}\left(\tinycap\right)_A &= H\left(\tinycap;A\right) \\
    Z_{SN}\left(\tinypants\right)^A_{B,C} &= H\left(\tinypants;A,B,C\right) \\
    Z_{SN}\left(\tinycopants\right)^{A,B}_C &= H\left(\tinycopants;A,B,C\right)
\end{align}

Here $A$, $B$, $C \in \{B_0,B_1\}$. In other words, the vector spaces assigned to the 1-generators are simply the string-net spaces on the surfaces naively suggested by the pictures, along with the stated boundary conditions.\\

Next we investigate the vector spaces assigned to a \emph{composite} of the generating 1-morphisms. Let $\Sigma$ be such a composite, with $m$ top (input) boundary circles and $n$ bottom (output) boundary circles, labeled by $X_1,\ldots,X_m$ and $Y_1,\ldots,Y_n$ respectively, with $X_i,Y_j \in \{B_0,B_1\}$ for each $i$ and $j$.  Then, according to Equation \ref{eqn:prof_comp}, we have

\begin{equation} \label{eqn:stringnet_1gencomposite}
    Z_{SN}\left( \Sigma \right)^{X_1,\ldots,X_m}_{Y_1,\ldots,Y_n} := \bigoplus_{L} \bigotimes_{\sigma} ~ \quotient{Z_{SN}(\sigma)}{\sim}
\end{equation}

where the tensor product ranges over all 1-generators $\sigma$ appearing in the decomposition of $\Sigma$ and the direct sum ranges over all possible labelings (with objects in \BV) of the internal boundary circles. Each vector spaces $Z_{SN}(\sigma)$ is computed using the boundary conditions induced by $L$ and the given boundary conditions on the top and bottom boundary circles.\\

Unpacking Equation \ref{eqn:stringnet_1gencomposite} we see that, in the string-net setting, the relation $\sim$ has the following graphical description:

\begin{equation}
\begin{aligned}
\begin{tikzpicture}[scale=0.65]

\draw [fill, white] (1.5,4.25) circle (1.5pt);
\draw [fill, white] (1.5,2.25) circle (1.5pt);

\draw [fill, LightBlue] (0,0) to [out=60,in=270] (0.5,1) -- (0.5,1.5) -- (2.5,1.5) -- (2.5,1) to [out=270,in=120] (3,0) -- (0,0);

\draw [fill, LightBlue] (0,8) to [out=-60,in=90] (0.5,7) -- (0.5,6.5) -- (2.5,6.5) -- (2.5,7) to [out=90,in=-120] (3,8) -- (0,8);

\draw [fill, pink] (1.5,5) ellipse (1cm and 0.25cm);
\draw [fill, pink] (0.5,1.5) -- (0.5,5) -- (2.5,5) -- (2.5,1.5) -- cycle;

\draw [fill, pink] (1.5,1.5) ellipse (1cm and 0.25cm);
\draw [dashed] (1.5,1.5) [partial ellipse=0:180:1cm and 0.25cm];
\draw (1.5,1.5) [partial ellipse=180:360:1cm and 0.25cm];

\draw [fill, LightBlue] (1.5,6.5) ellipse (1cm and 0.25cm);
\draw [dashed] (1.5,6.5) [partial ellipse=0:180:1cm and 0.25cm];
\draw (1.5,6.5) [partial ellipse=180:360:1cm and 0.25cm];

\draw (0,0) to [out=60,in=270] (0.5,1);
\draw (0.5,1) -- (0.5,5);
\draw (2.5,1) -- (2.5,5);
\draw (3,0) to [out=120,in=270] (2.5,1);

\draw (0.5,6.5) -- (0.5,7) to [out=90, in=-60] (0,8);
\draw (2.5,6.5) -- (2.5,7) to [out=90, in=-120] (3,8);

\draw (1.5,5) ellipse (1cm and 0.25cm);

\node at (1.5,7.5) {$\Gamma_1$};
\node at (1.5,0.5) {$\Gamma_2$};
\node at (1.5,3) {$\Gamma$};

\node at (0,6.5) {$X$};
\node at (0,5) {$X$};
\node at (0,1.5) {$Y$};

\end{tikzpicture}
\end{aligned}
\qquad
=
\qquad
\begin{aligned}
\begin{tikzpicture}[scale=0.65]

\draw [fill, pink] (1.5,3) ellipse (1cm and 0.25cm);
\draw [fill, pink] (0.5,3) -- (0.5,6.5) -- (2.5,6.5) -- (2.5,3) -- cycle;

\draw [fill, LightBlue] (0,0) to [out=60,in=270] (0.5,1) -- (0.5,1.5) -- (2.5,1.5) -- (2.5,1) to [out=270,in=120] (3,0) -- (0,0);
\draw [fill, VLB] (1.5,1.5) ellipse (1cm and 0.25cm);
\draw (1.5,1.5) ellipse (1cm and 0.25cm);

\draw (0,0) to [out=60,in=270] (0.5,1) -- (0.5,1.5);
\draw (3,0) to [out=120,in=270] (2.5,1) -- (2.5,1.5);

\draw [fill, LightBlue] (0,8) to [out=-60,in=90] (0.5,7) -- (0.5,6.5) -- (2.5,6.5) -- (2.5,7) to [out=90,in=-120] (3,8) -- (0,8);
\draw [dashed] (1.5,3) [partial ellipse=0:180:1cm and 0.25cm];
\draw (1.5,3) [partial ellipse=180:360:1cm and 0.25cm];
\draw [LightBlue, fill] (1.5,6.5) ellipse (1cm and 0.25cm);
\draw [dashed] (1.5,6.5) [partial ellipse=0:180:1cm and 0.25cm];
\draw (1.5,6.5) [partial ellipse=180:360:1cm and 0.25cm];
\draw (0.5,3) -- (0.5,7) to [out=90,in=-60] (0,8);
\draw (2.5,3) -- (2.5,7) to [out=90,in=-120] (3,8);

\node at (1.5,4.5) {$\Gamma$};
\node at (1.5,7.5) {$\Gamma_1$};
\node at (1.5,0.5) {$\Gamma_2$};

\node at (3,6.5) {$X$};
\node at (3,3) {$Y$};
\node at (3,1.5) {$Y$};

\end{tikzpicture}
\end{aligned}
\end{equation}

Here $\lab\Gamma_1\rab$ (resp. $\lab\Gamma_2\rab$) is a string-net local to a 1-generator $\sigma_1$ (resp. $\sigma_2$) appearing in $\Sigma$, with boundary value $X$ (resp. $Y$) on the displayed boundary circle. Finally, $\lab\Gamma\rab \in H(\tinycyl; X,Y)$.\\

In other words, the relation $\sim$ allows us to move string-net data living on $\sigma_1$ ``isotopically'' through in the internal boundary circle to $\sigma_2$, as though $\sigma_1$ and $\sigma_2$ were actually glued together. If we let $\widehat{\Sigma}$ denote the surface obtained by gluing together all the 1-generators in $\Sigma$ along the internal boundary circles in the obvious way, then $\sim$ suggests that elements of $Z_{SN}\left( \Sigma \right)^{X_1,\ldots,X_m}_{Y_1,\ldots,Y_n}$ behave just like the string-nets in $H(\widehat{\Sigma};X_1,\ldots,X_m,Y_1,\ldots,Y_n)$. We prove this fact below.

\begin{theorem} \label{thm:gluing}
Let $\Sigma$ and $\widehat{\Sigma}$ be defined as above (including the boundary values). Then there is a canonical isomorphism
\[
Z_{SN}\left( \Sigma \right)^{X_1,\ldots,X_m}_{Y_1,\ldots,Y_n}
\cong
H(\widehat{\Sigma};X_1,\ldots,X_m,Y_1,\ldots,Y_n) \, .
\]
\end{theorem}
\begin{proof}
We define a map

\[
\Psi : \bigoplus_{L} \bigotimes_{\sigma} ~ \quotient{Z_{SN}(\sigma)}{\sim}
\quad
\longrightarrow
\quad
H(\widehat{\Sigma};X_1,\ldots,X_m,Y_1,\ldots,Y_n)
\]

and demonstrate that it is an isomorphism. Fix a labeling $L$, and let $\Gamma_{\sigma}$ be a graph on $\sigma$ (subject to the obvious boundary conditions) for each 1-generator $\sigma$ appearing in the decomposition of $\Sigma$. Then we define $\Psi\left(\bigotimes_{\sigma} \lab\Gamma_{\sigma}\rab\right)$ to be the string-net on $\widehat{\Sigma}$ represented by the graph which looks like $\Gamma_{\sigma}$ in each component $\sigma$, where the graphs $\Gamma_{\sigma}$ have been glued together along the internal boundary circles in the obvious way.

Clearly, $\Psi$ is well-defined and one-to-one, since two elements are equivalent under $\sim$ precisely when their images are equivalent under isotopy as well as the usual string-net relations (which are \emph{local}) inside each 1-generator $\sigma$. Finally, $\Psi$ is onto since any graph on $\widehat{\Sigma}$ may be perturbed via isotopy to be in general position with respect to each internal boundary circle.
\end{proof}

In other words, $Z_{SN}$ assigns to $\Sigma$ the string-net space on the surface naively suggested by the composite pictures, subject to the given boundary conditions.

\subsection{The Generating 2-Morphisms}

It follows from Lemma \ref{lem:sec4_parityiso} that it suffices to define the actions of the generating 2-morphims on 1-morphisms where the boundary circles are labeled by $B_0$ or $B_1$. We therefore make this assumption throughout the remainder of the paper. The assignment for the generating 2-morphisms are split into two parts: the invertible ones and the non-invertible ones.

\subsubsection{The invertible 2-generators}  \label{invertible-2-gen}

Recall that each invertible 2-generator arises as the mapping cylinder of a diffeomorphism of the pictured surface. We therefore define the action to be the pushforward of the string-nets along these diffeomorphisms, as in from Remark \ref{rem:pushforward}.  We illustrate this with a few examples.\\

\begin{equation}
\begin{aligned}
\includegraphics[scale=0.05]{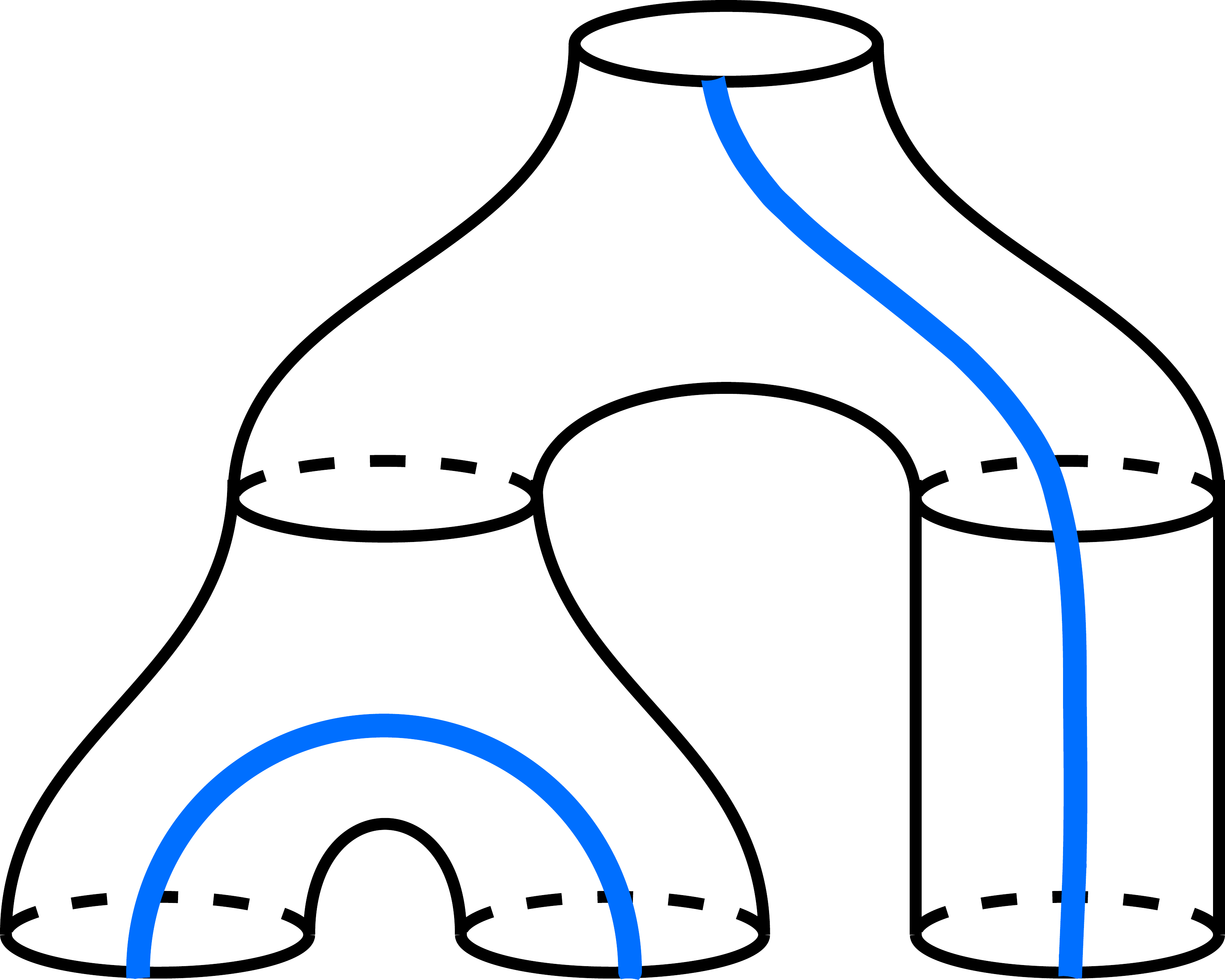}
\end{aligned}
\quad
\xmapsto{Z_{SN}(\alpha)}
\quad
\begin{aligned}
\includegraphics[scale=0.05]{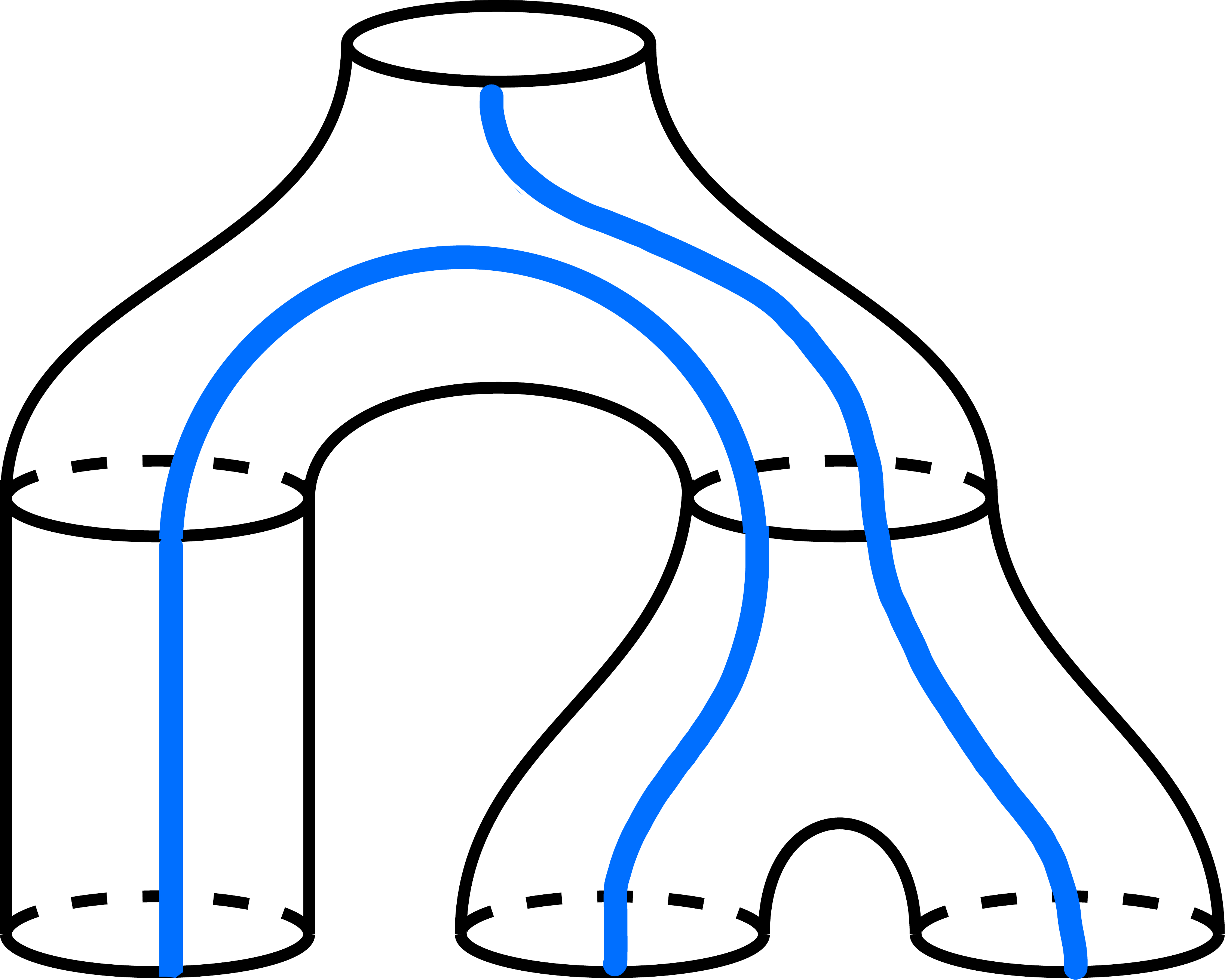}
\end{aligned}
\end{equation}

\begin{equation}
\begin{aligned}
\includegraphics[scale=0.05]{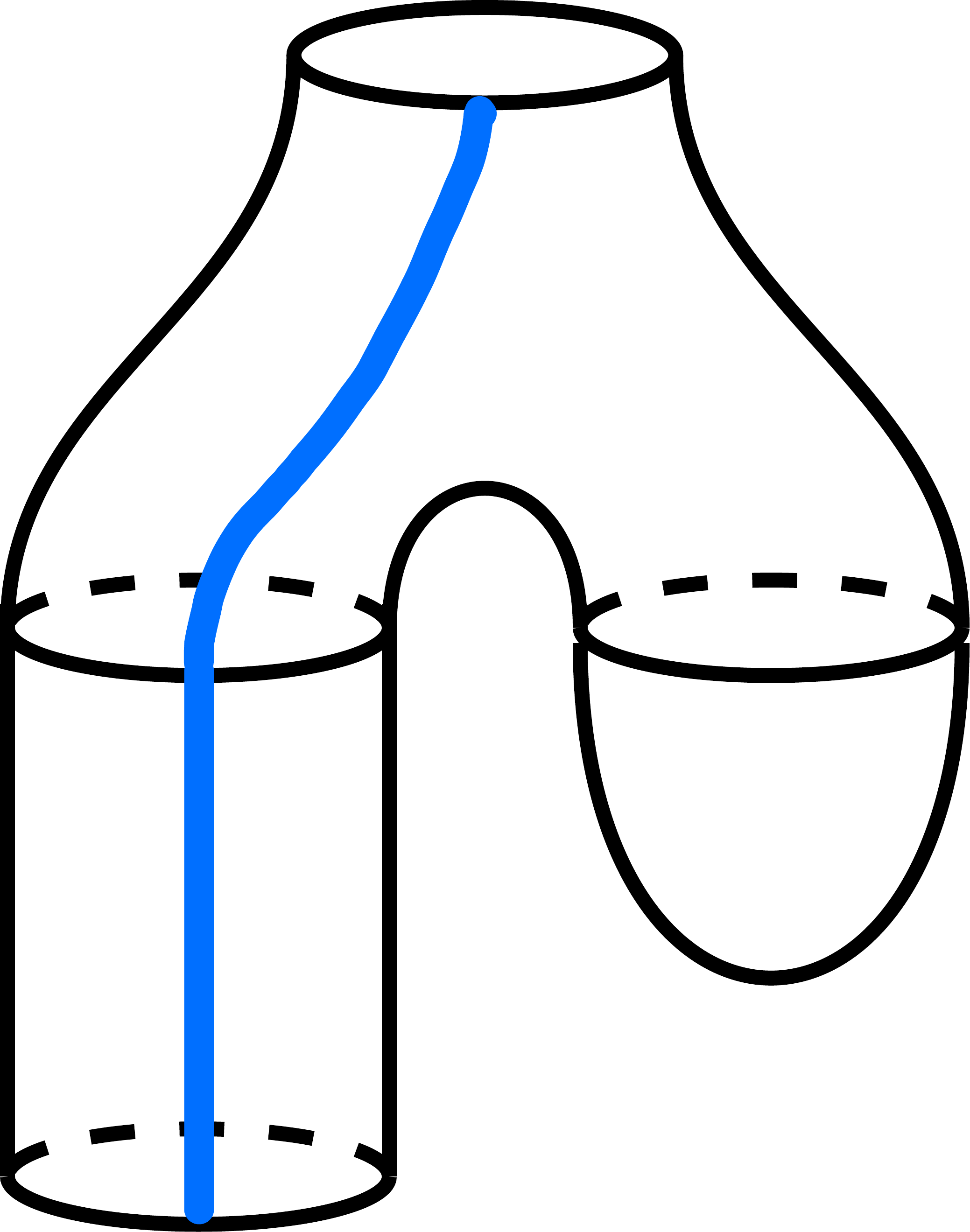}
\end{aligned}
\quad
\xmapsto{Z_{SN}(\rho)}
\quad
\begin{aligned}
\includegraphics[scale=0.05]{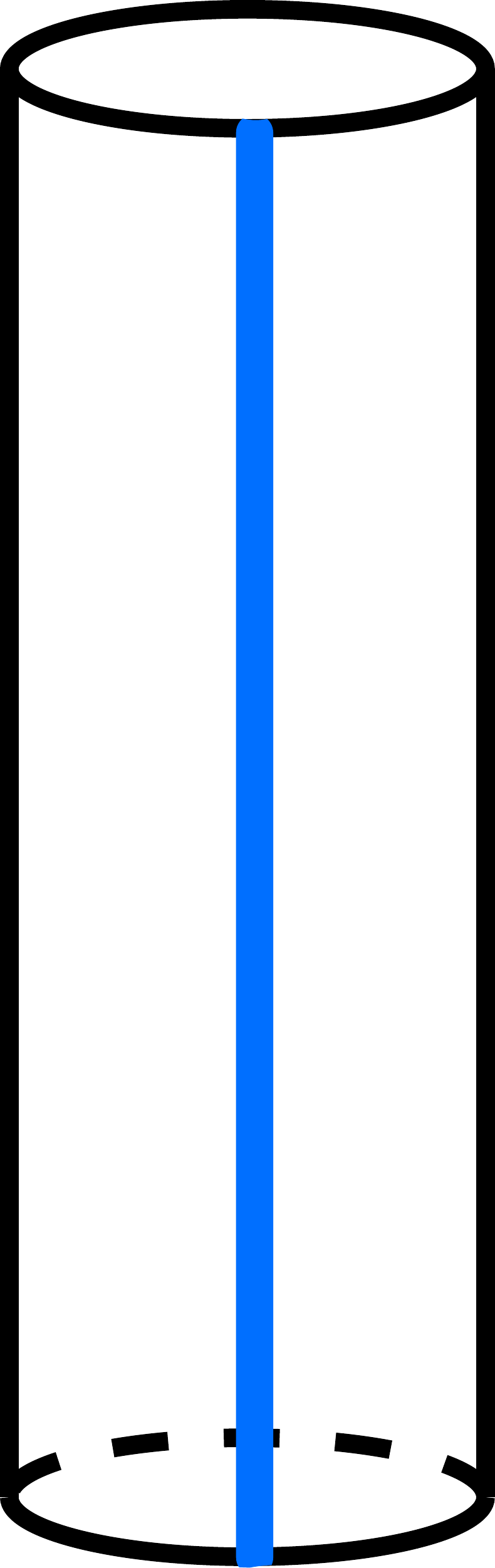}
\end{aligned}
\end{equation}

\begin{equation}
\begin{aligned}
\includegraphics[scale=0.05]{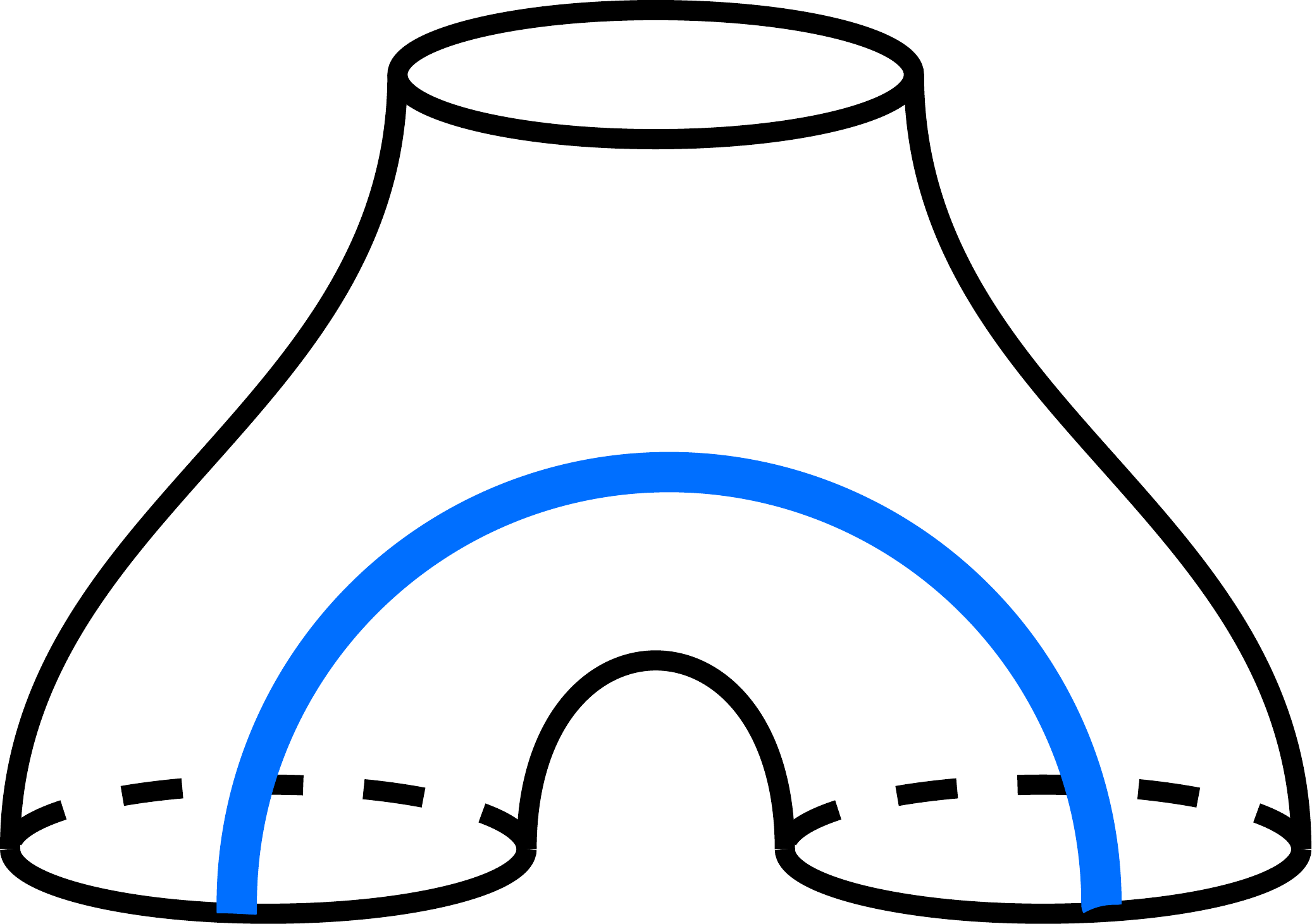}
\end{aligned}
\quad
\xmapsto{Z_{SN}(\beta)}
\quad
\begin{aligned}
\includegraphics[scale=0.05]{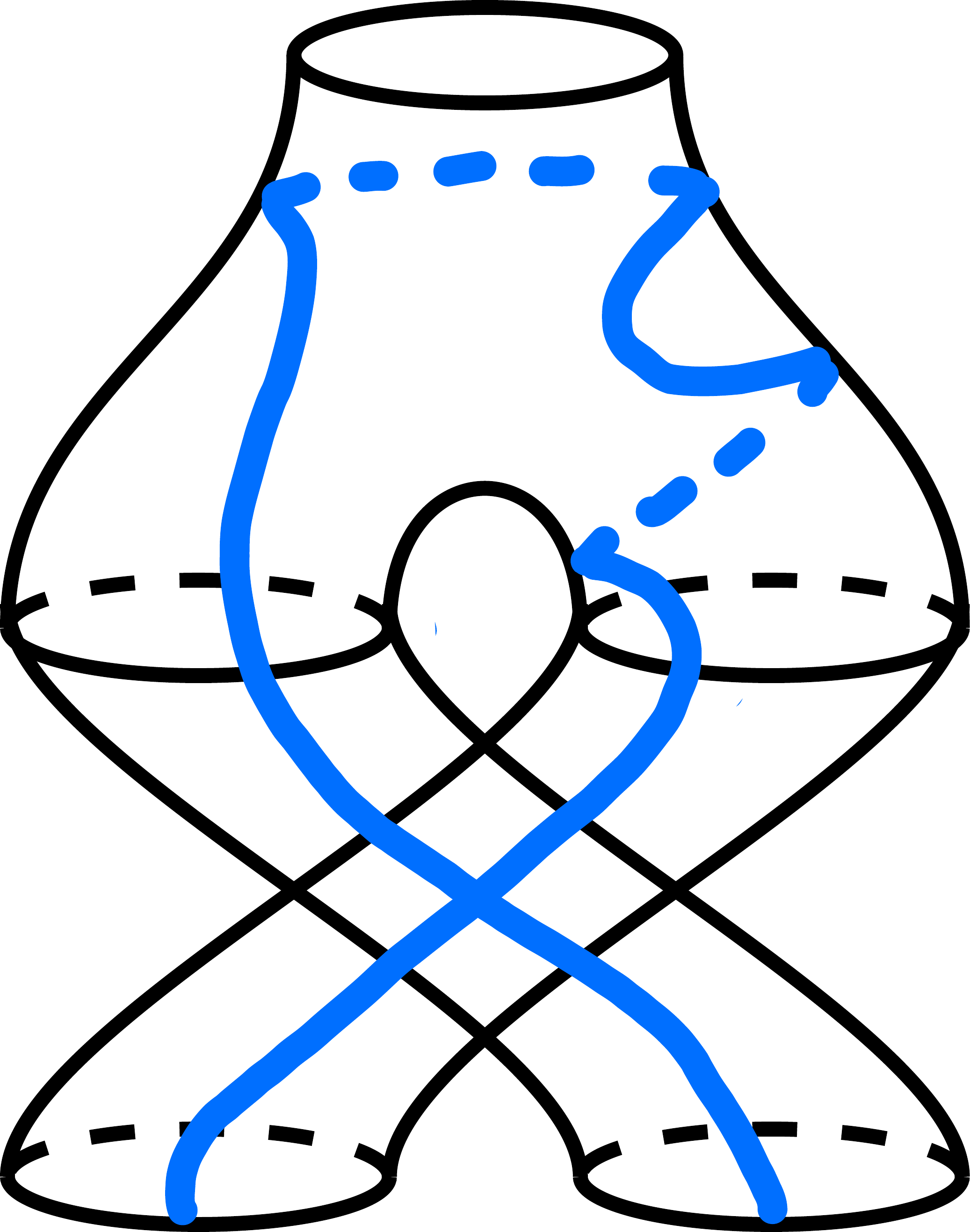}
\end{aligned}
\end{equation}

\begin{equation} \label{theta_map_action}
\begin{aligned}
\includegraphics[scale=0.05]{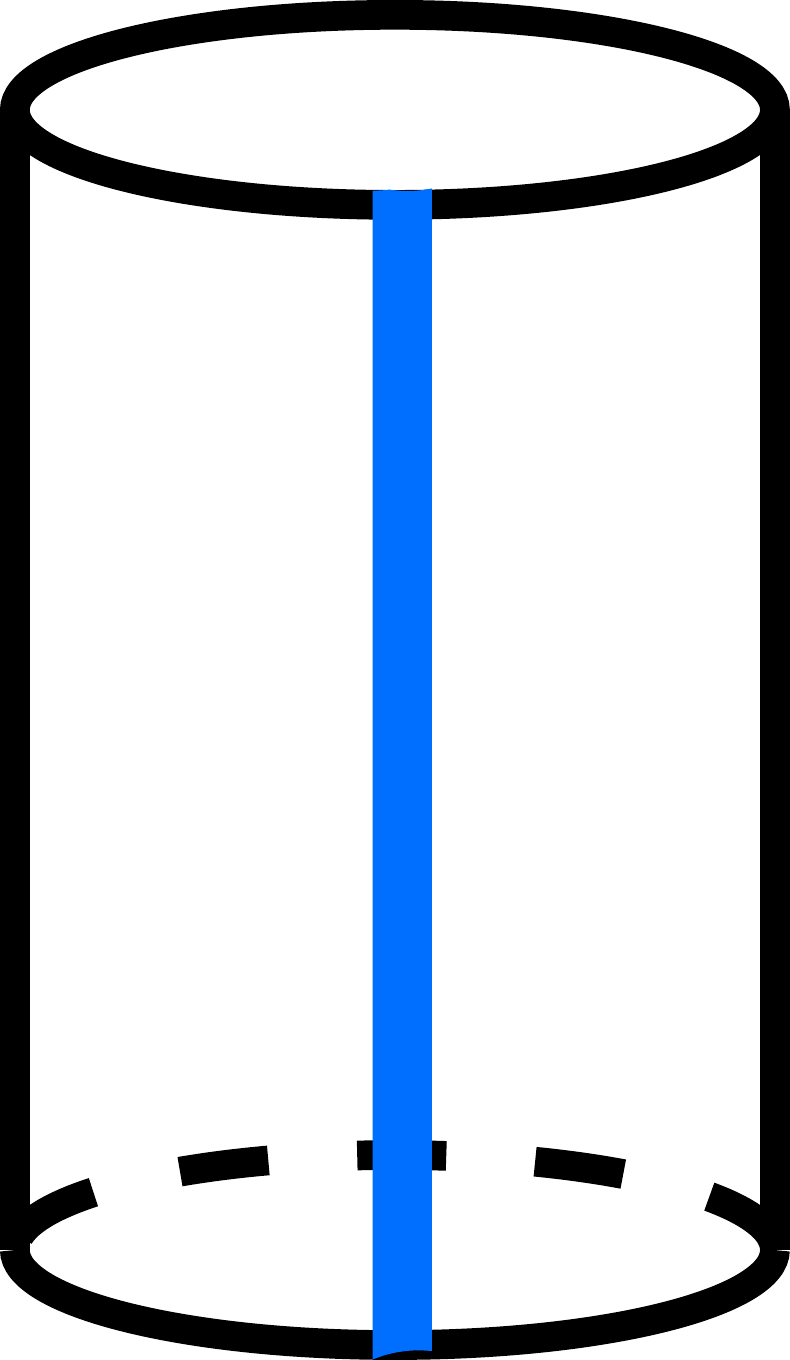}
\end{aligned}
\quad
\xmapsto{Z_{SN}(\theta)}
\quad
\begin{aligned}
\includegraphics[scale=0.05]{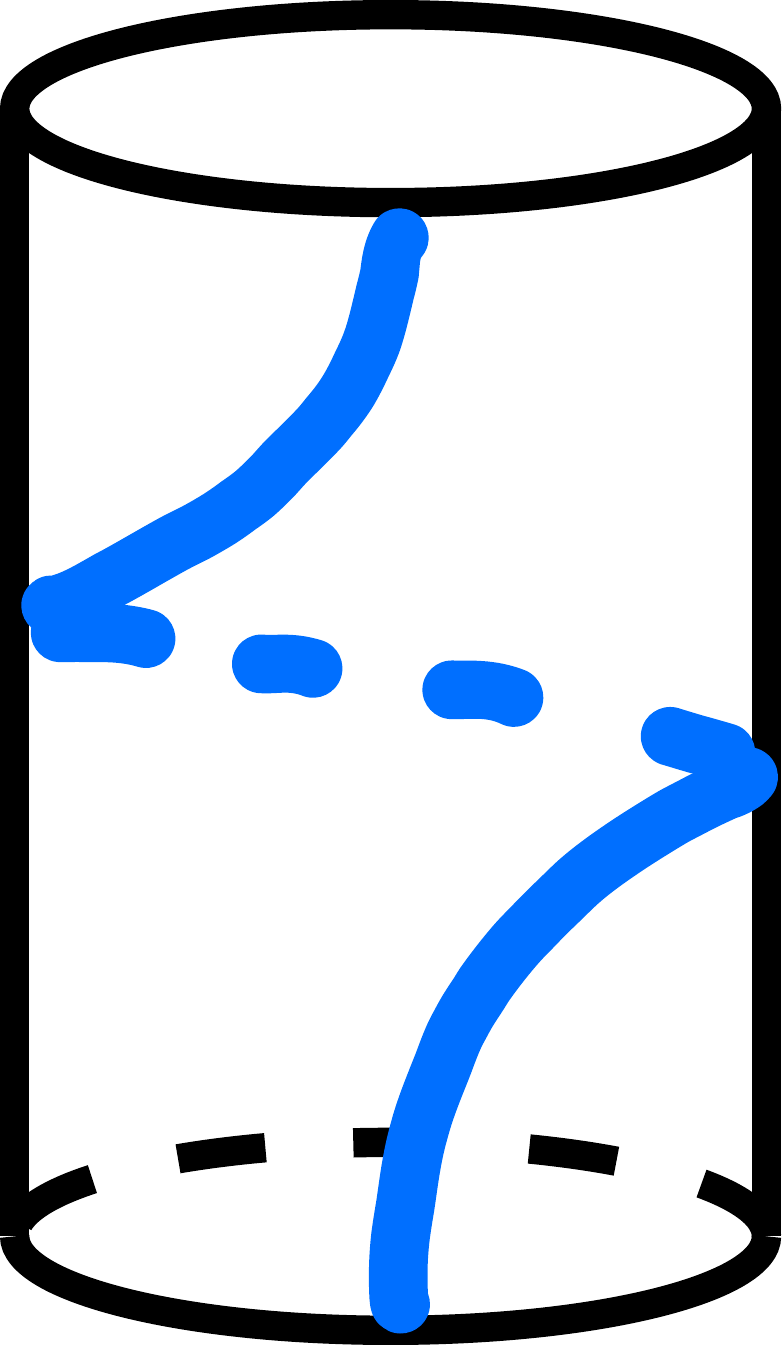}
\end{aligned}
\end{equation}

\subsubsection{The non-invertible 2-generators}  \label{noninvertible-2-gen}

We may divide the non-invertible 2-generators into three groups, based on their surgery actions.\footnote{See Equation \ref{fig:surgery_curves}.} 

\begin{itemize}

\item \textbf{Removing a 2-handle ($\mu$, $\eta$, and $\epsilon^{\dagger}$).} The first group consists of $\mu$, $\eta$, and $\epsilon^{\dagger}$, each of which represent the bordism implementing the removal of a 2-handle. That is, two disks are deleted and an annulus glued into their place. The recipe for how this procedure acts on the string-net space is the same in each case:

\begin{enumerate}
	\item Using isotopy, move all strings out of the area where the disks are to be removed.
	\item Cut out the disks, and glue in the annulus.
	\item Add an orange loop along the center of the annulus.
\end{enumerate}

\vspace{0.5cm}

\noindent Note that this action is \emph{locally} defined. Let us demonstrate the above procedure in ``slow motion'' for $\eta$:

\begin{equation}
	\begin{aligned}
	\includegraphics[scale=0.07]{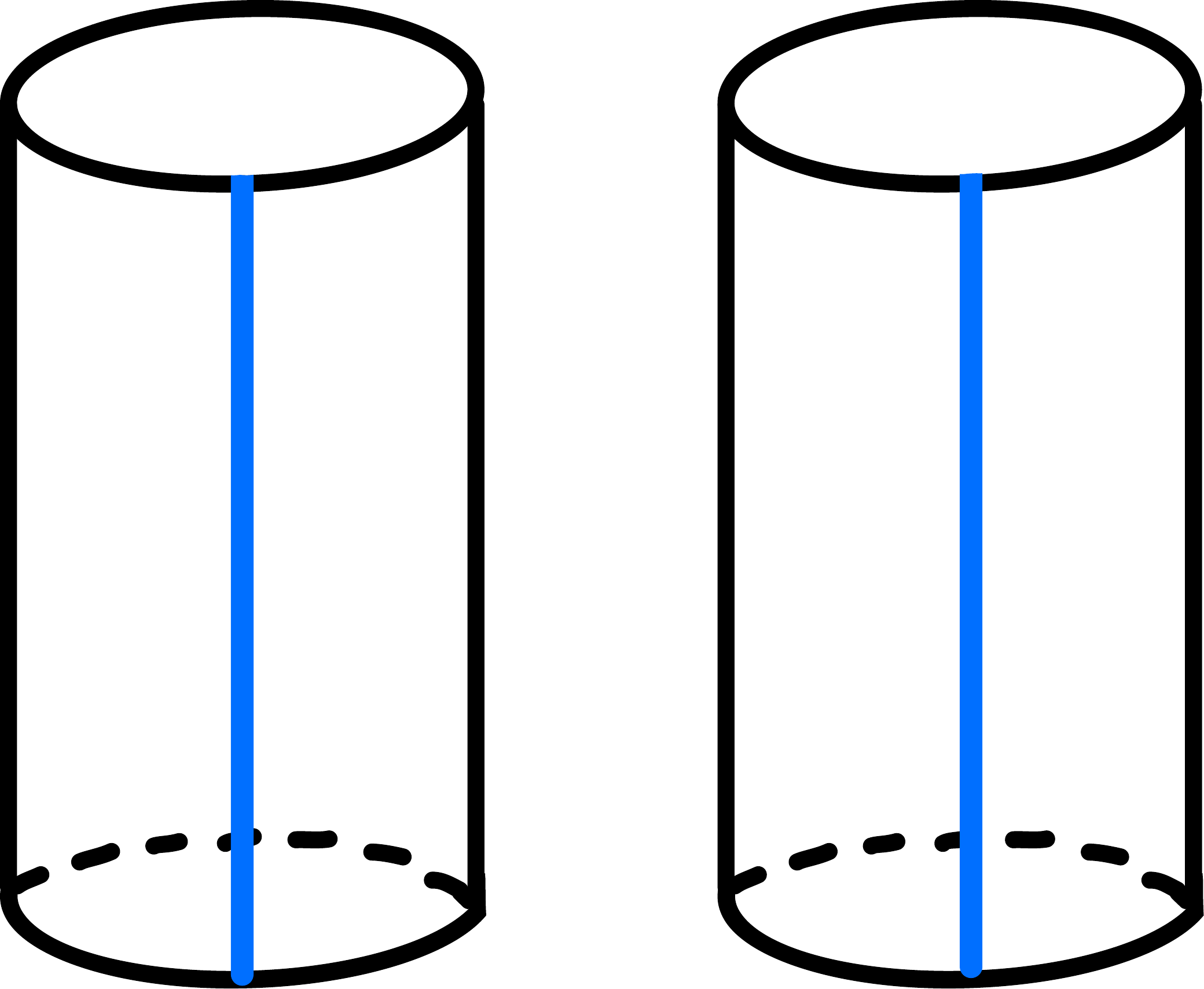}
	\end{aligned}
	\xrightarrow{\text{move}}
	\begin{aligned}
	\includegraphics[scale=0.07]{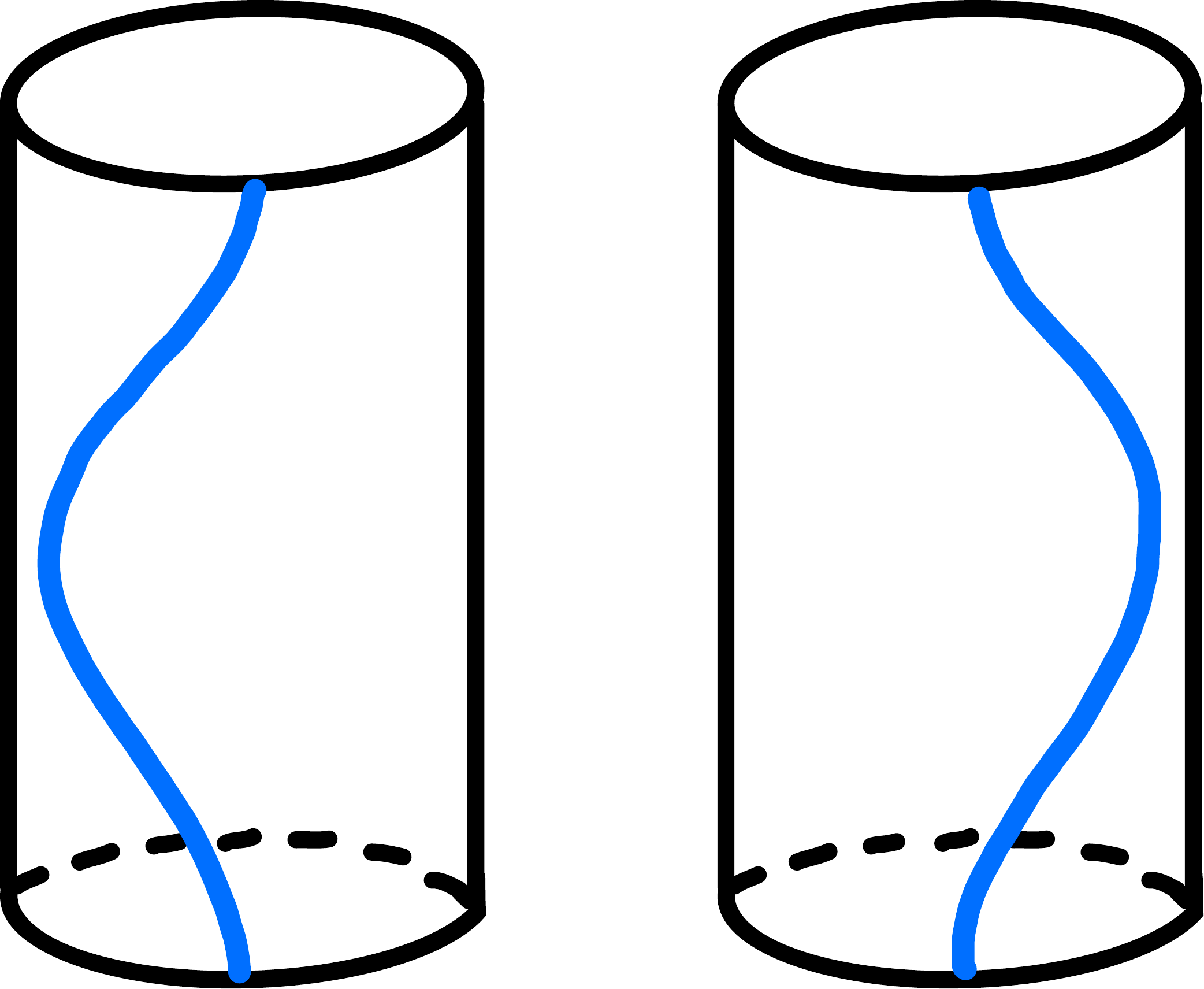}
	\end{aligned}
	\xrightarrow{\text{cut}}
	\begin{aligned}
	\includegraphics[scale=0.07]{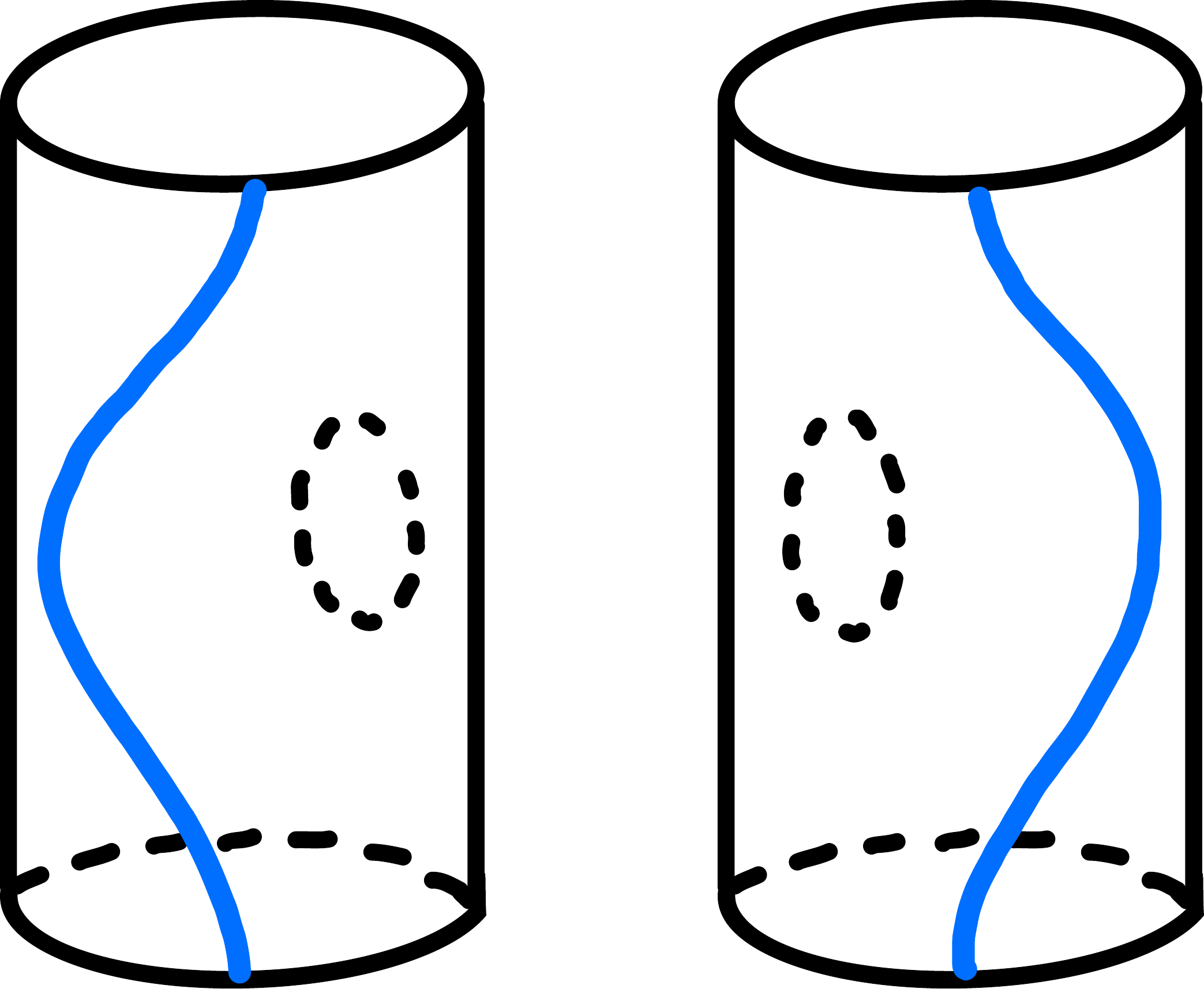}
	\end{aligned}
	\xrightarrow{\text{glue}}
	\begin{aligned}
	\includegraphics[scale=0.07]{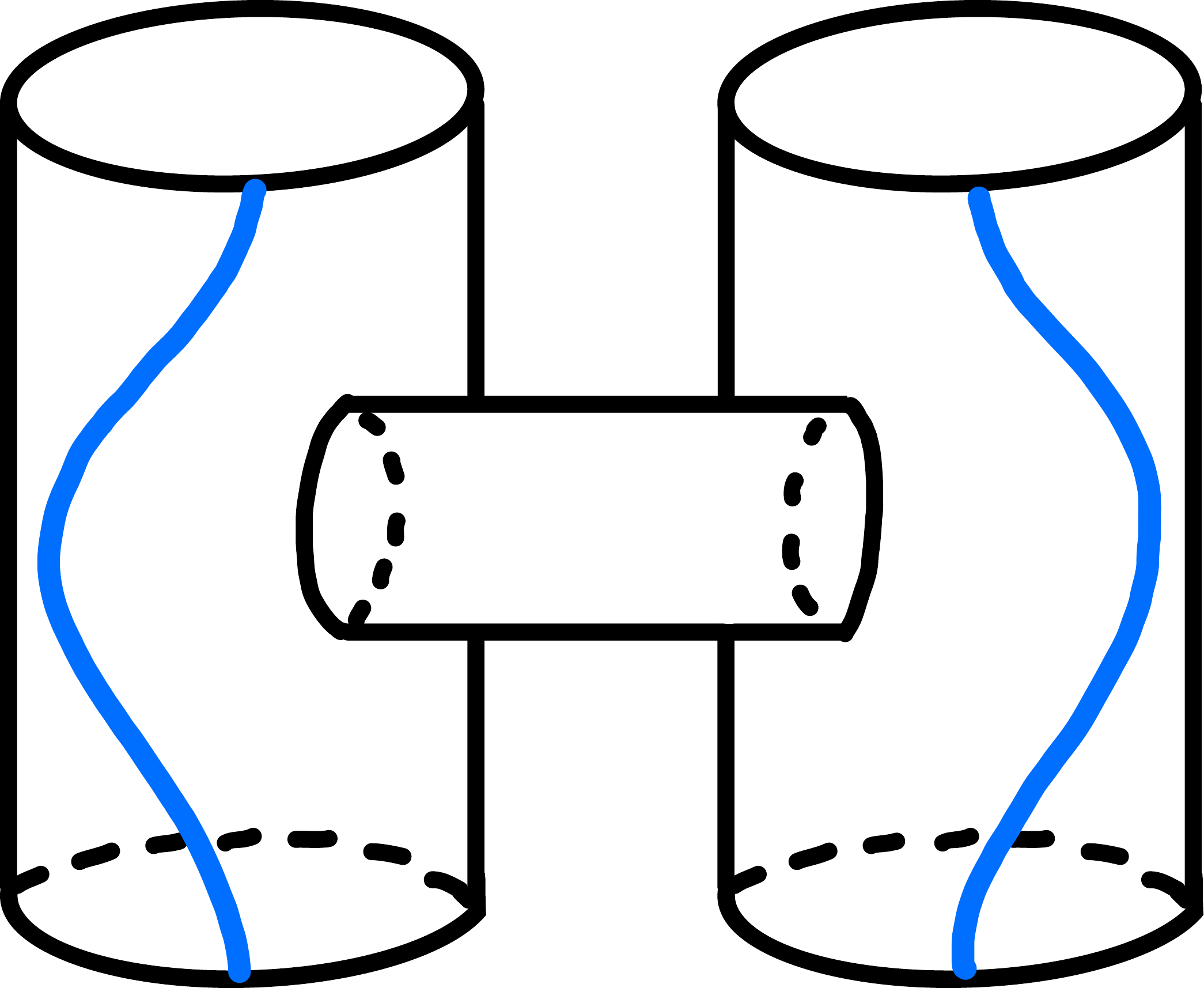}
	\end{aligned}
	\xrightarrow{\text{loop}}
	\begin{aligned}
	\includegraphics[scale=0.07]{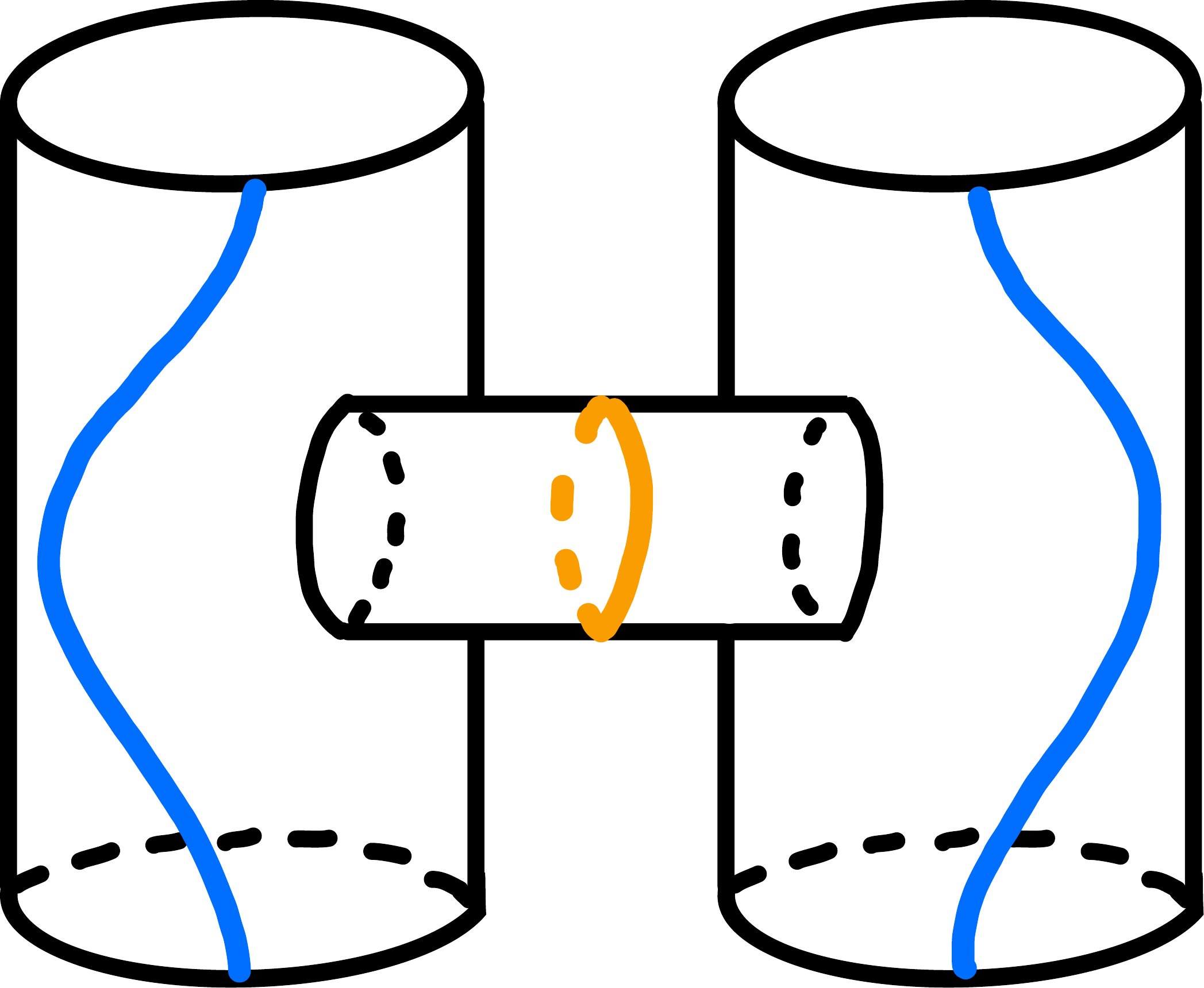}
	\end{aligned}
\end{equation}

\begin{remark}
Note that, via the \nameref{lem:cloaking} lemma, the presence of the orange loop removes the ambiguity in the manner in which strings are moved out of the cutting area in the first step.
\end{remark}

\noindent Below are some (global) examples showing how to use the (local) definition above.\\


\begin{equation}
	\begin{aligned}
	\includegraphics[scale=0.05]{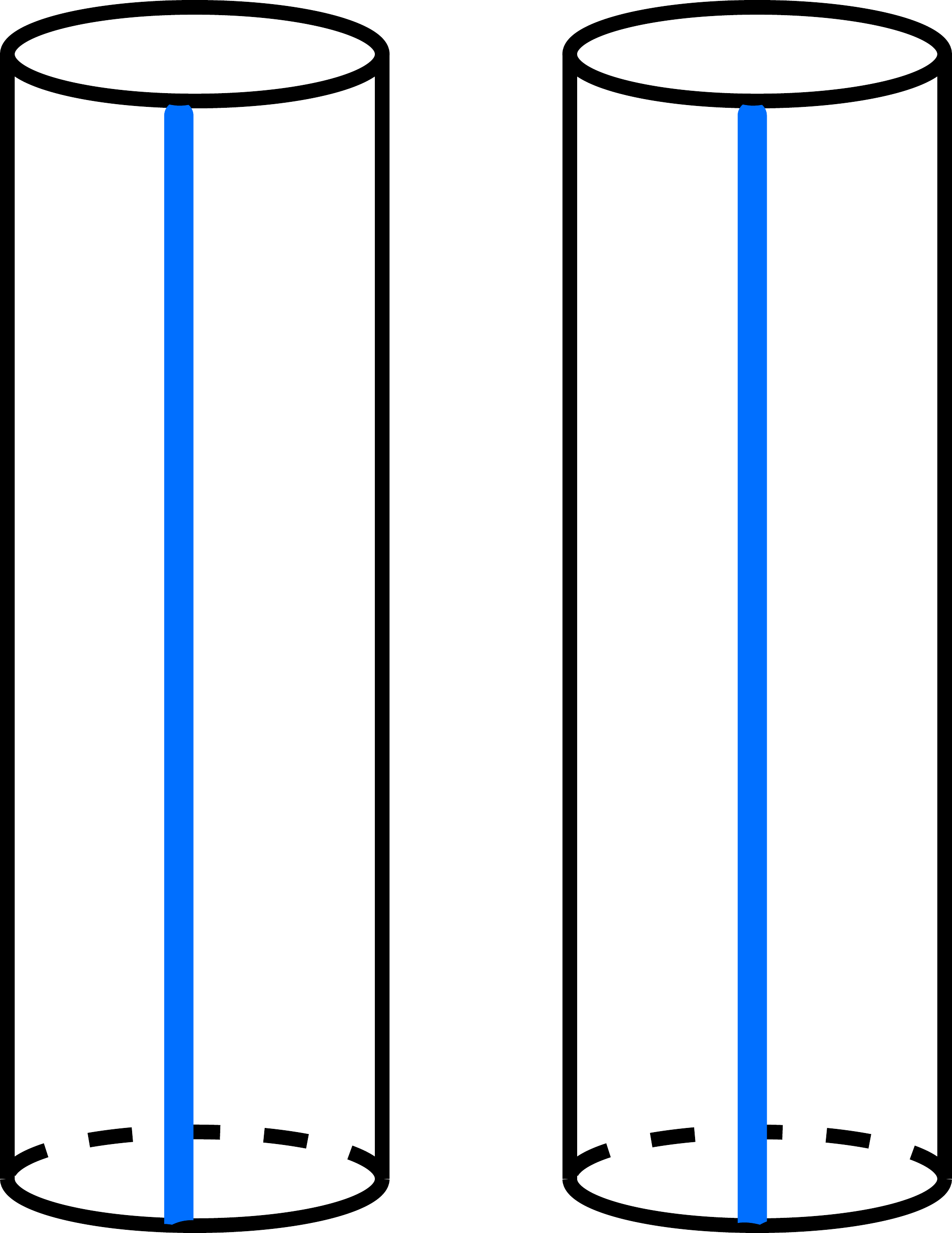}
	\end{aligned}
	\quad
	\xmapsto{Z_{SN}(\eta)}
	\quad
	\begin{aligned}
	\includegraphics[scale=0.05]{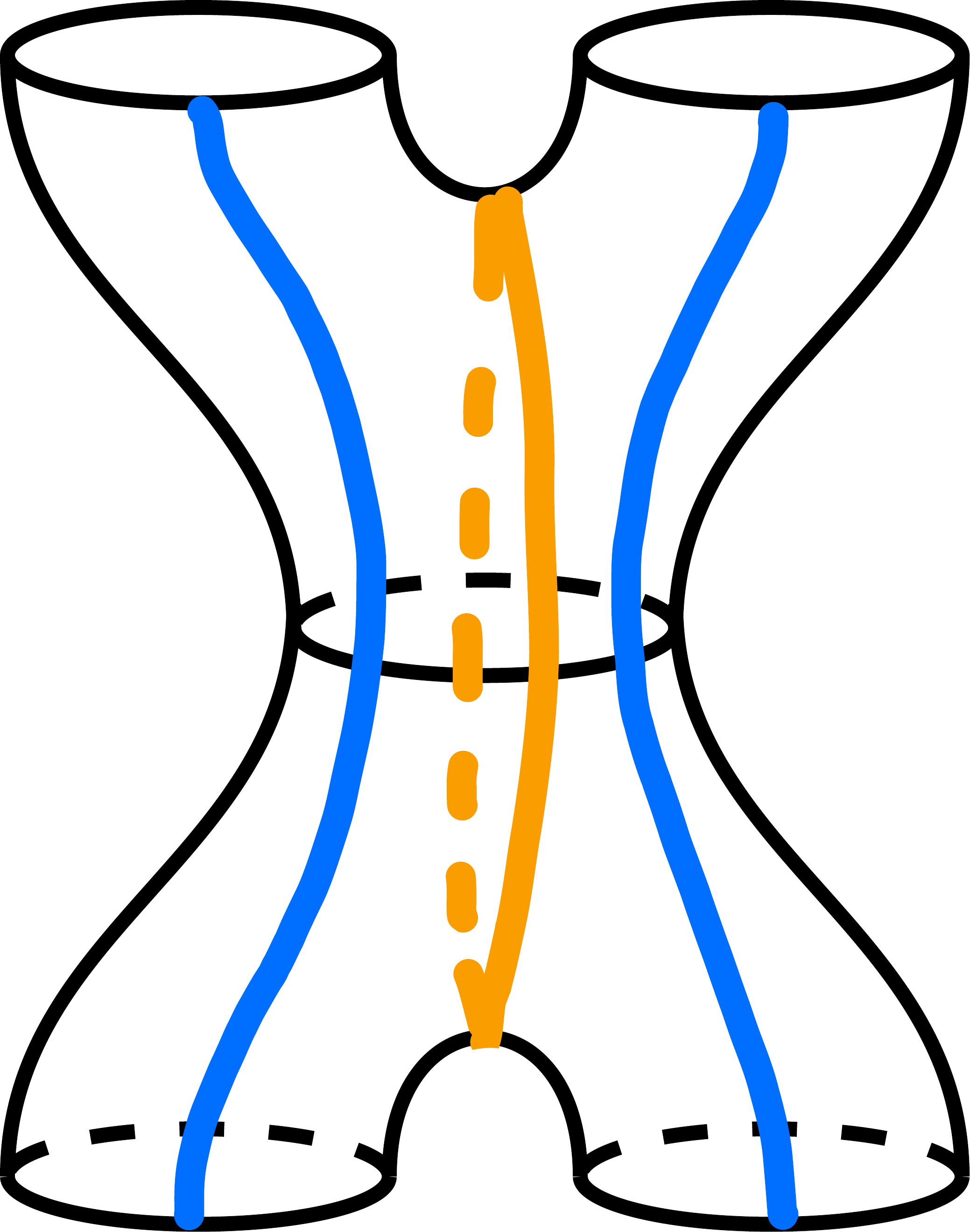}
	\end{aligned}
\end{equation}

\begin{equation}  \label{epsilon_dag_example_map}
	\begin{aligned}
	\includegraphics[scale=0.05]{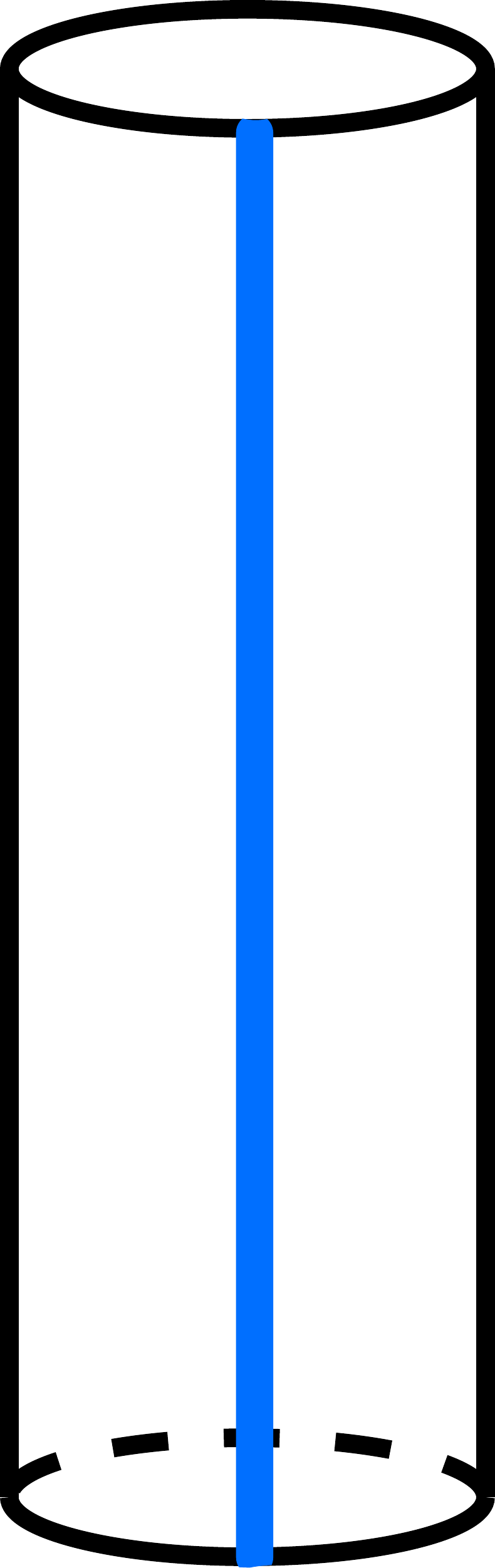}
	\end{aligned}
	\quad
	\xmapsto{Z_{SN}(\epsilon^{\dagger})}
	\quad
	\begin{aligned}
	\includegraphics[scale=0.05]{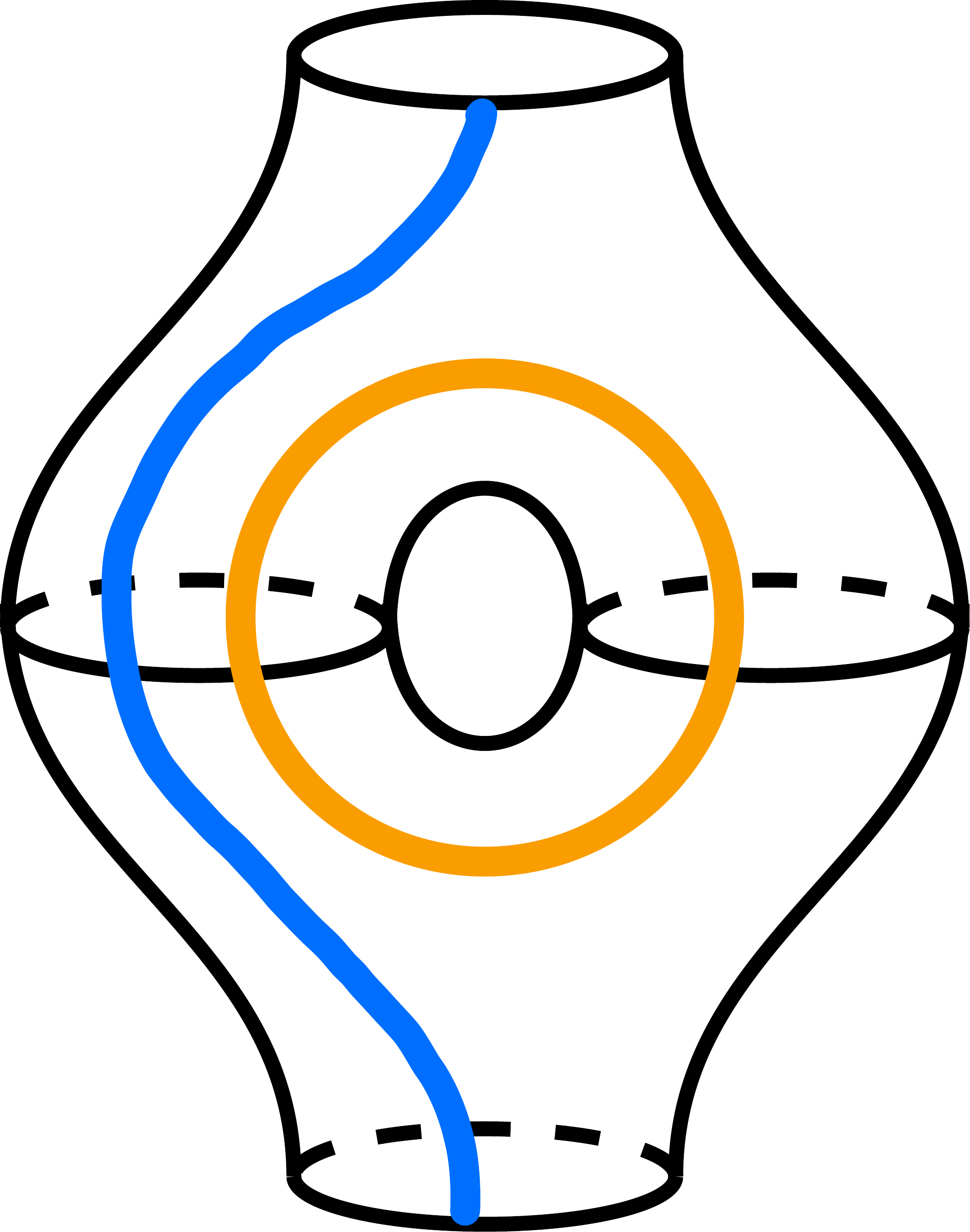}
	\end{aligned}
\end{equation}

\begin{equation}
	\begin{aligned}
	\includegraphics[scale=0.05]{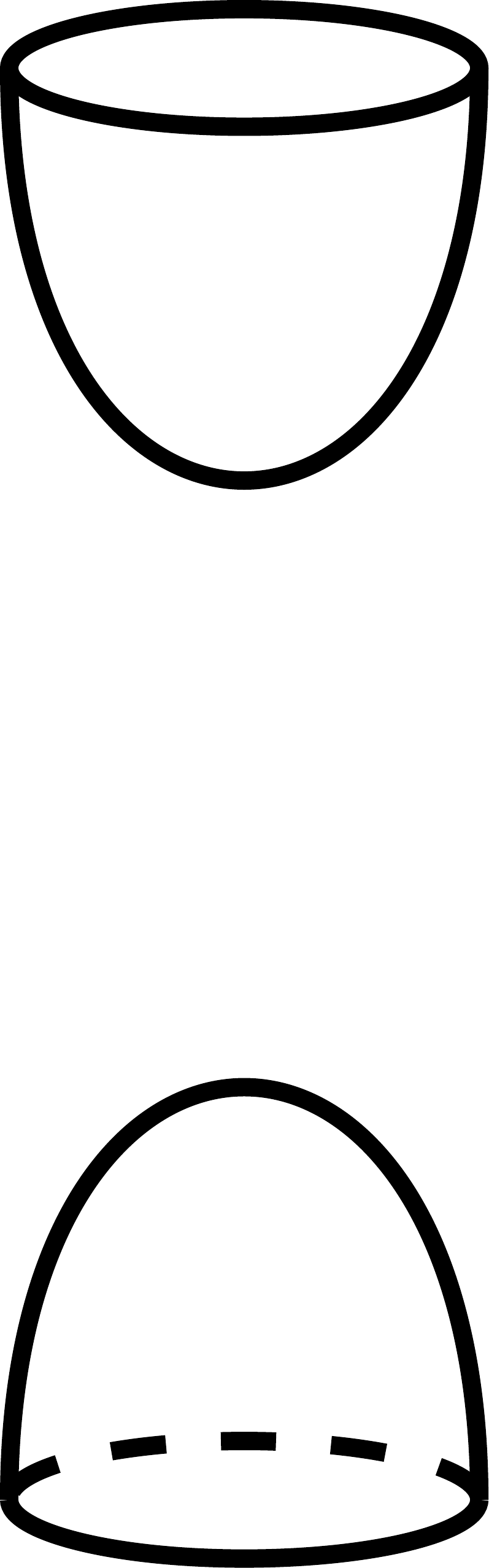}
	\end{aligned}
	\quad
	\xmapsto{Z_{SN}(\mu)}
	\quad
	\begin{aligned}
	\includegraphics[scale=0.05]{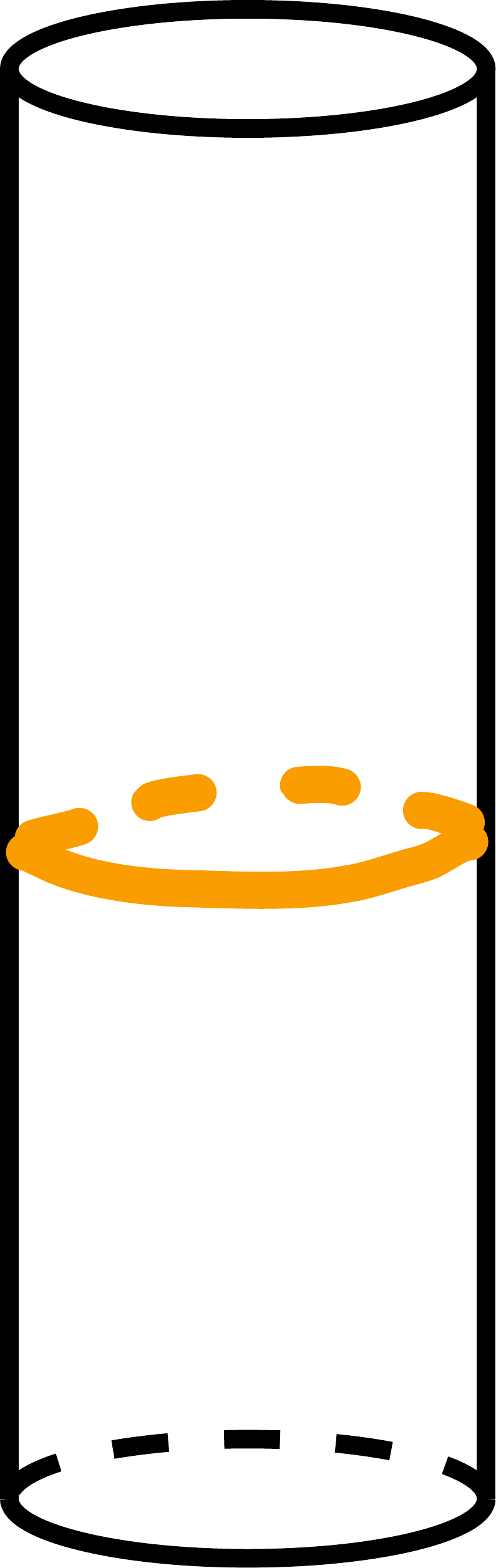}
	\end{aligned}
\end{equation}

\vspace{1cm}

\item \textbf{Adding a 2-handle ($\mu^{\dagger}$, $\eta^{\dagger}$, and $\epsilon$).} The second group of non-invertible 2-generators consists of $\mu^{\dagger}$, $\eta^{\dagger}$, and $\epsilon$, each of which represent the bordism implementing the addition of a 2-handle. That is, an annular region of the surface is excised and disks glued on in its place. Now because the annulus has nonzero genus, it is not always possible to ``evacuate'' strings safely before the annular region is removed. We interpret this as the trapped strings getting cut. In terms of the relations, if there are an even number of strings running through the annular region, then we cut by applying the F-move repeatedly. If there are an odd number of strings running through the annular region, then the string-net will get annihilated. We extend this procedure linearly over sums of embedded graphs.\\

\noindent Below are some (global) examples showing how to use the (local) definition above.\\

\begin{equation}
\begin{aligned}
\includegraphics[scale=0.75]{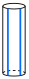}
\end{aligned}
\quad
\xmapsto{Z_{SN}(\mu^{\dagger})}
\quad
\begin{aligned}
\includegraphics[scale=0.75]{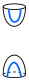}
\end{aligned}
\end{equation}

\begin{equation}
\begin{aligned}
\includegraphics[scale=0.05]{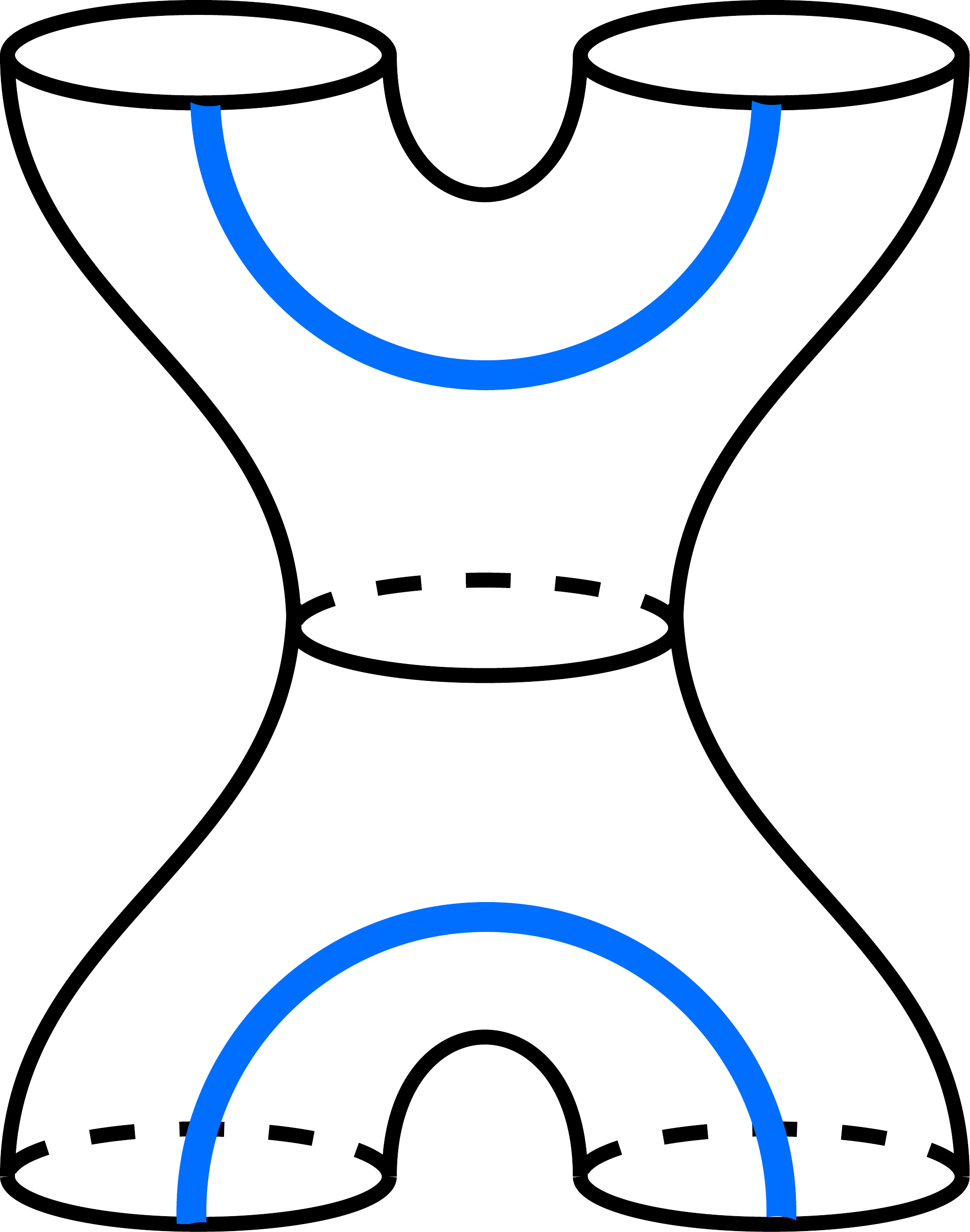}
\end{aligned}
\quad
\xmapsto{Z_{SN}(\eta^{\dagger})}
\quad
\begin{aligned}
\includegraphics[scale=0.05]{etadag_target.pdf}
\end{aligned}
\end{equation}

\begin{equation}
\begin{aligned}
\includegraphics[scale=0.05]{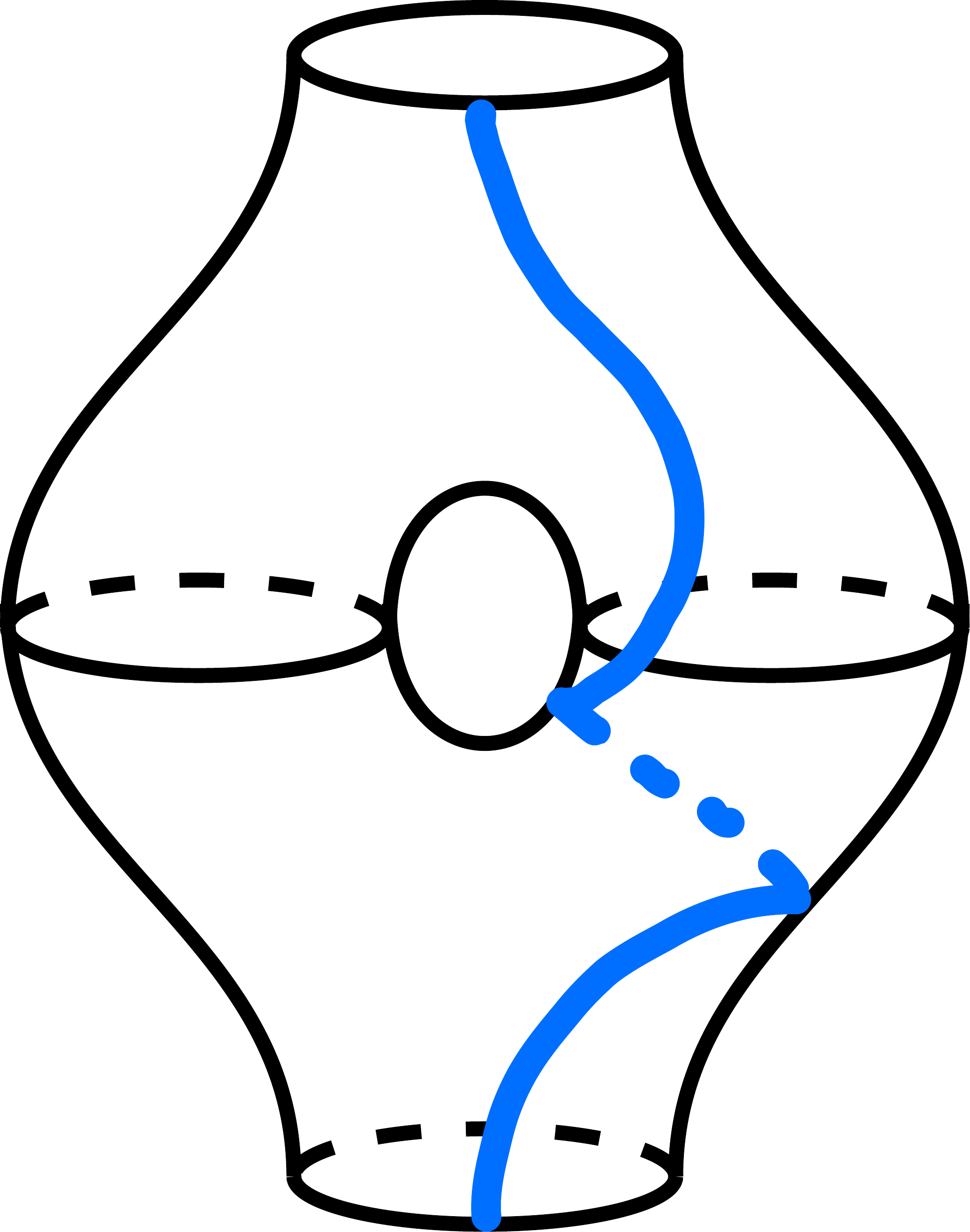}
\end{aligned}
\quad
\xmapsto{Z_{SN}(\epsilon)}
\quad
\begin{aligned}
0
\end{aligned}
\end{equation}

\item \textbf{Removing and adding a 3-handle ($\nu$ and $\nu^\dagger$).} The final group, consisting of $\nu$ and $\nu^\dagger$, have source and target string-net spaces each isomorphic to $\CC$, so in this case the actions are easier to guess.

\vspace{0.25cm}

\scalecobordisms{1}
\begin{equation}
\begin{tz}
\draw [green] (0,0) rectangle (1,1);
\end{tz}				
\quad
\xmapsto{~Z_{SN}(\nu)~}
\quad
\raisebox{0.5mm}{$\frac{1}{2}$}~
\begin{tz}
\node [Cap, bot] at (0,0) {};
\node [Cup] at (0,0) {};
\end{tz}
\qquad
\qquad
\begin{tz}
\node [Cap, bot] at (0,0) {};
\node [Cup] at (0,0) {};
\end{tz}
\quad
\xmapsto{~Z_{SN}(\nu^\dagger)~}
\quad
\begin{tz}
\draw [green] (0,0) rectangle (1,1);
\end{tz}
\end{equation}

The diagrams on the source and target (for $\nu$ or $\nu^{\dagger}$) represent the trivial string-nets on their respective spaces. The factor $\frac{1}{2}$ appearing above is difficult to motivate merely from the topology represented by $\nu$ - it is forced upon us by the adjunction relations.


\end{itemize}

\subsection{The Relations}

Finally, we come to the main result in this section where we show that, under the preceding definitions for the actions of the 2-generators, all the relations of the presentation for $\Bord$ are satisfied. 

\begin{theorem}
The relations are satisfied, and hence $Z_{SN}$ defines a 1-2-3 TQFT by Corollary \ref{co:123TQFTdata}.
\end{theorem}

\noindent The proof of the theorem above occupies the remainder of the section.

\subsubsection{Monoidal} Each of these relations involve only actions which push string-nets forward along a diffeomorphim isotopic to the identity. It follows that these relations are trivially satisfied.

\subsubsection{Balanced} Each of these relations also only involve pushing forward along diffeomorphisms; that these diffeomorphisms commute follows from standard facts about mapping class groups.\footnote{See \cite{farb2011primer}.}

\subsubsection{Modularity} It follows trivially from the calculations in Section \textcolor{red}{5.1} that 

\[
H(\tinycyl; B_0 \boxtimes B_0) = \text{span}\left\{\raisebox{-4.5mm}{\includegraphics[scale=0.04]{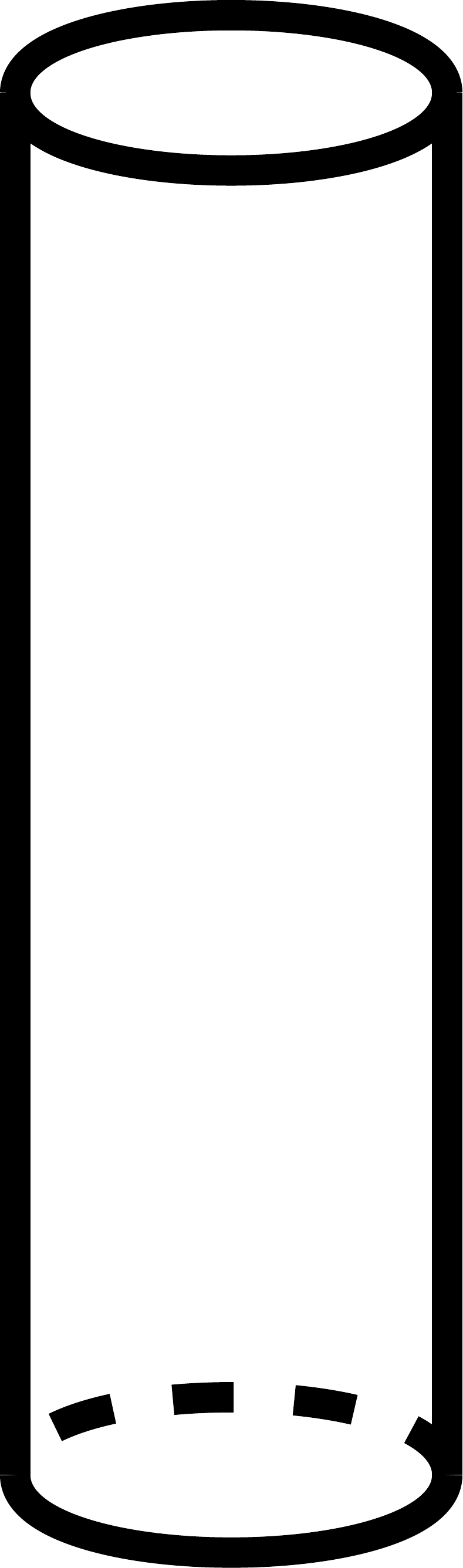}} \; , \raisebox{-4.5mm}{\includegraphics[scale=0.04]{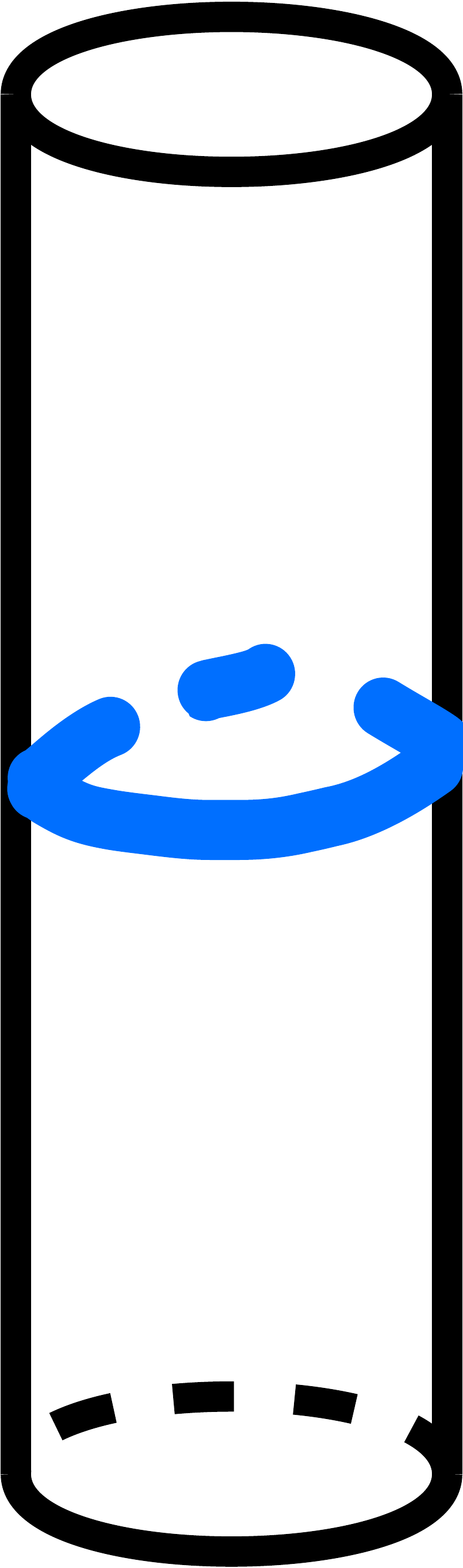}} \right\}
\]

\vspace{0.15cm}

\[
H(\tinycyl; B_1 \boxtimes B_1) = \text{span}\left\{\raisebox{-4.5mm}{\includegraphics[scale=0.04]{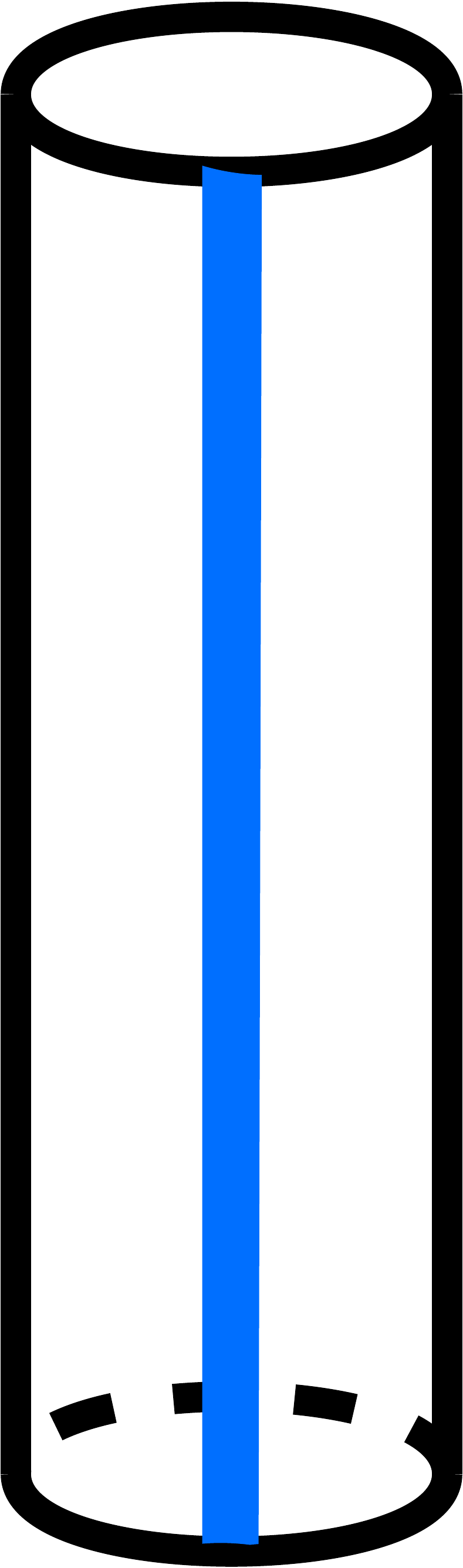}} \; , \raisebox{-4.5mm}{\includegraphics[scale=0.04]{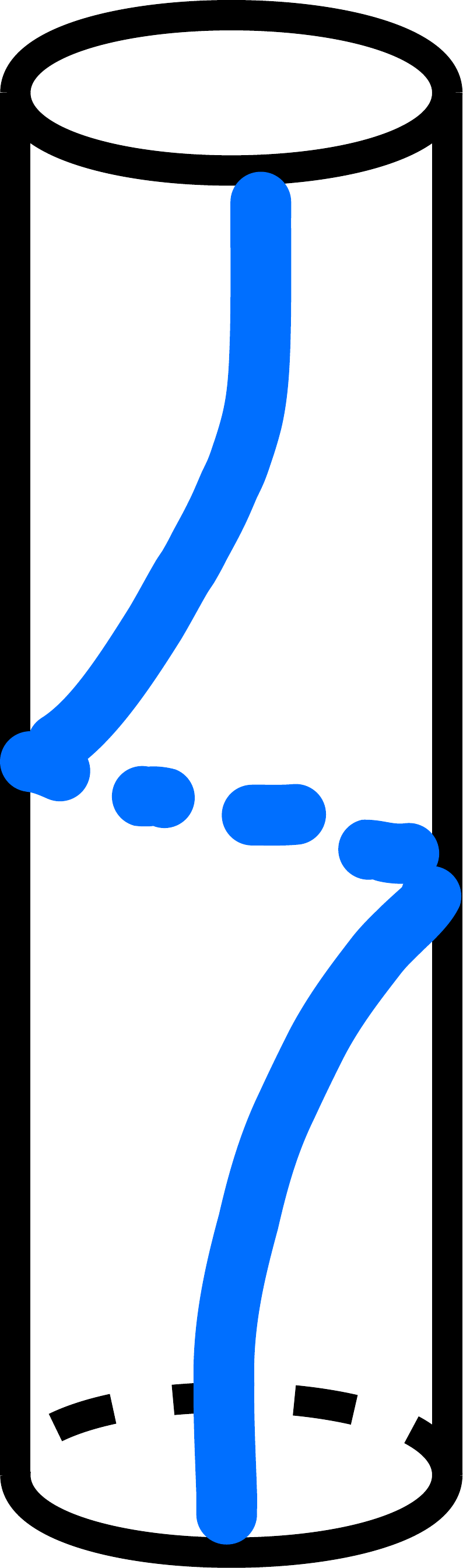}} \right\}
\]

\vspace{0.15cm}

\[
H(\tinycyl; B_0 \boxtimes B_1) = H(\tinycyl; B_1 \boxtimes B_0) = \{0\}
\]
We check that the relations are satisfied on the two relevant sets of basis vectors, and then extend linearly to the full space.

\newpage

\noindent The ``top'' composites are

	    \begin{align*}
			\begin{aligned}
			\includegraphics[scale=0.04]{modularity_1.pdf}
			\end{aligned}
			\quad
			\xmapsto{~Z_{SN}(\epsilon^{\dagger})~}
			\quad
			\begin{aligned}
			\includegraphics[scale=0.04]{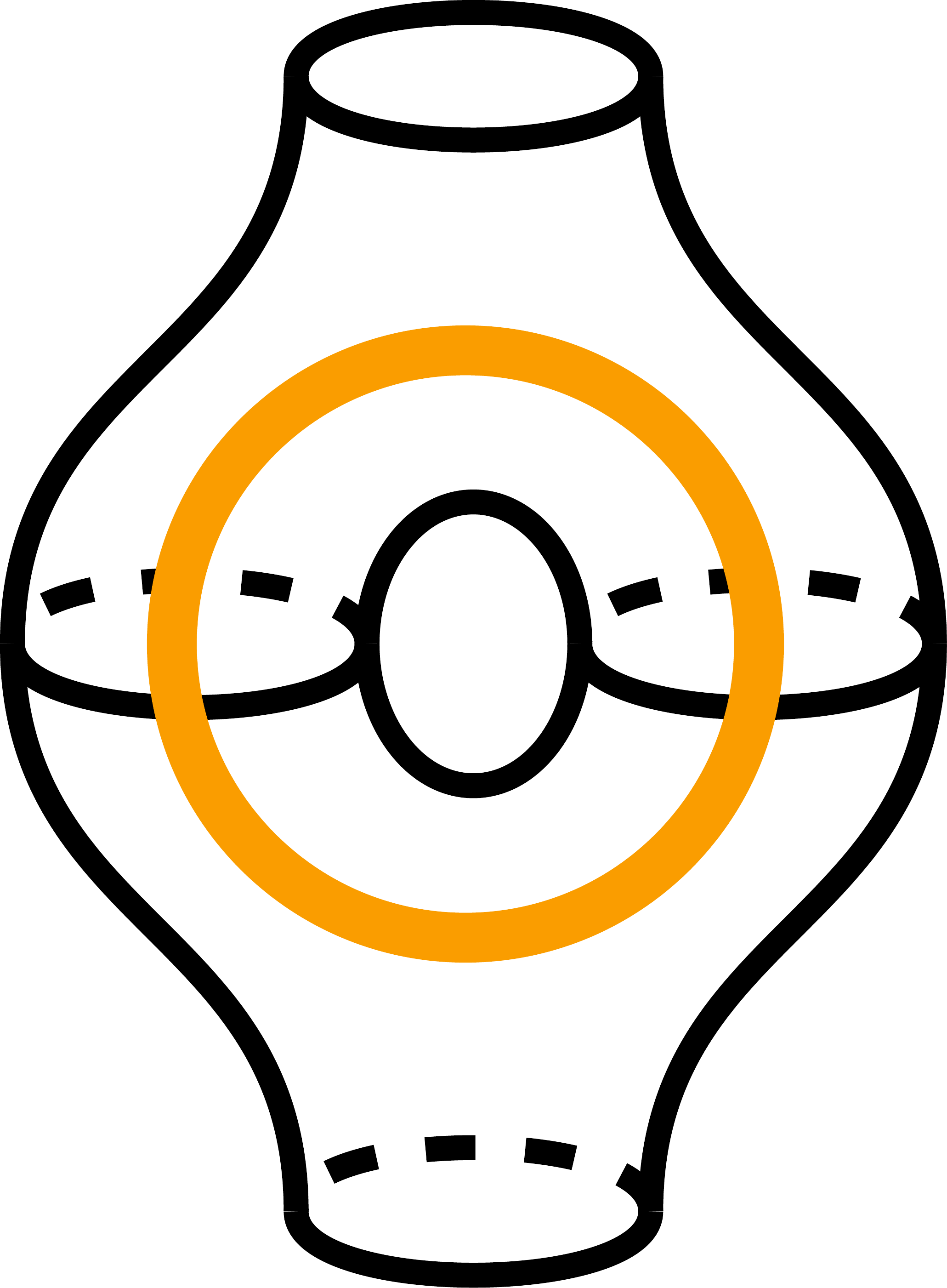}
			\end{aligned}
			\quad
			&\xmapsto{~Z_{SN}(\theta),\,Z_{SN}(\theta^{\inv})~}
			\quad
			\begin{aligned}
			\includegraphics[scale=0.04]{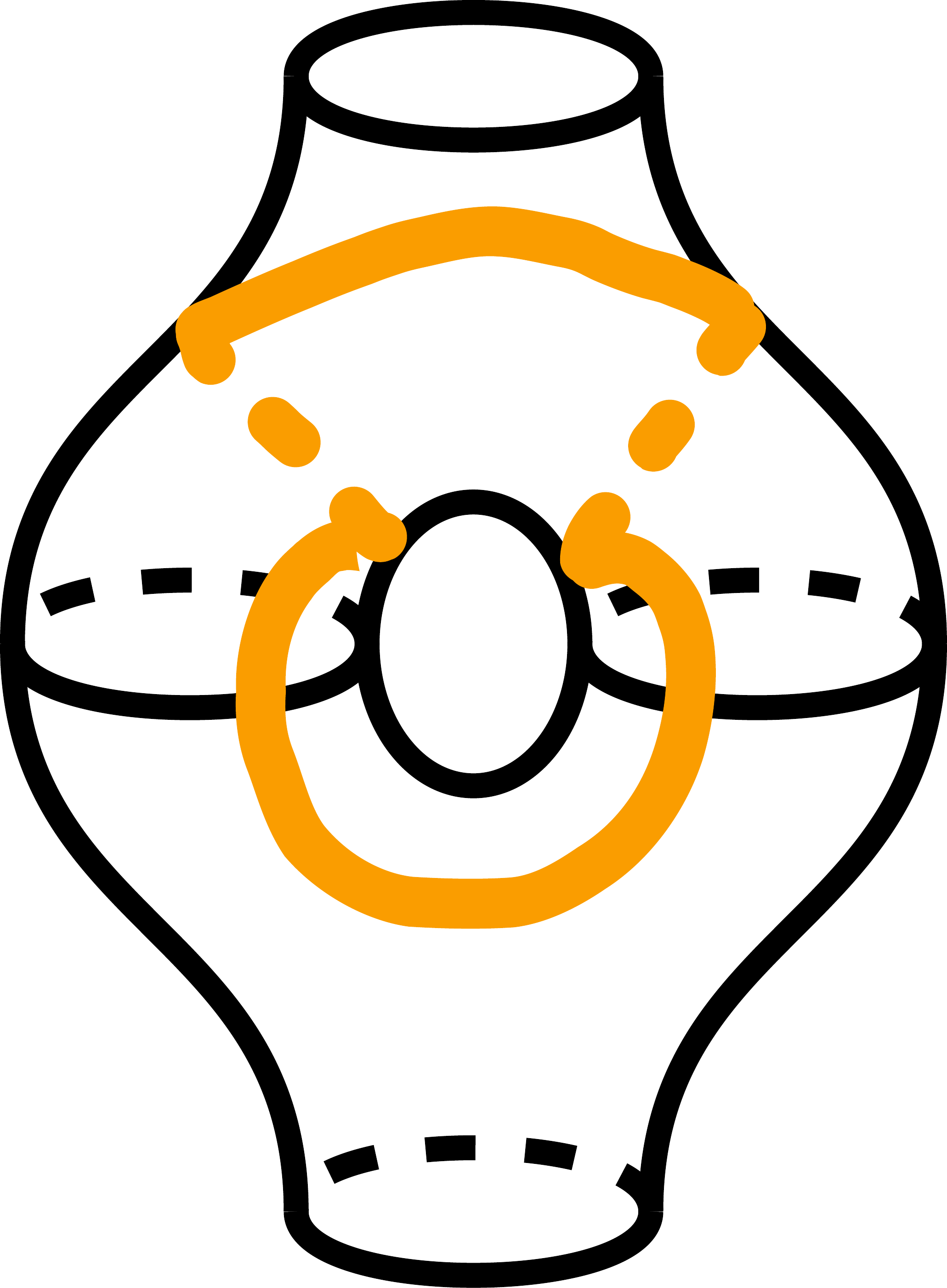}
			\end{aligned} \\
			&=
			\quad
			\begin{aligned}
			\includegraphics[scale=0.04]{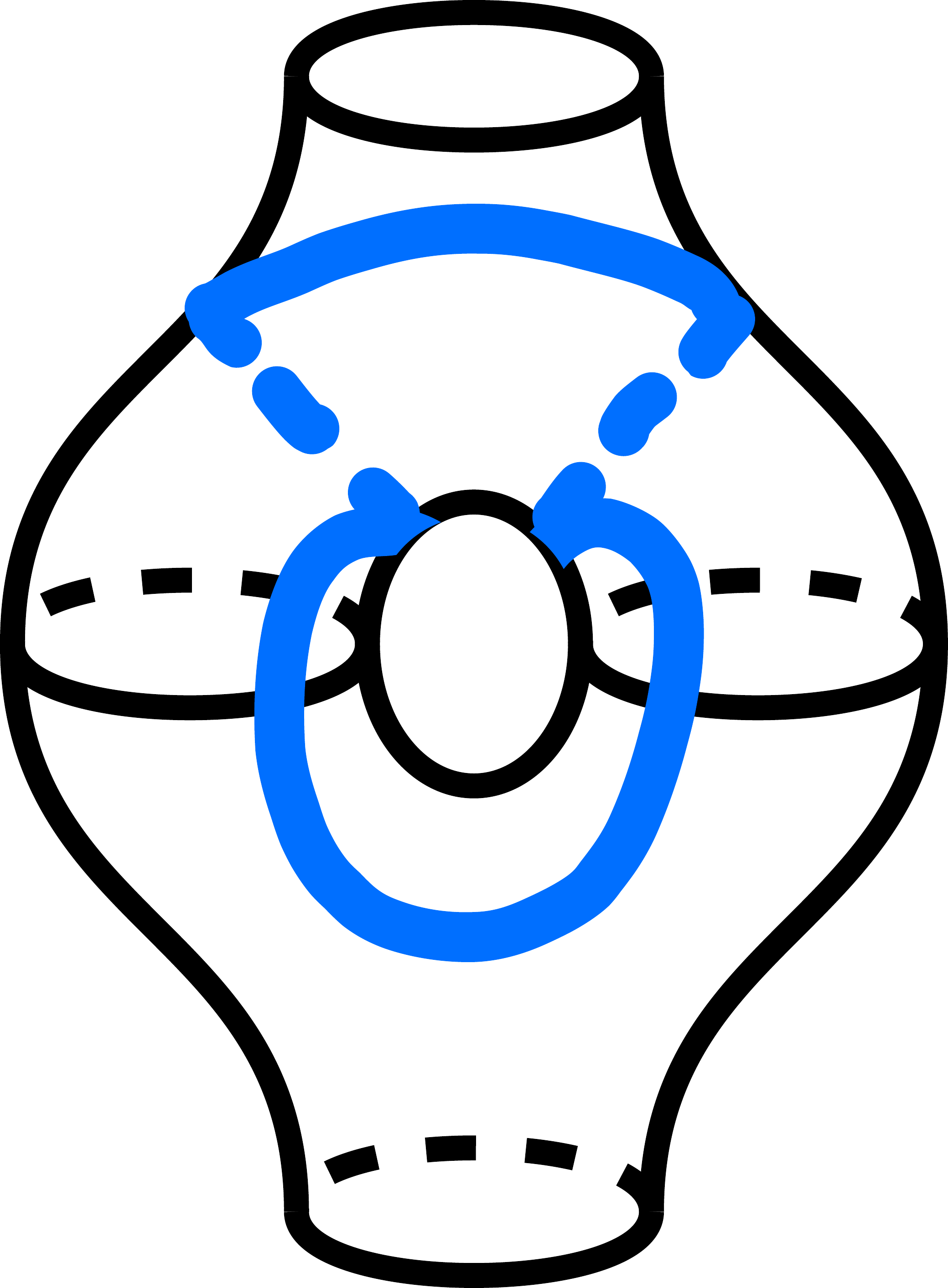}
			\end{aligned}
			\quad
			+
			\quad
			\begin{aligned}
			\includegraphics[scale=0.04]{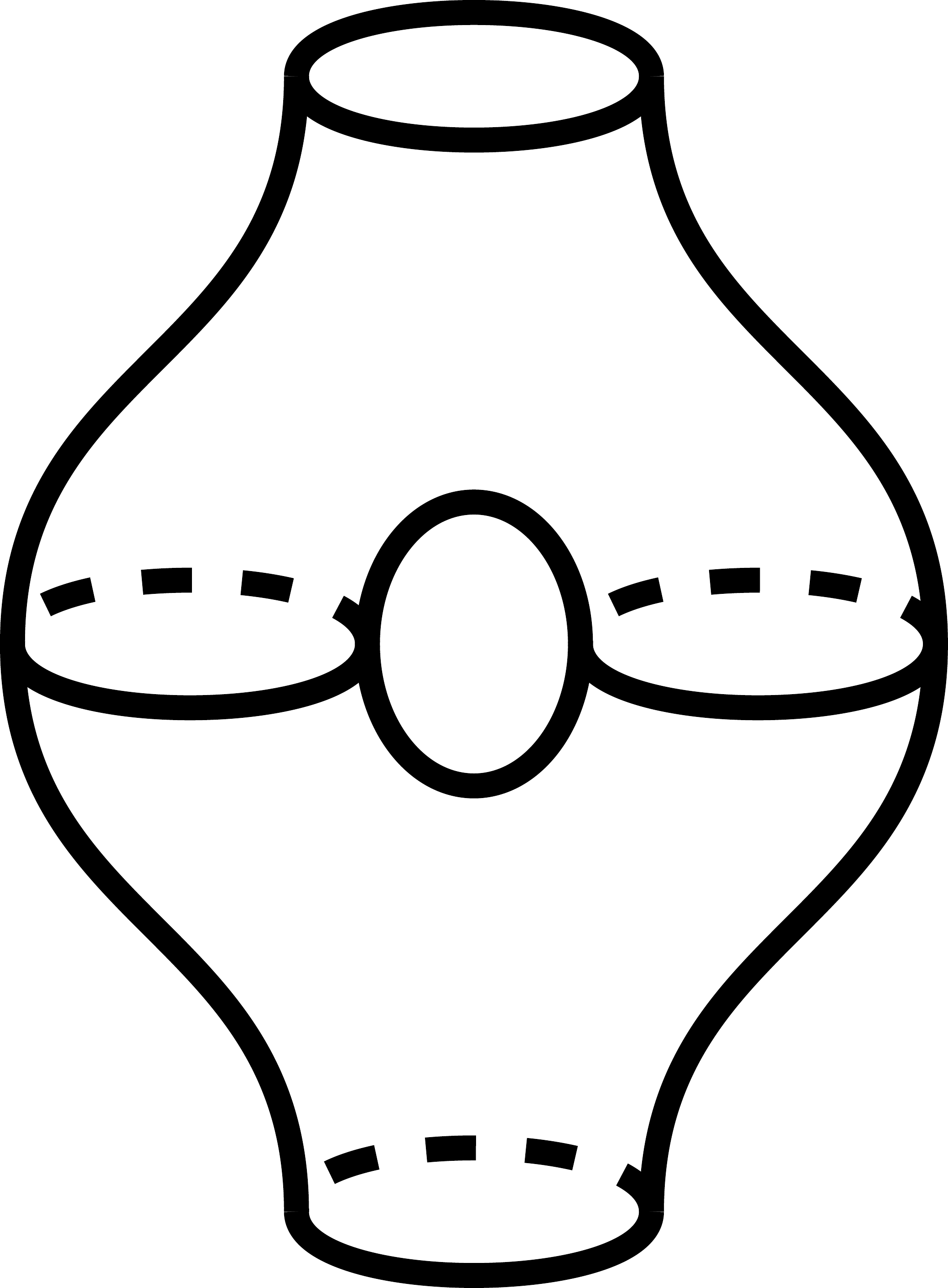}
			\end{aligned} \\
			&=
			\quad
			\begin{aligned}
			\includegraphics[scale=0.04]{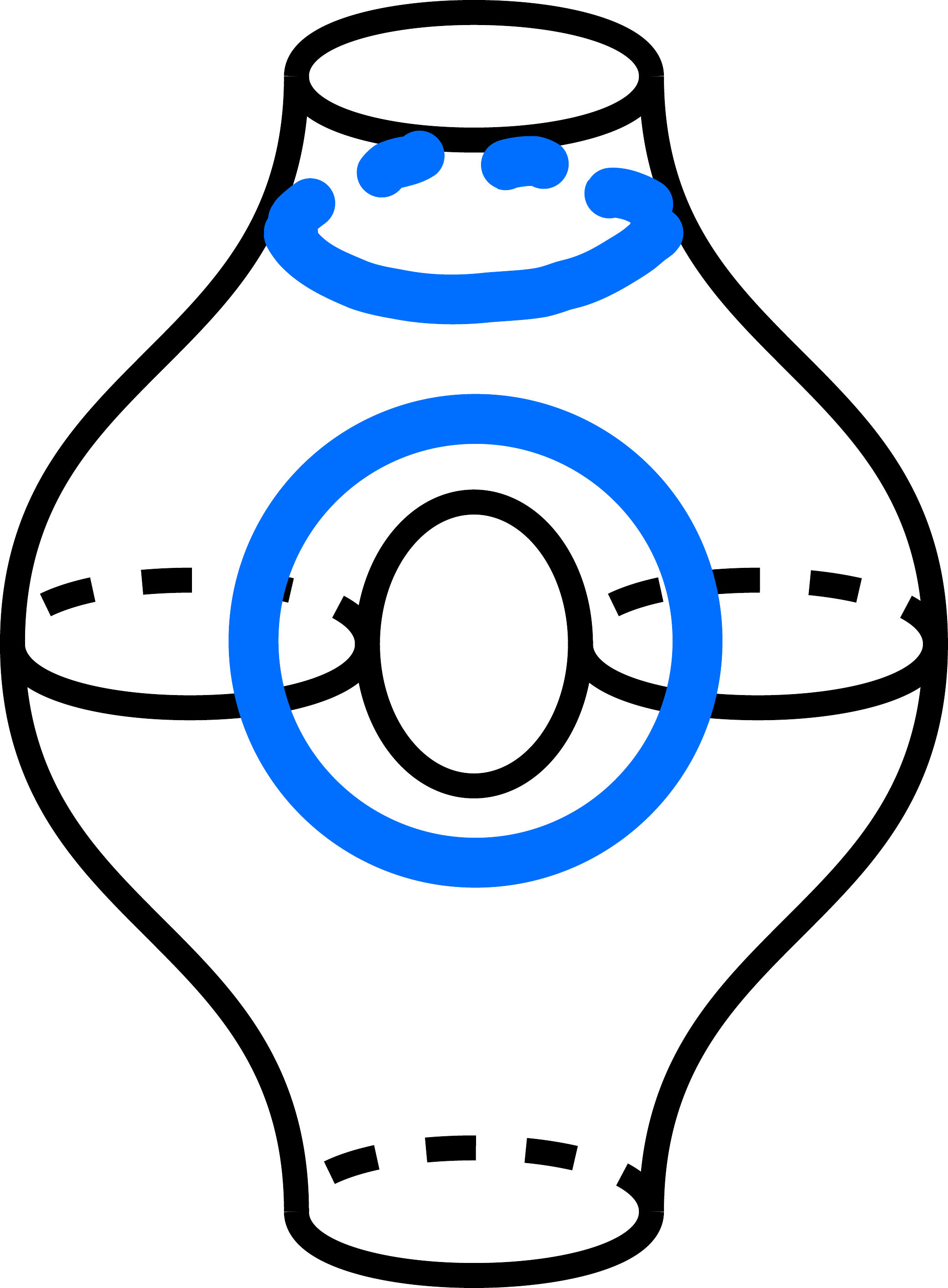}
			\end{aligned}
			\quad
			+
			\quad
			\begin{aligned}
			\includegraphics[scale=0.04]{modularity_25.pdf}
			\end{aligned} \\
			&\xmapsto{~Z_{SN}(\epsilon)~}
			\quad
			\begin{aligned}
			\includegraphics[scale=0.04]{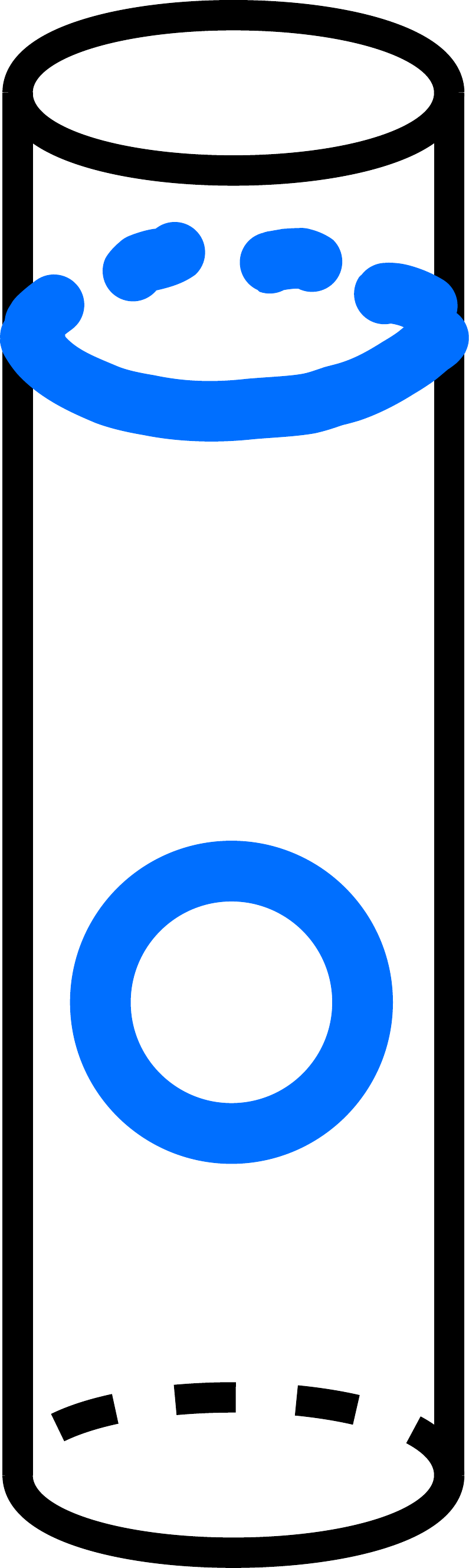}
			\end{aligned}
			\quad
			+
			\quad
			\begin{aligned}
			\includegraphics[scale=0.04]{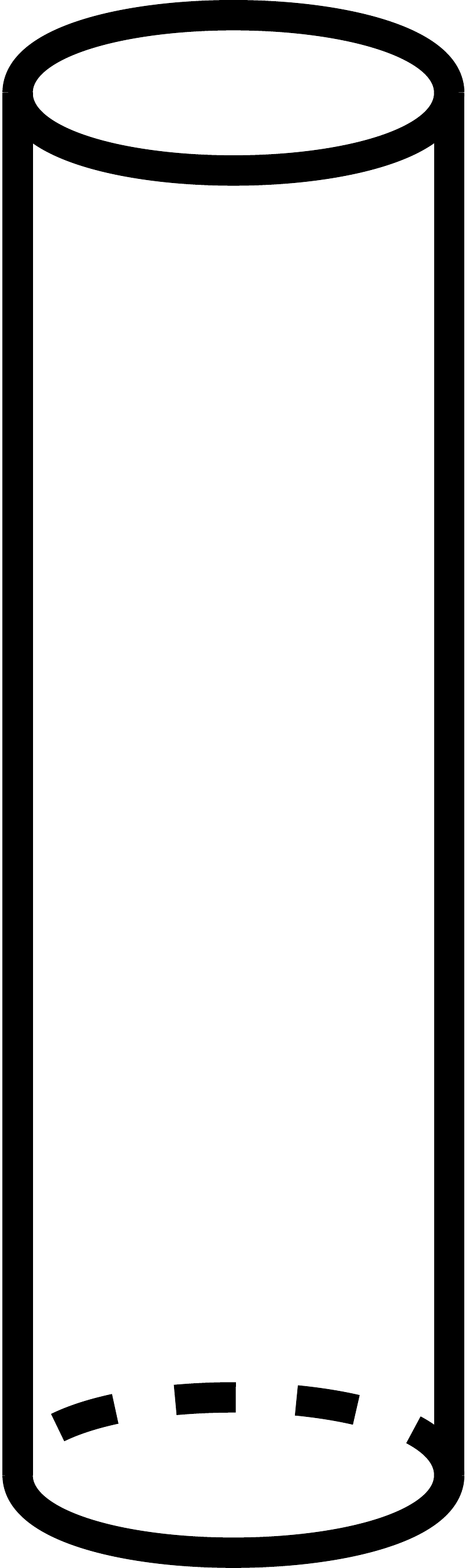}
			\end{aligned}
			\quad
			=
			\quad
			\begin{aligned}
			\includegraphics[scale=0.04]{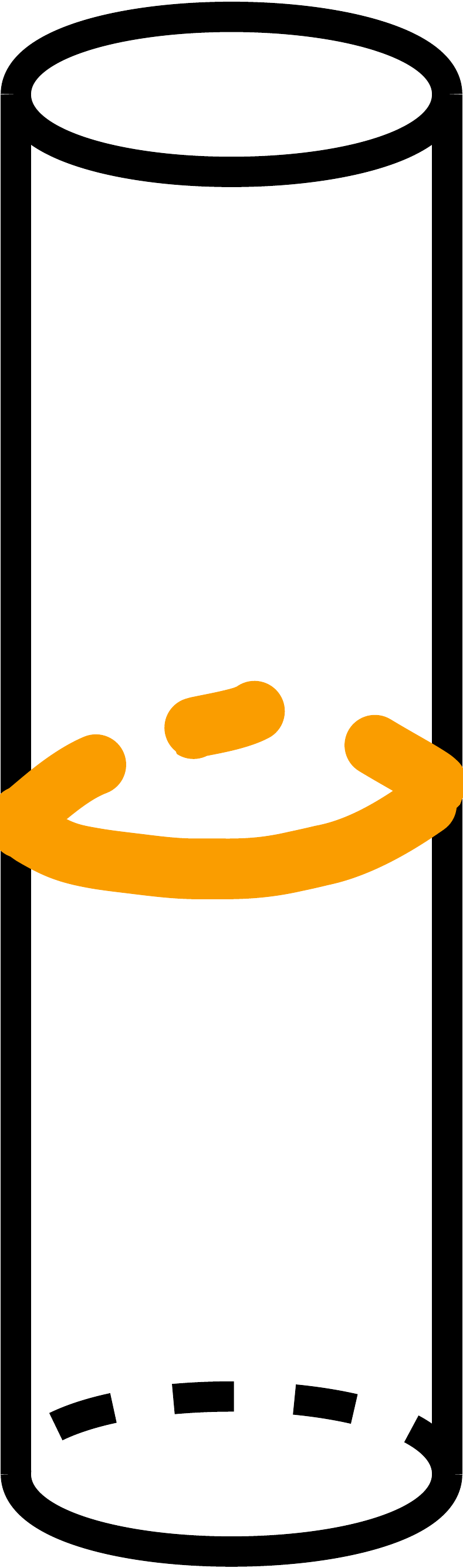}
			\end{aligned}
		\end{align*}
		
		\vspace{0.15cm}
		
		\begin{align*}
			\begin{aligned}
			\includegraphics[scale=0.04]{modularity_6.pdf}
			\end{aligned}
			\quad
			\xmapsto{~Z_{SN}(\epsilon^{\dagger})~}
			\quad
			\begin{aligned}
			\includegraphics[scale=0.04]{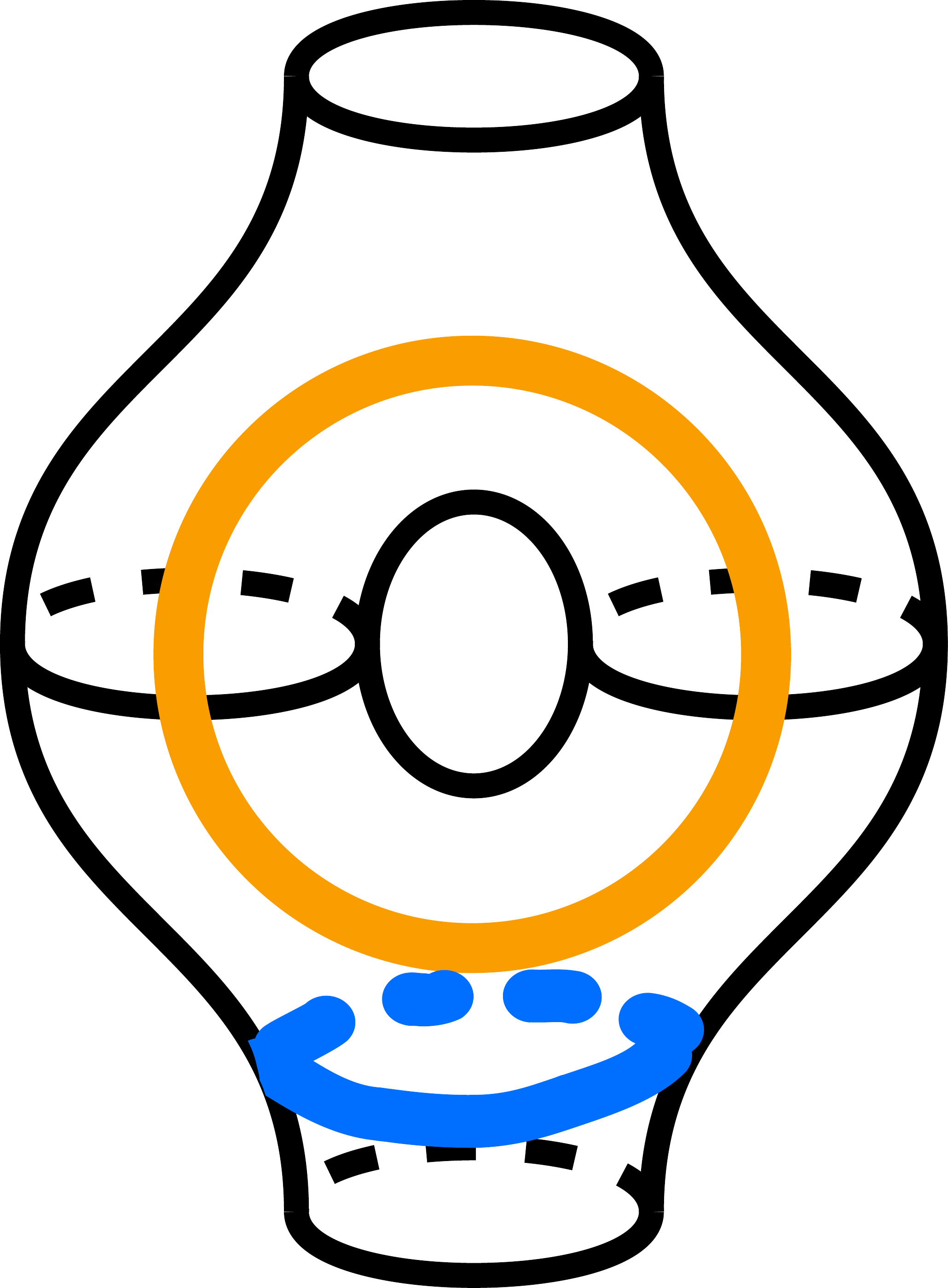}
			\end{aligned}
			\quad
			&\xmapsto{~Z_{SN}(\theta),\,Z_{SN}(\theta^{\inv})~}
			\quad
			\begin{aligned}
			\includegraphics[scale=0.04]{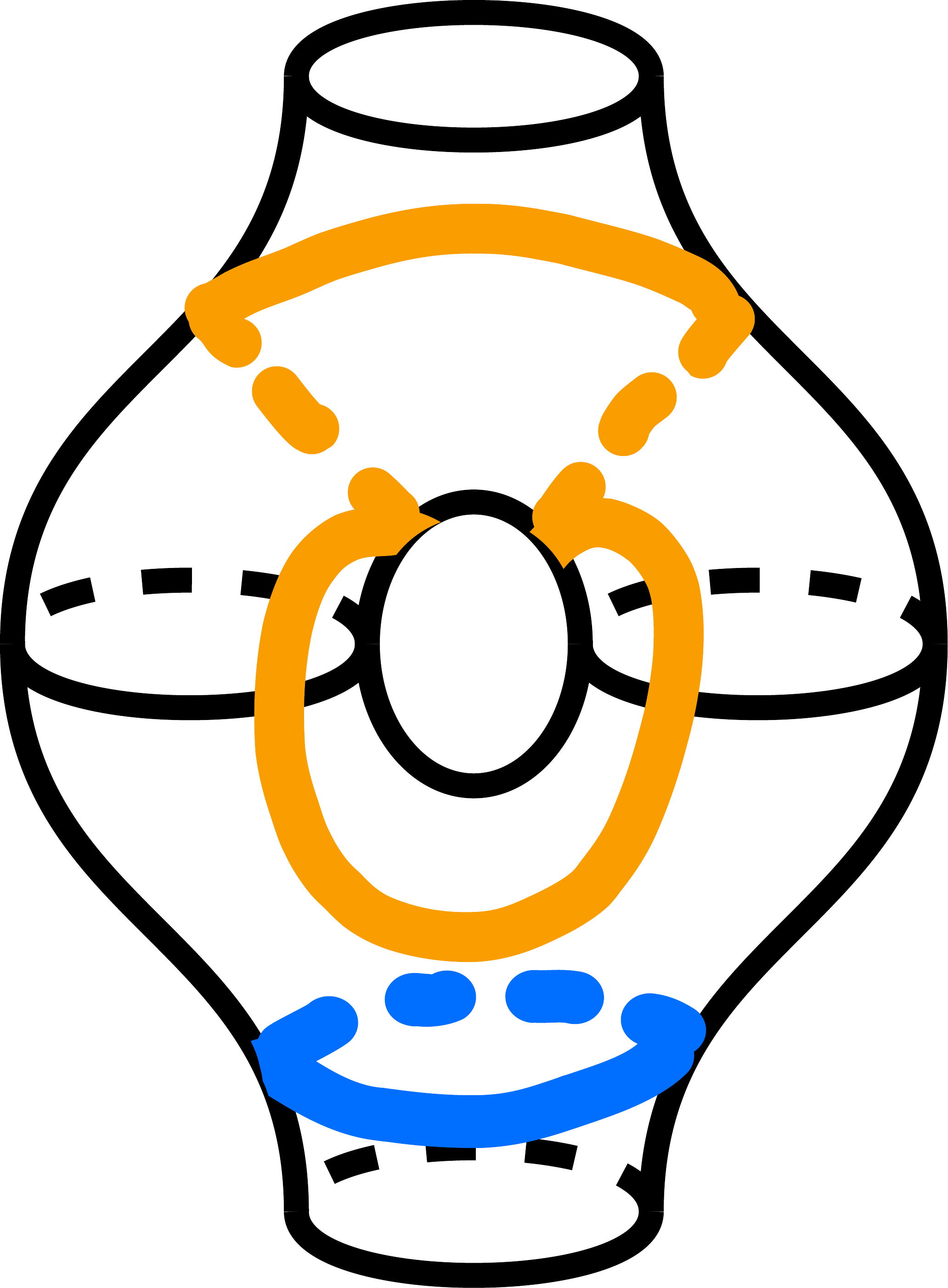}
			\end{aligned} \\
			&=
			\quad
			\begin{aligned}
			\includegraphics[scale=0.04]{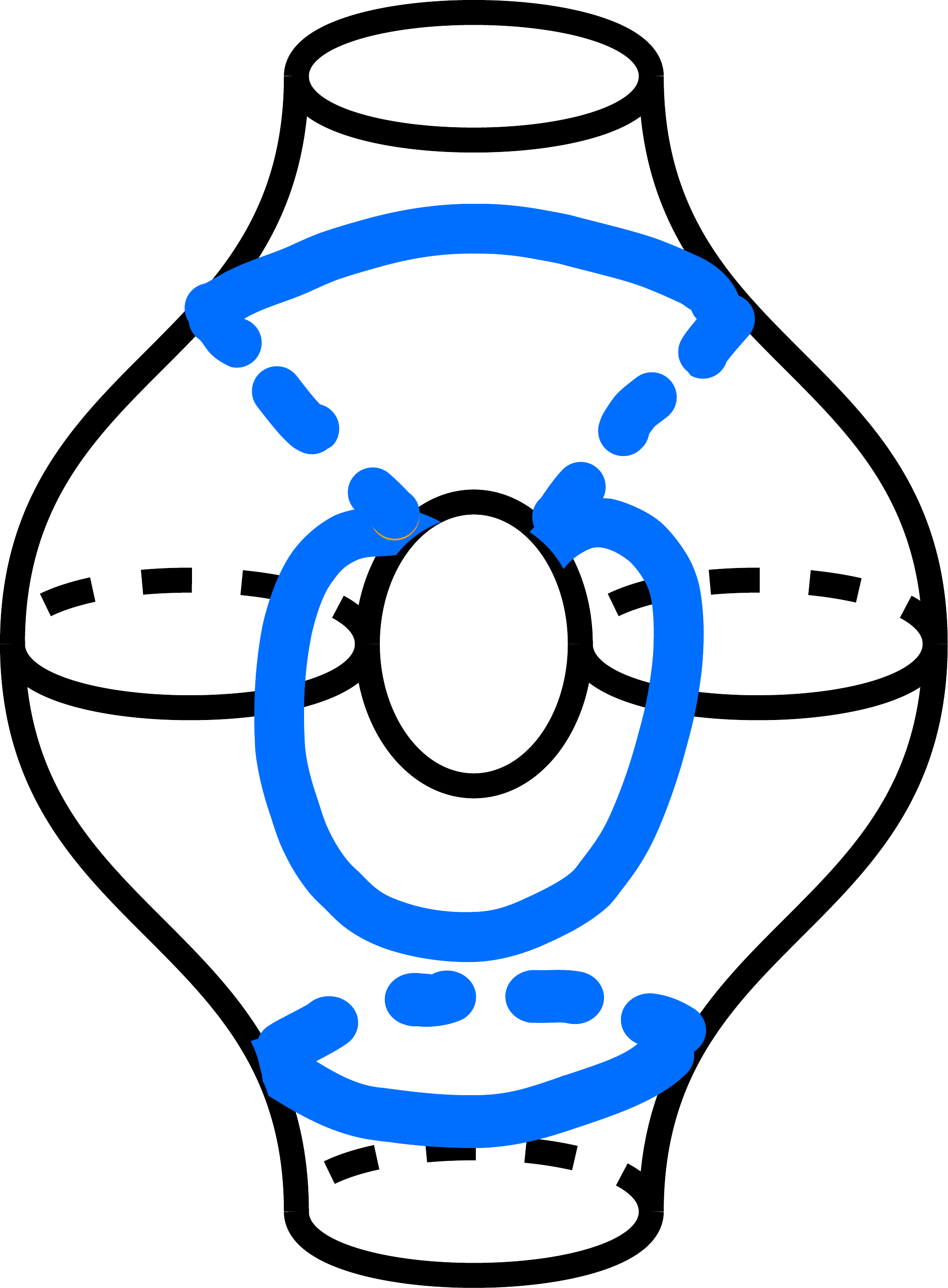}
			\end{aligned}
			\quad
			+
			\quad
			\begin{aligned}
			\includegraphics[scale=0.04]{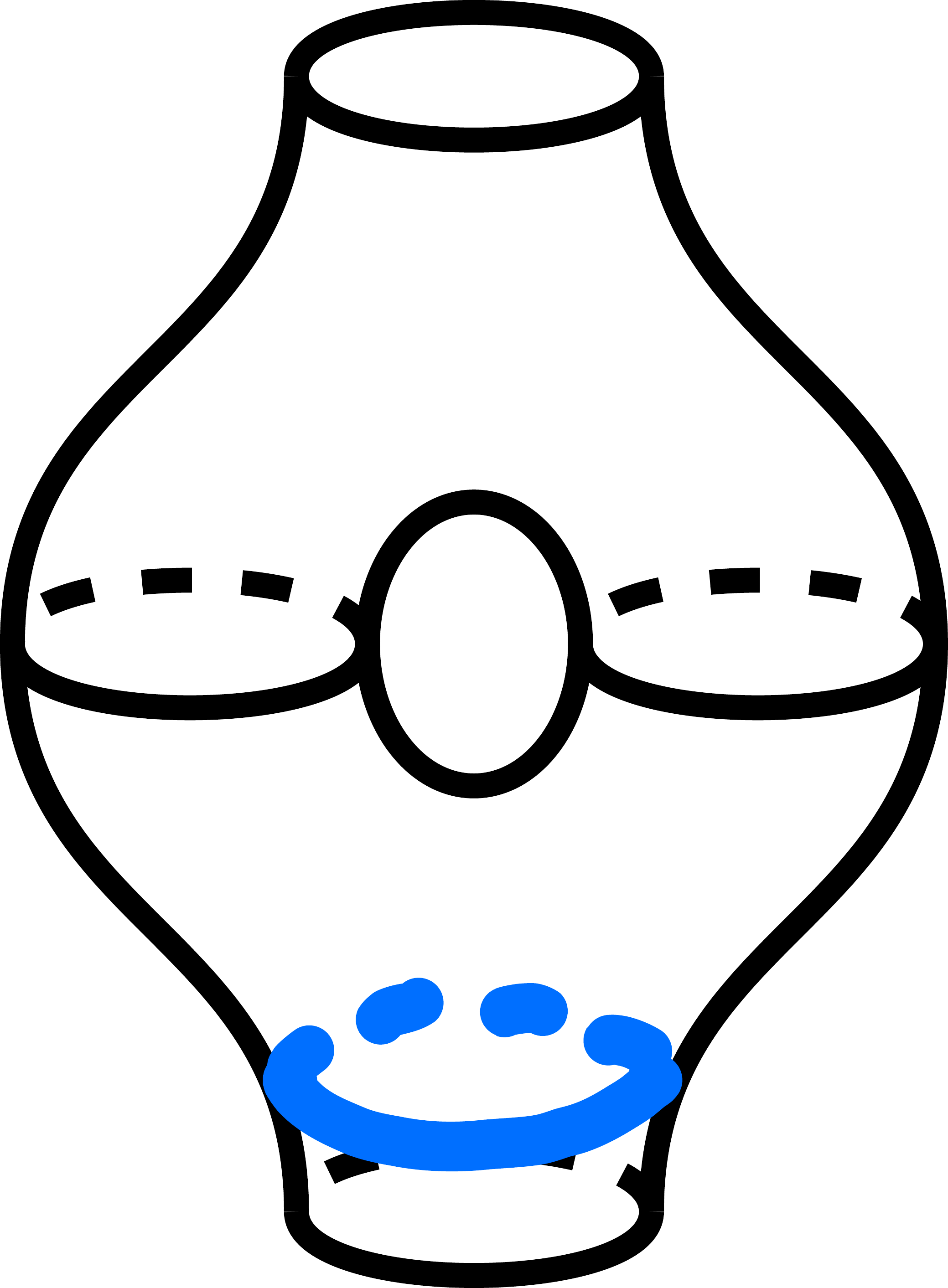}
			\end{aligned} \\
			&=
			\quad
			\begin{aligned}
			\includegraphics[scale=0.04]{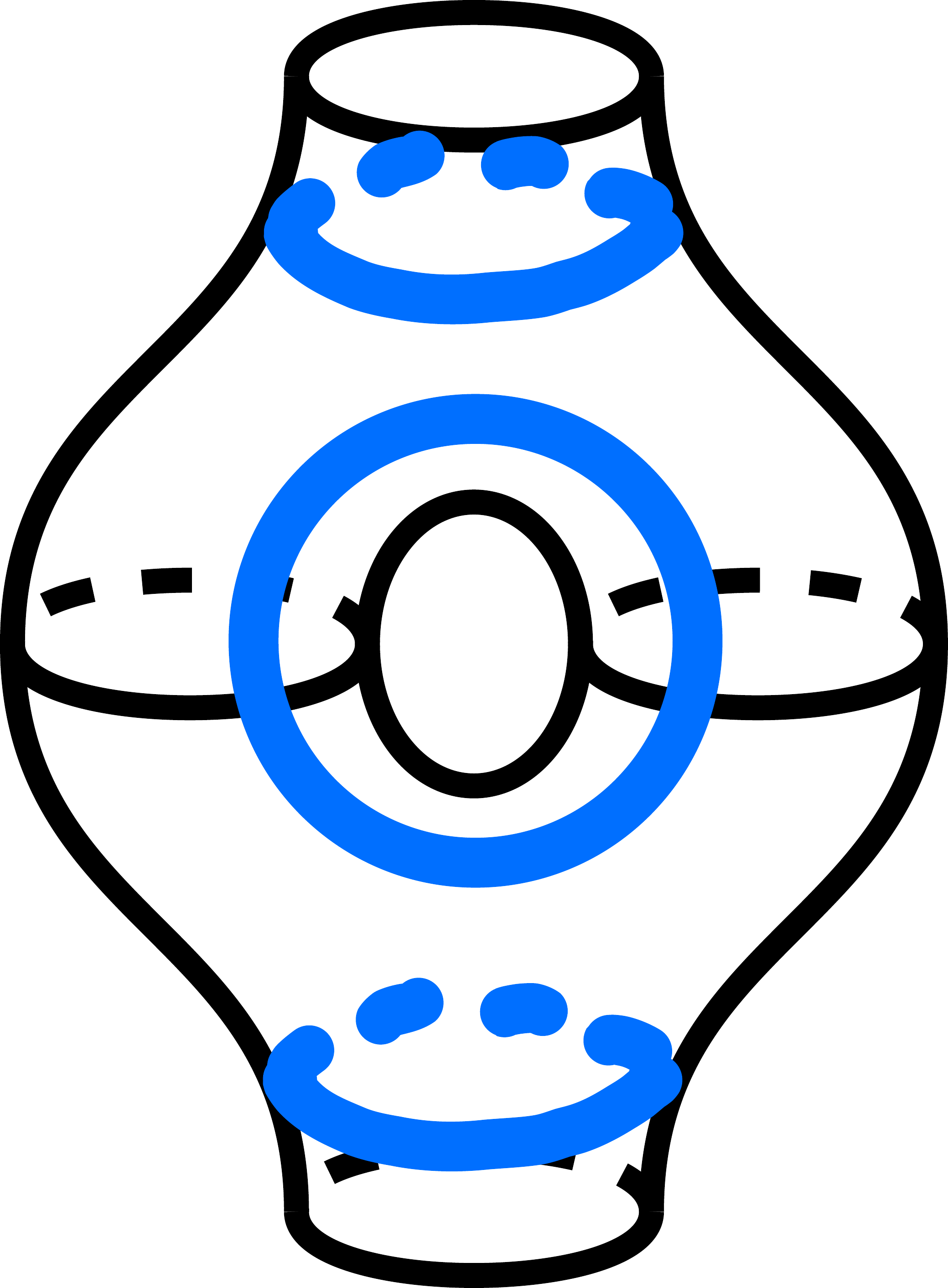}
			\end{aligned}
			\quad
			+
			\quad
			\begin{aligned}
			\includegraphics[scale=0.04]{modularity_30.pdf}
			\end{aligned} \\
			&\xmapsto{~Z_{SN}(\epsilon)~}
			\quad
			\begin{aligned}
			\includegraphics[scale=0.04]{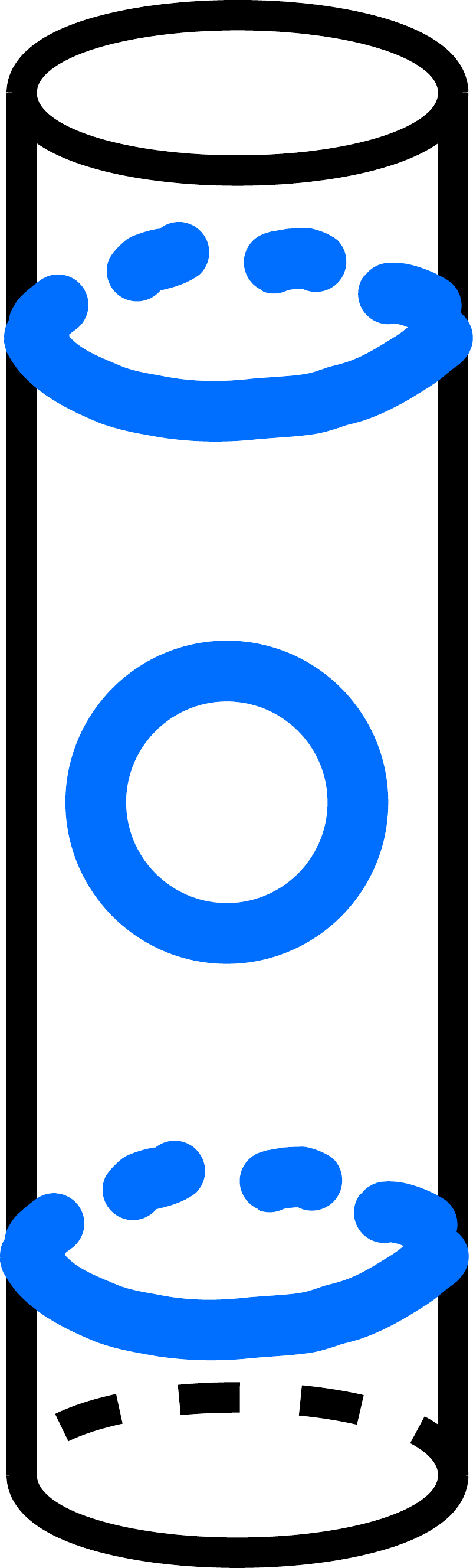}
			\end{aligned}
			\quad
			+
			\quad
			\begin{aligned}
			\includegraphics[scale=0.04]{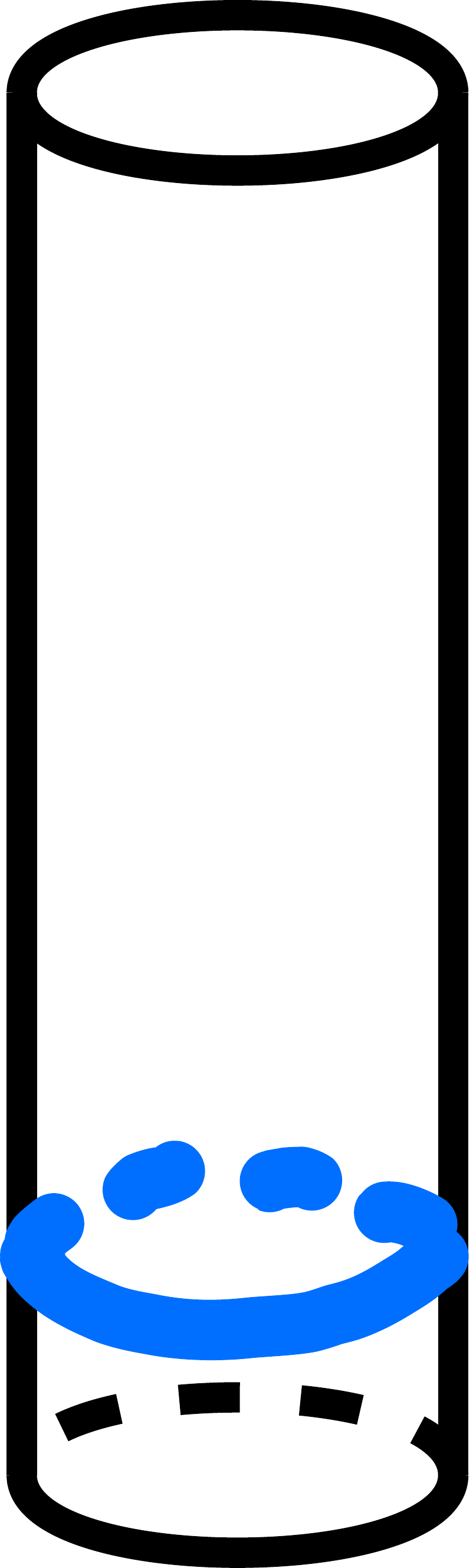}
			\end{aligned}
			\quad \\
			&=
			\quad
			\begin{aligned}
			\includegraphics[scale=0.04]{modularity_27.pdf}
			\end{aligned}
			\quad
			+
			\quad
			\begin{aligned}
			\includegraphics[scale=0.04]{modularity_32.pdf}
			\end{aligned}
			\quad
			=
			\quad
			\begin{aligned}
			\includegraphics[scale=0.04]{modularity_5.pdf}
			\end{aligned}
		\end{align*}
		
		\vspace{0.15cm}
		
		
		
		\begin{align*}
			\begin{aligned}
			\includegraphics[scale=0.04]{modularity_7.pdf}
			\end{aligned}
			\quad
			&\xmapsto{~Z_{SN}(\epsilon^{\dagger})~}
			\quad
			\begin{aligned}
			\includegraphics[scale=0.04]{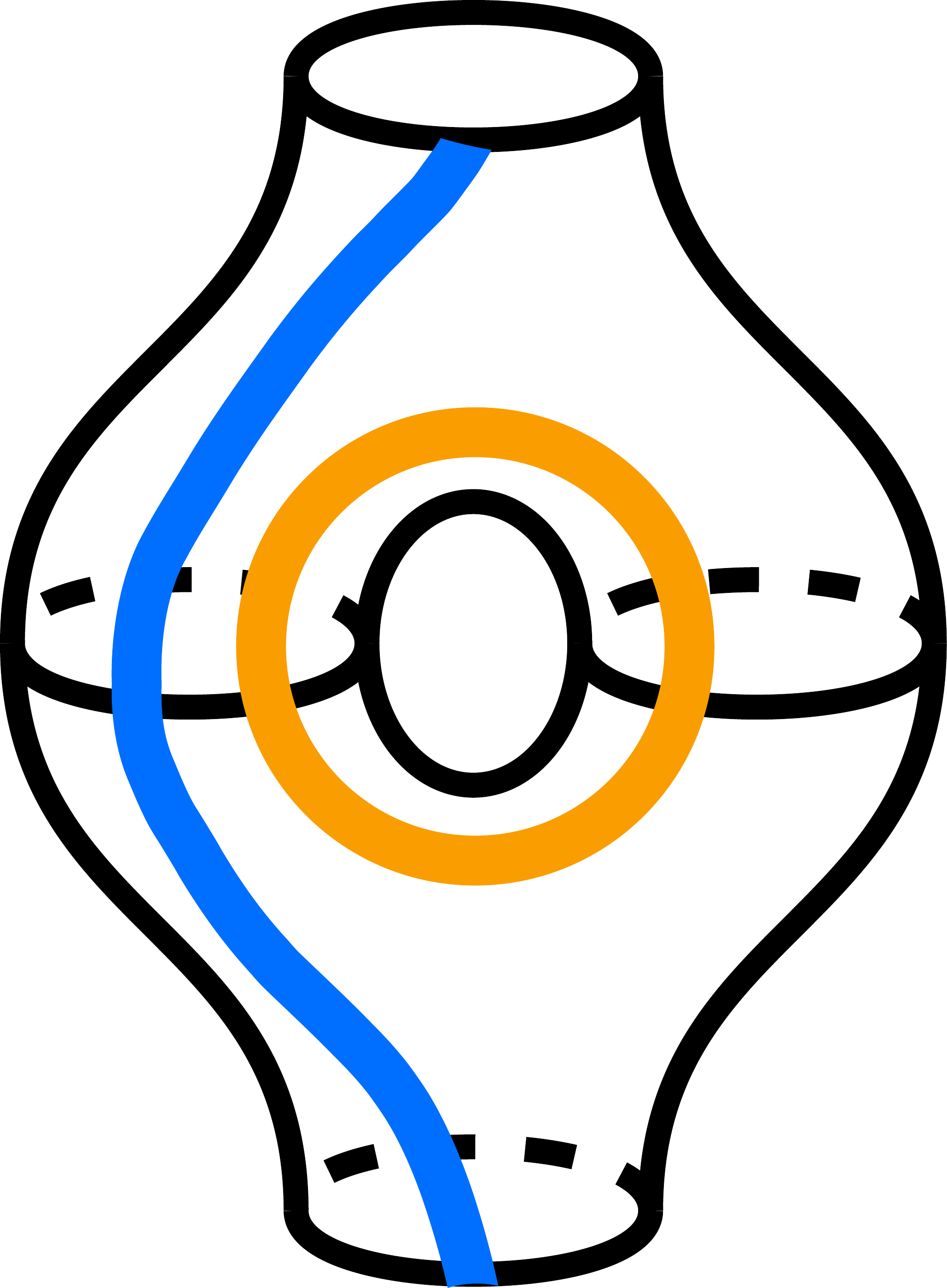}
			\end{aligned}
 			\quad
			=
			\quad
			\begin{aligned}
			\includegraphics[scale=0.04]{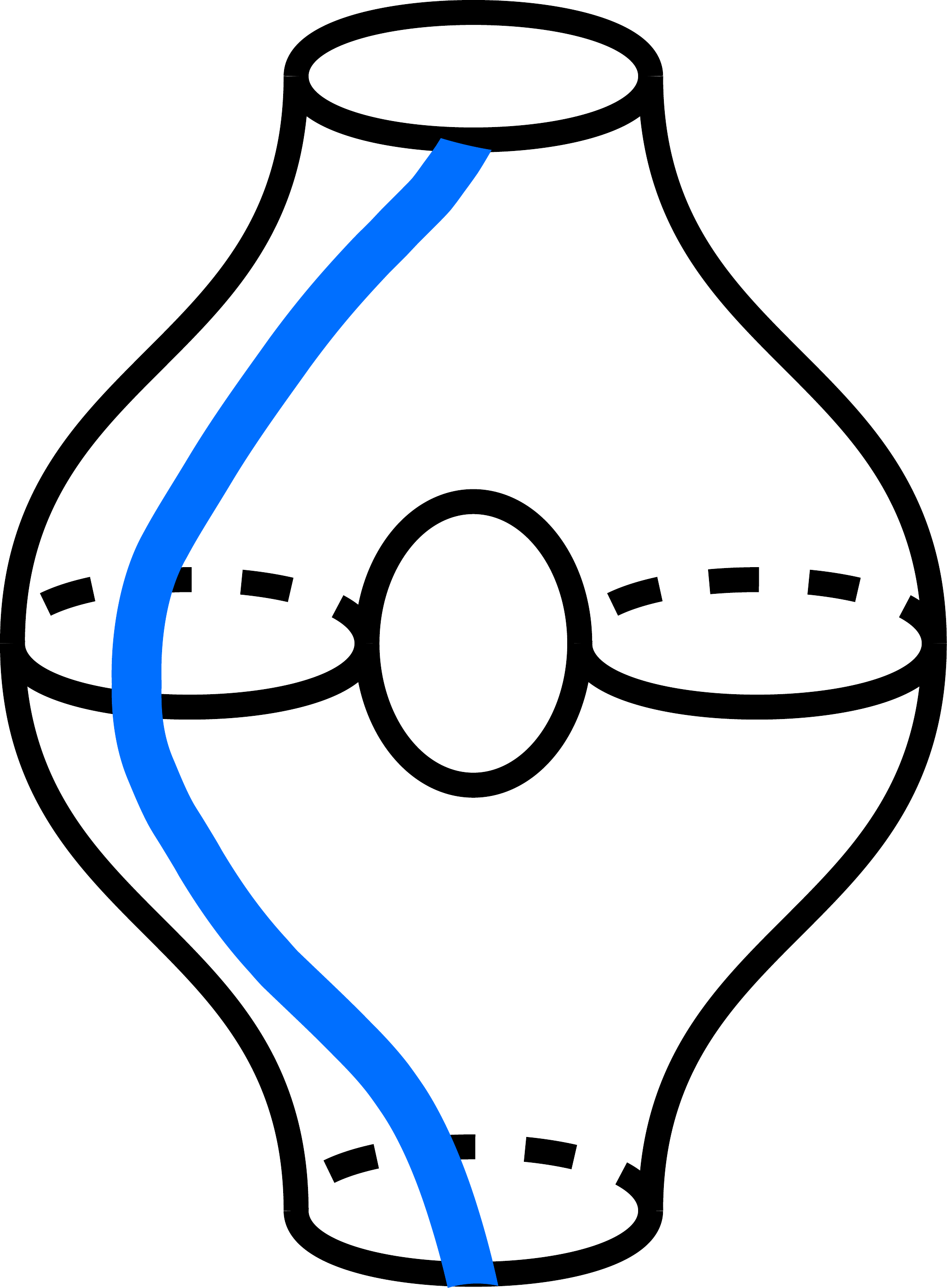}
			\end{aligned}
			\quad
			+
			\quad
			\begin{aligned}
			\includegraphics[scale=0.04]{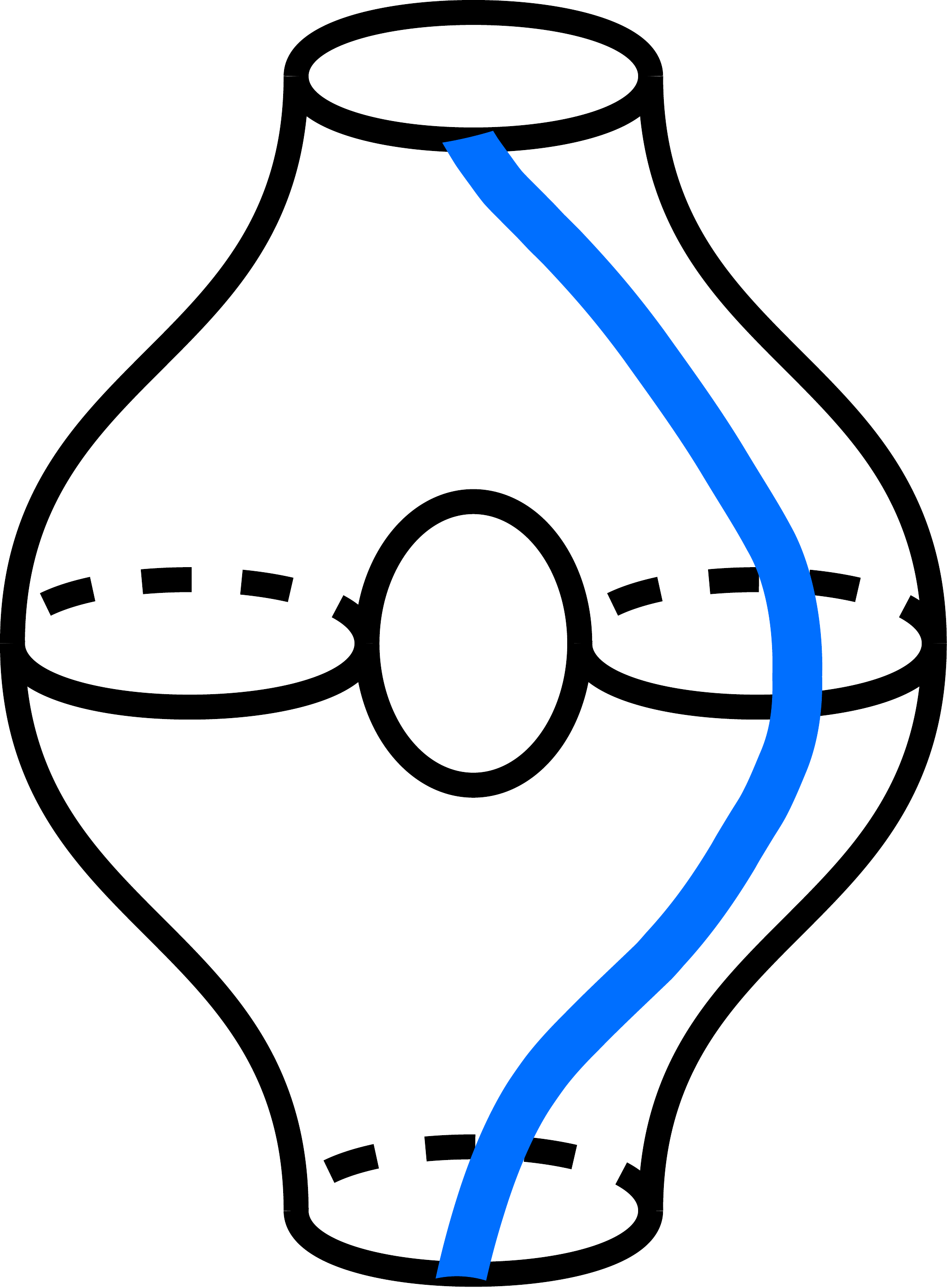}
			\end{aligned} \\
			&\xmapsto{~Z_{SN}(\theta),\,Z_{SN}(\theta^{\inv})~}
			\quad
			\begin{aligned}
			\includegraphics[scale=0.04]{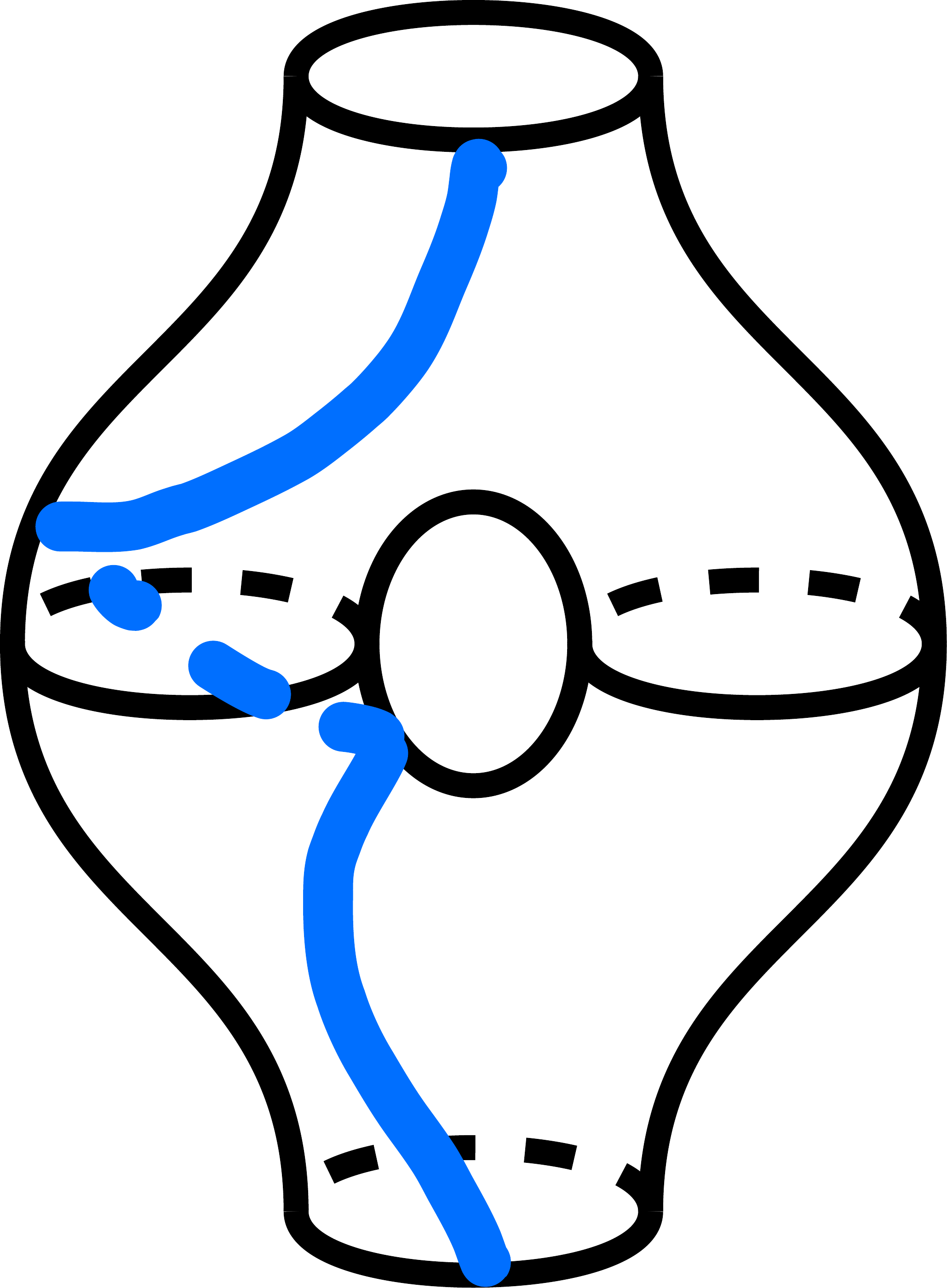}
			\end{aligned}
			\quad
			+
			\quad
			\begin{aligned}
			\includegraphics[scale=0.04]{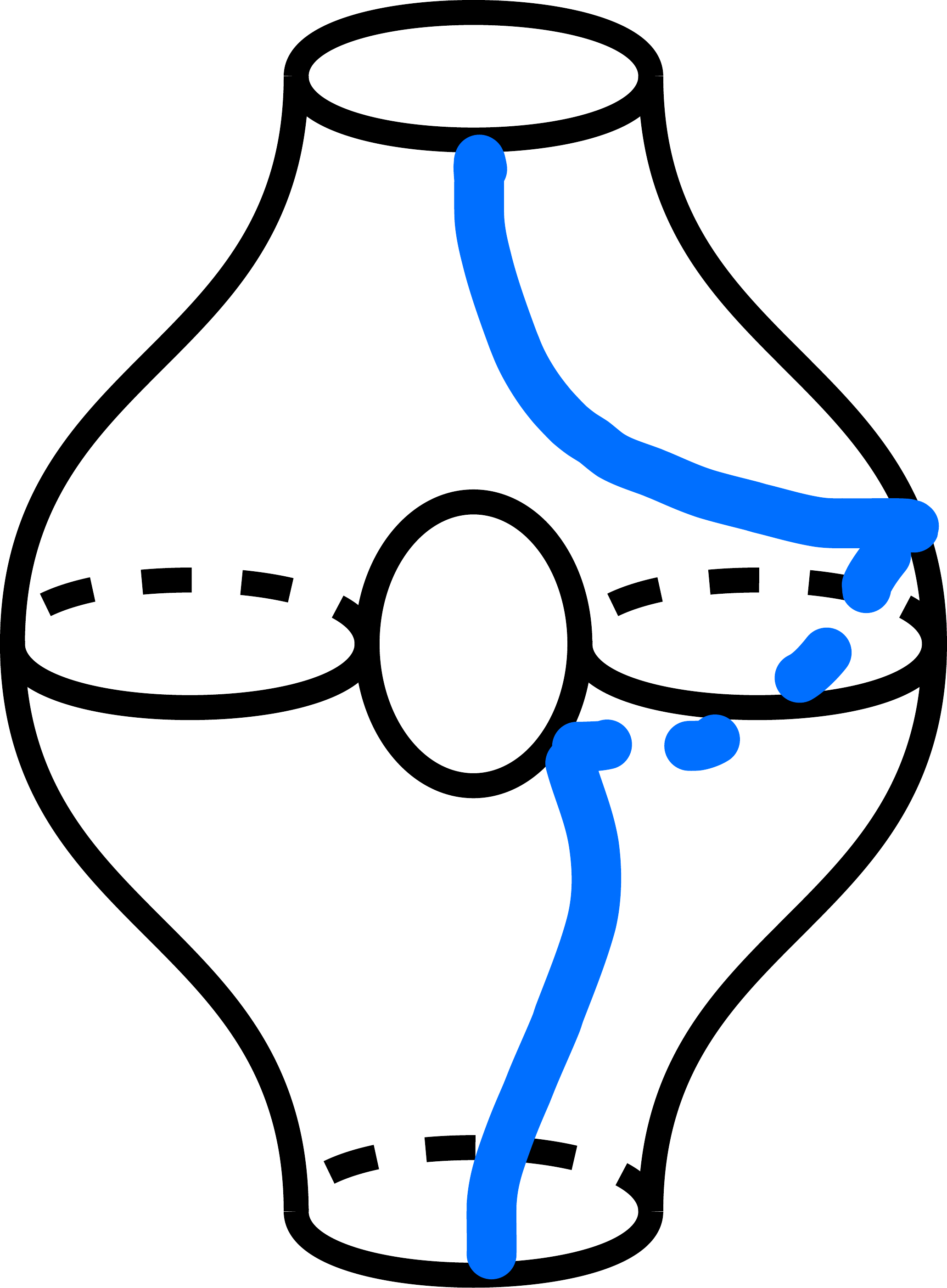}
			\end{aligned}
			\quad
			\xmapsto{~Z_{SN}(\epsilon)~}
			\quad
			0
		\end{align*}
		
		\vspace{0.15cm}
		
		\begin{align*}
			\begin{aligned}
			\includegraphics[scale=0.04]{modularity_10.pdf}
			\end{aligned}
			\quad
			&\xmapsto{~Z_{SN}(\epsilon^{\dagger})~}
			\quad
			\begin{aligned}
			\includegraphics[scale=0.04]{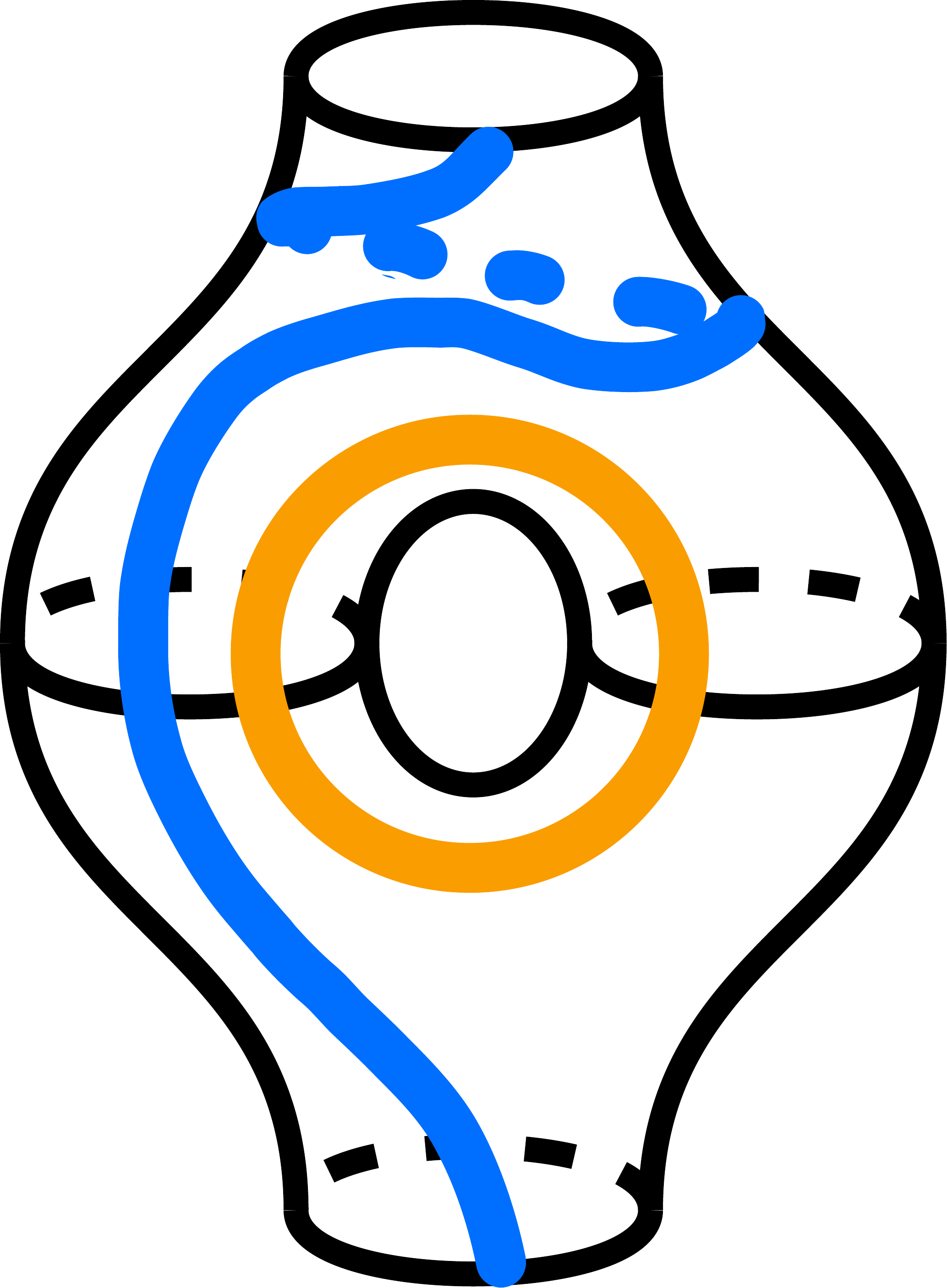}
			\end{aligned}
 			\quad
			=
			\quad
			\begin{aligned}
			\includegraphics[scale=0.04]{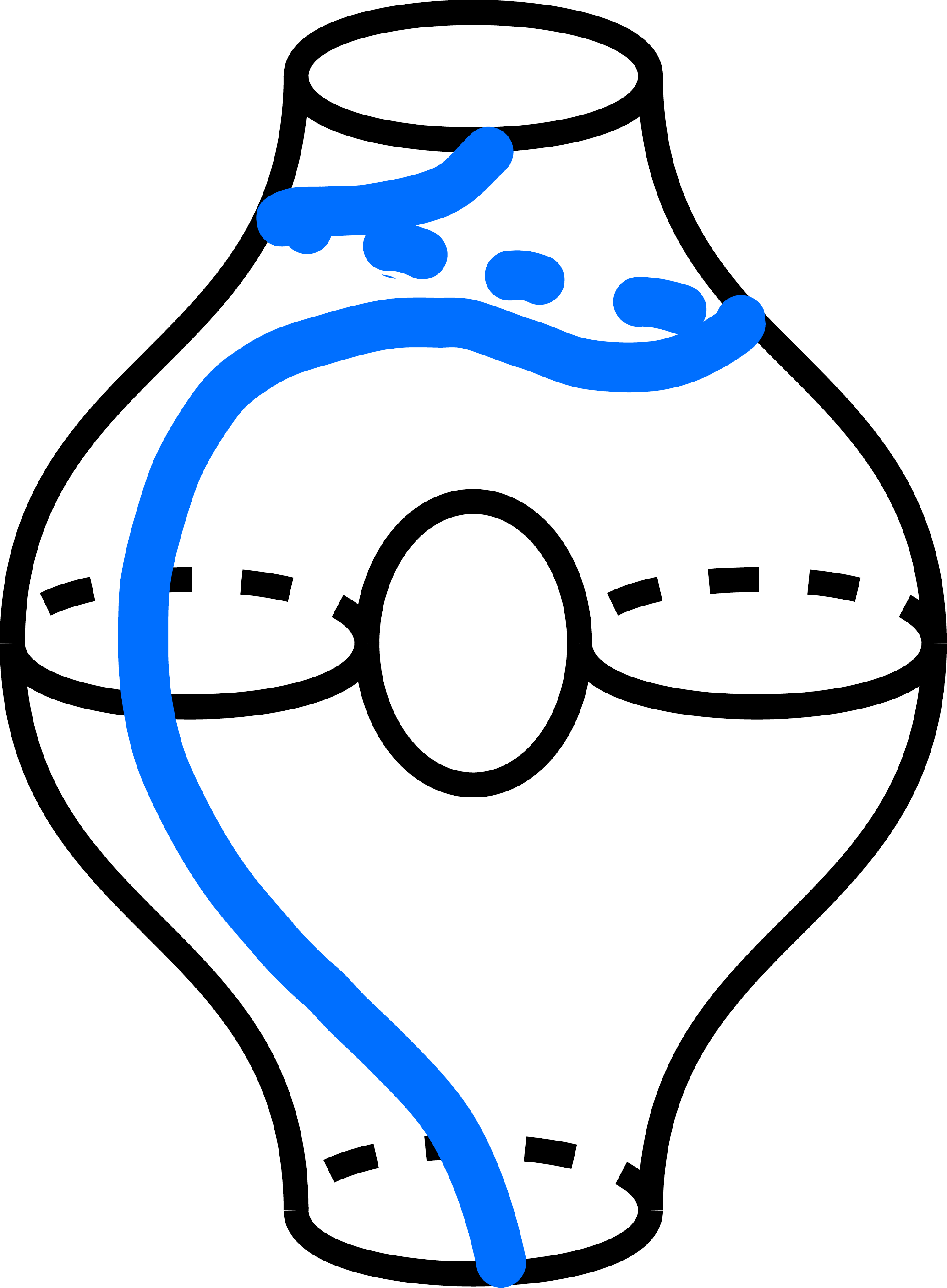}
			\end{aligned}
			\quad
			+
			\quad
			\begin{aligned}
			\includegraphics[scale=0.04]{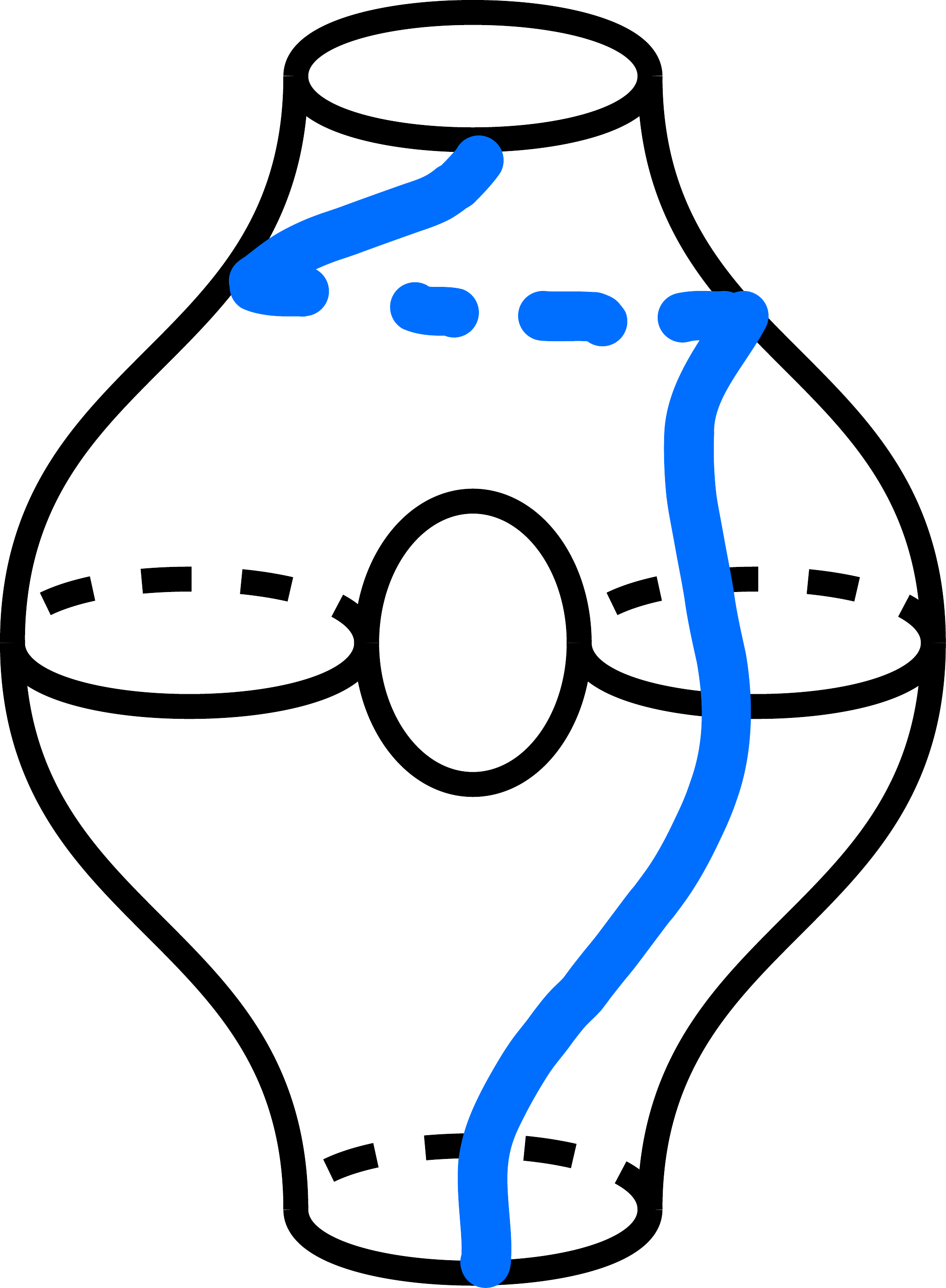}
			\end{aligned} \\
			&\xmapsto{~Z_{SN}(\theta),\,Z_{SN}(\theta^{\inv})~}
			\quad
			\begin{aligned}
			\includegraphics[scale=0.04]{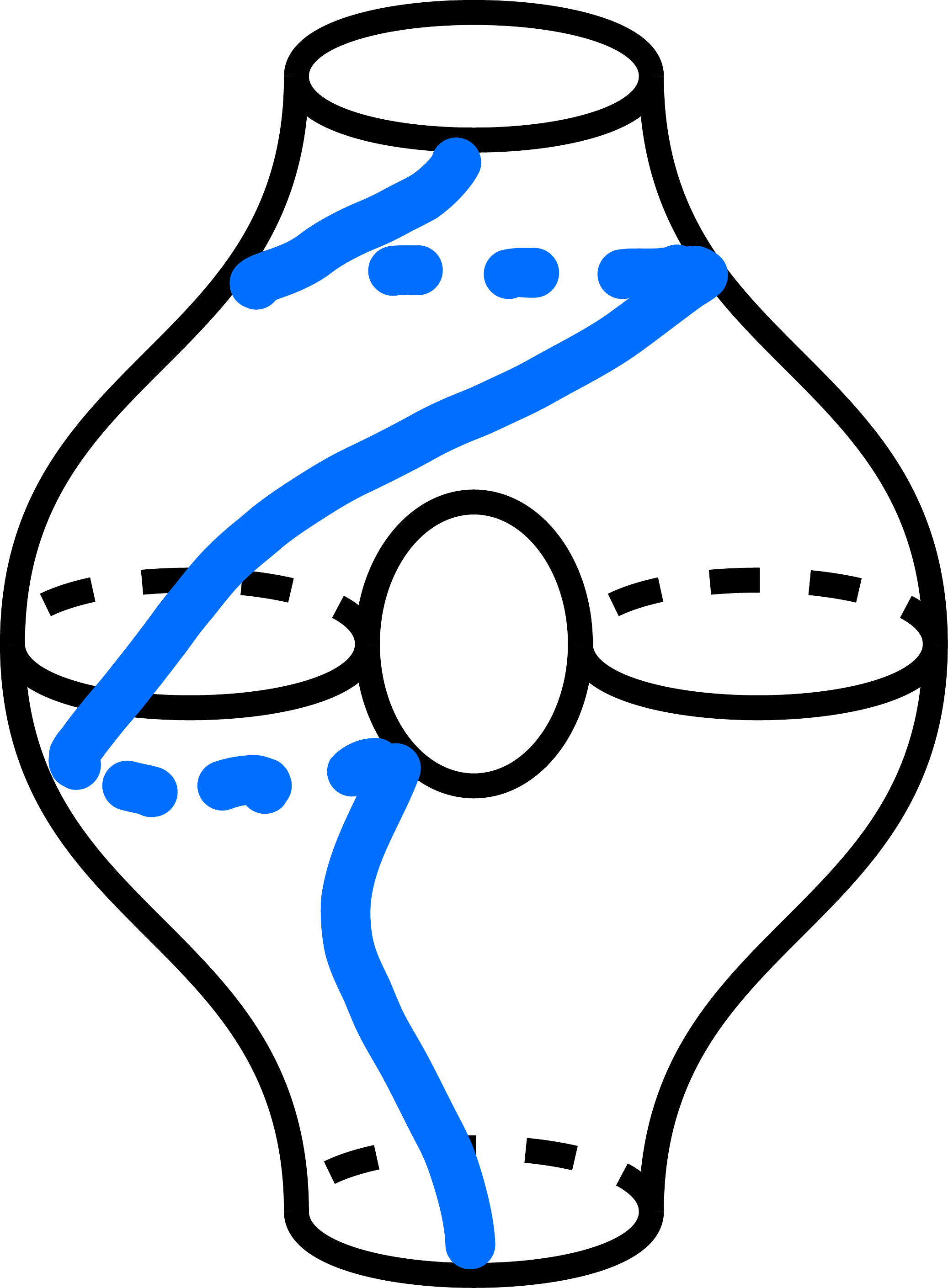}
			\end{aligned}
			\quad
			+
			\quad
			\begin{aligned}
			\includegraphics[scale=0.04]{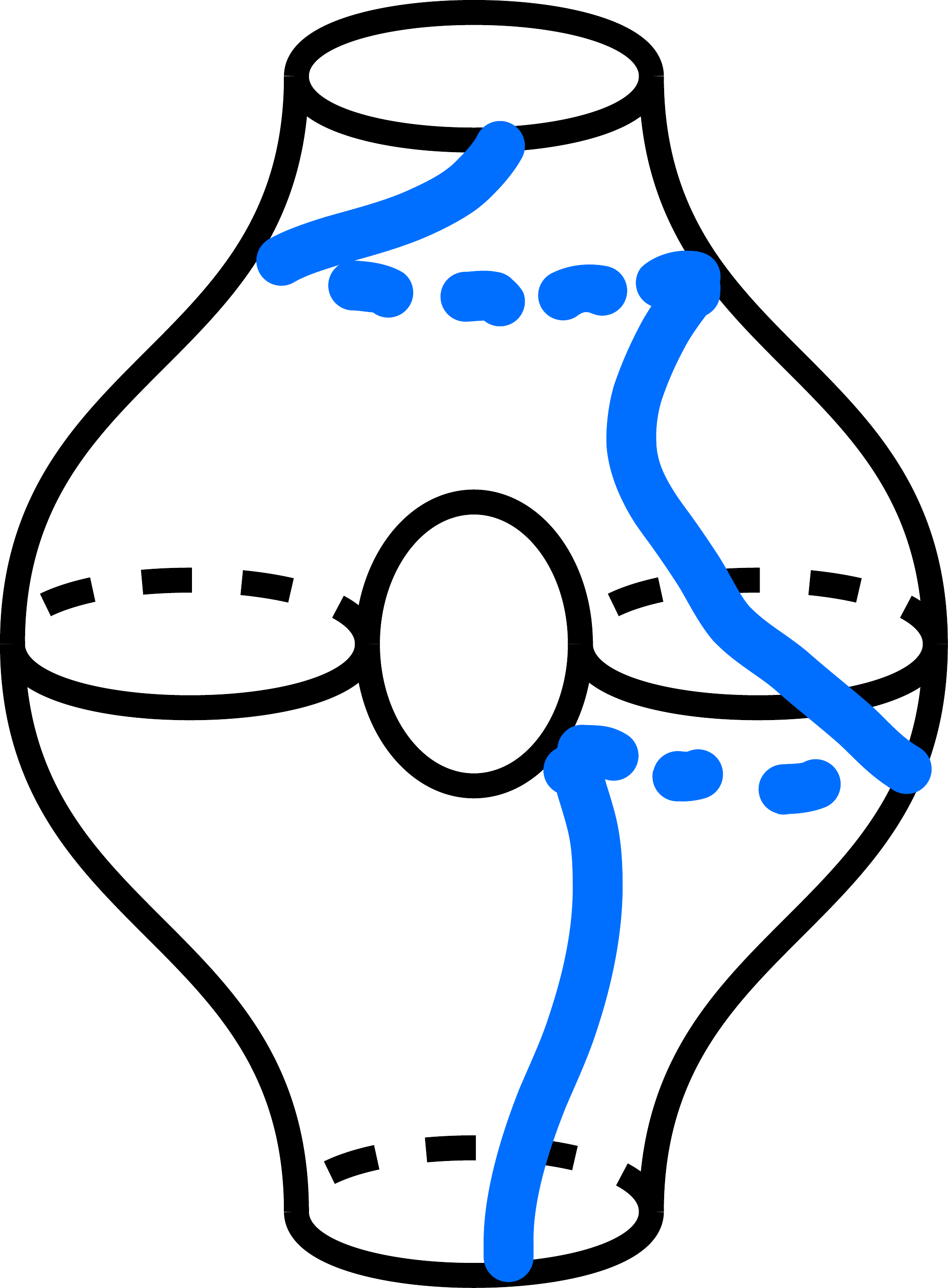}
			\end{aligned}
			\quad
			\xmapsto{~Z_{SN}(\epsilon)~}
			\quad
			0
		\end{align*}

\noindent On the other hand, the ``bottom'' composites are

        \begin{equation*}
			\begin{aligned}
			\includegraphics[scale=0.04]{modularity_1.pdf}
			\end{aligned}
			\quad
			\xmapsto{~Z_{SN}(\mu^{\dagger})~}
			\quad
			\begin{aligned}
			\includegraphics[scale=0.04]{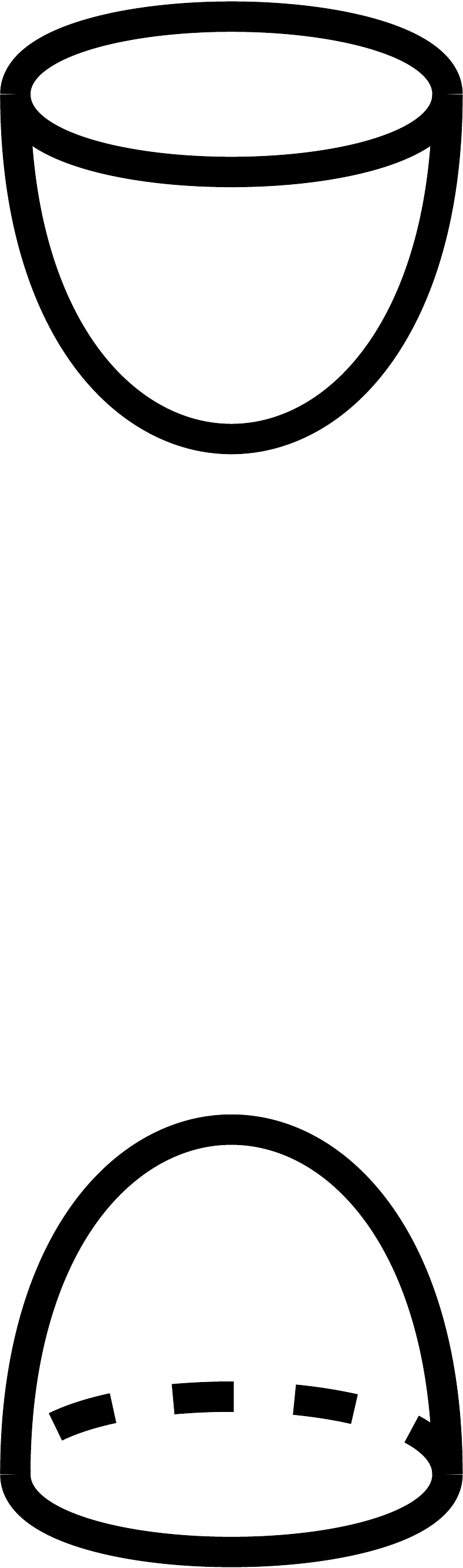}
			\end{aligned}
			\quad
			\xmapsto{~Z_{SN}(\mu)~}
			\quad
			\begin{aligned}
			\includegraphics[scale=0.04]{modularity_5.pdf}
			\end{aligned}
		\end{equation*}
		
		\vspace{0.15cm}
		
		\begin{equation*}
			\begin{aligned}
			\includegraphics[scale=0.04]{modularity_6.pdf}
			\end{aligned}
			\quad
			\xmapsto{~Z_{SN}(\mu^{\dagger})~}
			\quad
			\begin{aligned}
			\includegraphics[scale=0.04]{modularity_2_bot.pdf}
			\end{aligned}
			\quad
			\xmapsto{~Z_{SN}(\mu)~}
			\quad
			\begin{aligned}
			\includegraphics[scale=0.04]{modularity_5.pdf}
			\end{aligned}
		\end{equation*}
		
		\vspace{0.15cm}
		
		\begin{equation*}
			\begin{aligned}
			\includegraphics[scale=0.04]{modularity_7.pdf}
			\end{aligned}
			\quad
			\xmapsto{~Z_{SN}(\mu^{\dagger})~}
			\quad
			0
			\quad
			\xmapsto{~Z_{SN}(\mu)~}
			\quad
			0
		\end{equation*}
		
		\vspace{0.15cm}
		
		\begin{equation*}
			\begin{aligned}
			\includegraphics[scale=0.04]{modularity_10.pdf}
			\end{aligned}
			\quad
			\xmapsto{~Z_{SN}(\mu^{\dagger})~}
			\quad
			0
			\quad
			\xmapsto{~Z_{SN}(\mu)~}
			\quad
			0
		\end{equation*}

\noindent The corresponding pairs of composites are clearly equal.

\begin{remark}
Although we checked the modularity relation on each basis vector we will, for brevity, check some of the remaining relations only on sample vectors in the relevant source vector space.
\end{remark}

\subsubsection{Rigidity} The left Frobeniusator on a sample basis vector is given by the composite

\begin{equation*}
			\begin{aligned}
			\includegraphics[scale=0.04]{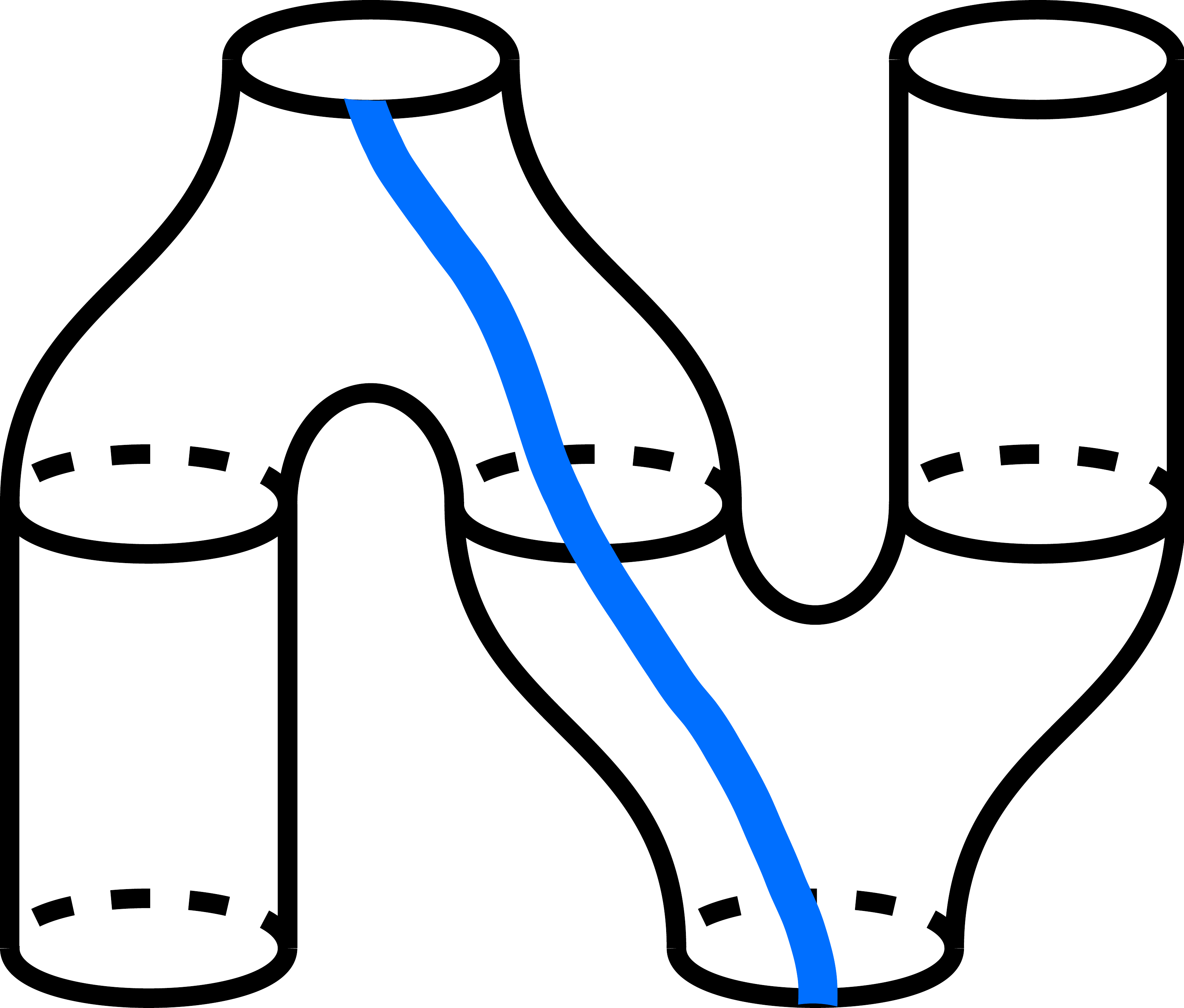}
			\end{aligned}
			\quad
			\xmapsto{~Z_{SN}(\eta)~}
			\quad
			\begin{aligned}
			\includegraphics[scale=0.04]{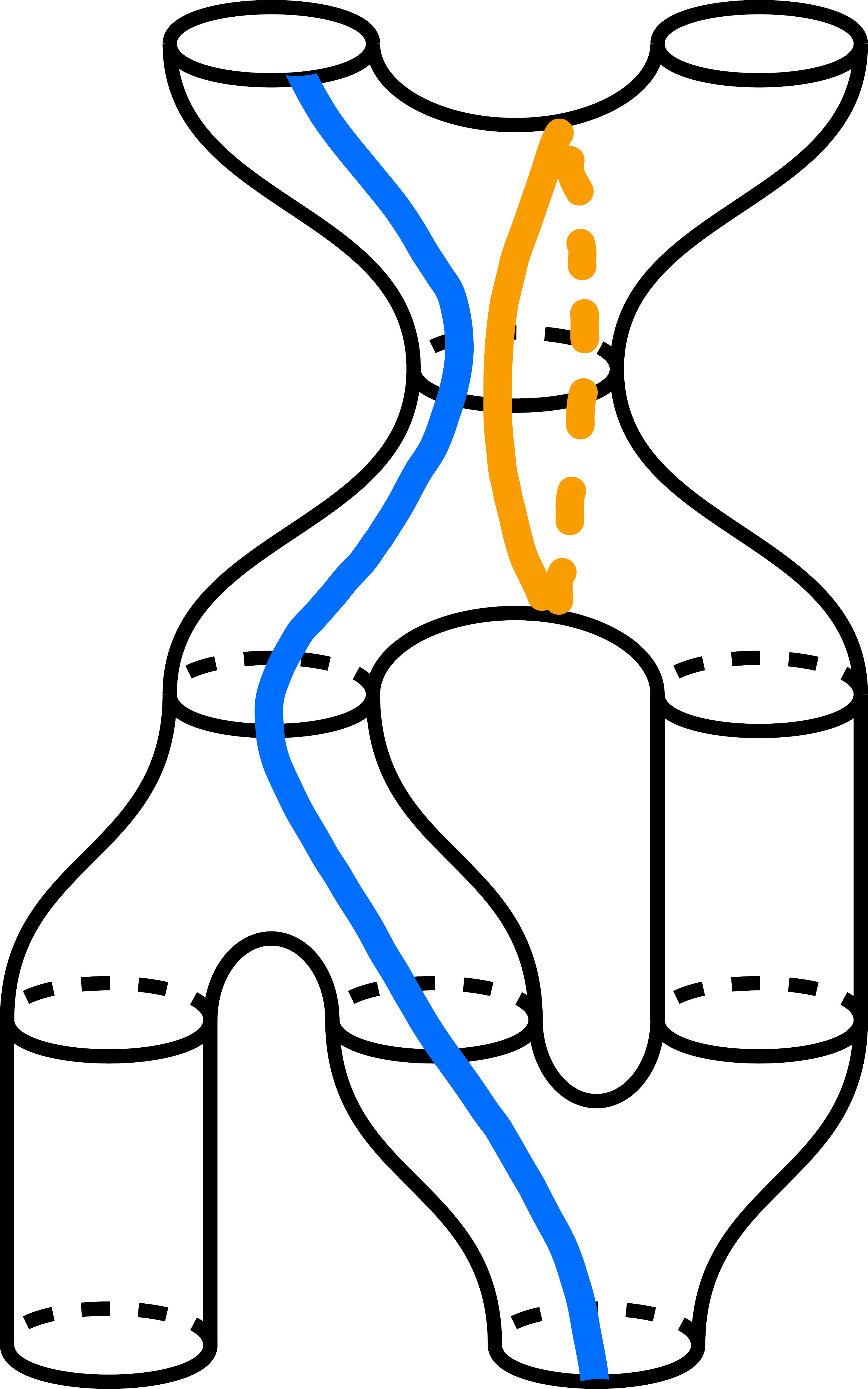}
			\end{aligned}
			\quad
			\xmapsto{~Z_{SN}(\alpha)~}
			\quad
			\begin{aligned}
			\includegraphics[scale=0.04]{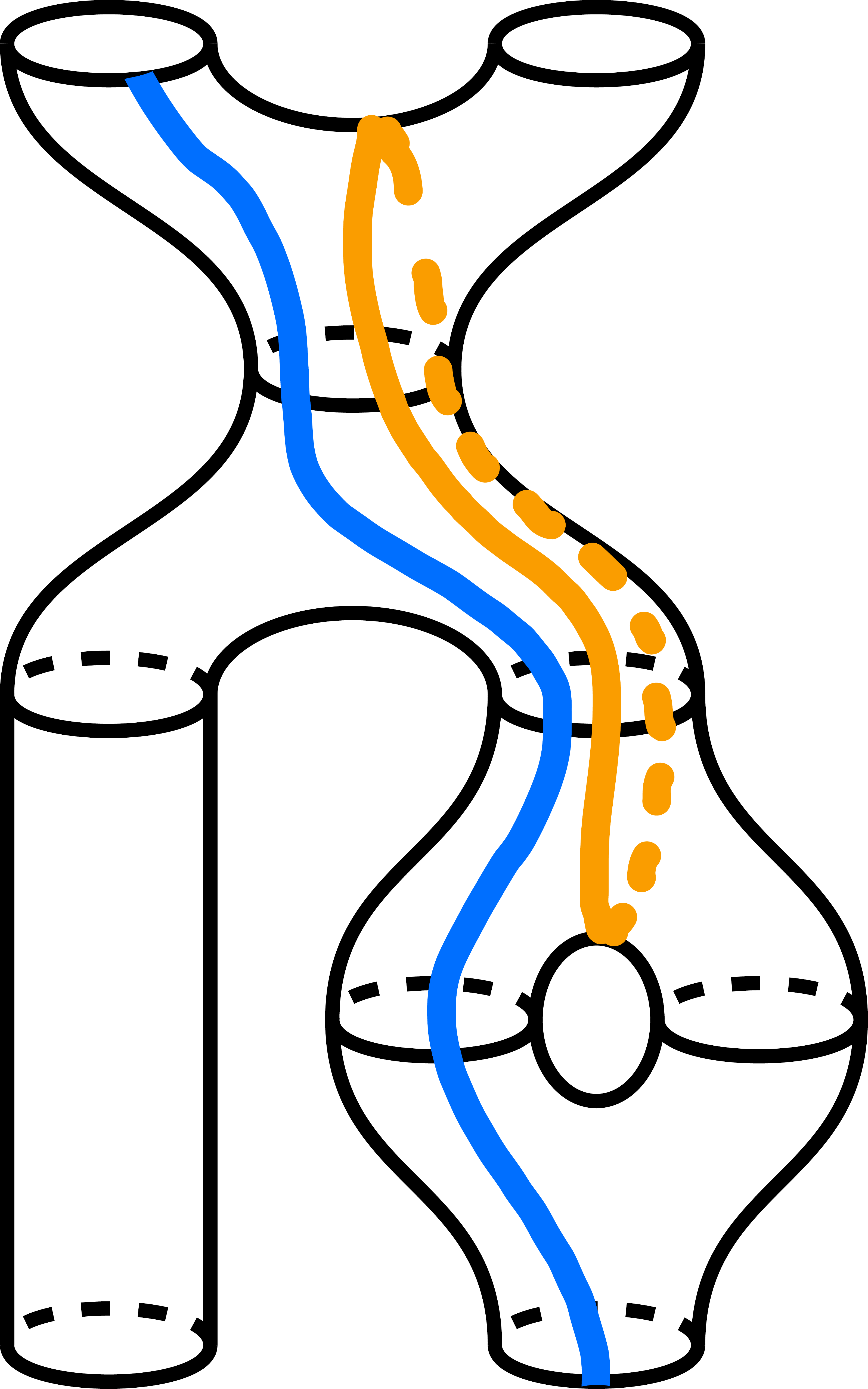}
			\end{aligned}
			\quad
			\xmapsto{~Z_{SN}(\epsilon)~}
			\quad
			\begin{aligned}
			\includegraphics[scale=0.04]{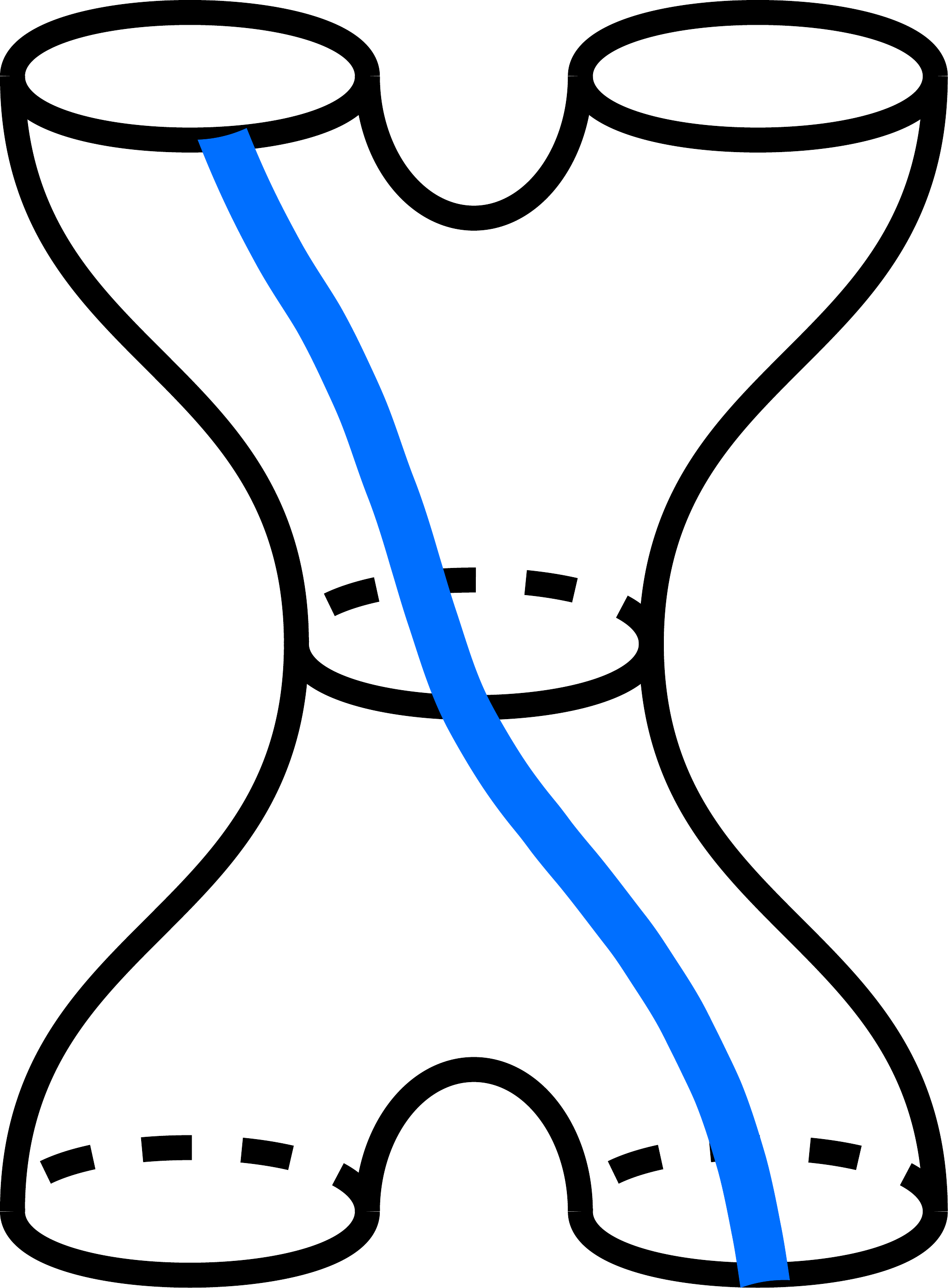}
			\end{aligned}
\end{equation*}

\noindent which is clearly the inverse of the composite

\begin{equation*}
			\begin{aligned}
			\includegraphics[scale=0.04]{rigidity_4.pdf}
			\end{aligned}
			\quad
			\xmapsto{~Z_{SN}(\epsilon^{\dagger})~}
			\quad
			\begin{aligned}
			\includegraphics[scale=0.04]{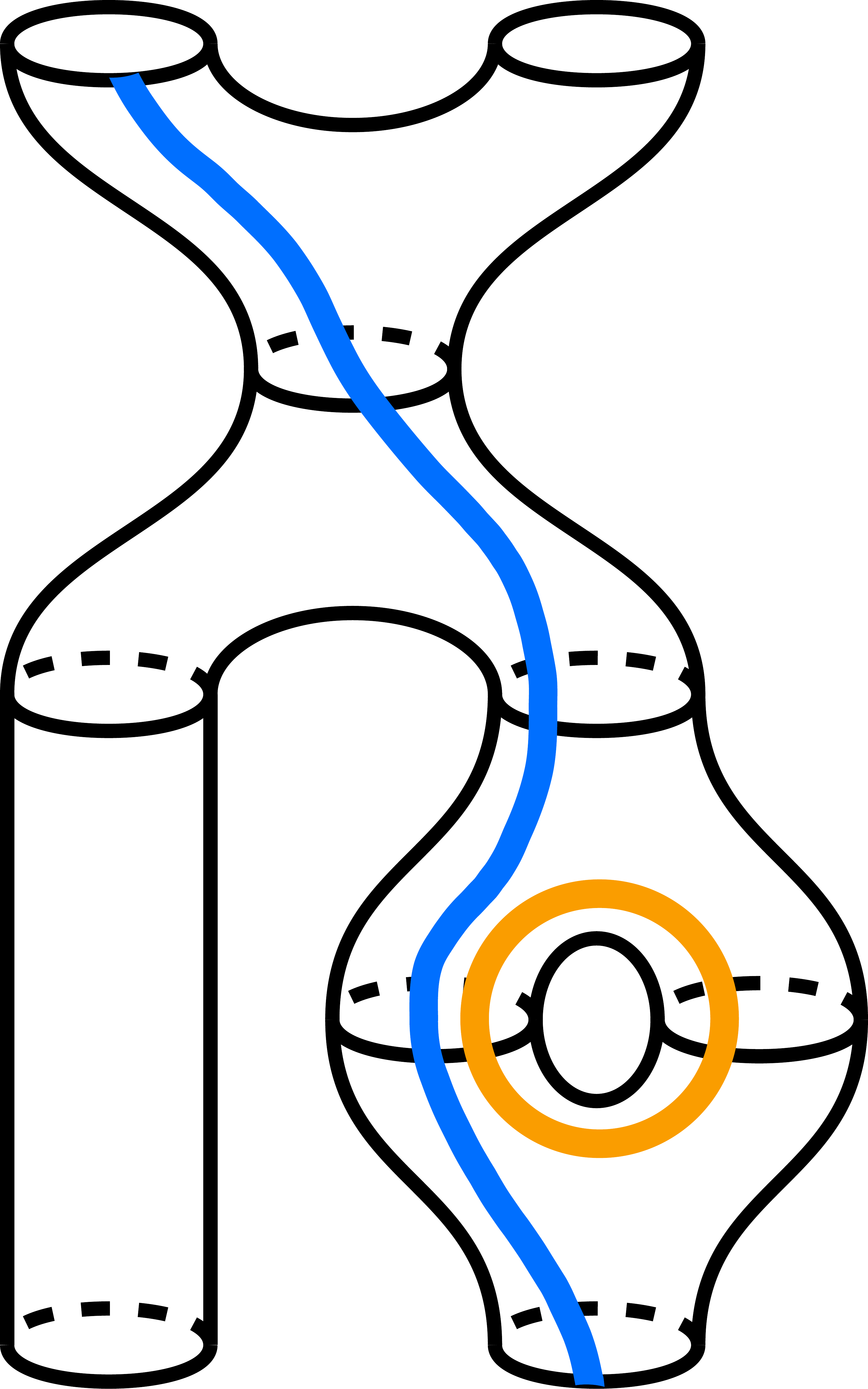}
			\end{aligned}
			\quad
			\xmapsto{~Z_{SN}(\alpha^{\inv})~}
			\quad
			\begin{aligned}
			\includegraphics[scale=0.04]{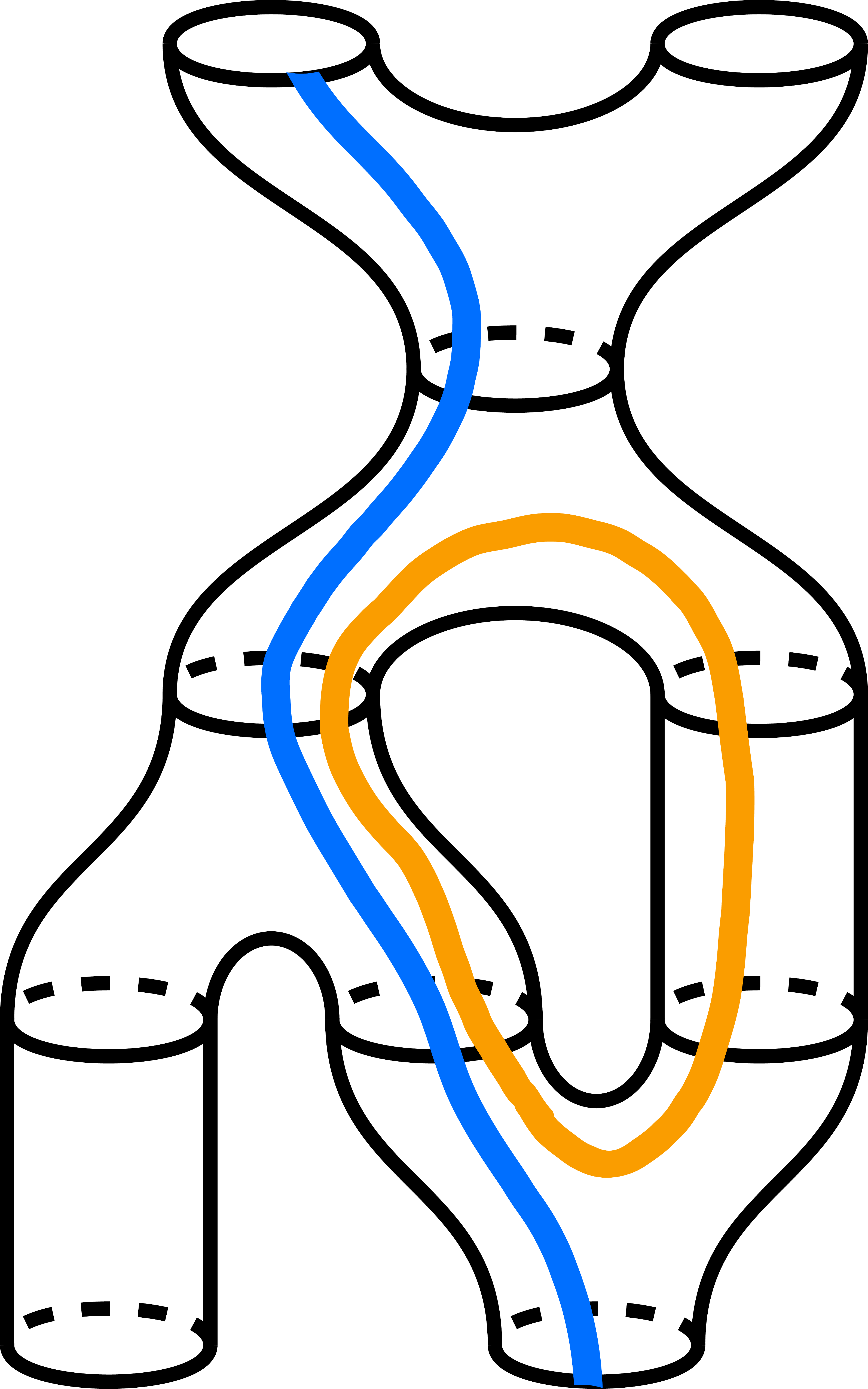}
			\end{aligned}
			\quad
			\xmapsto{~Z_{SN}(\eta^{\dagger})~}
			\quad
			\begin{aligned}
			\includegraphics[scale=0.04]{rigidity_1.pdf}
			\end{aligned}
\end{equation*}

\noindent The right Frobeniusator, and other cases, follows similarly.

\newpage

\subsubsection{Pivotality} The composite

	\begin{align*}
			\begin{aligned}
			\includegraphics[scale=0.04]{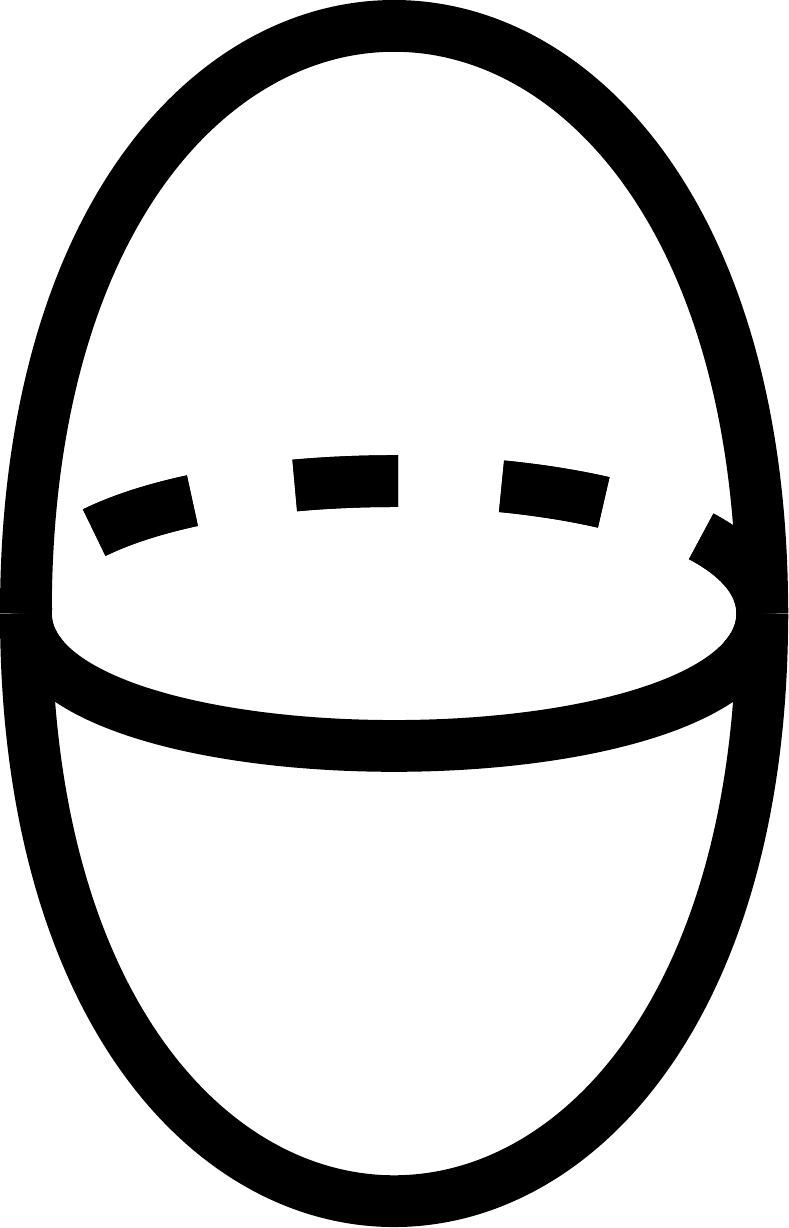}
			\end{aligned}
			\quad
			&\xmapsto{~Z_{SN}(\epsilon^{\dagger})~}
			\quad
			\begin{aligned}
			\includegraphics[scale=0.04]{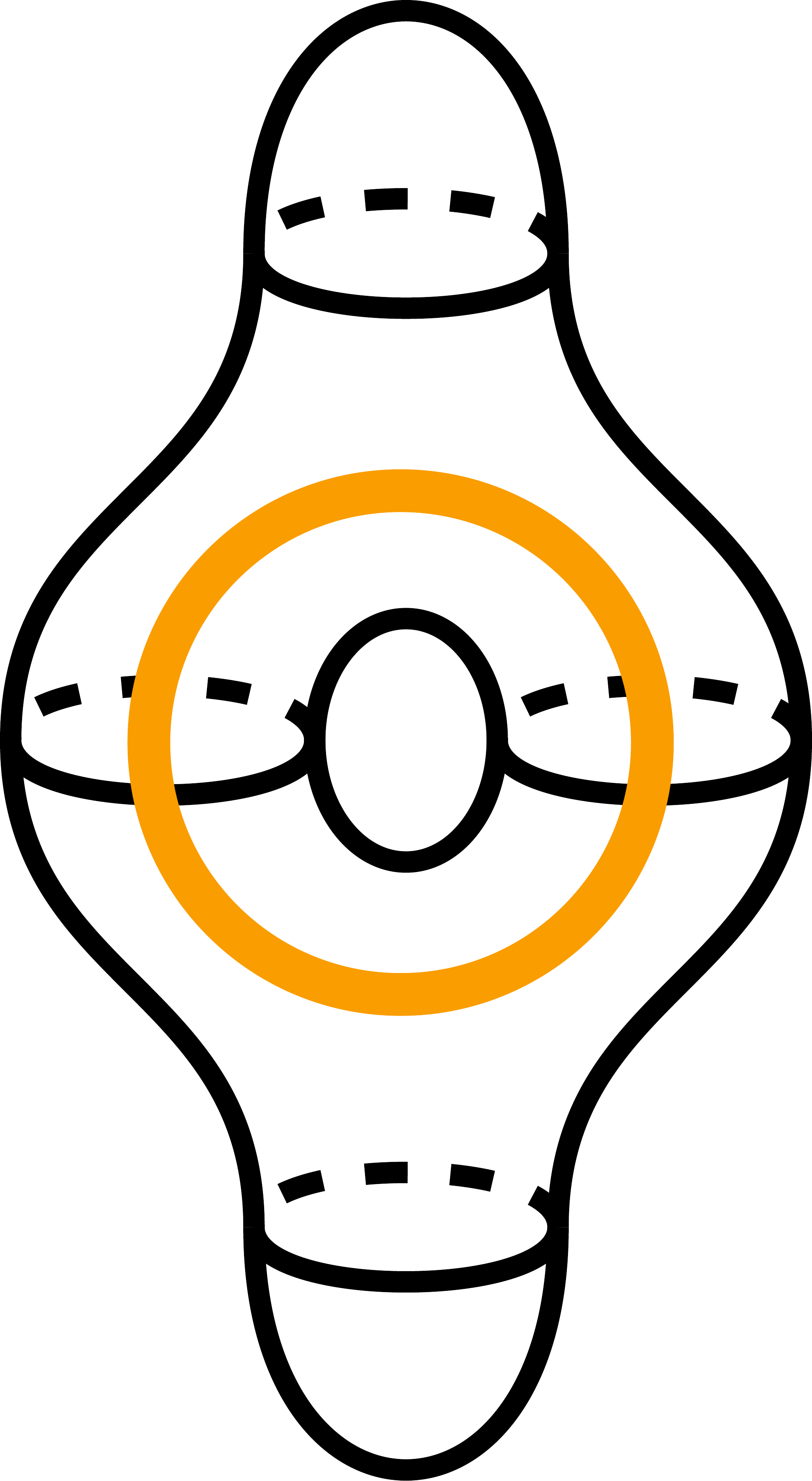}
			\end{aligned}
			\quad
			=
			\quad
			\begin{aligned}
			\includegraphics[scale=0.04]{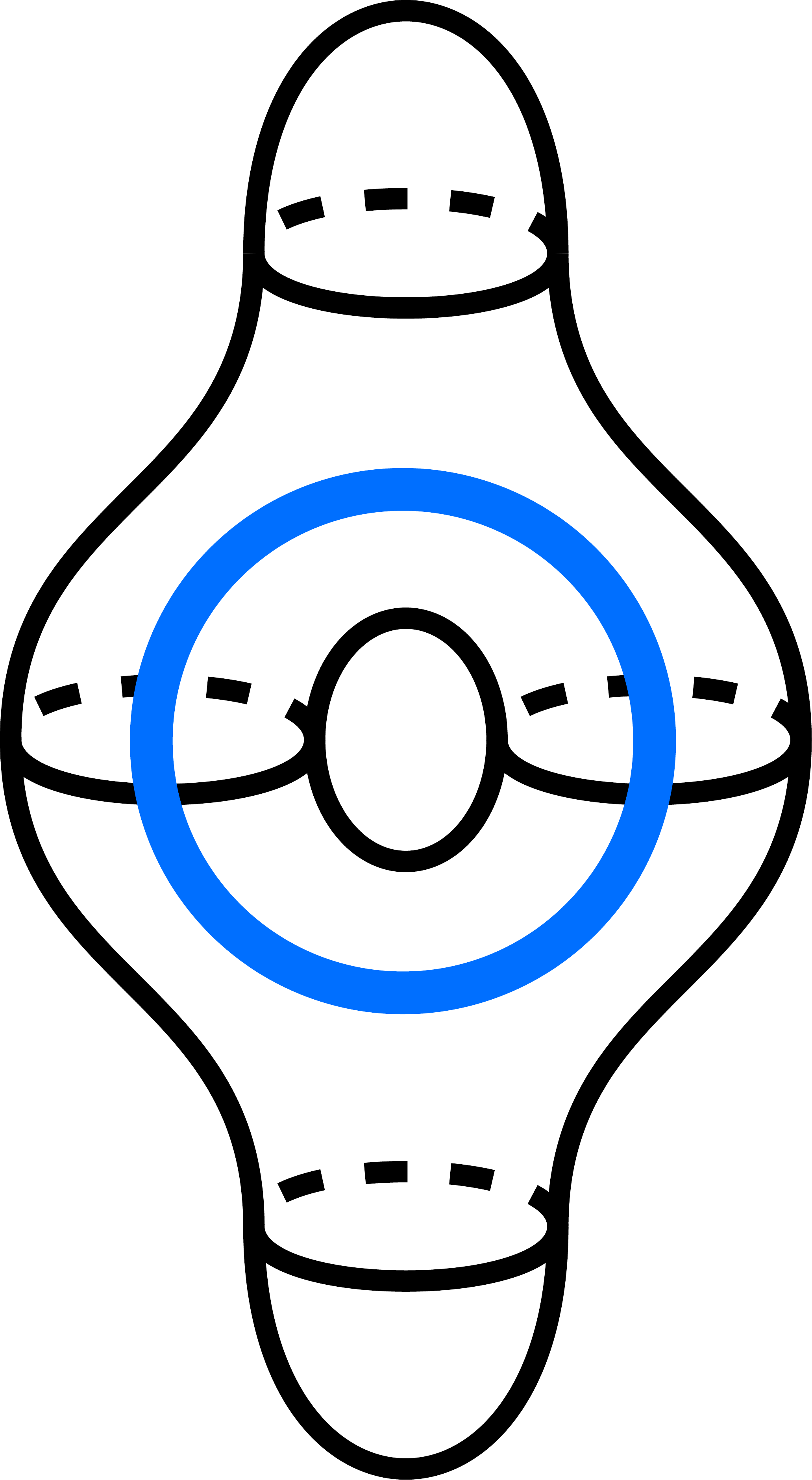}
			\end{aligned}
			\quad
			+
			\quad
			\begin{aligned}
			\includegraphics[scale=0.04]{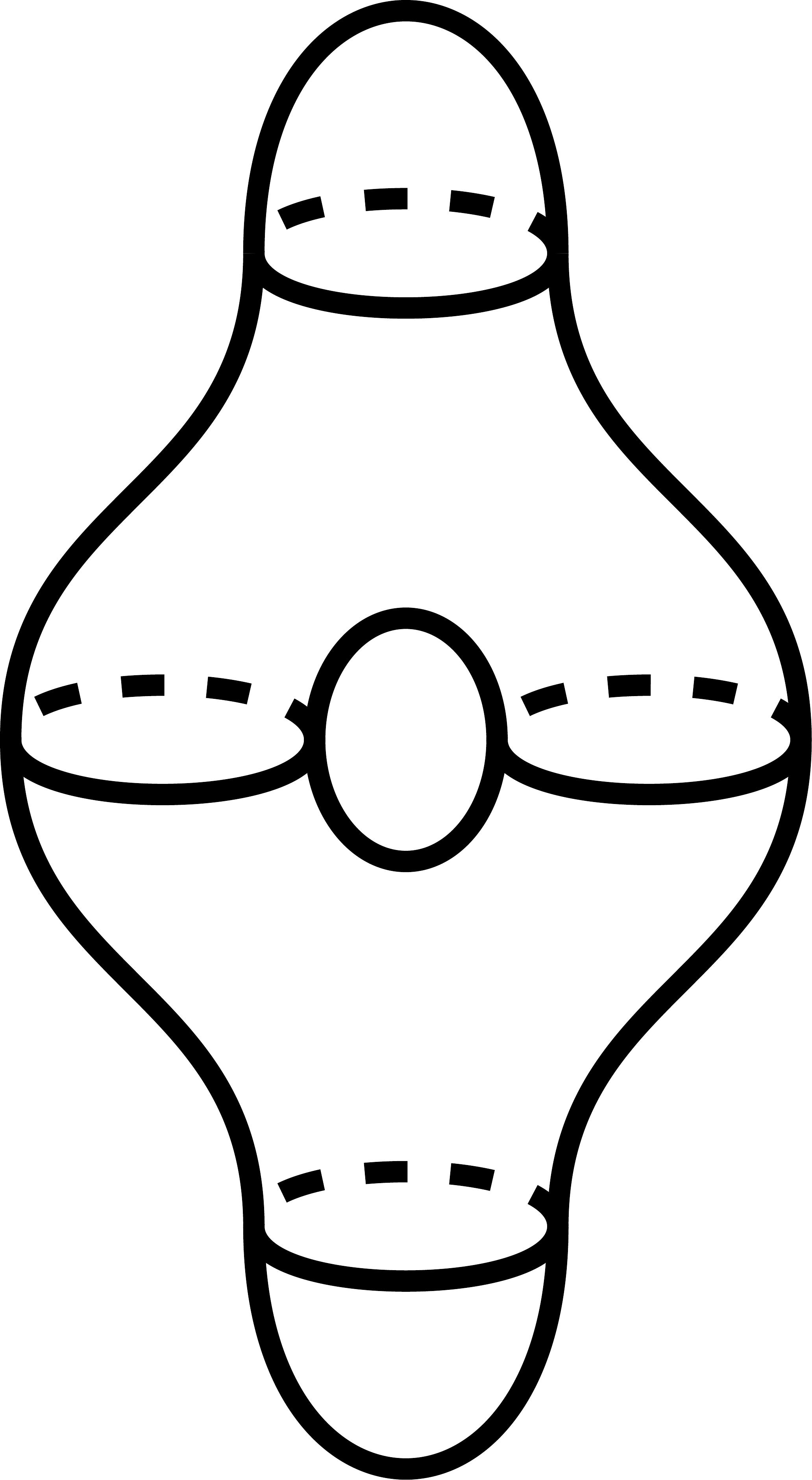}
			\end{aligned}\\
			&\xmapsto{~Z_{SN}(\mu^{\dagger})~}
			\quad
			0
			\quad
			+
			\quad
			\begin{aligned}
			\includegraphics[scale=0.04]{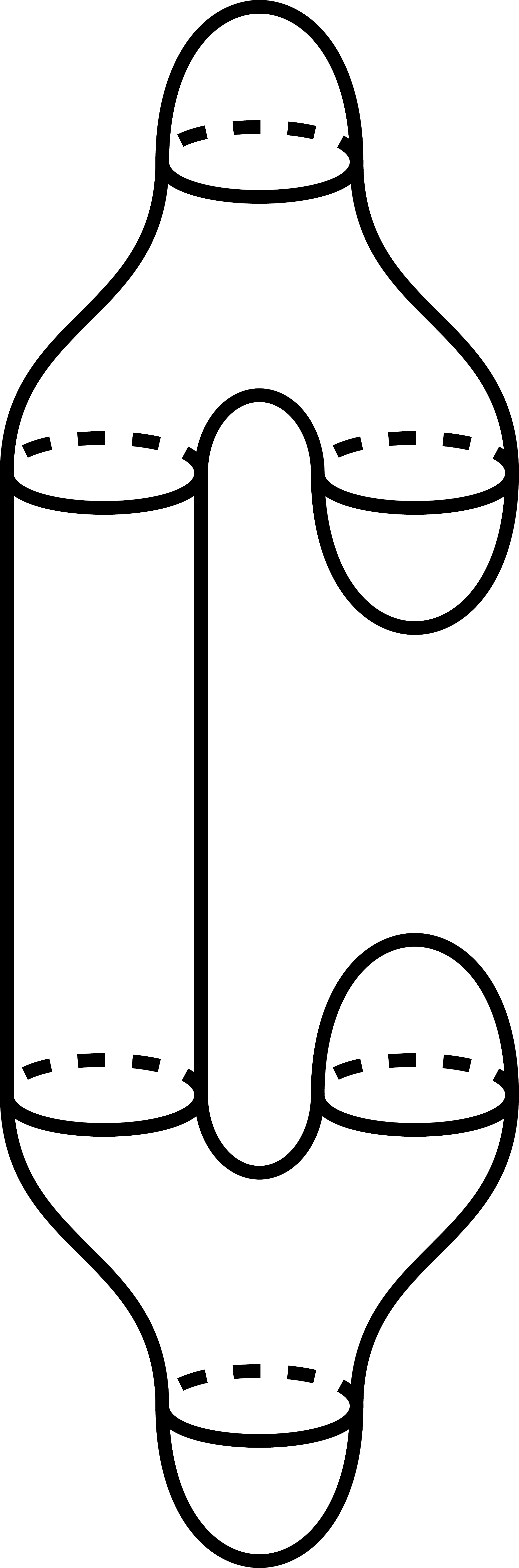}
			\end{aligned}\\
			&\xmapsto{~Z_{SN}(\mu)~}
			\quad
			\begin{aligned}
			\includegraphics[scale=0.04]{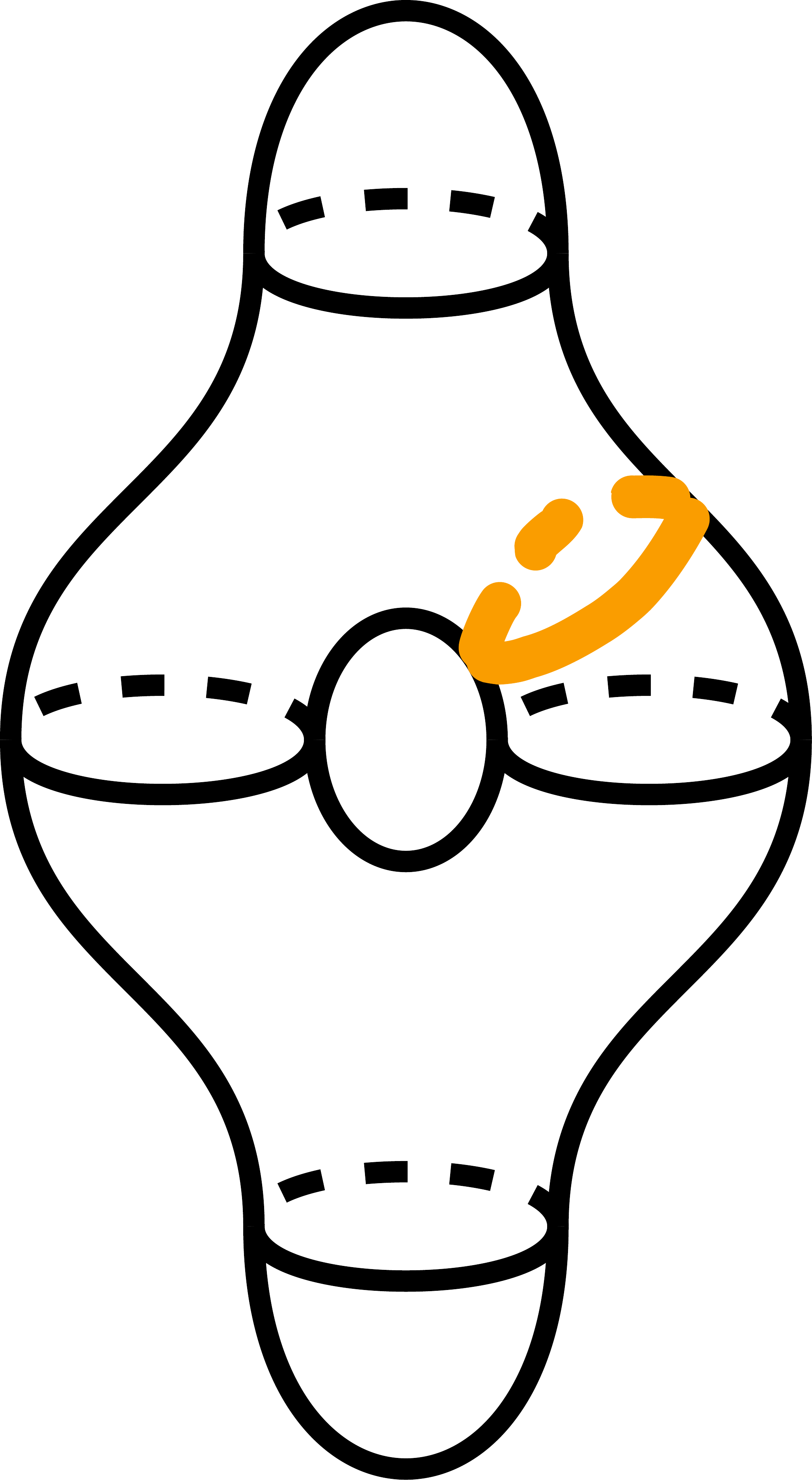}
			\end{aligned}
			\quad
			=
			\quad
			\begin{aligned}
			\includegraphics[scale=0.04]{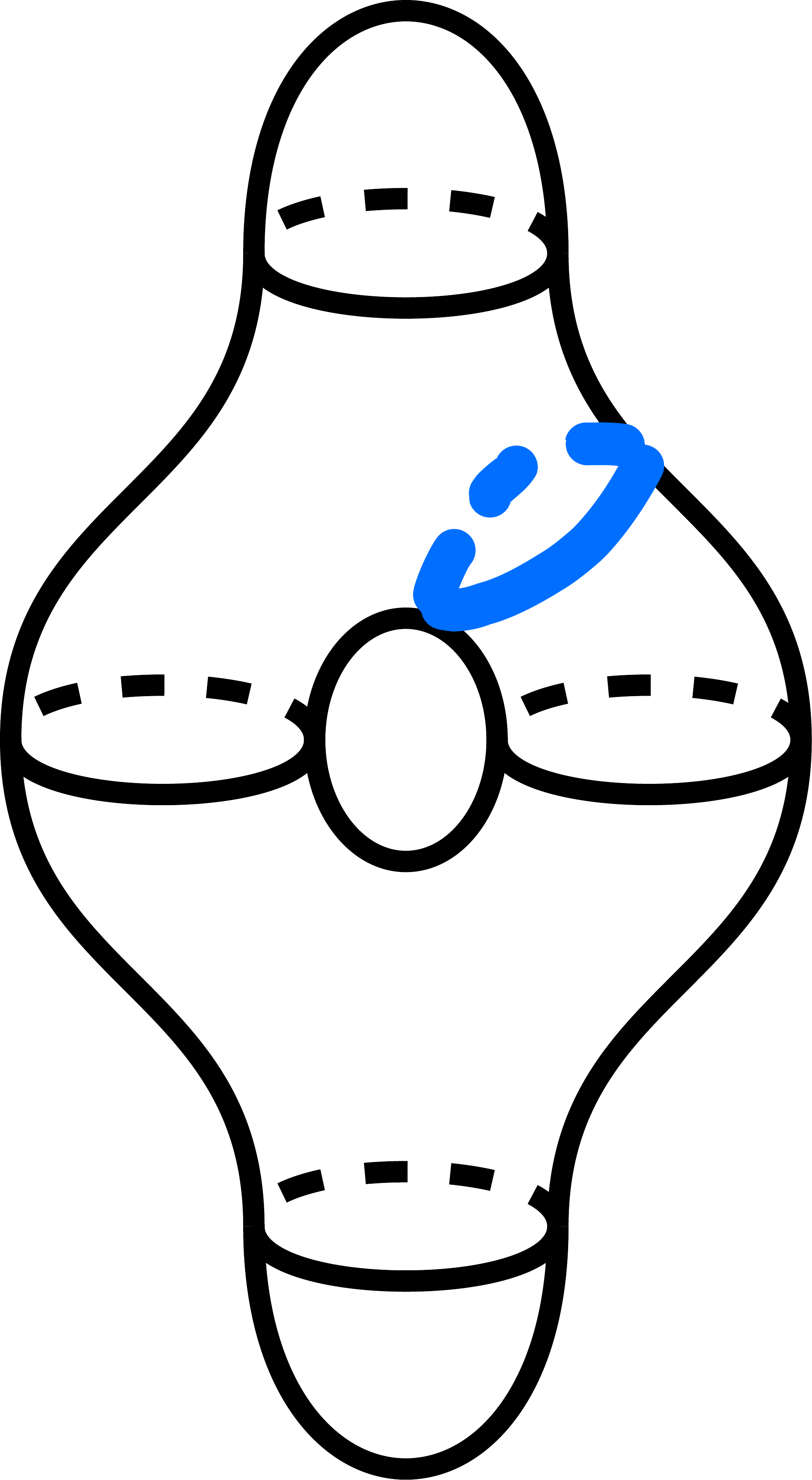}
			\end{aligned}
			\quad
			+
			\quad
			\begin{aligned}
			\includegraphics[scale=0.04]{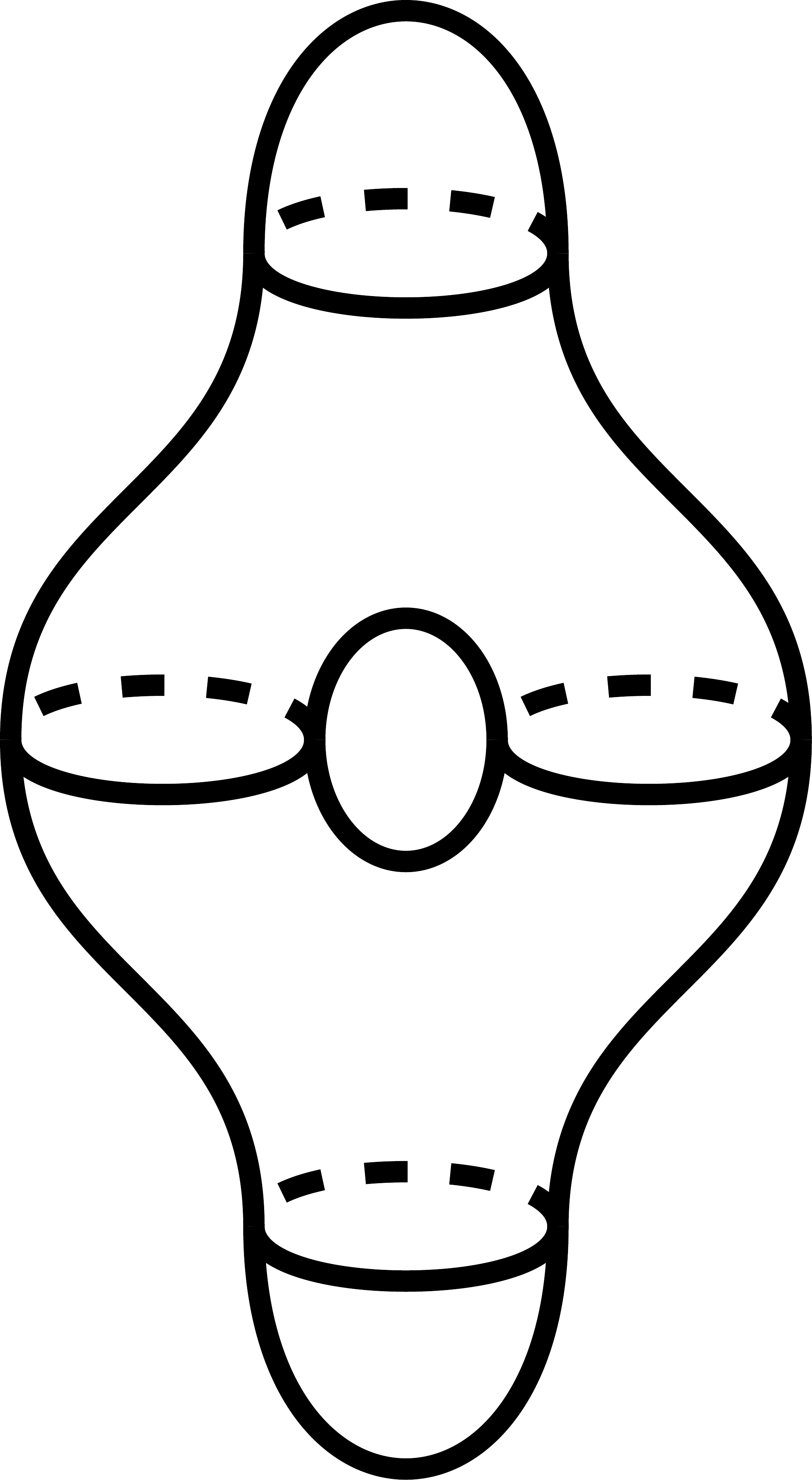}
			\end{aligned}\\
			&\xmapsto{~Z_{SN}(\epsilon)~}
			\quad
			0
			\quad
			+
			\quad
			\begin{aligned}
			\includegraphics[scale=0.04]{pivotality_1.pdf}
			\end{aligned}
	\end{align*}

\noindent clearly equals the identity\hspace{0.2cm} $\begin{aligned}
			\includegraphics[scale=0.04]{pivotality_1.pdf}
			\end{aligned}
			~~
			\xmapsto{~~\text{id}~~}
			~~
			\begin{aligned}
			\includegraphics[scale=0.04]{pivotality_1.pdf}
			\end{aligned}$. ~Note that this calculation is fully general since the sphere has a 1-dimensional string-net space.

\subsubsection{Adjunction} For the adjunction between $\nu$ and $\mu$, the equality
		
		\begin{equation*}
			\begin{aligned}
			\includegraphics[scale=0.04]{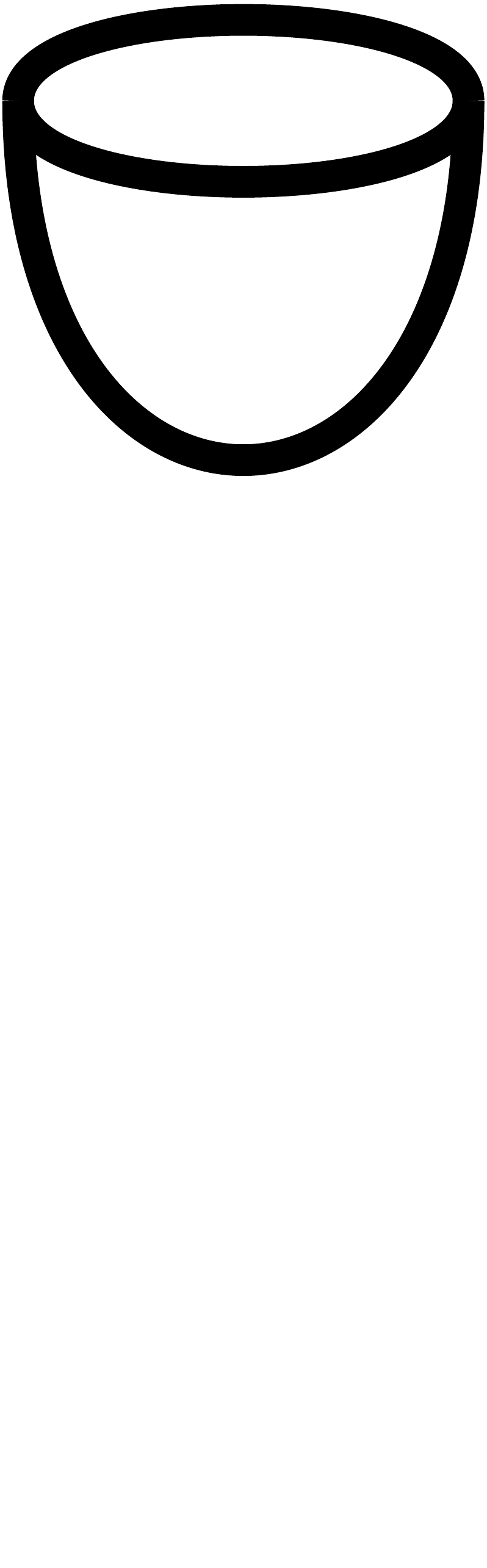}
			\end{aligned}
			\quad
			\xmapsto{~Z_{SN}(\nu)~}
			\quad
			\raisebox{0.2mm}{$\frac{1}{2}$}\;
			\begin{aligned}
			\includegraphics[scale=0.04]{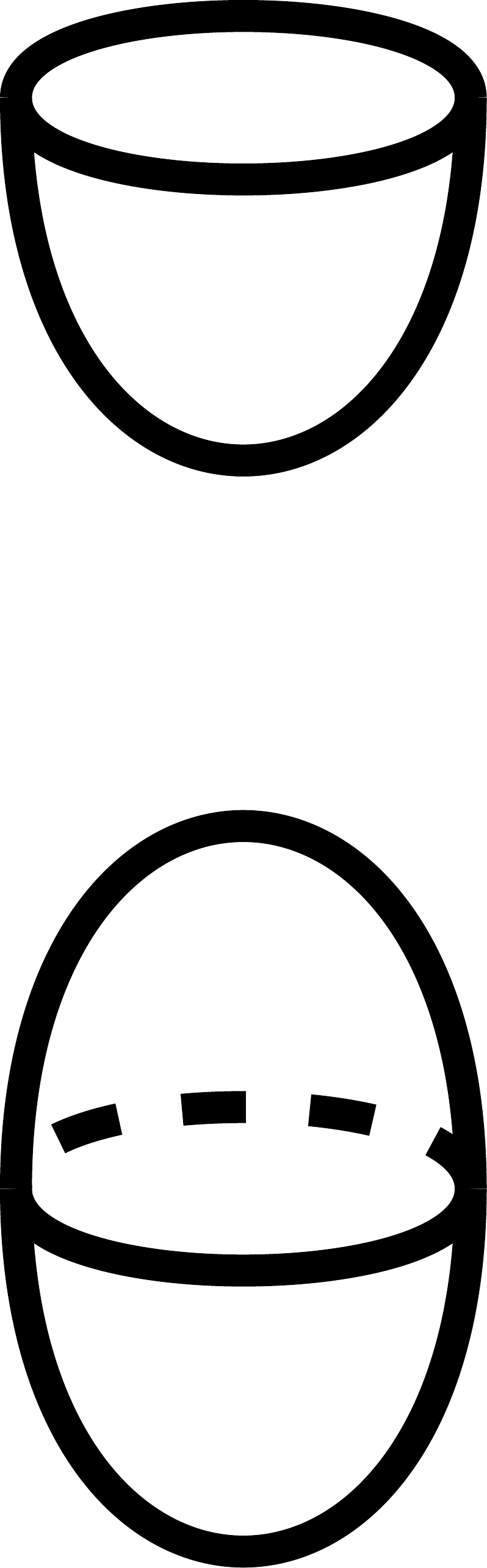}
			\end{aligned}
			\quad
			\xmapsto{~Z_{SN}(\mu)~}
			\quad
			\raisebox{0.2mm}{$\frac{1}{2}$}\;
			\begin{aligned}
			\includegraphics[scale=0.04]{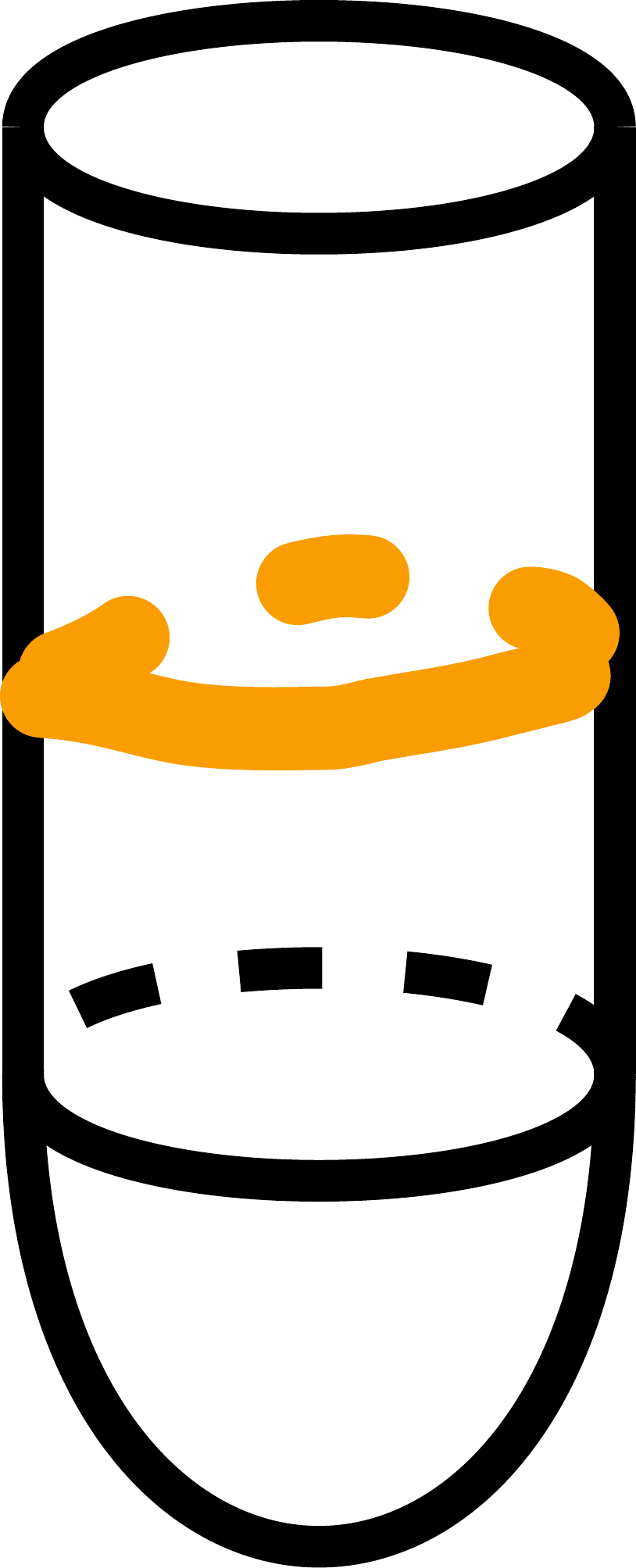}
			\end{aligned}
			\qquad
			=
			\qquad
			\begin{aligned}
			\includegraphics[scale=0.04]{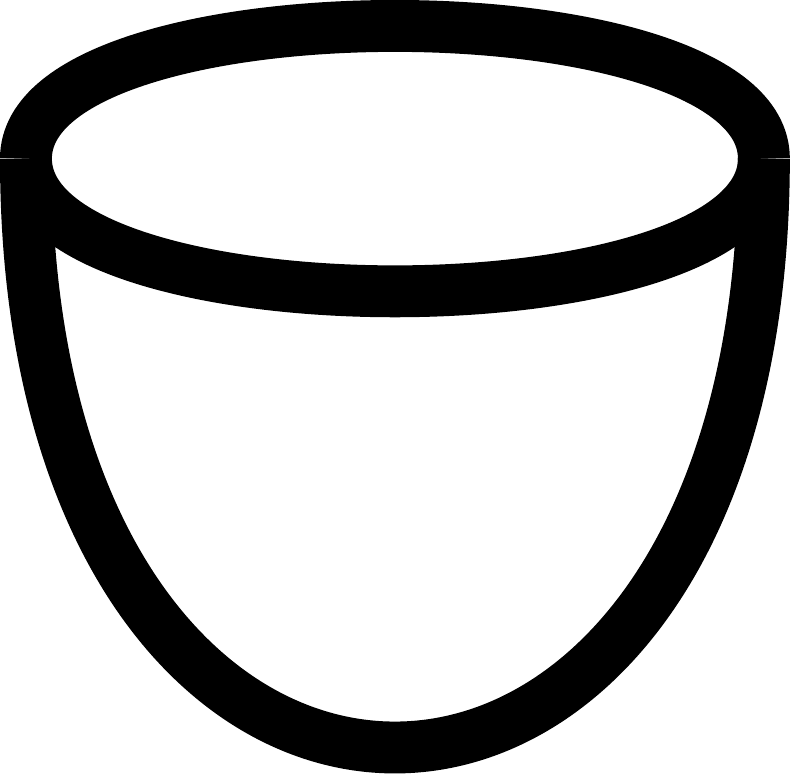}
			\end{aligned}
			\quad
			\xmapsto{~~\text{id}~~}
			\quad
			\begin{aligned}
			\includegraphics[scale=0.04]{adjunction_numu_4.pdf}
			\end{aligned}
		\end{equation*}
		
		\noindent holds since \hspace{0.15cm}$\begin{aligned} \includegraphics[scale=0.04]{adjunction_numu_3.pdf} \end{aligned} \;=\; \begin{aligned} \includegraphics[scale=0.04]{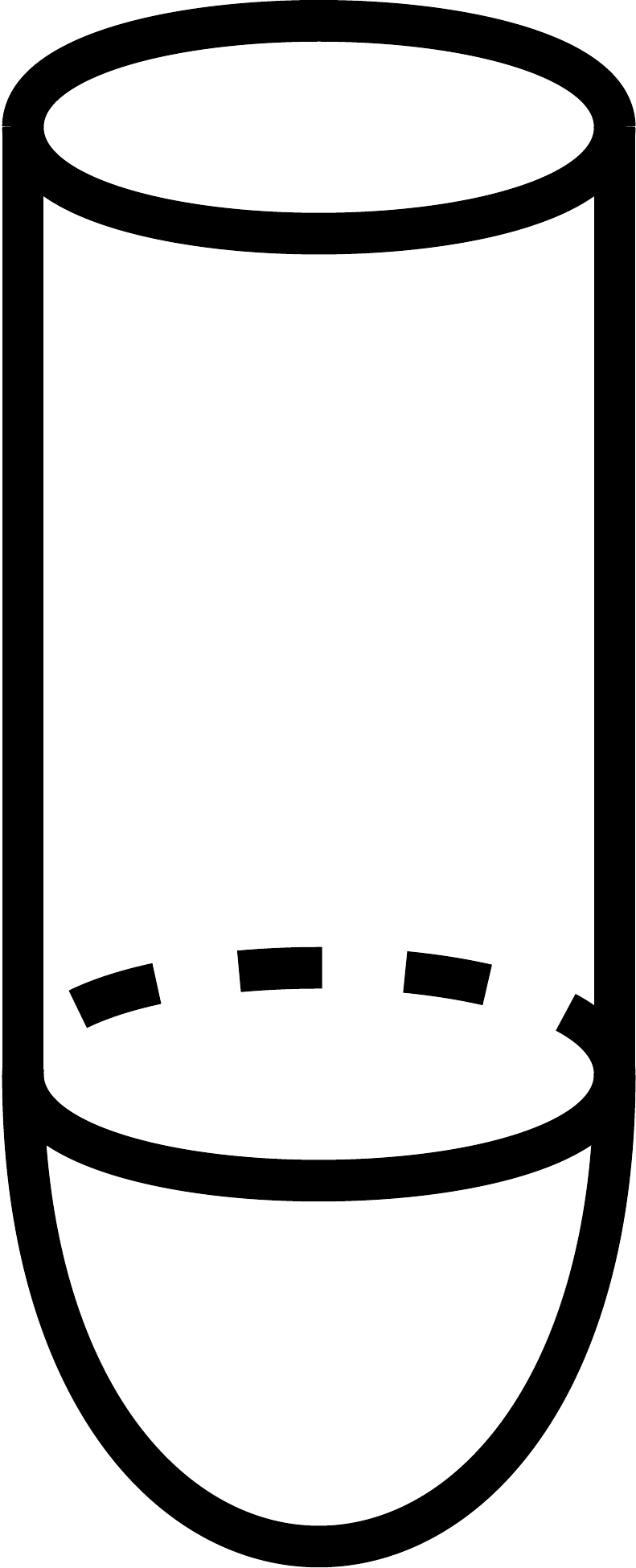} \end{aligned} \;+\; \begin{aligned} \includegraphics[scale=0.04]{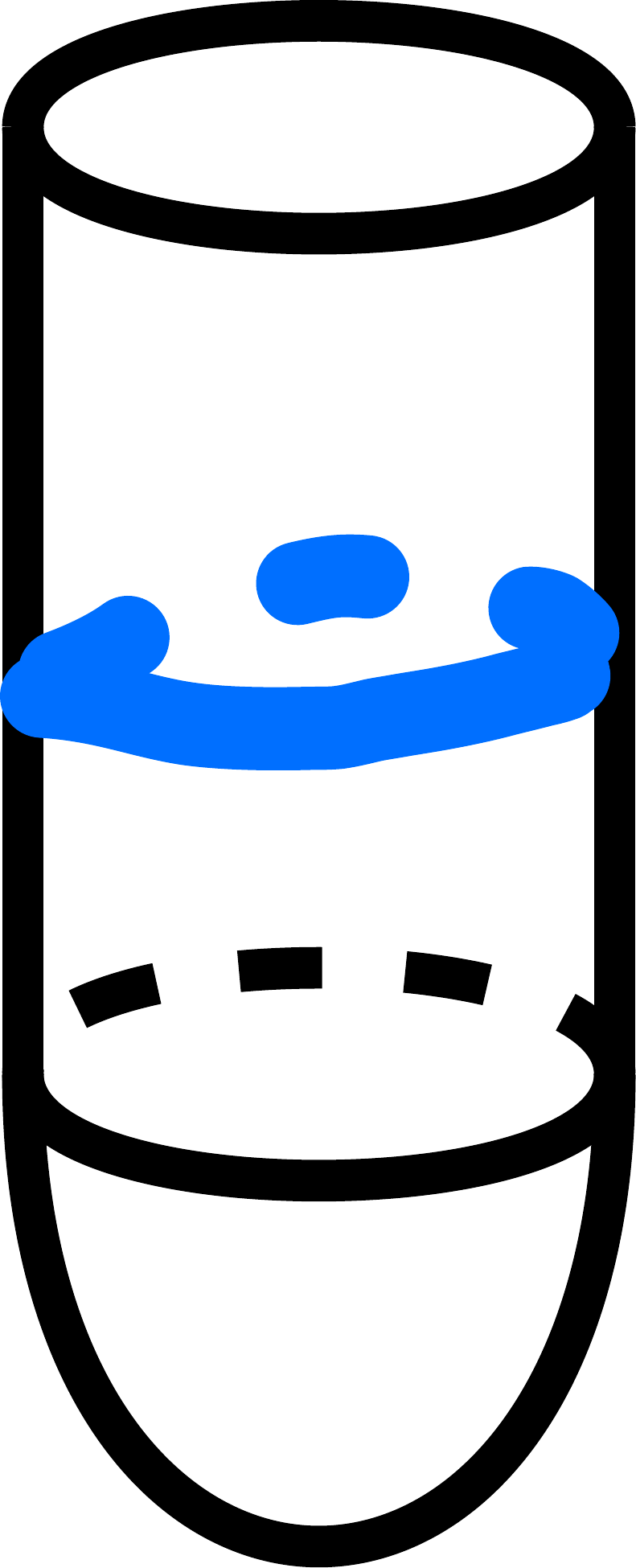} \end{aligned} \;=\; \raisebox{0.2mm}{$2$}\; \begin{aligned}
			\includegraphics[scale=0.04]{adjunction_numu_3a.pdf} \end{aligned}$. The daggered version of the above is given by
		
		\begin{equation*}
			\begin{aligned}
			\includegraphics[scale=0.04]{adjunction_numu_3a.pdf}
			\end{aligned}
			\quad
			\xmapsto{~Z_{SN}(\mu^{\dagger})~}
			\quad
			\begin{aligned}
			\includegraphics[scale=0.04]{adjunction_numu_2.pdf}
			\end{aligned}
			\quad
			\xmapsto{~Z_{SN}(\nu^{\dagger})~}
			\quad
			\begin{aligned}
			\includegraphics[scale=0.04]{adjunction_numu_1.pdf}
			\end{aligned}
			\qquad
			=
			\qquad
			\begin{aligned}
			\includegraphics[scale=0.04]{adjunction_numu_1.pdf}
			\end{aligned}
			\quad
			\xmapsto{~~\text{id}~~}
			\quad
			\begin{aligned}
			\includegraphics[scale=0.04]{adjunction_numu_1.pdf}
			\end{aligned}
		\end{equation*}
		
		\noindent The calculations above are once again fully general since the spaces involved are 1-dimensional.\\
		
		\noindent For the adjunction between $\eta$ and $\epsilon$, on a sample basis vector, the equality
		
		\begin{equation*}
			\begin{aligned}
			\includegraphics[scale=0.04]{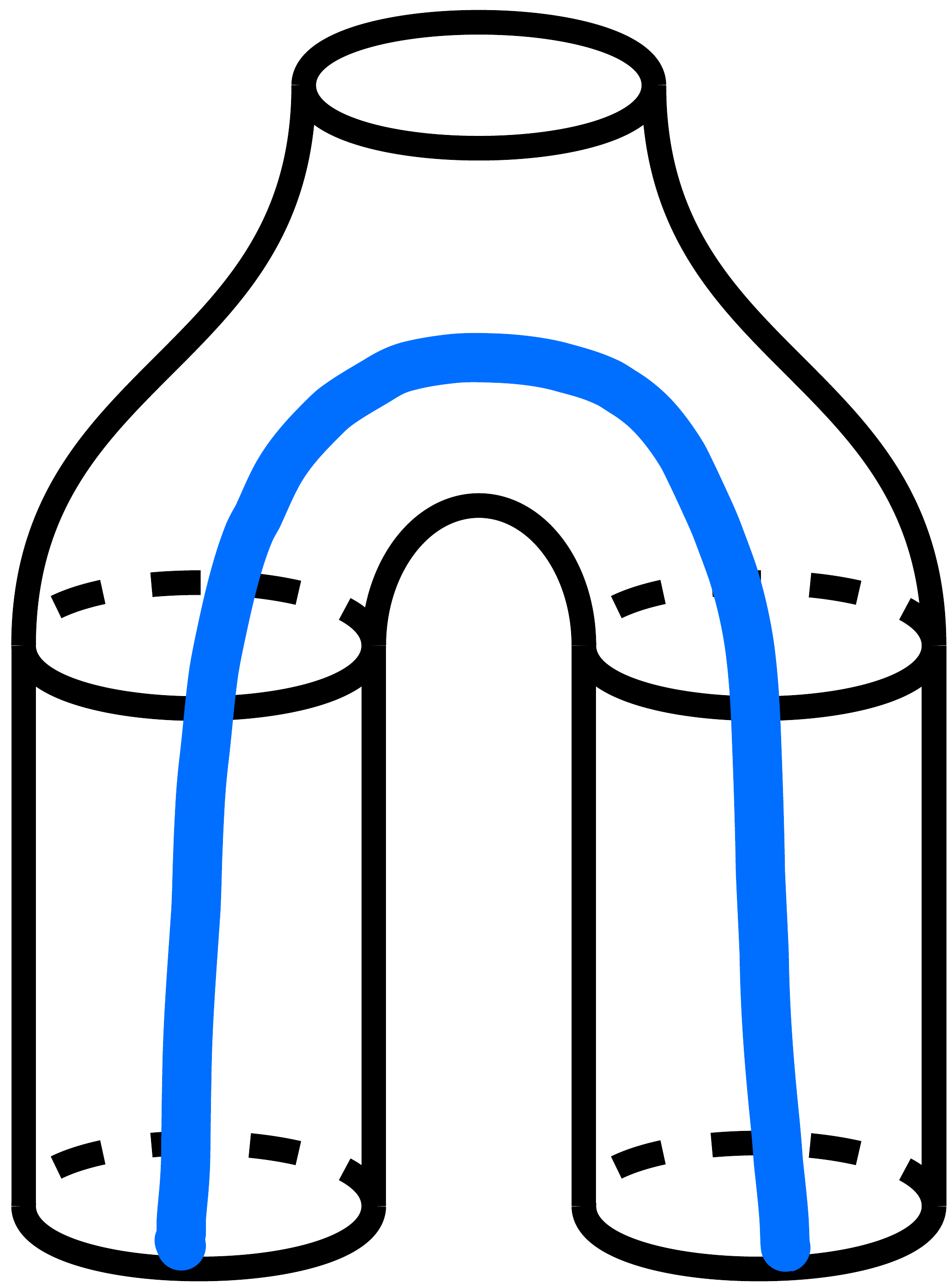}
			\end{aligned}
			\quad
			\xmapsto{~Z_{SN}(\eta)~}
			\quad
			\begin{aligned}
			\includegraphics[scale=0.04]{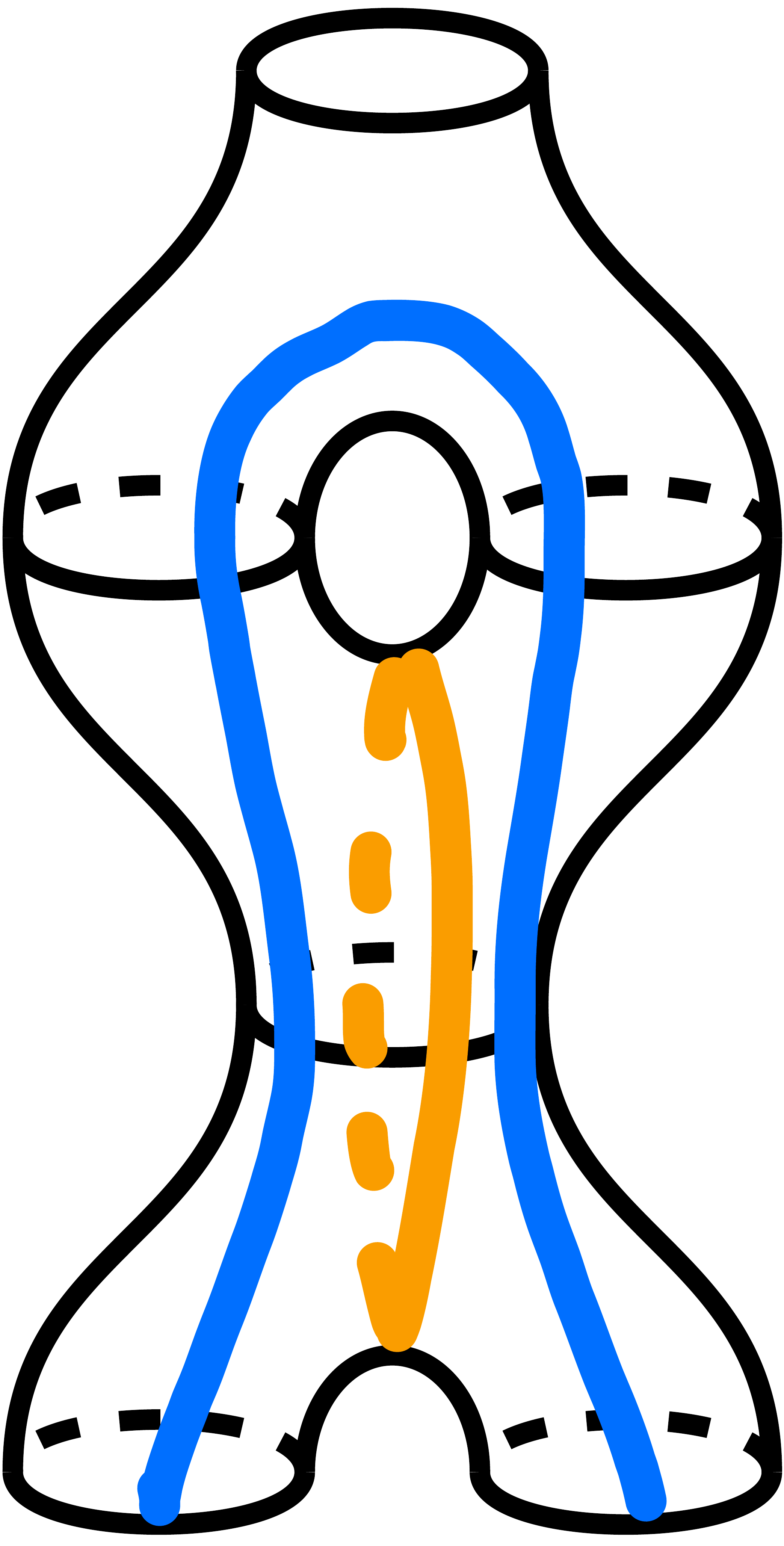}
			\end{aligned}
			\quad
			\xmapsto{~Z_{SN}(\epsilon)~}
			\quad
			\begin{aligned}
			\includegraphics[scale=0.04]{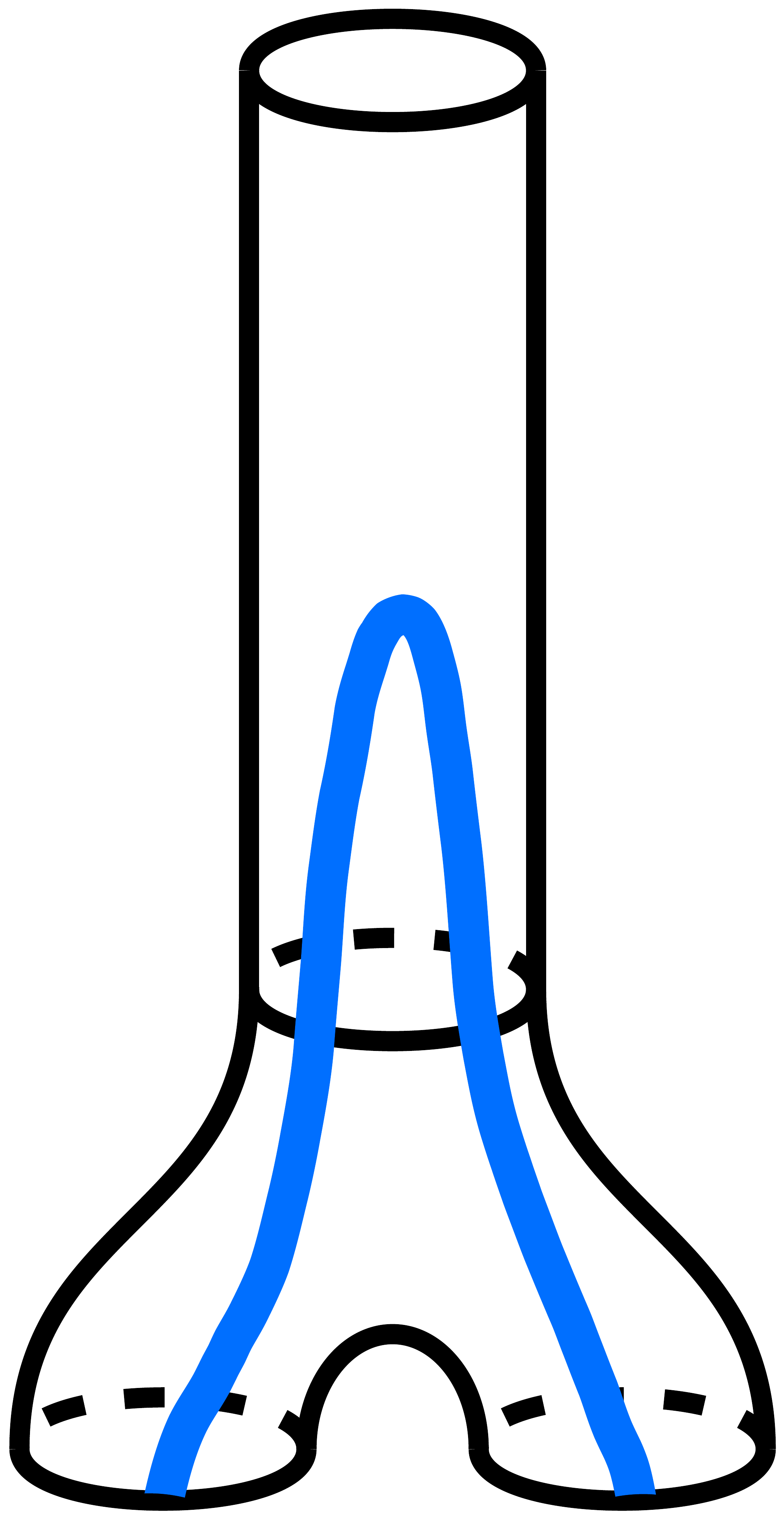}
			\end{aligned}
			\qquad
			=
			\qquad
			\begin{aligned}
			\includegraphics[scale=0.04]{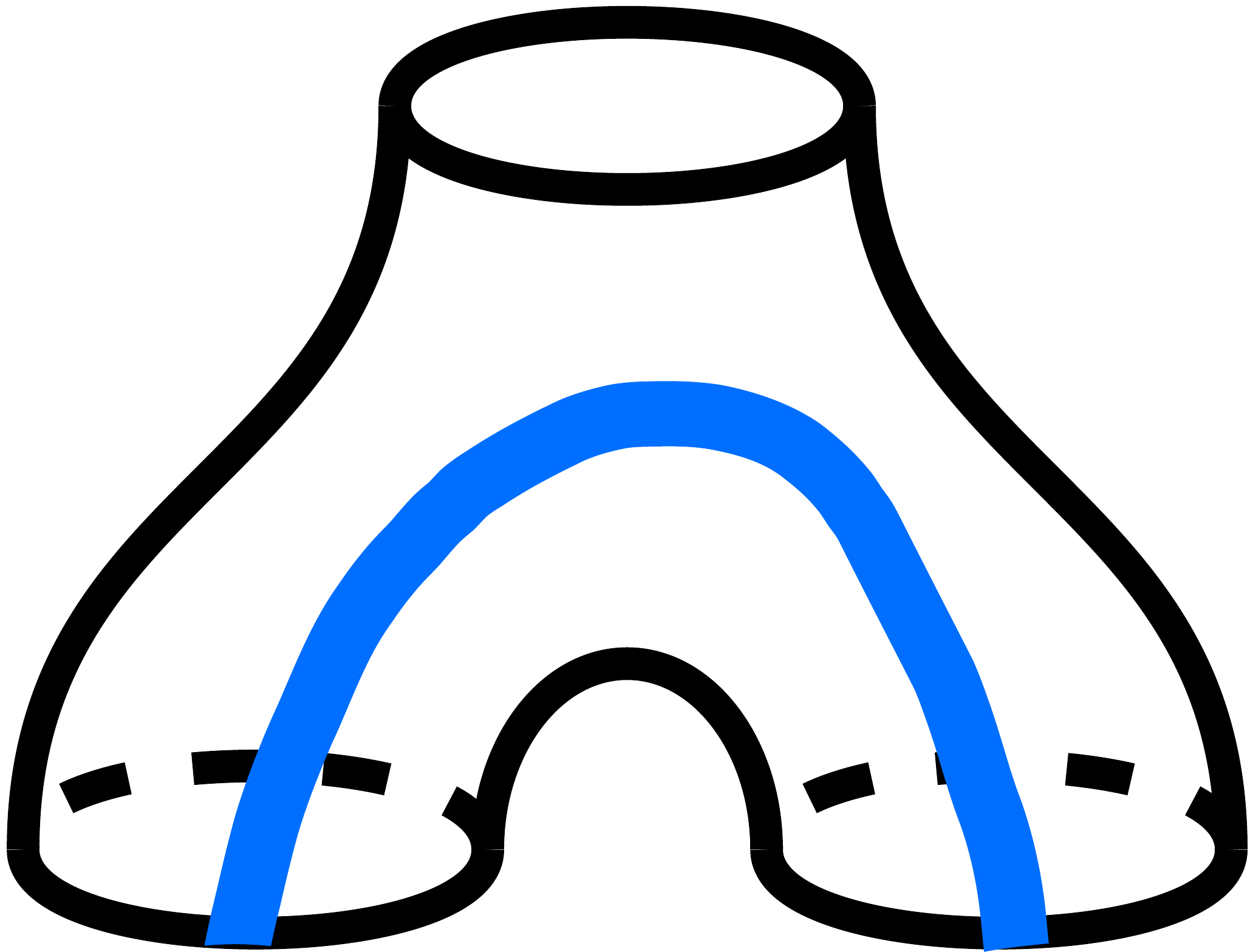}
			\end{aligned}
			\quad
			\xmapsto{~~\text{id}~~}
			\quad
			\begin{aligned}
			\includegraphics[scale=0.04]{adjunction_etaepsilon_4.pdf}
			\end{aligned}
		\end{equation*}
		
		\noindent clearly holds. The daggered version of the above is given by
		
		\begin{equation*}
			\begin{aligned}
			\includegraphics[scale=0.04]{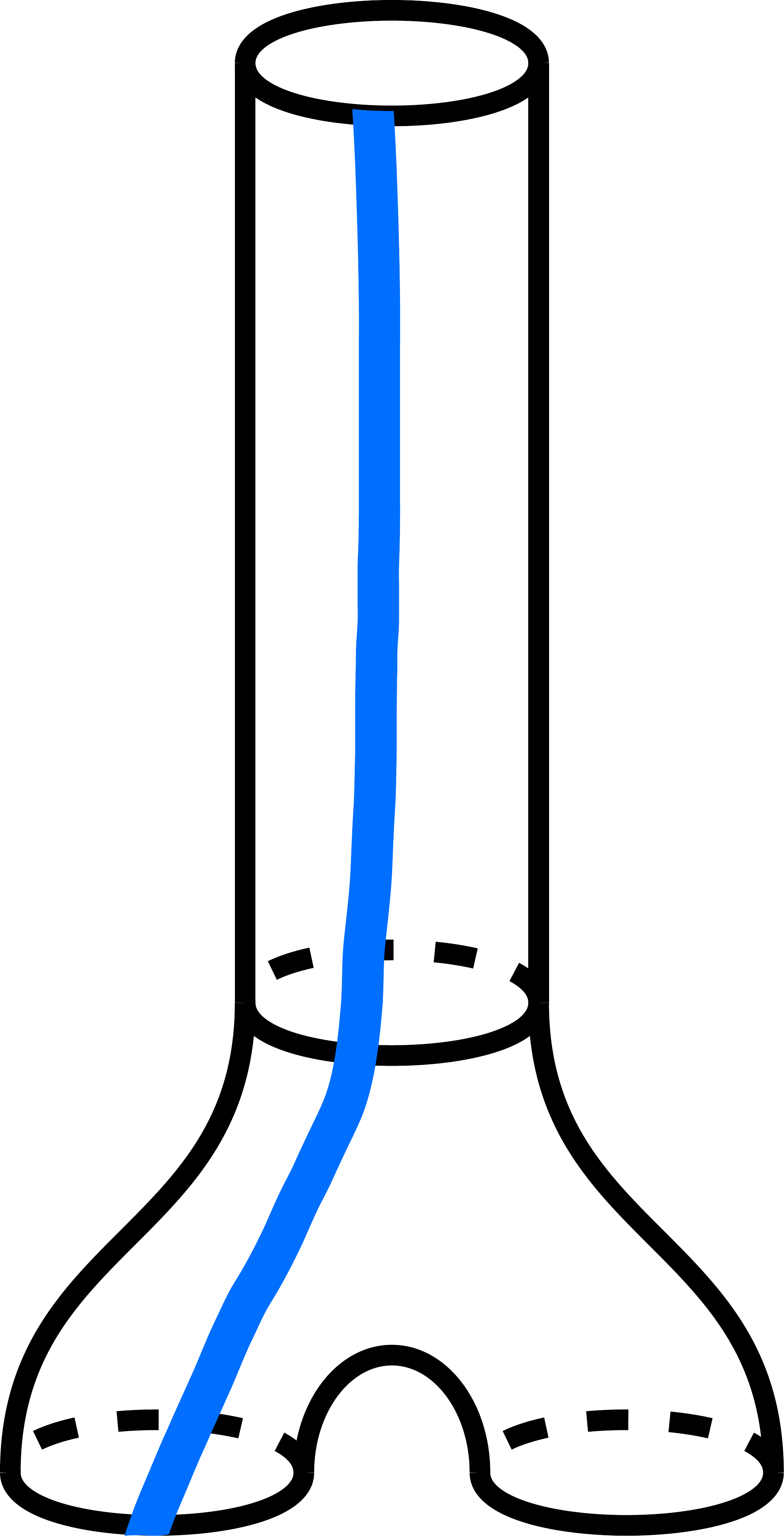}
			\end{aligned}
			\quad
			\xmapsto{~Z_{SN}(\epsilon^{\dagger})~}
			\quad
			\begin{aligned}
			\includegraphics[scale=0.04]{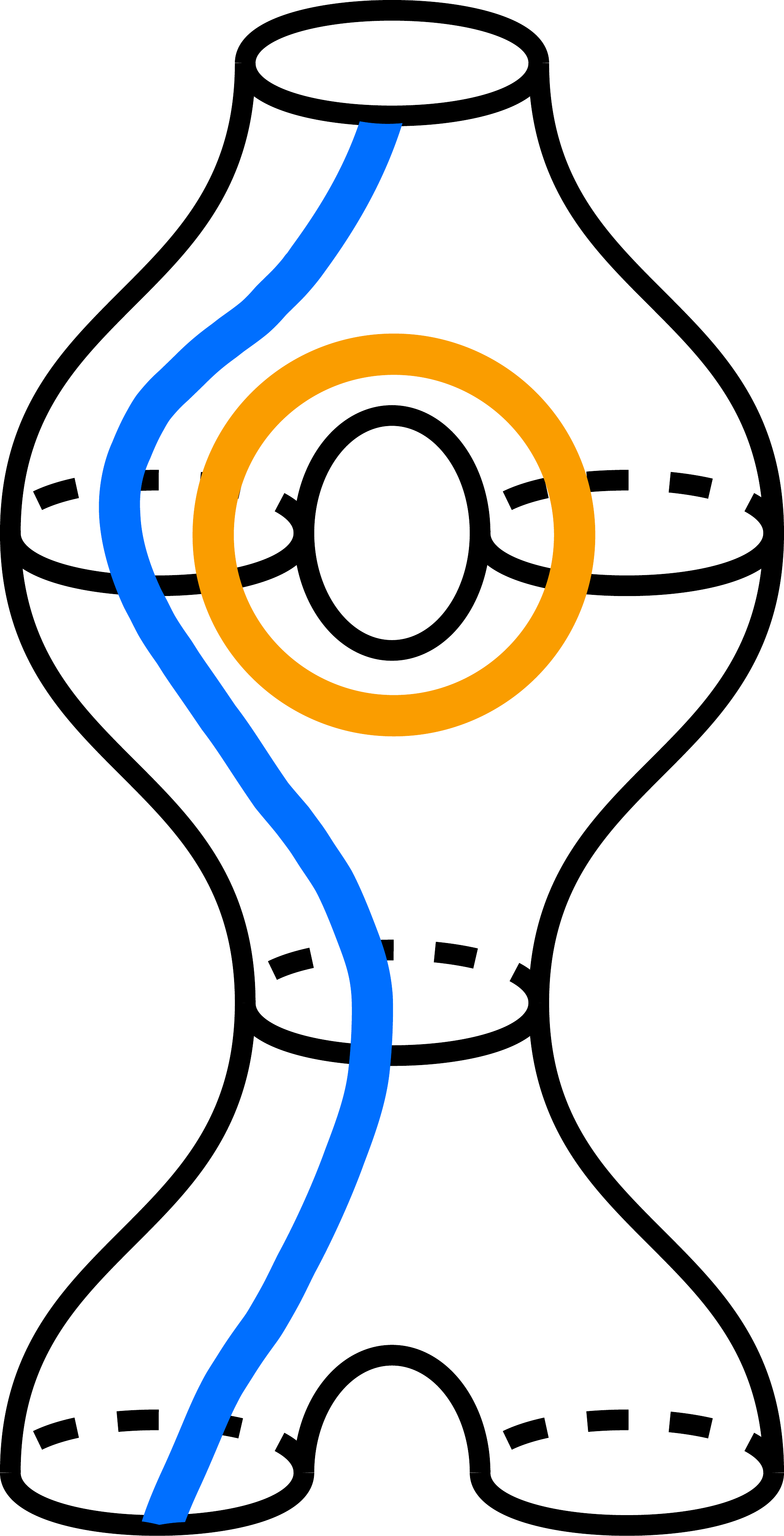}
			\end{aligned}
			\quad
			\xmapsto{~Z_{SN}(\eta^{\dagger})~}
			\quad
			\begin{aligned}
			\includegraphics[scale=0.04]{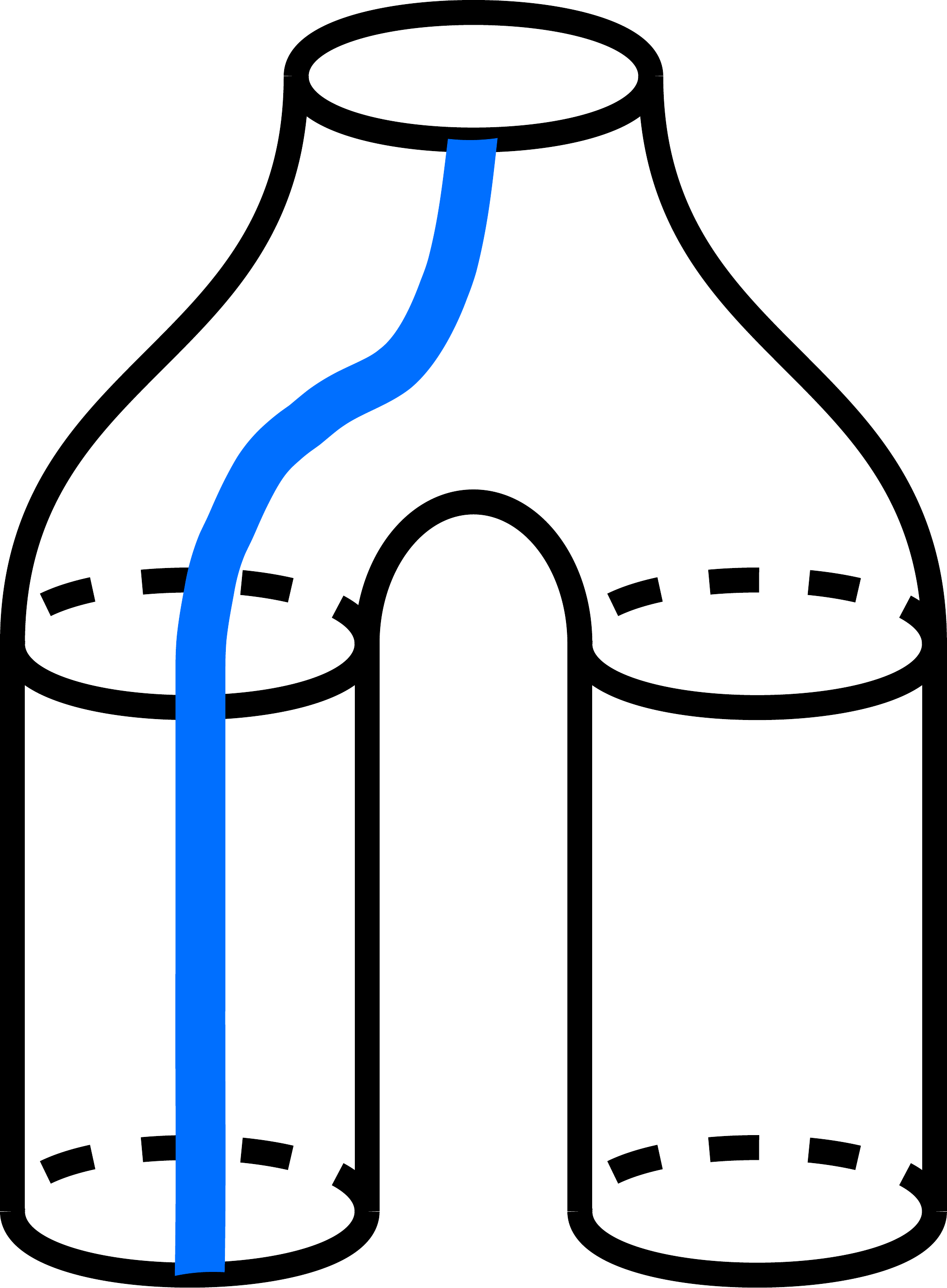}
			\end{aligned}
			\qquad
			=
			\qquad
			\begin{aligned}
			\includegraphics[scale=0.04]{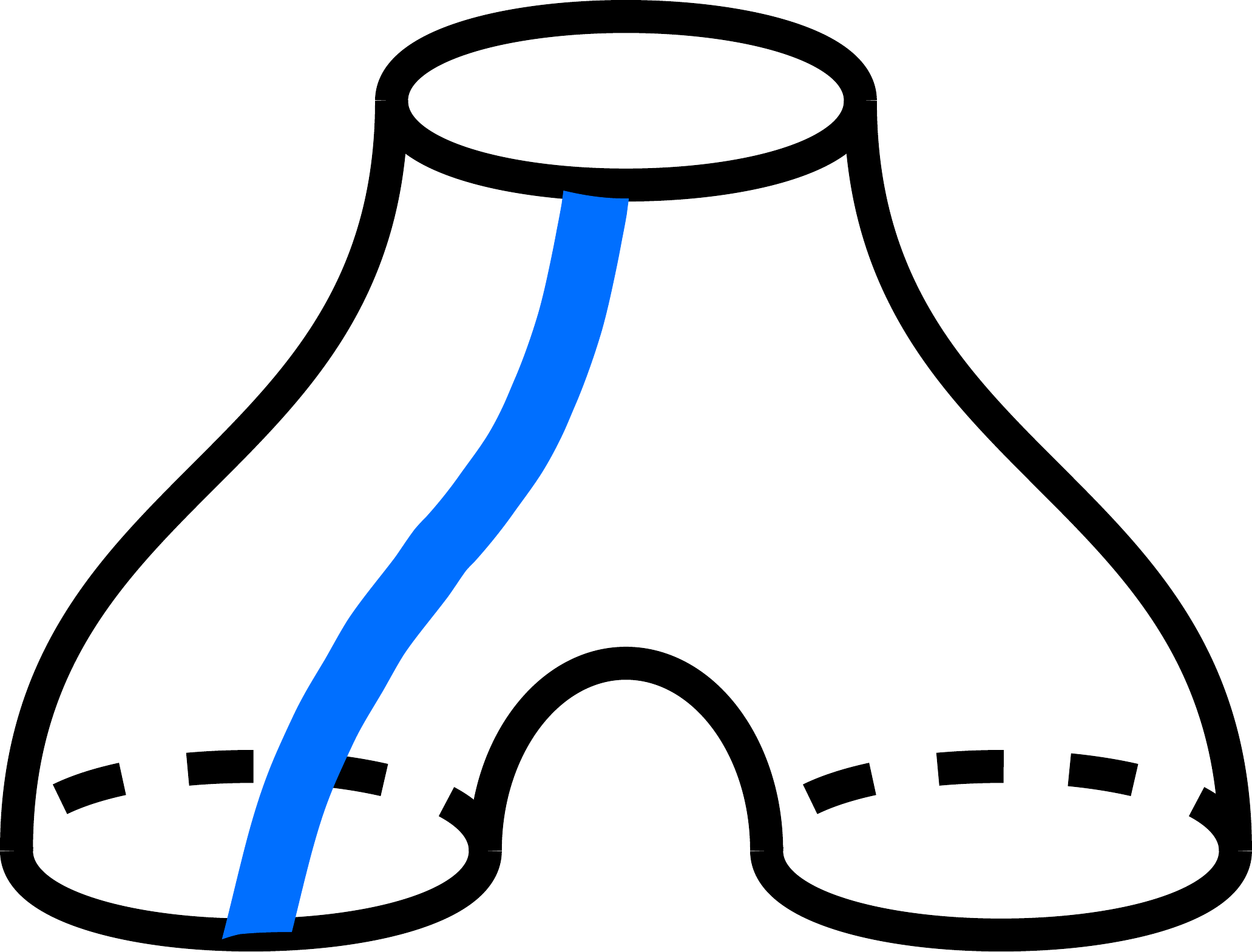}
			\end{aligned}
			\quad
			\xmapsto{~~\text{id}~~}
			\quad
			\begin{aligned}
			\includegraphics[scale=0.04]{adjunction_etaepsilon_dagger_4.pdf}
			\end{aligned}
		\end{equation*}
\noindent Other cases are similar.

\subsubsection{Anomaly-freeness} The composite

\begin{align*}
			\begin{aligned}
			\includegraphics[scale=0.04]{pivotality_1.pdf}
			\end{aligned}
			\quad
			&\xmapsto{~Z_{SN}(\epsilon^{\dagger})~}
			\quad
			\begin{aligned}
			\includegraphics[scale=0.04]{pivotality_2.pdf}
			\end{aligned}
			\quad
			=
			\quad
			\begin{aligned}
			\includegraphics[scale=0.04]{pivotality_2a.pdf}
			\end{aligned}
			\quad
			+
			\quad
			\begin{aligned}
			\includegraphics[scale=0.04]{pivotality_2b.pdf}
			\end{aligned}\\
			&\xmapsto{~Z_{SN}(\theta)~}
			\quad
			\begin{aligned}
			\includegraphics[scale=0.04]{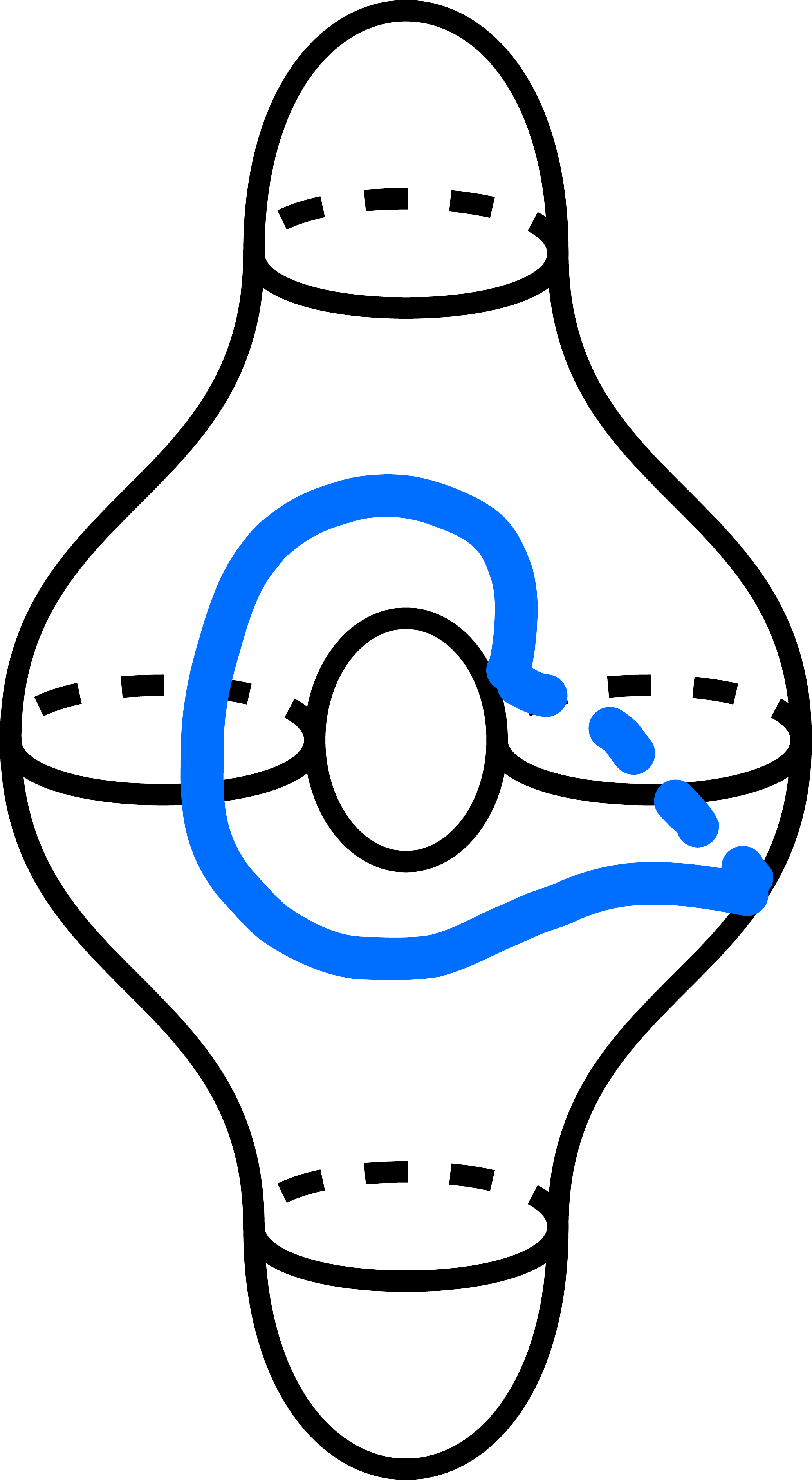}
			\end{aligned}
			\quad
			+
			\quad
			\begin{aligned}
			\includegraphics[scale=0.04]{pivotality_2b.pdf}
			\end{aligned}\\
			&\xmapsto{~Z_{SN}(\epsilon)~}
			\quad
			0
			\quad
			+
			\quad
			\begin{aligned}
			\includegraphics[scale=0.04]{pivotality_1.pdf}
			\end{aligned}
	\end{align*}

\noindent clearly equals the identity\hspace{0.2cm} $\begin{aligned}
			\includegraphics[scale=0.04]{pivotality_1.pdf}
			\end{aligned}
			~~
			\xmapsto{~~\text{id}~~}
			~~
			\begin{aligned}
			\includegraphics[scale=0.04]{pivotality_1.pdf}
			\end{aligned}$. ~Note that this calculation is fully general since the sphere has a 1-dimensional string-net space.

\section{Example Calculations}


\label{computations_section}

\noindent In this section we demonstrate how to apply the string-net formalism we have developed in this paper to carry out some standard TQFT calculations. 

\subsection{Dimension of string net spaces}
\scalecobordisms{0.5}

Let \sgn ~denote the oriented surface with genus $g$ and $n$ boundary circles. In this section we compute the dimension of $H(\sgn ; B)$, where $B$ is any admissable boundary condition. In the case of $B_0$ boundary conditions, we can interpret the string-net space in terms of homology (see Remark \ref{homology_remark}) and the dimension calculation is a standard result in algebraic topology. It is nevertheless instructive to derive the same result here, for arbitary boundary conditions, using only simple string-net arguments. 

\subsubsection{Reduction to $B_0$ case}

The following lemma shows that it is sufficent to consider the case where each boundary component is coloured by $B_0$. We leave the proof to the reader. 

\begin{lemma}  \label{lem:zerotoone}
Let $B$ be any boundary value on \sgn, and let $A$ be a collection of disjoints arcs connecting the $B_1$ coloured boundary components in pairs. Consider the map
\begin{align*}
f_A : H(\sgn;B_0^{\boxtimes n}) & \to H(\sgn;B) \\
\Gamma & \mapsto \operatorname{smoothing}( \Gamma \cup A )
\end{align*}
\noindent where by `smoothing' we mean the resolution of any possible crossings introduced. (There are two possibilities for smoothing a crossing but they are equivalent by the $F$-move). Then $f_A$ is an isomorphism.
\end{lemma}

\noindent As a simple example, we compute $f_A$ on the following string net:

\begin{equation}
    \begin{aligned}
	\includegraphics[scale=0.06]{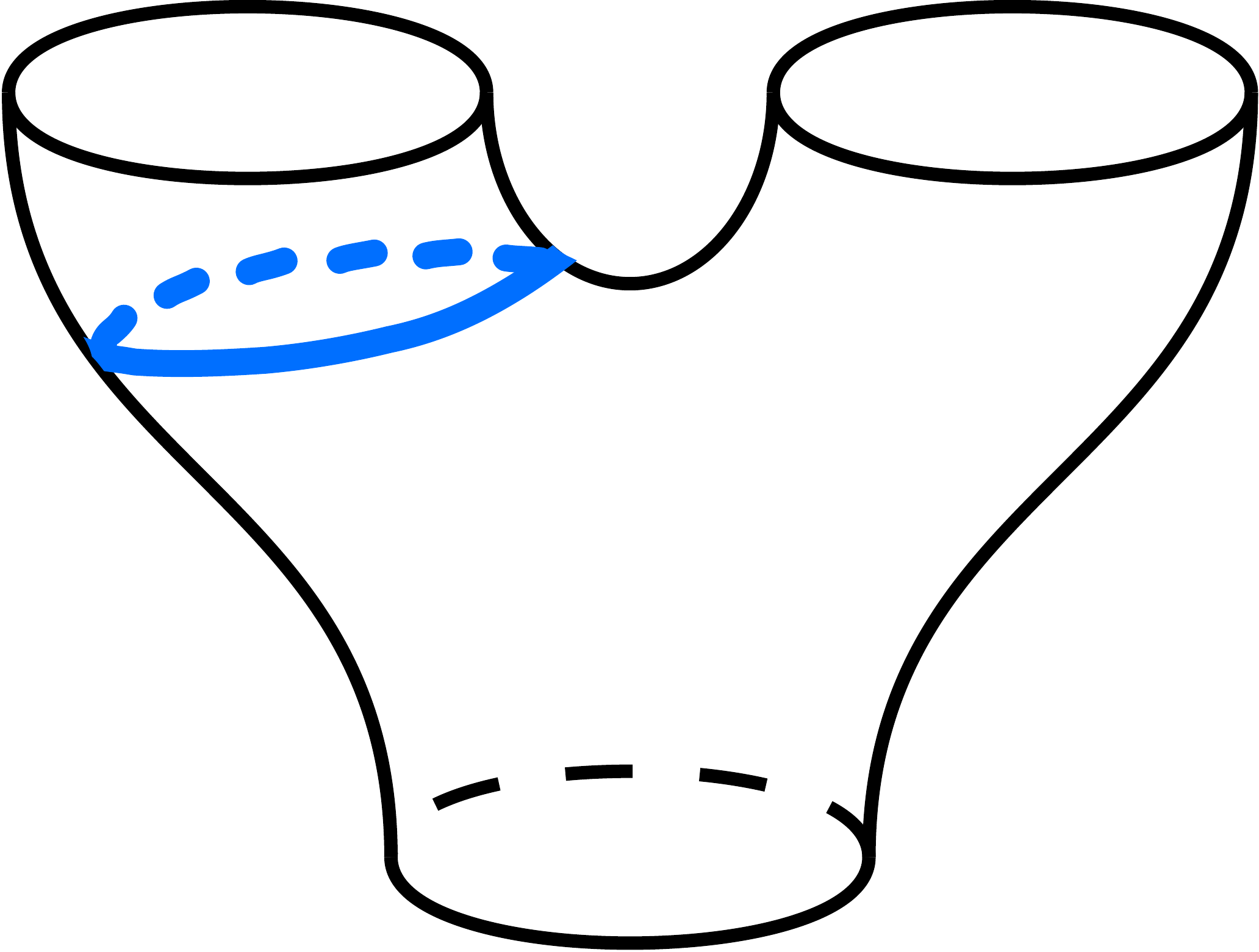}
	\end{aligned}
	\quad
	\xmapsto{-\, \cup A}
	\quad
	\begin{aligned}
	\includegraphics[scale=0.06]{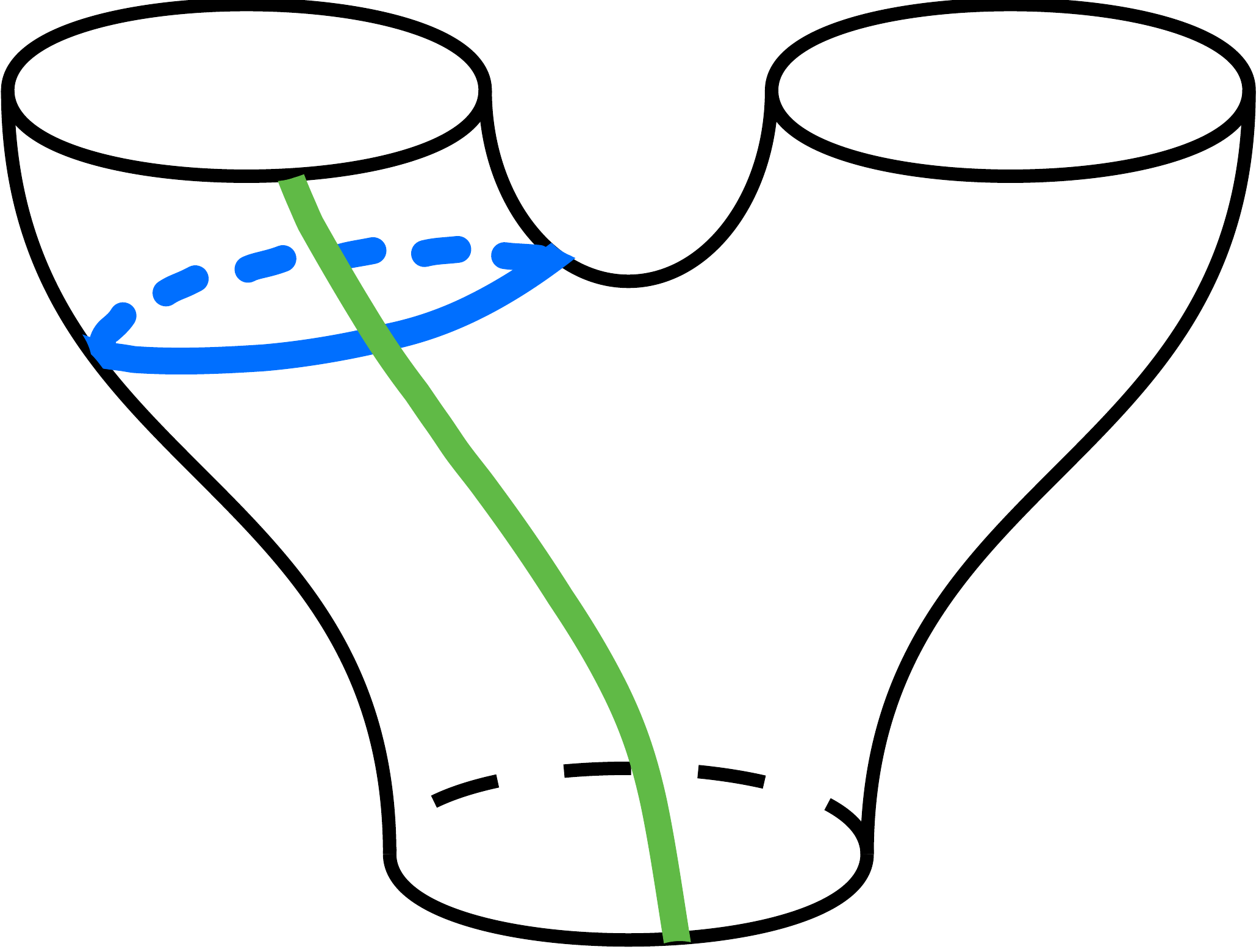}
	\end{aligned}
	\quad
	\xmapsto{\text{smooth crossings}}
	\quad
	\begin{aligned}
	\includegraphics[scale=0.06]{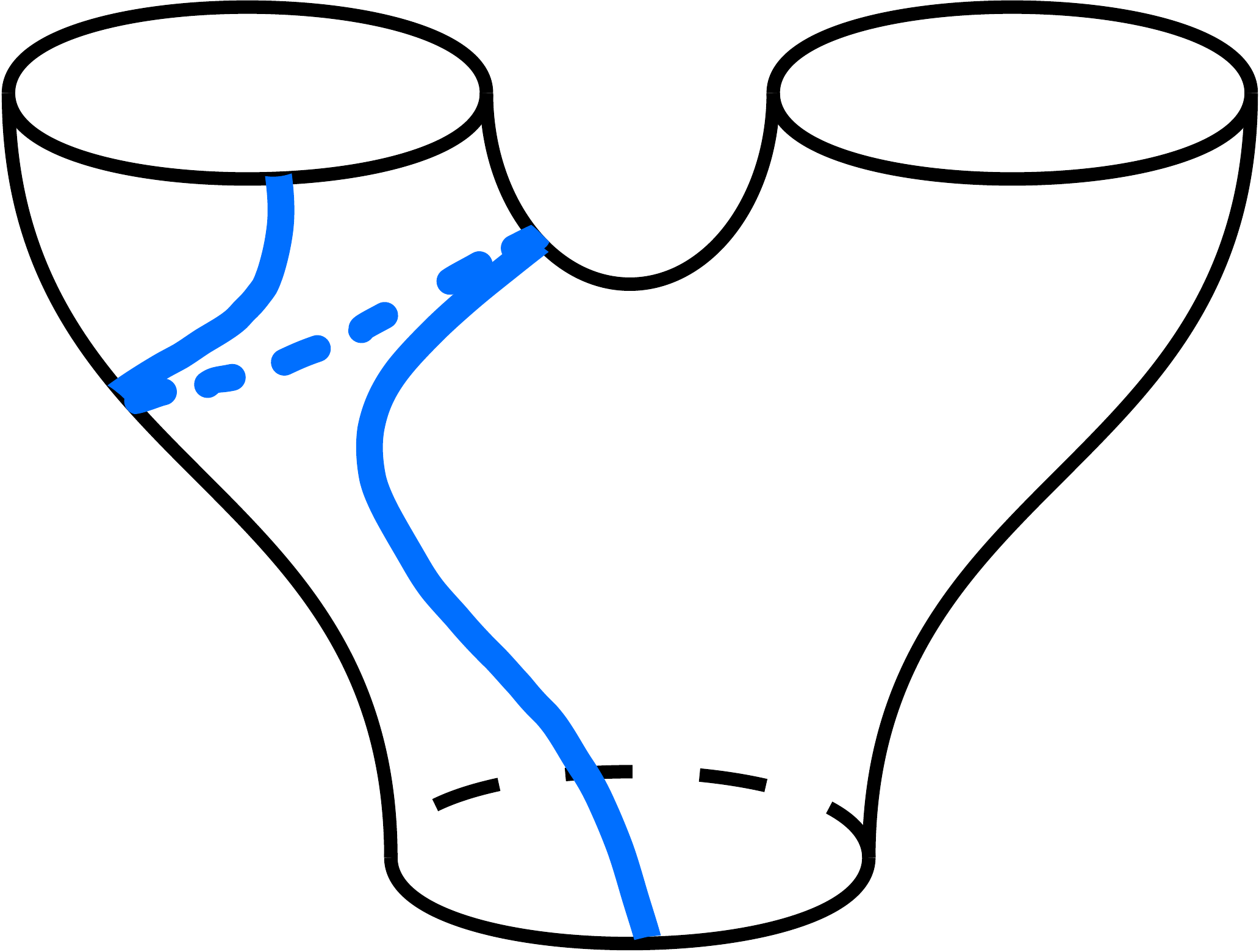}
	\end{aligned}
\end{equation}

\subsubsection{Dimensions for $B_0$ boundary conditions}

We work in stages, beginning with the surfaces underlying the 1-generators, and then generalize to $\Sigma_{g,n}$ via repeated applications of Theorem \ref{thm:gluing}.\\ 

\noindent  The surface underlying the cup and cap is $\Sigma_{0,1}$. In this case the string-net space is very simple; one easily sees that

\begin{equation}
H(\Sigma_{0,1};B_0) = \text{span}\left\{\tinycup\right\} \quad \text{and} \quad H(\Sigma_{0,1};B_1) = 0.
\end{equation}

\noindent Note that $\Dim \, H(\Sigma_{0,1};B_0) = 1 = 2^{2 \cdot 0 + 1 - 1}$. Next we proceed to the pants and copants; the surface underlying these bordisms is $\Sigma_{0,3}$. In this case there are only two (up to diffeomorphism of \sgn) admissible boundary conditions: $B = B_0 \boxtimes B_0 \boxtimes B_0$ and $B' = B_0 \boxtimes B_1 \boxtimes B_1$. In the first case, one can show that

\begin{equation} \label{eq:basis_3punc_sphere}
			H(\spants;B)
			\quad
			=
			\quad
			\text{span}\left\{\;
			\begin{aligned}
			\includegraphics[scale=0.06]{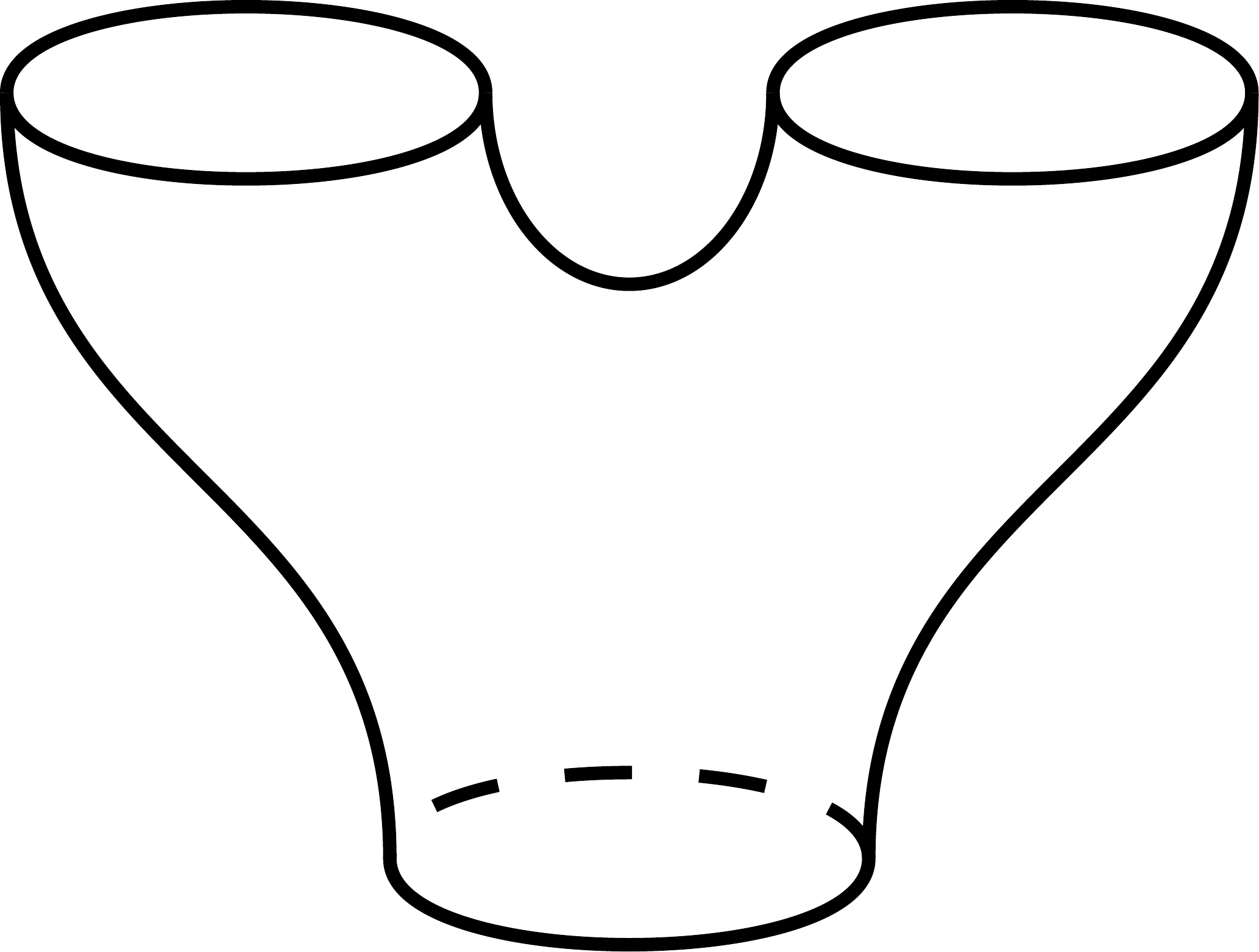}\;,\;
			\end{aligned}
			\begin{aligned}
			\includegraphics[scale=0.06]{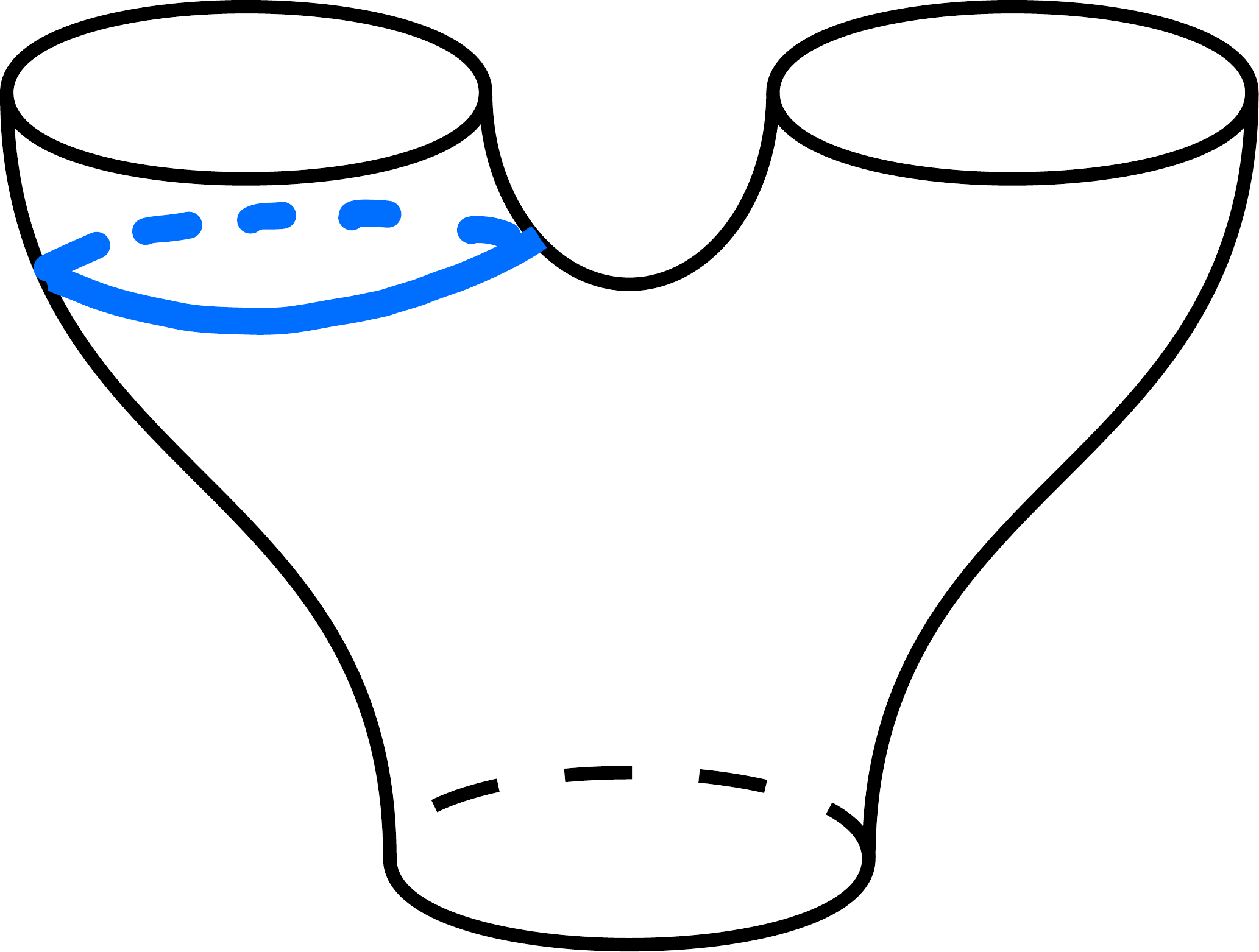}\;,\;
			\end{aligned}
			\begin{aligned}
			\includegraphics[scale=0.06]{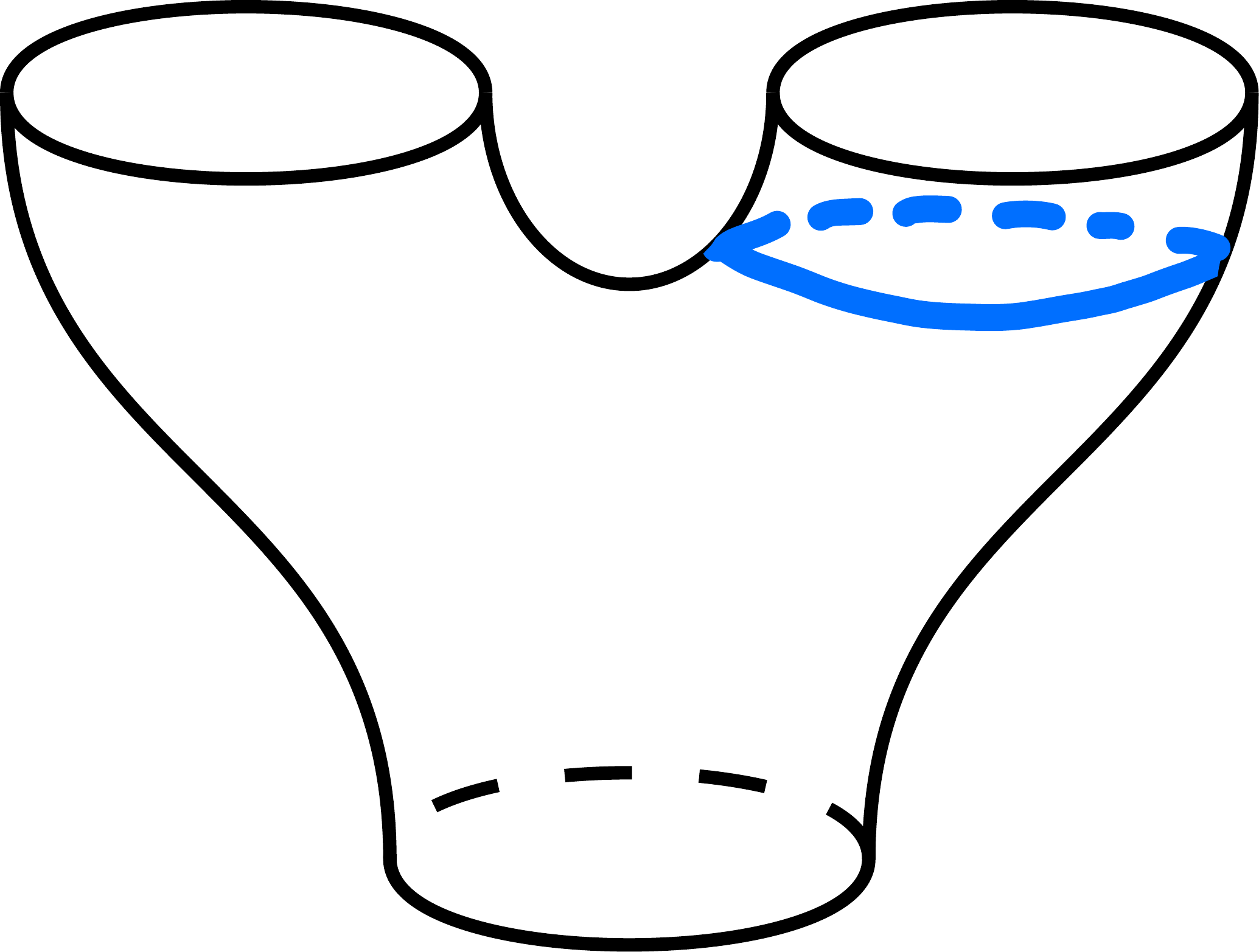}\;,\;
			\end{aligned}
			\begin{aligned}
			\includegraphics[scale=0.06]{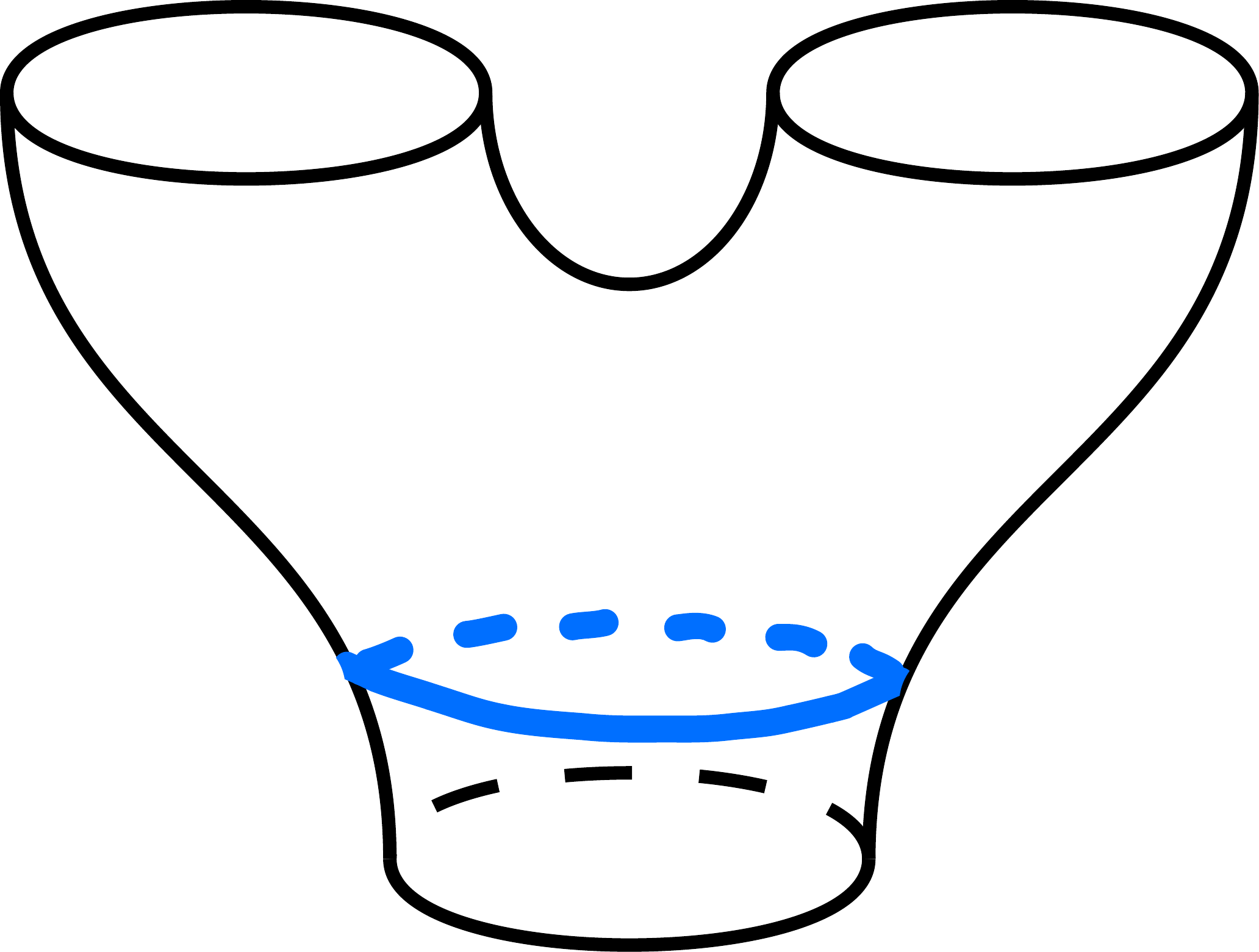}
			\end{aligned}
			\;\right\}
\end{equation}

\noindent In the second case, one likewise sees that

\begin{equation} \label{eq:basis_3punc_sphere_1}
			H(\spants;B')
			\quad
			=
			\quad
			\text{span}\left\{\;
			\begin{aligned}
			\includegraphics[scale=0.06]{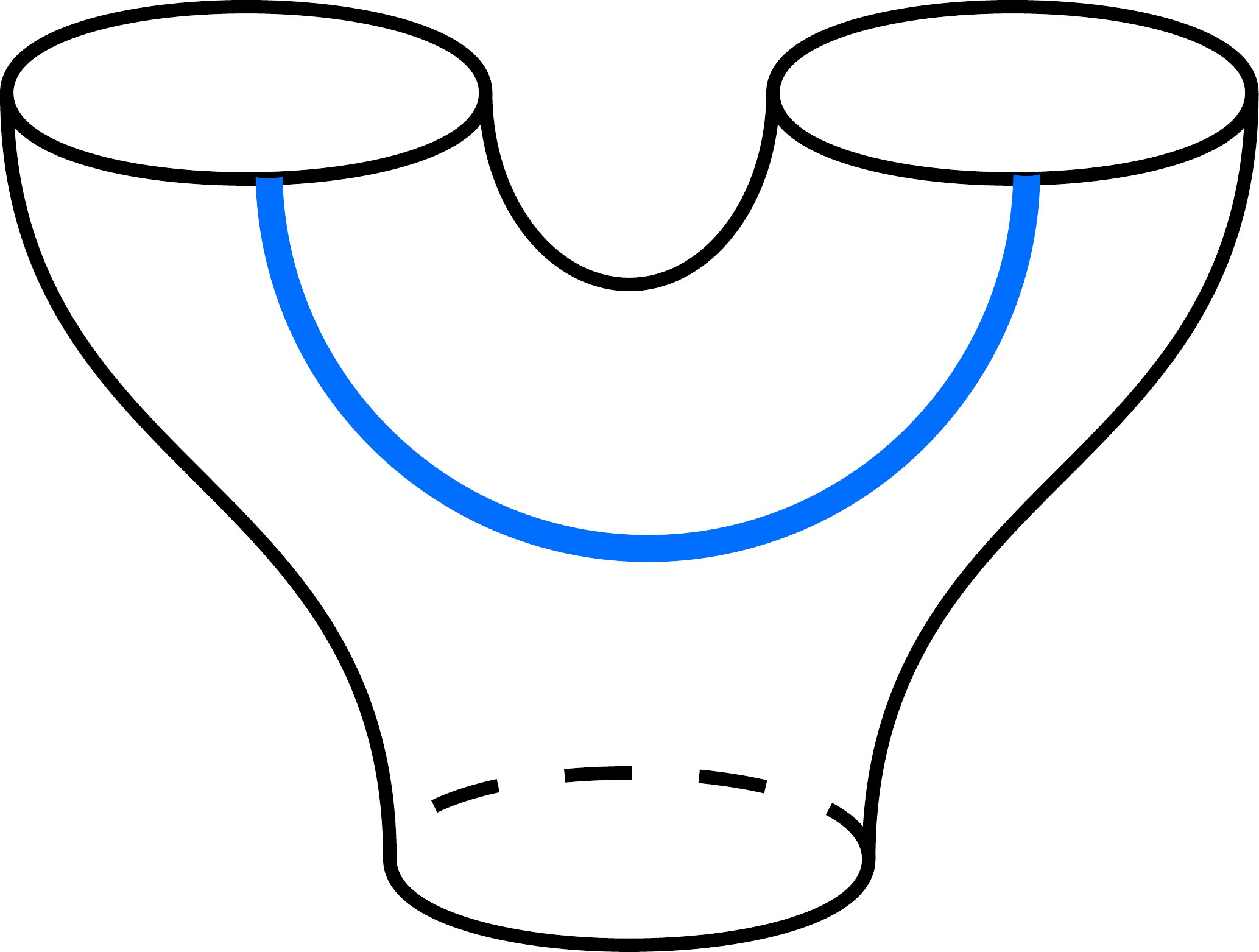}\;,\;
			\end{aligned}
			\begin{aligned}
			\includegraphics[scale=0.06]{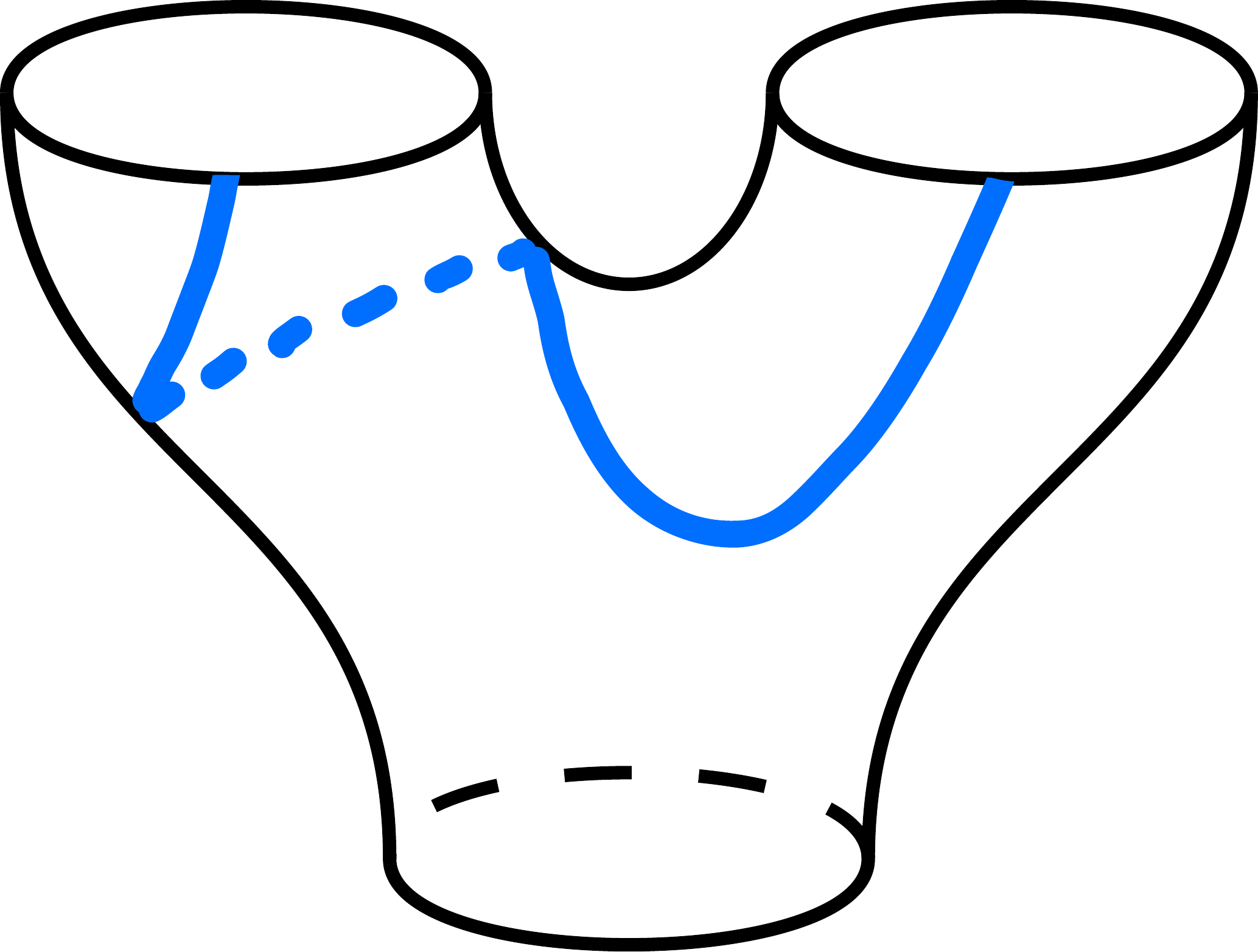}\;,\;
			\end{aligned}
			\begin{aligned}
			\includegraphics[scale=0.06]{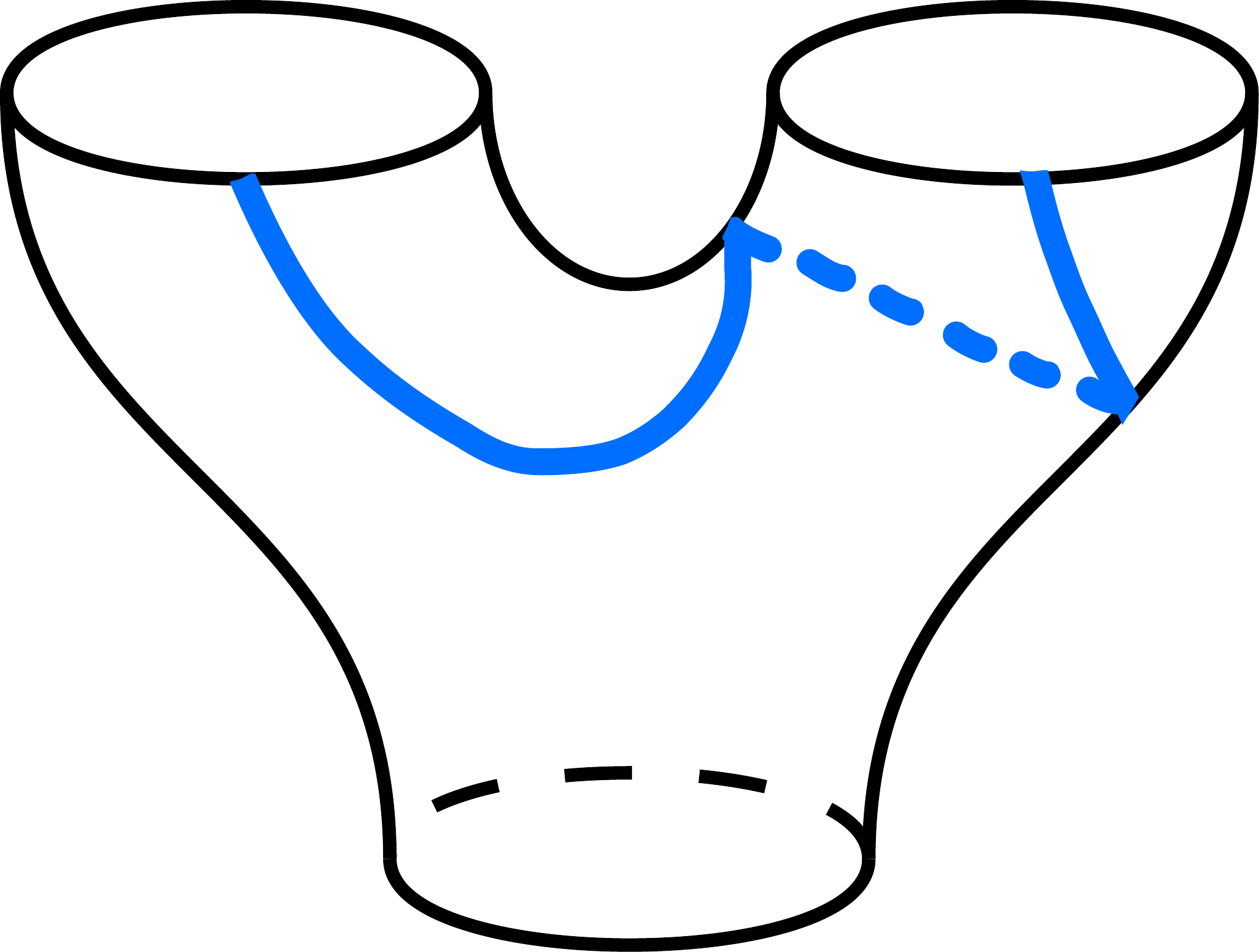}\;,\;
			\end{aligned}
			\begin{aligned}
			\includegraphics[scale=0.06]{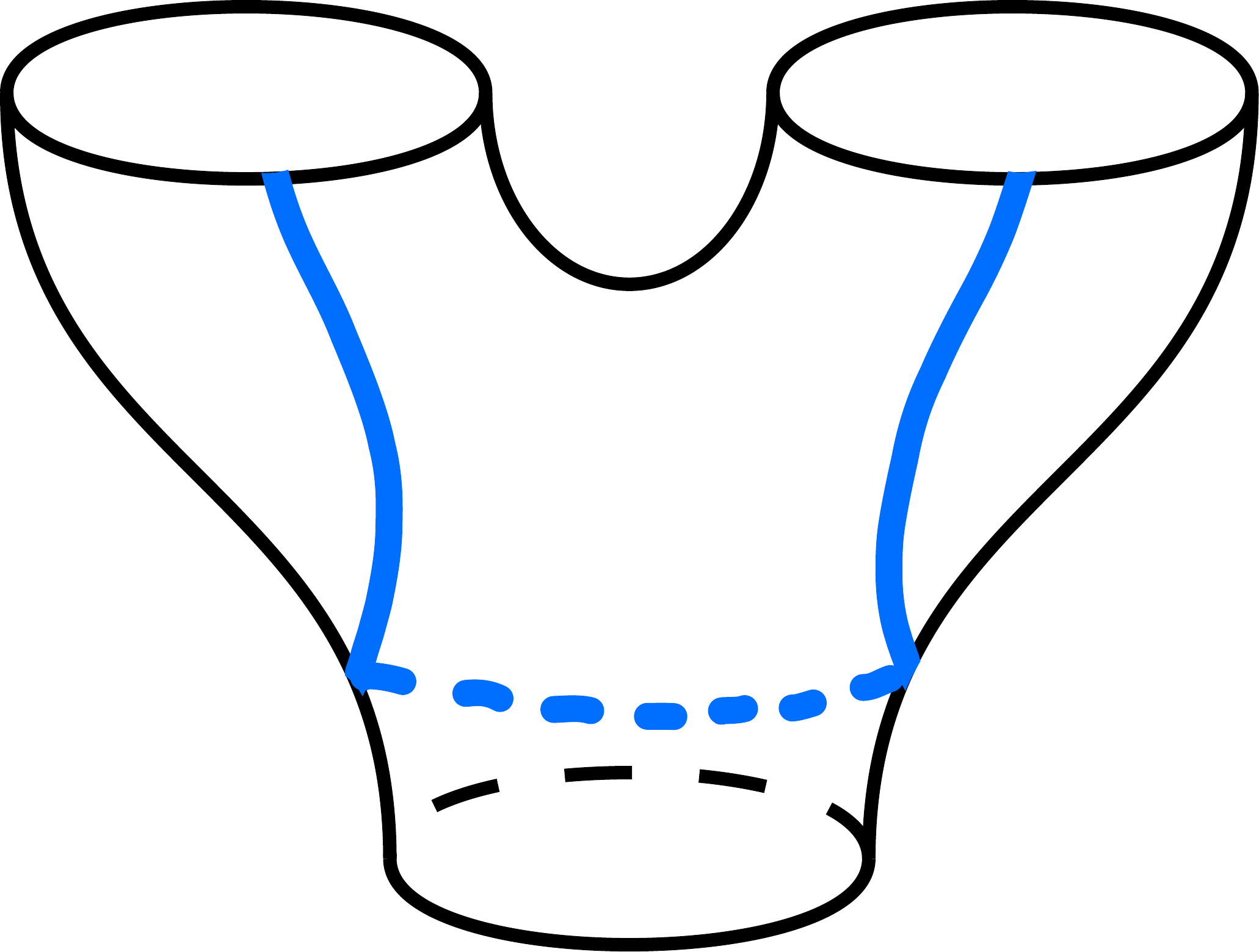}
			\end{aligned}
			\;\right\}
\end{equation}

\noindent In either case, therefore, $\Dim \, H(\spants) = 4 = 2^{2 \cdot 0 + 3 - 1}$.\\

\noindent For a composite surface we split the dimension calculation into two cases, that of $\Sigma_{g,0}$ and $\Sigma_{g,n}$ (for $n \geq 1$). These can be built by repeatedly gluing $\Sigma_{0,1}$ and $\Sigma_{0,3}$ together:
\scalecobordisms{0.4}
\[
\Sigma_{g,0} = \quad \begin{aligned}
	\begin{tz}[xscale=2.1, yscale=1.2]
\node (1) at (0,0)
{
$\begin{tikzpicture}
	\node[Cup, top] (A) at (0,0) {};
	\node[Copants, anchor=belt] (B) at ([yshift=5pt] A) {};
	\node[Pants, top, anchor=leftleg] (C) at (B.leftleg) {};
	\node (D) at ([yshift = 15pt] C.belt) {\vdots};	
	\node[Copants, anchor=belt] (E) at (D.north) {};
	\node[Pants, top, anchor=leftleg] (F) at (E.leftleg) {};
	\node[Cap] (G) at ([yshift=5pt] F.belt) {};
	\node[yscale=6.7, anchor=south] (M) at (0.6,-0.75) {$\}$};
	\node at (M.east) {$\, g$};
\end{tikzpicture}$
};
\end{tz}
\end{aligned}
\qquad
\Sigma_{g,n} = \quad \begin{aligned}
\begin{tz}[xscale=2.1, yscale=1.2]
\node (1) at (0,0)
{
$\begin{tikzpicture}
	\node[Cup, top] (A) at (0,0) {};
	\node[Copants, anchor=belt] (B) at ([yshift=5pt] A) {};
	\node[Pants, top, anchor=leftleg] (C) at (B.leftleg) {};
	\node (D) at ([yshift = 15pt] C.belt) {\vdots};	
	\node[Copants, anchor=belt] (E) at (D.north) {};
	\node[Pants, top, anchor=leftleg] (F) at (E.leftleg) {};
	\node[Copants, anchor=belt] (G) at ([yshift=5pt] F.belt) {};
	\node[Copants, top, anchor=belt] (H) at (G.leftleg) {};
	\node (I) at ([yshift=15pt, xshift=-10pt] H.leftleg) {$\ddots$};
	\node[Cobordism Top End] (J) at (G.rightleg) {};
	\node[Copants, top, anchor=belt] (K) at ([yshift=25pt, xshift=-20pt] H.leftleg) {};
	\node[yscale=6.6] (L) at (0.8,1.2) {$\}$};
	\node[anchor=west] at (L.east) {$\!\! g$};
	\node[yscale=5.5, anchor=south] (M) at (0.8,1.85) {$\!\! \}$};
	\node[anchor=west] at (M.east) {$\!\! n-1$};	
\end{tikzpicture}$
};
\end{tz}
\end{aligned}
\]

\noindent This decomposition allows us to use Theorem \ref{thm:gluing} to compute $\Dim \, H(\Sigma_{g,n},B)$.

\begin{lemma} In the extended toric code, the dimension of the string net space is
\[
 \Dim H(\sgn ; B) = \begin{cases} 2^{2g}, & n = 0 \\
 2^{2g+n-1}, & n \geq 1 
 \end{cases}
\]
\end{lemma}
\begin{proof}
\noindent Using Theorem \ref{thm:gluing}, one can show that

\begin{equation}
    H(\shandle;B_0^{\boxtimes 2})
    =
    \text{span}\left\{\;
    \begin{aligned}
	\includegraphics[scale=0.04]{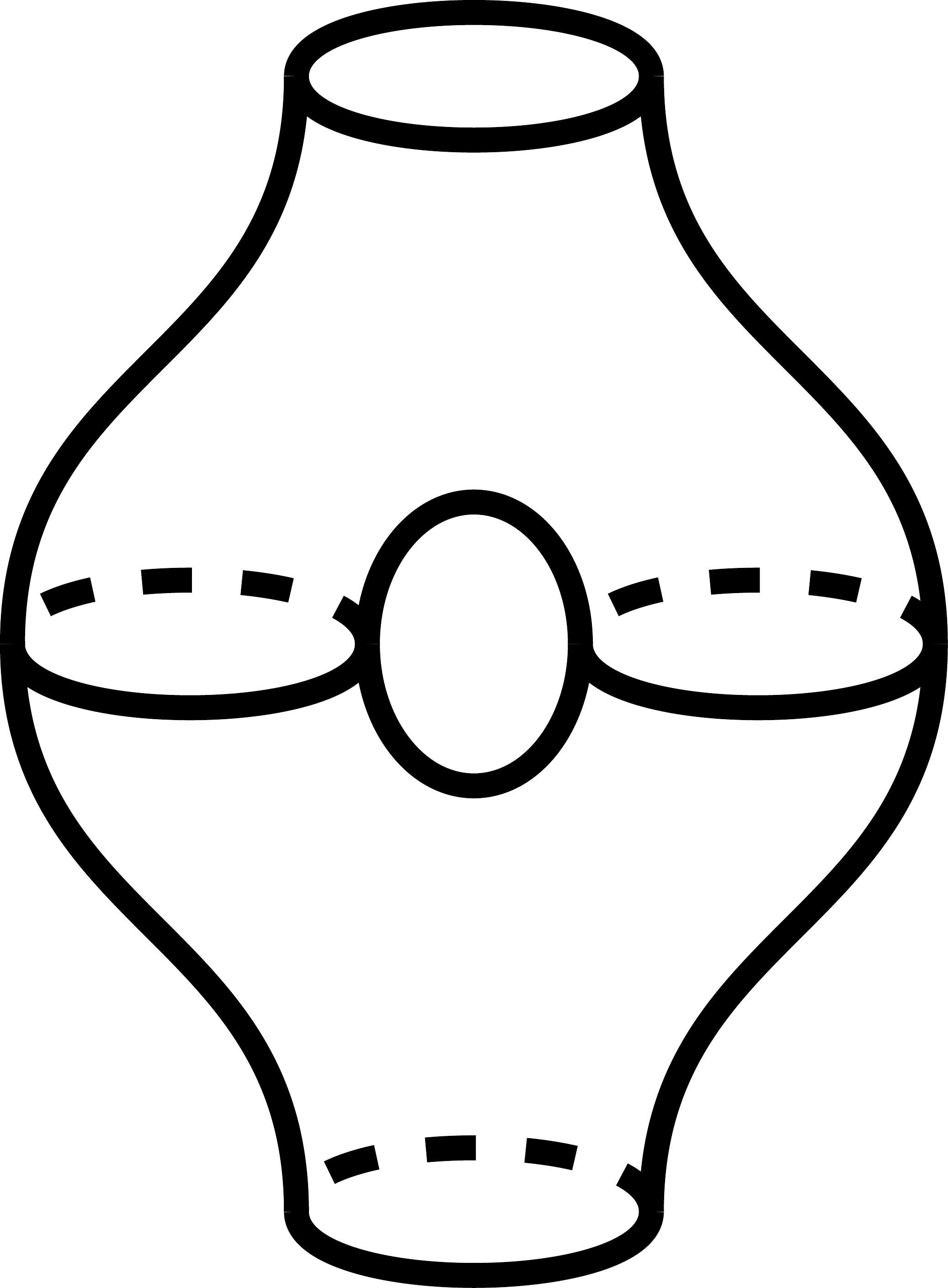}\;,\;
	\end{aligned}
	\begin{aligned}
	\includegraphics[scale=0.04]{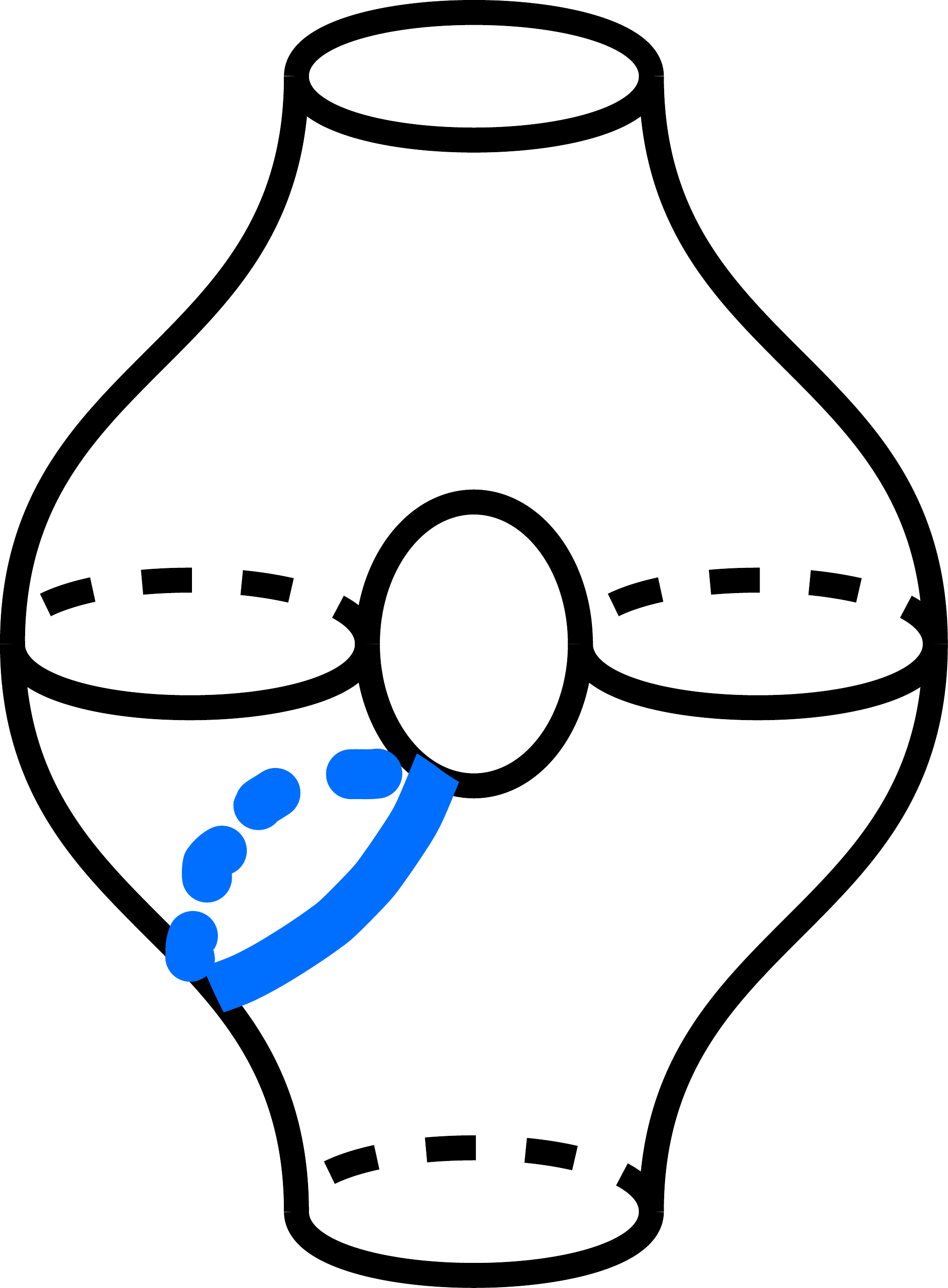}\;,\;
	\end{aligned}
	\begin{aligned}
	\includegraphics[scale=0.04]{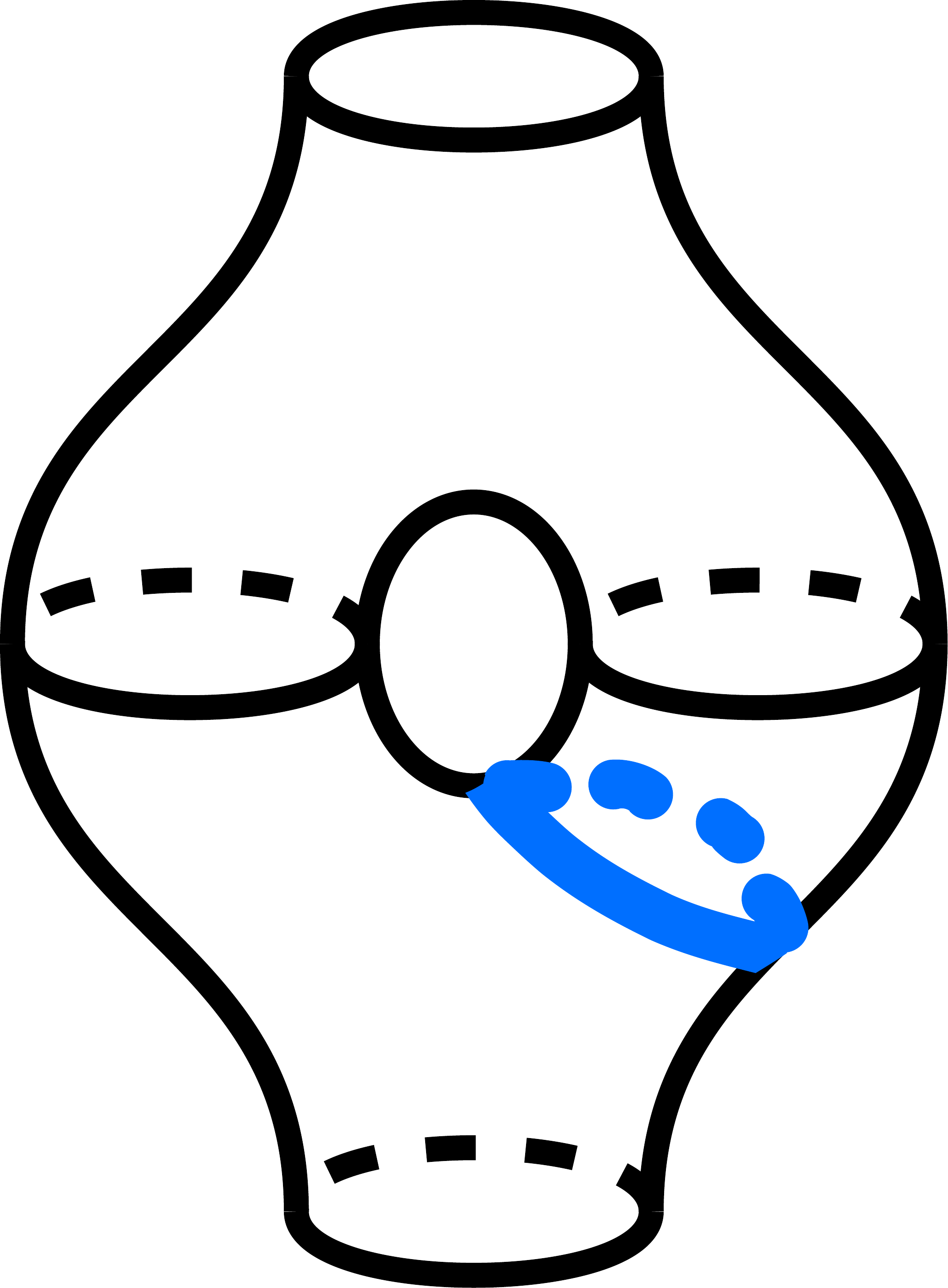}\;,\;
	\end{aligned}
	\begin{aligned}
	\includegraphics[scale=0.04]{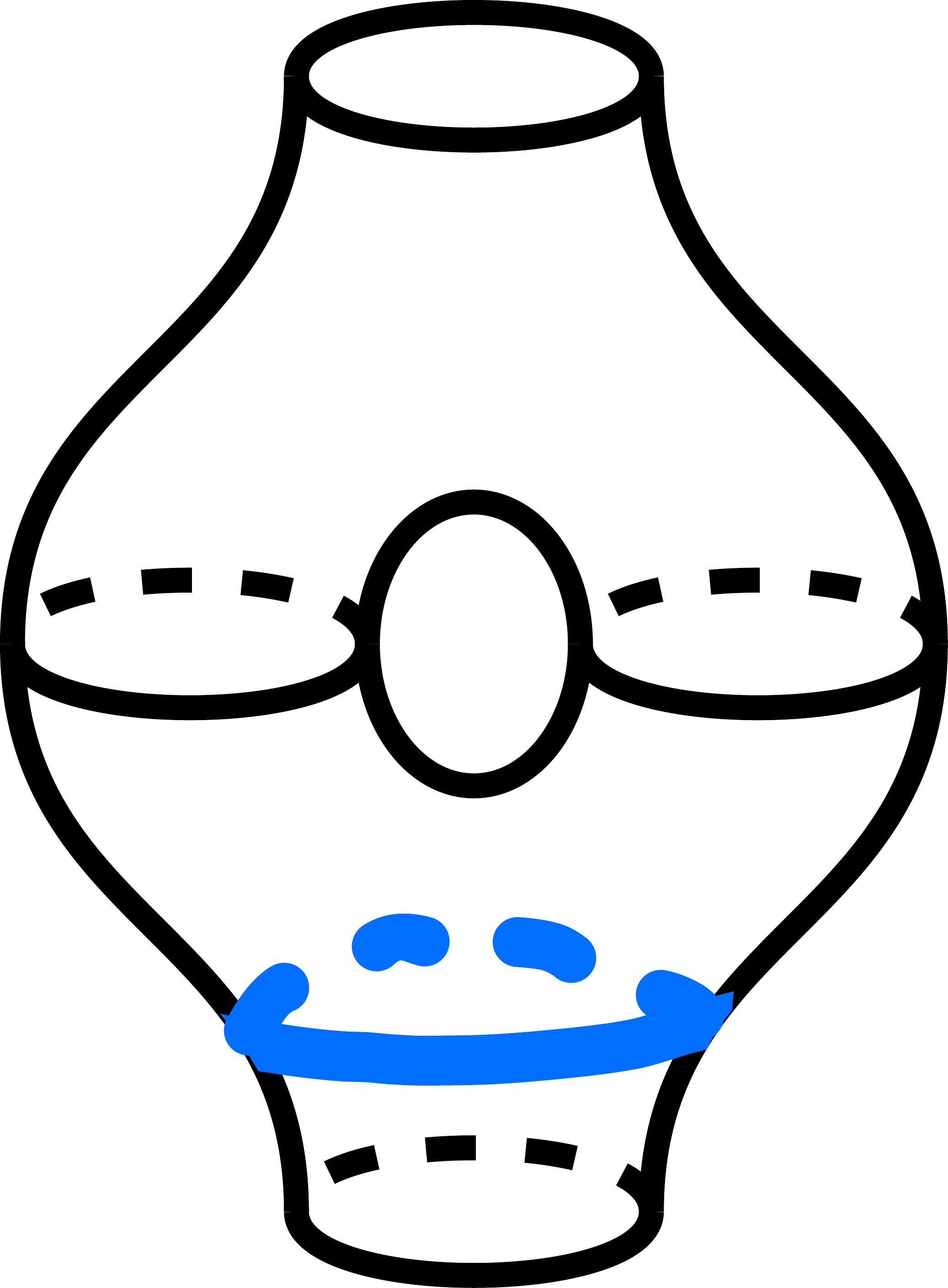}\;,\;
	\end{aligned}
	\begin{aligned}
	\includegraphics[scale=0.04]{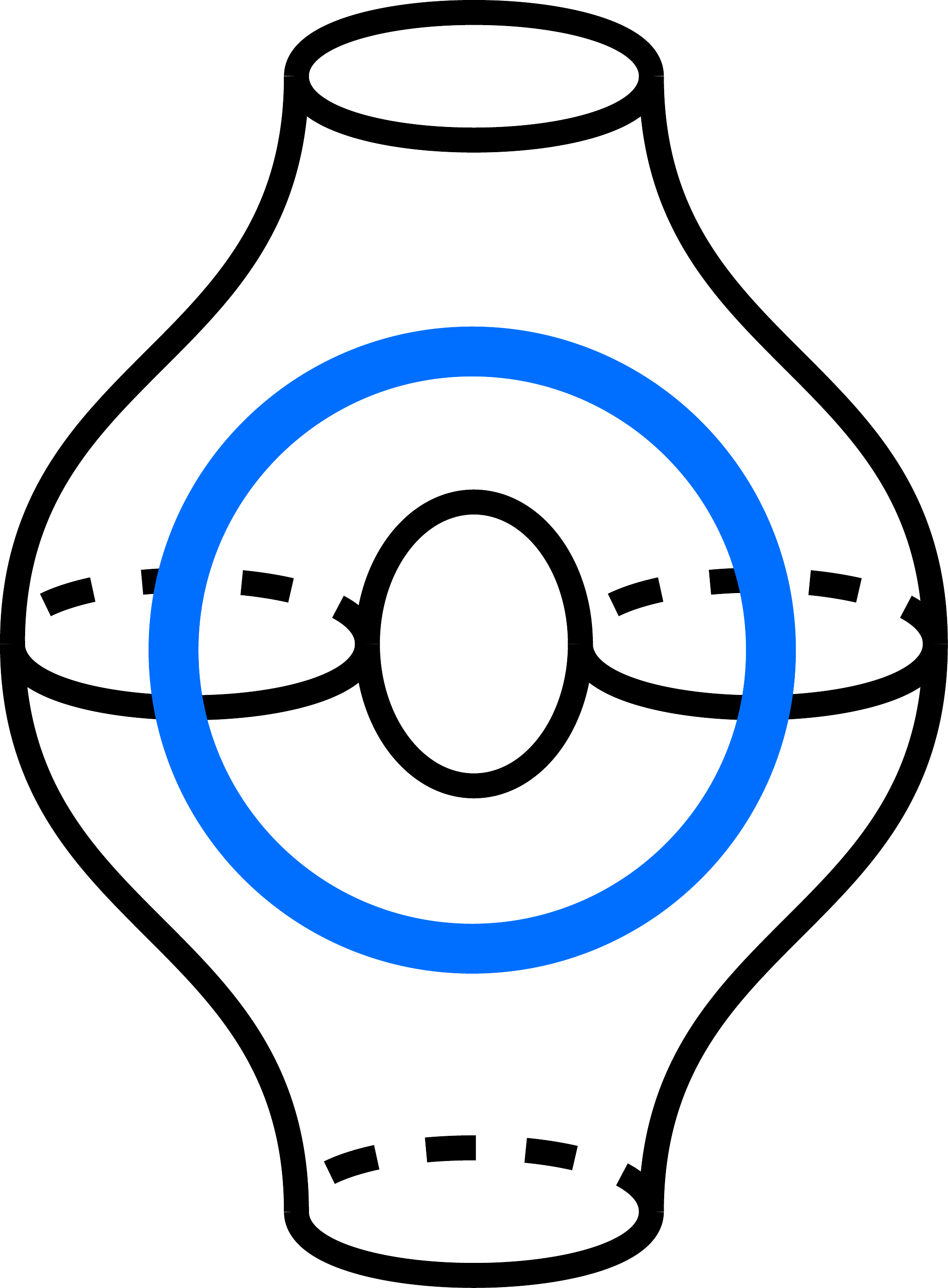}\;,\;
	\end{aligned}
	\begin{aligned}
	\includegraphics[scale=0.04]{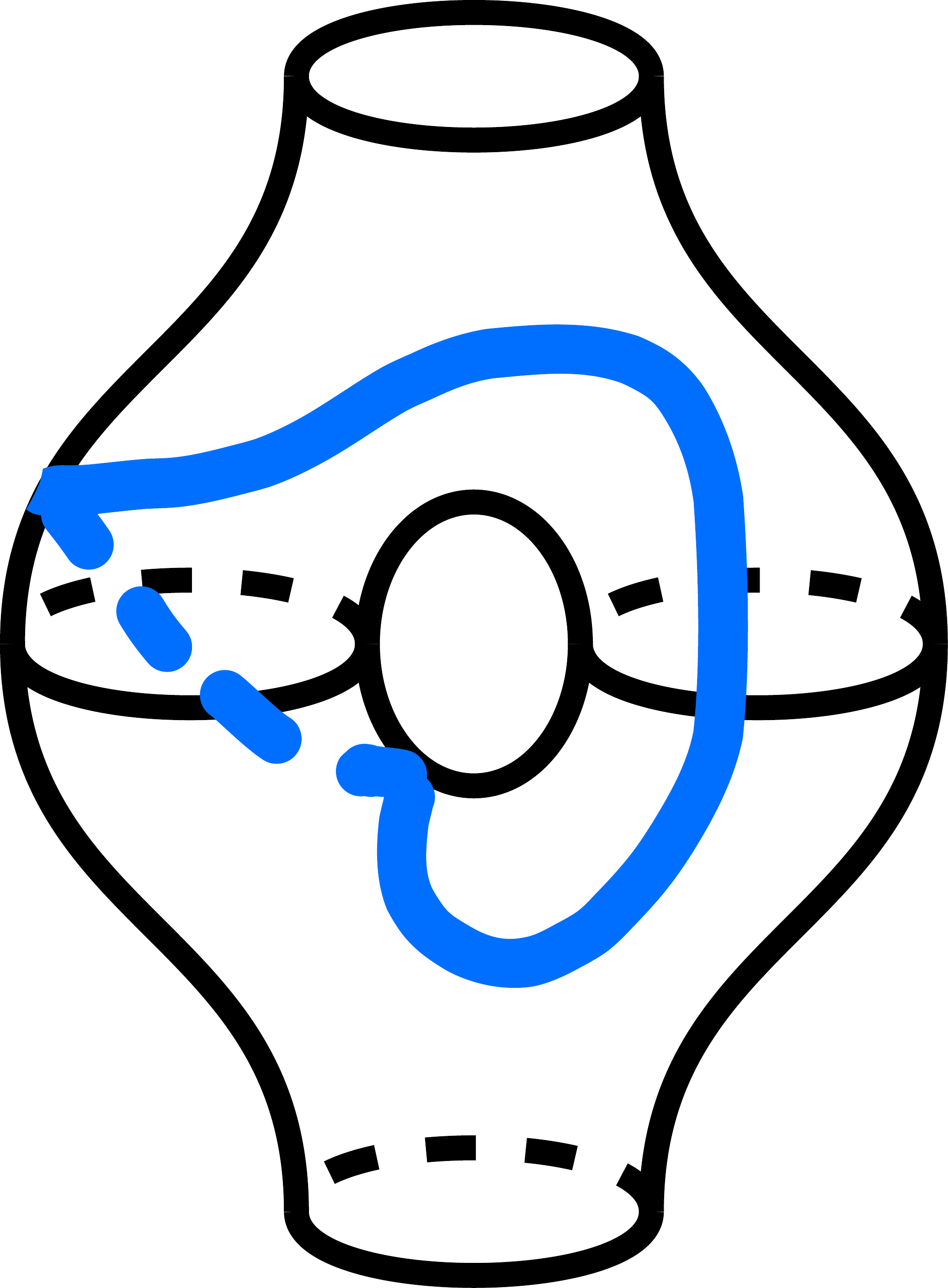}\;,\;
	\end{aligned}
	\begin{aligned}
	\includegraphics[scale=0.04]{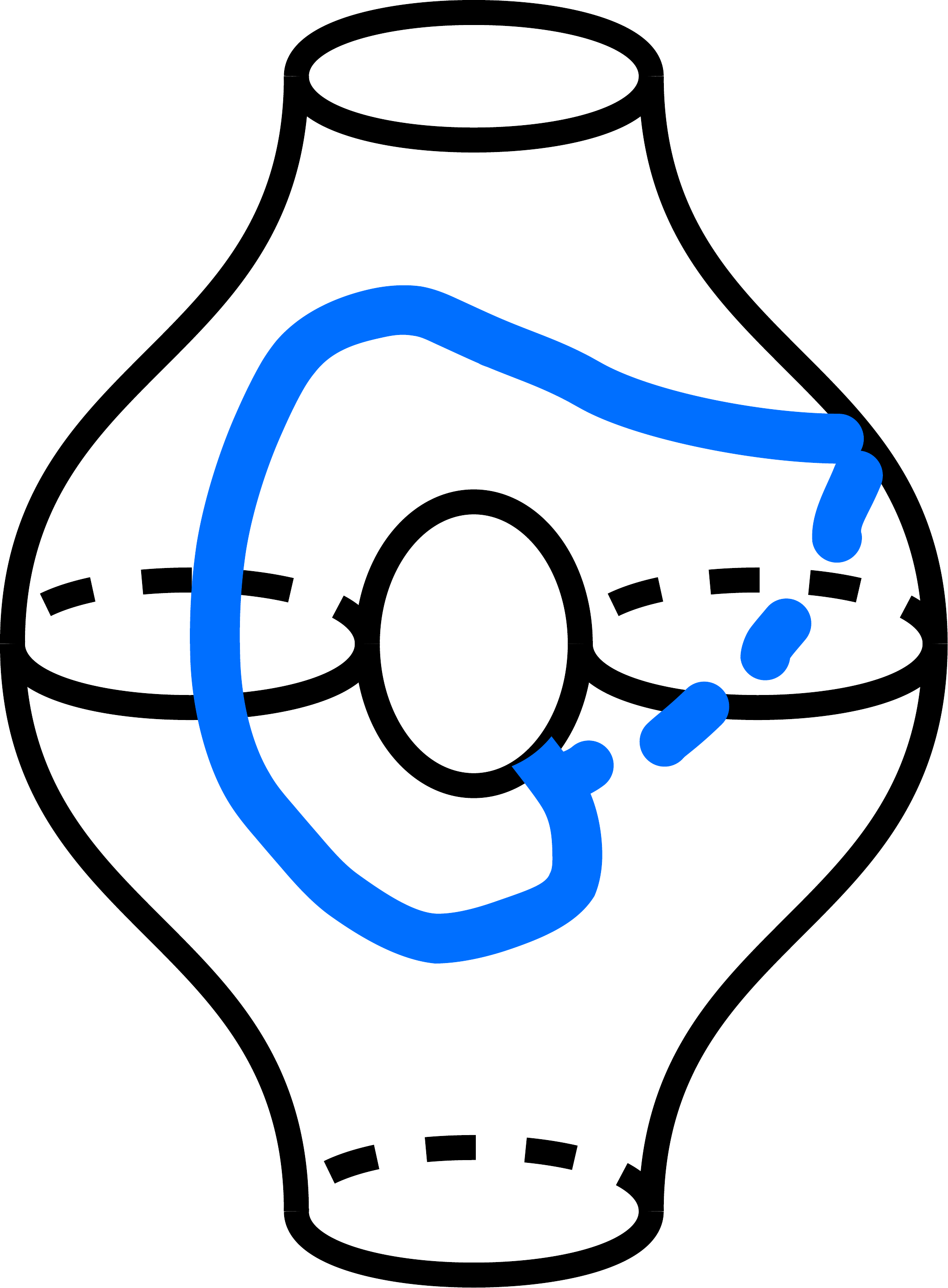}\;,\;
	\end{aligned}
	\begin{aligned}
	\includegraphics[scale=0.04]{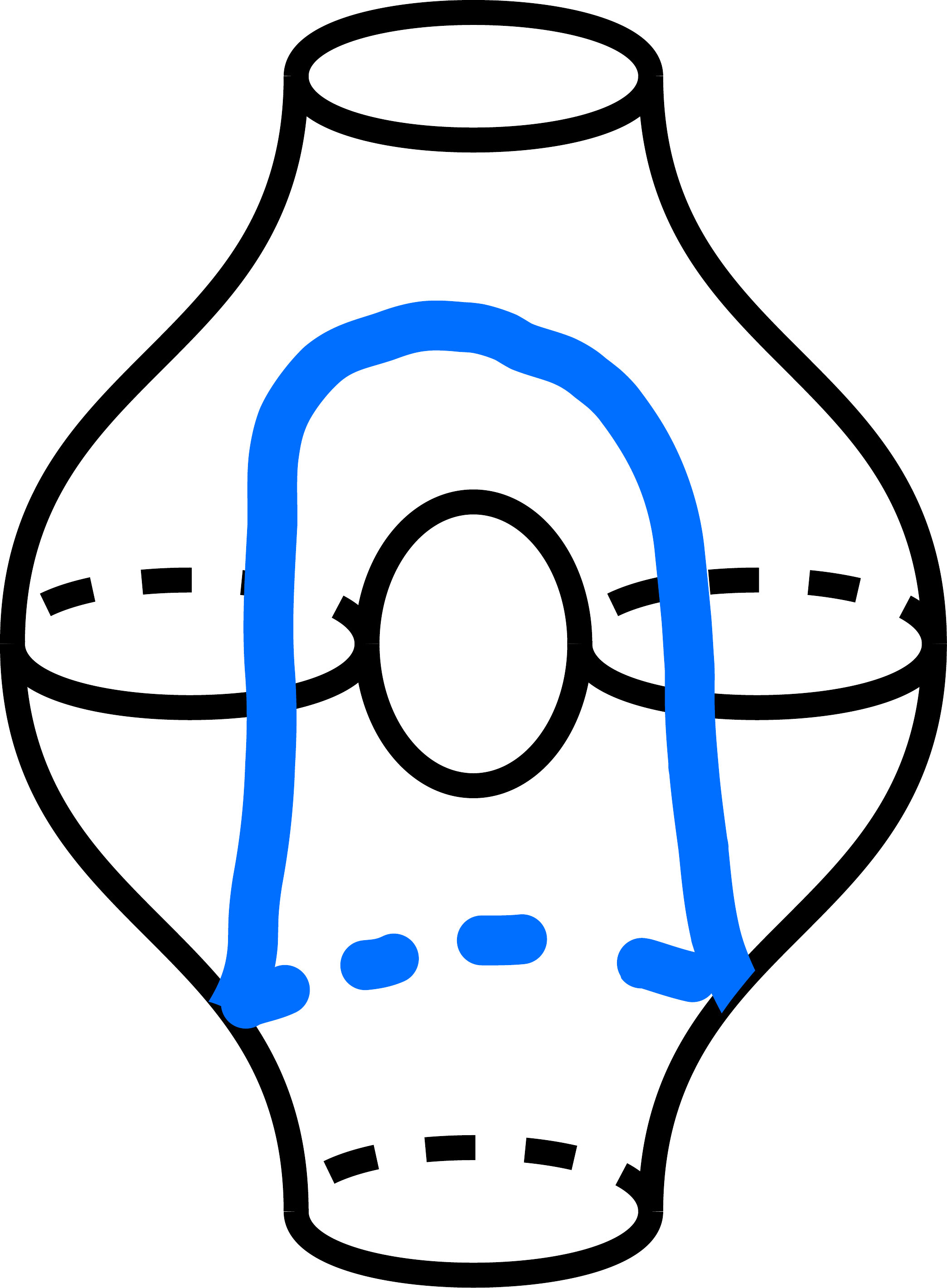}
	\end{aligned}\;\right\}
\end{equation}

\noindent So, for $n \geq 1$, we can compute the dimension of \sgn ~by the following procedure:

\begin{itemize}
	\item [(1)] Start with the cup (i.e. \scup), whose string-net space is 1-dimensional, as was shown above.
	\item [(2)] Glue on $g$ copies of \shandle. Each such iteration quadruples the dimension.
	\item [(3)] Glue on $(n-1)$ copies of \spants. Each such iteration doubles the dimension.
\end{itemize}
We conclude that $\Dim H(\Sigma_{g, n}; B) = 2^{2g + n - 1}$ for $n \geq 1$, independent of the boundary condition $B$.
In the case n = 0 we leave out step (3) and glue on a cap (i.e. $\Sigma_{0,1}$) to close the hole. This gluing leaves the dimension unchanged. 
\end{proof}

\subsection{Mapping class group action}

Every 1-2-3 TQFT $Z$ gives rise to a system of mapping class group representations
\[
	  \rho_{\Sigma_{g,n}} : \Gamma(\Sigma_{g,n}) \rightarrow \Aut(Z(\Sigma_{g,n}))
\]
where $\Sigma_{g,n}$ is a closed surface of genus $g$ with $n$ boundary circles, and $\Gamma(\Sigma_{g,n})$ is its mapping class group. For simplicity we focus here on closed surfaces ($n=0$).

In contrast to other approaches of constructing TQFTs (surgery on links, or state-sum models), the string-net model gives a very straightforward geometric way to describe this mapping class group representation. The string-net must simply be {\em pushed forward along the diffeomorphism}. 

However, there is a caveat. In this paper we have not actually defined the string-net TQFT geometrically, but rather combinatorially via generators-and-relations (Sections \ref{invertible-2-gen} and \ref{noninvertible-2-gen}). We therefore need the following lemma.

\begin{lemma} \label{action_of_mcg_lemma} The assignment in Sections \ref{invertible-2-gen} and \ref{noninvertible-2-gen} of linear maps to generating 2-morphisms of $\Bord$, when evaluated on cobordisms arising as the mapping cylinders of diffeomorphisms, is equal to the push-forward map along the diffeomorphism.
\end{lemma}
\begin{proof} It suffices to check this on a system of generators for the mapping class group $\Gamma(\Sigma_g)$. Figure \ref{mcg-generators} shows the Humphries generators for $\Gamma(\Sigma_g)$ as Dehn twists $a_i$ about certain closed curves $c_i$. We observe:
\begin{itemize}
	\item The Dehn twists $a_0$ and $a_1$ correspond to the generating 2-morphism $\theta$, which indeed operates on string-nets via the push-forward map (see \eqref{theta_map_action}).
	
	\item The Dehn twists $a_3, a_5, \ldots, a_{2g-1}$ correspond to the following composite of generating 2-morphisms (see \cite[Defn 21]{BDSPV2014}):
	\begin{equation} \label{action_of_II}
		\includegraphics{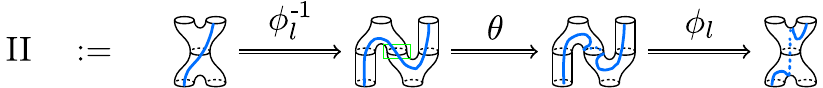}
	\end{equation}
	In equation \eqref{action_of_II}, we have also indicated the action of the associated composite of linear maps on a representative string-net basis vector. We see that the composite $\II$ indeed operates as the push-forward along one of the Dehn-twists $a_3, a_5, \ldots, a_{2g-1}$. (This is also true for the action of $\II$ on the other string-net basis vectors, as the reader will check).

	\item The Dehn twists $a_2, a_4, \ldots, a_{2g}$ correspond to the following composite of generating 2-morphisms (see \cite[Defn 21]{BDSPV2014}):
	\begin{equation} \label{action_of_A}
			\includegraphics{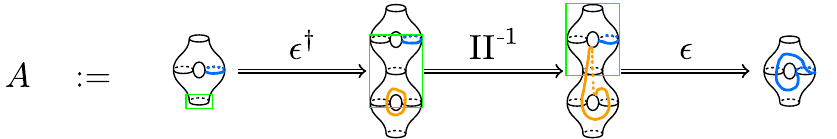}
		\end{equation}
	In equation \eqref{action_of_A}, we have also indicated the action of the associated composite of linear maps on a representative string-net basis vector. We see that the composite $A$ indeed operates as the push-forward along one of the Dehn-twists $a_3, a_5, \ldots, a_{2g-1}$. (This is also true for the action of $A$ on the other string-net basis vectors, as the reader will check).

\end{itemize}
\end{proof}
\begin{figure}
\begin{center}
 \includegraphics{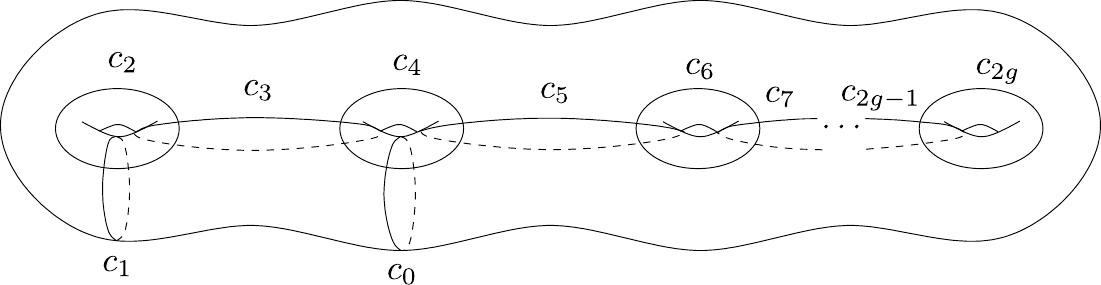}
\end{center}

\caption{\label{mcg-generators} The Humphries generators for $\Gamma(\Sigma_g)$. Taken from \cite{farb2011primer}.}
\end{figure}
Combining this with Remark \ref{homology_remark} gives:
\begin{corollary} \label{push-forward-map-lem} On closed surfaces, the action of the mapping class group on the toric code string-net space is equal to the push-forward on homology:
\[
 Z(\gamma) = \gamma_{*} : \mathbb{C}[H_1(\Sigma_g, \mathbb{Z}/2\mathbb{Z})] \rightarrow \mathbb{C}[H_1(\Sigma_g, \mathbb{Z}/2\mathbb{Z})], \quad \gamma \in \Gamma(\Sigma_g)
\]
\end{corollary}

\subsection{Invariants of lens spaces}
Given relatively prime integers $p, q \in \mathbb{Z}$, the lens space $L(p,q)$ is the 3-dimensional manifold obtained by gluing two solid tori $T$ and $T'$ together along a diffeomorphism $h : \partial T \rightarrow \partial T'$ whose action on integral homology is given by
\[
 h_*(\alpha) = q \alpha' + p \beta'
\]
where $\alpha$ (resp. $\alpha'$) is the meridian of $T$ (resp. $T'$) and $\beta'$ is the longitude of $T'$ (see \cite{prasolov1997knots}).

Now, the toric code is known to correspond to the finite group model TQFT where $G = \mathbb{Z}/2\mathbb{Z}$. And, the invariant of a closed 3-manifold $M$ in the finite group model is simply a weighted sum of the principal $G$-bundles on $M$,
\[
 Z(M) = \frac{|\Hom(\pi_1 M, G)|}{|G|} \, .
\]
Since $\pi_1( L(p,q)) = \mathbb{Z}/p\mathbb{Z}$, we conclude that from the framework of counting principal bundles,
\begin{equation} \label{count_prin_bundles}
 Z(L(p,q)) = \begin{cases} 1 & \text{ if $p$ is even} \\
  \frac{1}{2} & \text{ if $p$ is odd} 
  \end{cases}
\end{equation}
We find it instructive here to rederive this result from our string-net framework. 

We can create $T$ in terms of our generating 2-morphisms as the composite $\epsilon^\dagger \circ \nu$, and $T'$ as the reversed composite $\nu^\dagger \circ \epsilon$. Therefore, $L(p,q)$ can be expressed in terms of the generators as

\begin{equation}
\label{lens_eq}
L(p,q) =
\begin{tz}
\draw[green] (0,0) rectangle (\cobwidth, \cobwidth);
\end{tz}
\Rarrow{\nu}
\begin{tz}
    \node[Cap] (A) at (0,0) {};
    \node[Cup] (B) at (0,0) {};
    \selectpart[green, inner sep=1pt] {(A-center)};
\end{tz}
\Rarrow{\epsilon ^\dagger}
\begin{tz}
    \node[Cap] (A) at (0,0) {};
    \node[Pants, anchor=belt] (B) at (A) {};
    \node[Copants, anchor=leftleg] (C) at (B.leftleg) {};
    \node[Cup] (D) at (C.belt) {};
    \selectpart[green] {(A-center) (B-leftleg) (B-rightleg) (C-belt)};
\end{tz}
\Rarrow{\hat{h}}
\begin{tz}
    \node[Cap, bot=false] (A) at (0,0) {};
    \node[Pants, anchor=belt] (B) at (A) {};
    \node[Copants, anchor=leftleg] (C) at (B.leftleg) {};
    \node[Cup] (D) at (C.belt) {};
    \node (E) [Cobordism Bottom End 3D] at (0,0) {};    
    \selectpart[green] {(E) (B-leftleg) (B-rightleg) (C-belt)};
\end{tz}
\Rarrow{\epsilon}
\begin{tz}
    \node[Cap] (A) at (0,0) {};
    \node[Cup] (B) at (0,0) {};
    \selectpart[green, inner sep=1pt] {(A-center)};
\end{tz}
\Rarrow{\nu ^\dagger}
\begin{tz}
\draw[green] (0,0) rectangle (\cobwidth, \cobwidth);
\end{tz}
\end{equation}
Here, $\hat{h}$ is the composite of the generating 2-morphisms which represents the diffeomorphism $h$. The utility of the string-net framework is that to evaluate $Z(\hat{h})$ it is not necessary to compute the composite $\hat{h}$: we can simply push-forward the string-net along the diffeomorphism $h$, using Lemma \ref{action_of_mcg_lemma}. (This is in contrast to a surgery-on-a-link description of the TQFT, where $\hat{h}$ must be calculated. This is not difficult but it does involve a continued fraction expansion, see \cite[Section 3]{jeffrey1992chern}). 

Note that when we construct the solid torus $T$ as $\epsilon^\dagger \circ \nu$, the meridian $\alpha$ and longitude $\beta$ of $T$ are given by
\[
	\includegraphics{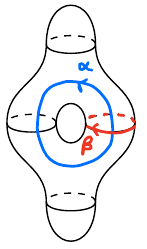}
\]
because $\alpha$ is the loop which contracts in $T$, not $\beta$. Therefore, evaluating \eqref{lens_eq} in the string-net model, up to the application of $Z(\hat{h})$, gives
\scalecobordisms{0.5}
\begin{align}
Z(L(p,q)) & = \quad
\begin{aligned}
\begin{tikzpicture}
\draw[green] (0,0) rectangle (0.5,0.5);
\end{tikzpicture}
\end{aligned}
\quad
\xmapsto{Z(\nu)} 
\quad
\frac{1}{2}\,
\begin{aligned}
\includegraphics[scale=0.02]{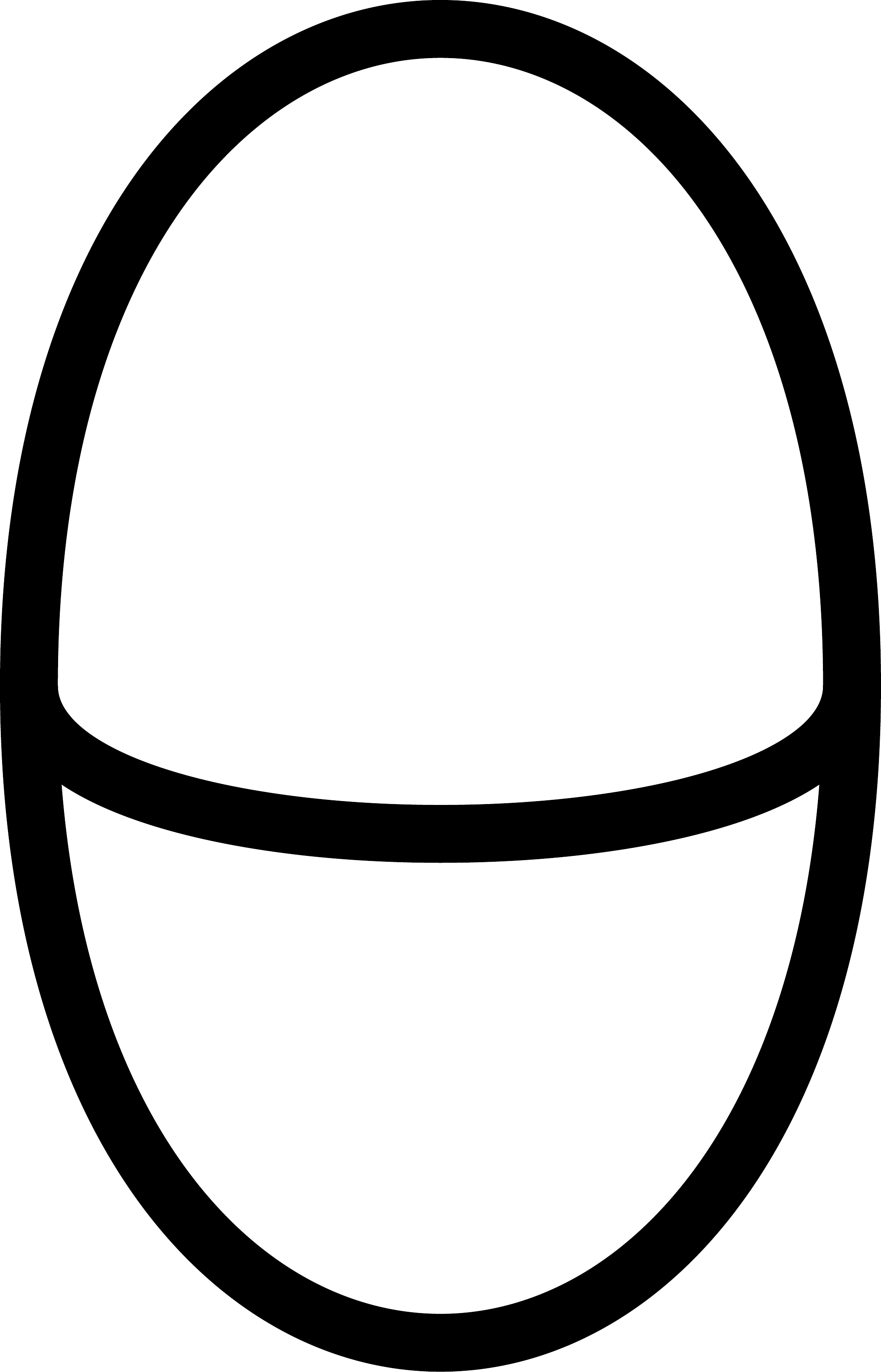}
\end{aligned}
\quad
\xmapsto{Z(\epsilon ^\dagger)}
\quad
\frac{1}{2} \, \left(
\begin{aligned}
\includegraphics[scale=0.04]{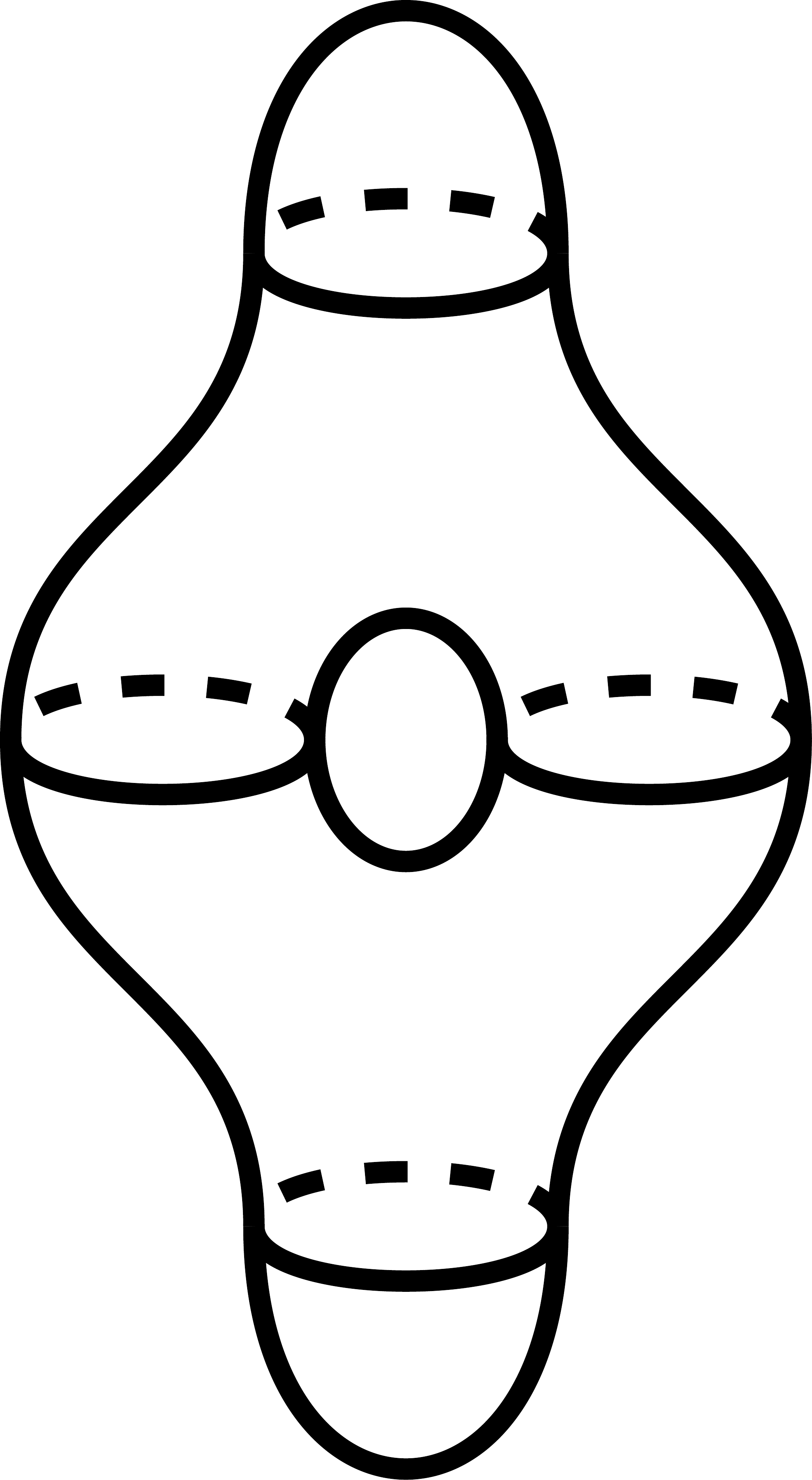}
\end{aligned}
 \, + \,
\begin{aligned}
\includegraphics[scale=0.04]{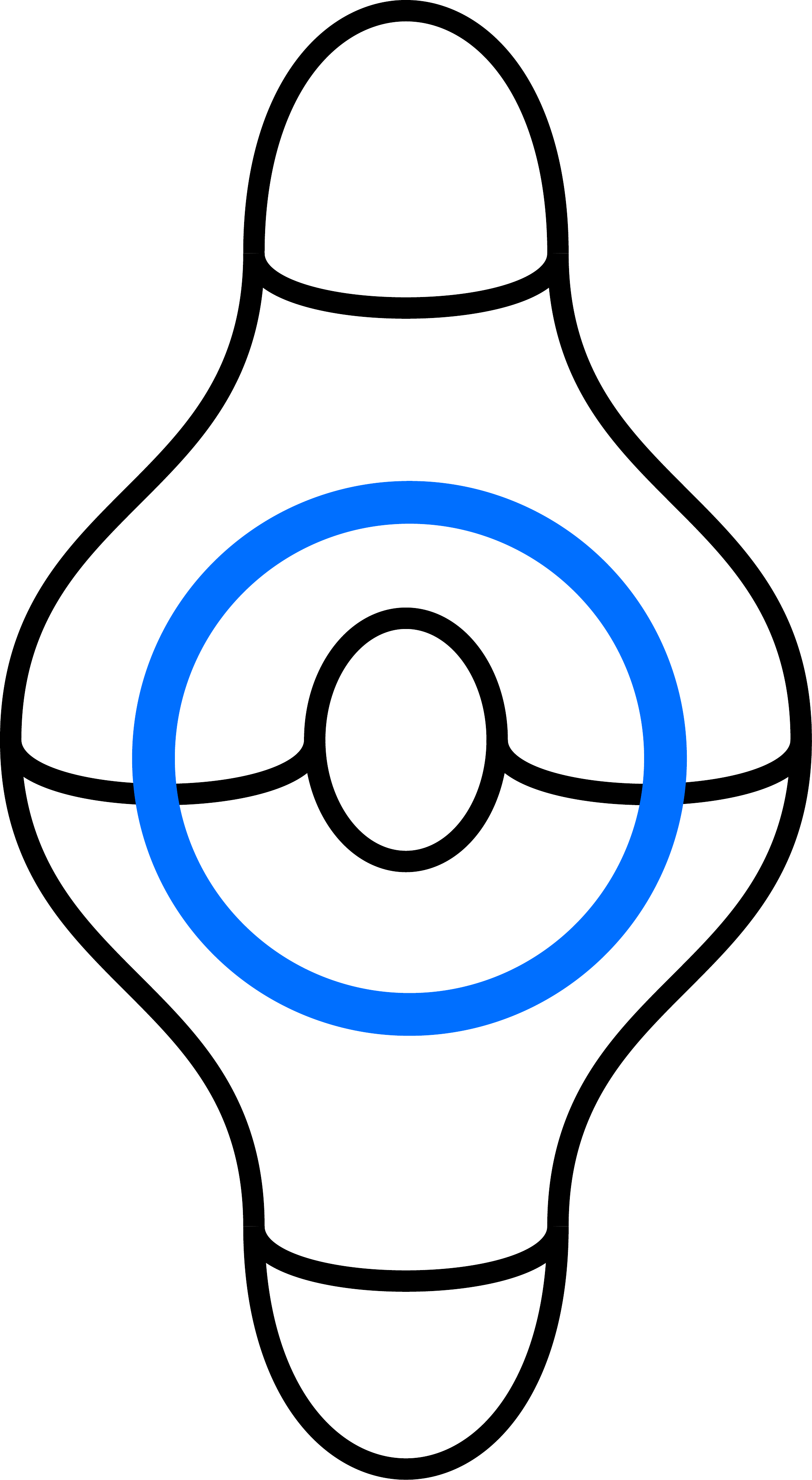}
\end{aligned}
\right) \\
& \xmapsto{Z(\hat{h})}
\, \left(
\begin{aligned}
\includegraphics[scale=0.04]{meridianblank.pdf}
\end{aligned}
 \, + \,
 \begin{aligned}
\includegraphics[scale=0.04]{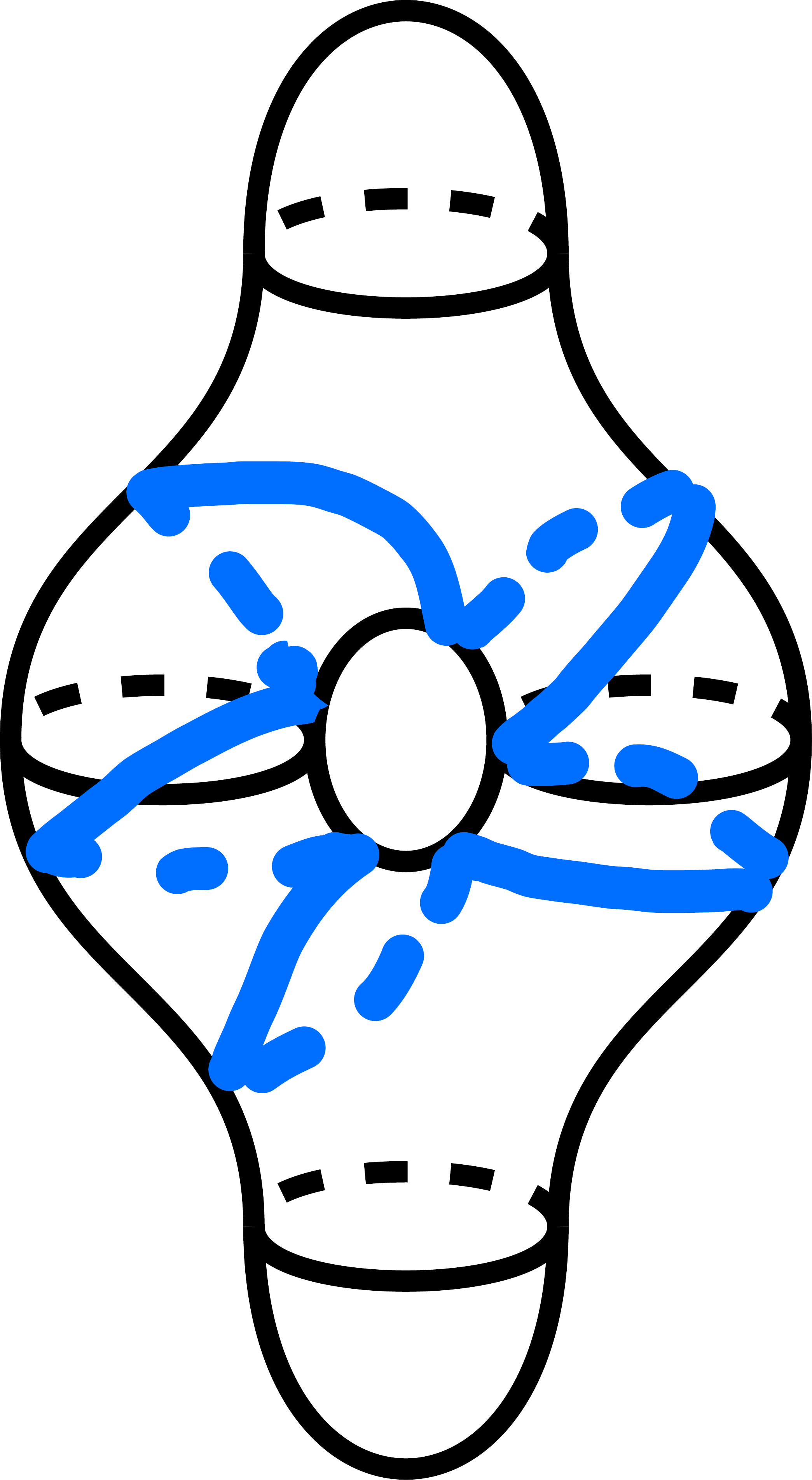}
\end{aligned}
\right)
\end{align}
where the symbolic string-net in the second term wraps $p$ times around the longitude of $T'$. Hence, when we apply $Z(\epsilon)$ to this term, $p$ strands will need to be cut. If $p$ is odd, this will result in $0$ as in \eqref{epsilon_dag_example_map}, while if $p$ is even, this will result in a collection of loops via the F-move, which collectively evaluate to $1$ via loop contraction. Hence, continuing the calculation,
\begin{align*}
 Z(L(p,q)) &= \cdots \\
 &\xmapsto{Z(\epsilon)} \, \frac{1}{2} \left(
\begin{aligned}
\includegraphics[scale=0.02]{meridiansphere.pdf}
\end{aligned}
\, + \, \delta_{p, \text{even}} \,
\begin{aligned}
\includegraphics[scale=0.02]{meridiansphere.pdf}
\end{aligned}
\right) \,\,
 \xmapsto{Z(\nu^\dagger)} 
\, \frac{1}{2} \left( \,
 \begin{tz}
\draw[green] (0,0) rectangle (0.5, 0.5);
\end{tz}
\, + \, \delta_{p, \text{even}} \,
 \begin{tz}
\draw[green] (0,0) rectangle (0.5, 0.5);
\end{tz}
\, \right) \\
 &= \begin{cases} 1 \text{ if $p$ is even} \\
  \frac{1}{2} \text{ if $p$ is odd} 
  \end{cases} 
\end{align*}
reproducing \eqref{count_prin_bundles}.

\bibliographystyle{amsplain}
\bibliography{mybibliography}

\end{document}